\documentclass
[numbers=enddot,12pt,final,onecolumn,notitlepage,abstracton]{scrartcl}%
\usepackage[headsepline,footsepline,manualmark]{scrlayer-scrpage}
\usepackage[all,cmtip]{xy}
\usepackage{amsfonts}
\usepackage{amssymb}
\usepackage{framed}
\usepackage{amsmath}
\usepackage{comment}
\usepackage{color}
\usepackage[breaklinks]{hyperref}
\usepackage[sc]{mathpazo}
\usepackage[T1]{fontenc}
\usepackage{amsthm}
\usepackage{ytableau}
\usepackage{needspace}
\usepackage{allrunes}
\providecommand{\U}[1]{\protect\rule{.1in}{.1in}}
\theoremstyle{definition}
\newtheorem{theo}{Theorem}[section]
\newenvironment{theorem}[1][]
{\begin{theo}[#1]\begin{leftbar}}
{\end{leftbar}\end{theo}}
\newtheorem{lem}[theo]{Lemma}
\newenvironment{lemma}[1][]
{\begin{lem}[#1]\begin{leftbar}}
{\end{leftbar}\end{lem}}
\newtheorem{prop}[theo]{Proposition}
\newenvironment{proposition}[1][]
{\begin{prop}[#1]\begin{leftbar}}
{\end{leftbar}\end{prop}}
\newtheorem{defi}[theo]{Definition}
\newenvironment{definition}[1][]
{\begin{defi}[#1]\begin{leftbar}}
{\end{leftbar}\end{defi}}
\newtheorem{remk}[theo]{Remark}
\newenvironment{remark}[1][]
{\begin{remk}[#1]\begin{leftbar}}
{\end{leftbar}\end{remk}}
\newtheorem{coro}[theo]{Corollary}
\newenvironment{corollary}[1][]
{\begin{coro}[#1]\begin{leftbar}}
{\end{leftbar}\end{coro}}
\newtheorem{conv}[theo]{Convention}

\newtheorem{quest}[theo]{Question}
\newenvironment{question}[1][]
{\begin{quest}[#1]\begin{leftbar}}
{\end{leftbar}\end{quest}}
\newtheorem{warn}[theo]{Warning}

\newtheorem{conj}[theo]{Conjecture}

\newenvironment{statement}{\begin{quote}}{\end{quote}}

\newenvironment{verlong}{}{}
\newenvironment{vershort}{}{}
\newenvironment{noncompile}{}{}
\excludecomment{verlong}
\includecomment{vershort}
\excludecomment{noncompile}

\newcommand{\tvi}{\left. \textarm{\tvimadur} \right.}
\newcommand{\bel}{\left. \textarm{\belgthor} \right.}

\let\sumnonlimits\sum
\let\prodnonlimits\prod
\renewcommand{\sum}{\sumnonlimits\limits}
\renewcommand{\prod}{\prodnonlimits\limits}
\ihead{Dual immaculate creation operators}
\ohead{August 31, 2026}
\begin{document}

\title{Dual immaculate creation operators and a dendriform algebra structure on the
quasisymmetric functions}
\author{Darij Grinberg}
\date{version 10.0, August 31, 2026 (corrected and updated)}
\maketitle

\begin{abstract}
The dual immaculate functions are a basis of the ring $\operatorname*{QSym}$
of quasisymmetric functions, and form one of the most natural analogues of the
Schur functions. The dual immaculate function corresponding to a composition
is a weighted generating function for immaculate tableaux in the same way as a
Schur function is for semistandard Young tableaux; an \textquotedblleft
immaculate tableau\textquotedblright\ is defined similarly to a semistandard
Young tableau, but the shape is a composition rather than a partition, and
only the first column is required to strictly increase (whereas the other
columns can be arbitrary; but each row has to weakly increase). Dual
immaculate functions have been introduced by Berg, Bergeron, Saliola, Serrano
and Zabrocki in arXiv:1208.5191, and have since been found to possess numerous
nontrivial properties.

In this note, we prove a conjecture of Mike Zabrocki which provides an
alternative construction for the dual immaculate functions in terms of certain
"vertex operators". The proof uses a dendriform structure on the ring
$\operatorname*{QSym}$; we discuss the relation of this structure to known
dendriform structures on the combinatorial Hopf algebras
$\operatorname*{FQSym}$ and $\operatorname*{WQSym}$.

\end{abstract}

\section{Introduction}

\begin{noncompile}
The main purpose of this note is to prove a conjecture by Mike Zabrocki on
dual immaculate quasi-symmetric functions.
\end{noncompile}

\begin{vershort}
The three most well-known combinatorial Hopf algebras that are defined over
any commutative ring $\mathbf{k}$ are the Hopf algebra of symmetric functions
(denoted $\operatorname*{Sym}$), the Hopf algebra of quasisymmetric functions
(denoted $\operatorname*{QSym}$), and that of noncommutative symmetric
functions (denoted $\operatorname*{NSym}$). The first of these three has been
studied for several decades, while the latter two are newer; we refer to
\cite[Chapters 4 and 6]{HGK} and \cite[Chapters 2 and 5]{Reiner} for
expositions of them\footnote{Historically, the origin of the noncommutative
symmetric functions is in \cite{NCSF1}, whereas the quasisymmetric functions
have been introduced in \cite{Gessel}. See also \cite[Section 7.19]%
{Stanley-EC2} specifically for the quasisymmetric functions and their
enumerative applications (although the Hopf algebra structure does not appear
in this source).}. All three of these Hopf algebras are known to carry
multiple algebraic structures, and have several bases of combinatorial and
algebraic significance. The Schur functions -- forming a basis of
$\operatorname*{Sym}$ -- are probably the most important of these bases; a
natural question is thus to seek similar bases for $\operatorname*{QSym}$ and
$\operatorname*{NSym}$.
\end{vershort}

\begin{verlong}
The three most well-known combinatorial Hopf algebras that are defined over
any commutative ring $\mathbf{k}$ are the Hopf algebra of symmetric functions,
the Hopf algebra of quasisymmetric functions, and that of noncommutative
symmetric functions. The first of these three Hopf algebras has been studied
for several decades, while the latter two are newer (the quasisymmetric
functions, for example, have been first defined by Ira M. Gessel in 1984); we
refer to \cite[Chapters 4 and 6]{HGK} and \cite[Chapters 2 and 5]{Reiner} for
expositions of them\footnote{Historically, the origin of the noncommutative
symmetric functions is in \cite{NCSF1}, whereas the quasisymmetric functions
have been introduced in \cite{Gessel}. See also \cite[Section 7.19]%
{Stanley-EC2} specifically for the quasisymmetric functions and their
enumerative applications (although the Hopf algebra structure does not appear
in this source).}. All three of these Hopf algebras are known to carry
multiple algebraic structures (such as additional products, skewing operators,
pairings etc.) and have several bases of combinatorial and algebraic
significance. The Schur functions -- forming a basis of the symmetric
functions -- are probably the most important of these bases (certainly the
most natural in terms of relations to representation theory and several other
applications); a natural question is thus to seek similar bases for
quasisymmetric and noncommutative symmetric functions.
\end{verlong}

\begin{vershort}
Several answers to this question have been suggested, but the simplest one
appears to be given in a 2013 paper by Berg, Bergeron, Saliola, Serrano and
Zabrocki \cite{BBSSZ}: They define the \textit{immaculate (noncommutative
symmetric) functions} (which form a basis of $\operatorname*{NSym}$) and the
\textit{dual immaculate (quasi-symmetric) functions} (which form a basis of
$\operatorname*{QSym}$). These two bases are mutually dual and satisfy
analogues of various properties of the Schur functions. Among these are a
Littlewood-Richardson rule \cite{BBSSZ14}, a Pieri rule \cite{BSOZ13}, and a
representation-theoretical interpretation \cite{BBSSZ13c}. The immaculate
functions can be defined by an analogue of the Jacobi-Trudi identity (see
\cite[Remark 3.28]{BBSSZ} for details), whereas the dual immaculate functions
can be defined as generating functions for \textquotedblleft immaculate
tableaux\textquotedblright\ in analogy to the Schur functions being generating
functions for semistandard tableaux (see Proposition \ref{prop.dualImm} below).
\end{vershort}

\begin{verlong}
Several answers to this question have been suggested, but the simplest one
appears to be given in a 2013 paper by Berg, Bergeron, Saliola, Serrano and
Zabrocki \cite{BBSSZ}: They define the \textit{immaculate (noncommutative
symmetric) functions} (which form a basis of the noncommutative symmetric
functions) and the \textit{dual immaculate (quasi-symmetric) functions} (which
form a basis of the quasisymmetric functions). These two bases are mutually
dual and satisfy analogues of various properties of the Schur basis (i.e., the
basis of the symmetric functions consisting of the Schur functions). Among
these properties are a Littlewood-Richardson rule \cite{BBSSZ14}, a Pieri rule
\cite{BSOZ13} (which is not a consequence of the Littlewood-Richardson rule),
and a representation-theoretical interpretation \cite{BBSSZ13c}. The
immaculate functions can be defined by an analogue of the Jacobi-Trudi
identity (see \cite[Remark 3.28]{BBSSZ} for details), whereas the dual
immaculate functions can be defined as generating functions for
\textquotedblleft immaculate tableaux\textquotedblright\ in analogy to the
Schur functions being generating functions for semistandard tableaux (see
Proposition \ref{prop.dualImm} below for details).
\end{verlong}

\begin{vershort}
The original definition of the immaculate functions (\cite[Definition
3.2]{BBSSZ}) is by applying a sequence of so-called \textit{noncommutative
Bernstein operators} to the constant power series $1\in\operatorname*{NSym}$.
Around 2013, Mike Zabrocki conjectured that the dual immaculate functions can
be obtained by a similar use of \textquotedblleft quasi-symmetric Bernstein
operators\textquotedblright. The purpose of this note is to prove this
conjecture (Corollary \ref{cor.zabrocki} below). Along the way, we define
certain new binary operations on $\operatorname*{QSym}$; two of them give rise
to a structure of a dendriform algebra \cite{EbrFar08}, which seems to be
interesting in its own right.
\end{vershort}

\begin{verlong}
The original definition of the immaculate functions (\cite[Definition
3.2]{BBSSZ}) is by applying a sequence of so-called \textit{noncommutative
Bernstein operators} to the constant power series $1$. Around 2013, Mike
Zabrocki conjectured that the dual immaculate functions can be obtained by a
similar use of \textquotedblleft quasi-symmetric Bernstein
operators\textquotedblright. The purpose of this note is to prove this
conjecture (Corollary \ref{cor.zabrocki} below). Along the way, we define
certain new binary operations on $\operatorname*{QSym}$ (the ring of
quasisymmetric functions); two of them give rise to a structure of a
dendriform algebra \cite{EbrFar08}, which seems to be interesting in its own right.
\end{verlong}

This note is organized as follows: In Section \ref{sect.qsym}, we recall basic
properties of quasisymmetric (and symmetric) functions and introduce the
notations that we shall use. In Section \ref{sect.dendri}, we define two
binary operations $\left.  \prec\right.  $ and $\bel $ on the power series
ring $\mathbf{k}\left[  \left[  x_{1},x_{2},x_{3},\ldots\right]  \right]  $
and show that they restrict to operations on $\operatorname*{QSym}$ which
interact with the Hopf algebra structure of $\operatorname*{QSym}$ in a useful
way. In Section \ref{sect.dualimmac}, we define the dual immaculate functions,
and show that this definition agrees with the one given in \cite[Remark
3.28]{BBSSZ}; we then give a combinatorial interpretation of dual immaculate
functions (which is not new, but has apparently never been explicitly stated).
In Section \ref{sect.zabrocki}, we prove Zabrocki's conjecture. In Section
\ref{sect.WQSym}, we discuss how our binary operations can be lifted to
noncommutative power series and restrict to operations on
$\operatorname*{WQSym}$, which are closely related to similar operations that
have appeared in the literature. In the final Section \ref{sect.epilogue}, we
ask some further questions.

\begin{vershort}
A detailed version of this note is available on the arXiv (as ancillary file
to \href{https://arxiv.org/abs/1410.0079}{preprint arXiv:1410.0079}); it is
longer and contains more details in some of the arguments.
\end{vershort}

\begin{verlong}
This note is available in two versions: a short one and a long one (with more
details, mainly in proofs). The former is available at \newline%
\url{https://www.cip.ifi.lmu.de/~grinberg/algebra/dimcreation.pdf} , the
latter at \newline%
\url{https://www.cip.ifi.lmu.de/~grinberg/algebra/dimcreation-long.pdf} . The
version you are currently reading is the long (detailed) one. Both versions
are compiled from the same sourcecode (the short one compiles by default; see
the comments at front of the TeX file for precise instructions to get the long
one). Both versions appear on the arXiv as
\href{https://arxiv.org/abs/1410.0079}{preprint arXiv:1410.0079} (the short
version being the regular PDF download, while the long version is an ancillary file).
\end{verlong}

This note has been published as:

\begin{quote}
Darij Grinberg, \textit{Dual Creation Operators and a Dendriform Algebra
Structure on the Quasisymmetric Functions}, Canad. J. Math. \textbf{69} (1),
2017, pp. 21--53, \url{https://doi.org/10.4153/CJM-2016-018-8} .
\end{quote}

\begin{vershort}
The published version differs insignificantly from the version you are
reading. (The former has editorial changes; the latter has some trivial
corrections and updated references.)
\end{vershort}

\begin{verlong}
The published version differs insignificantly from the above-mentioned short
version of this note. (The former has editorial changes; the latter has some
trivial corrections and updated references.)
\end{verlong}

\begin{noncompile}
A paper \cite{Gri-gammapart} which is currently being written might end up
containing alternative proofs to some of the results in the present notes. (It
also will give an exposition of the basic theory of quasisymmetric functions.)
\end{noncompile}

\subsection{Acknowledgments}

Mike Zabrocki kindly shared his conjecture with me during my visit to York
University, Toronto in March 2014. I am also grateful to Nantel Bergeron for
his invitation and hospitality. An anonymous referee made numerous helpful remarks.

\section{\label{sect.qsym}Quasisymmetric functions}

We assume that the reader is familiar with the basics of the theory of
symmetric and quasisymmetric functions (as presented, e.g., in \cite[Chapters
4 and 6]{HGK} and \cite[Chapters 2 and 5]{Reiner}). However, let us define all
the notations that we need (not least because they are not consistent across
the literature). We shall try to have our notations match those used in
\cite[Section 2]{BBSSZ} as much as possible.

We use $\mathbb{N}$ to denote the set $\left\{  0,1,2,\ldots\right\}  $.

A \textit{composition} means a finite sequence of positive integers. For
instance, $\left(  2,3\right)  $ and $\left(  1,5,1\right)  $ are
compositions. The \textit{empty composition} (i.e., the empty sequence
$\left(  {}\right)  $) is denoted by $\varnothing$. We denote by
$\operatorname*{Comp}$ the set of all compositions. For every composition
$\alpha=\left(  \alpha_{1},\alpha_{2},\ldots,\alpha_{\ell}\right)  $, we
denote by $\left\vert \alpha\right\vert $ the \textit{size} of the composition
$\alpha$; this is the nonnegative integer $\alpha_{1}+\alpha_{2}+\cdots
+\alpha_{\ell}$. If $n\in\mathbb{N}$, then a \textit{composition of }$n$
simply means a composition having size $n$. A \textit{nonempty composition}
means a composition that is not empty (or, equivalently, that has size $>0$).

Let $\mathbf{k}$ be a commutative ring (which, for us, means a commutative
ring with unity). This $\mathbf{k}$ will stay fixed throughout the paper. We
shall define our symmetric and quasisymmetric functions over this commutative
ring $\mathbf{k}$.\ \ \ \ \footnote{We do not require anything from
$\mathbf{k}$ other than being a commutative ring. Some authors prefer to work
only over specific rings $\mathbf{k}$, such as $\mathbb{Z}$ or $\mathbb{Q}$
(for example, \cite{BBSSZ} always works over $\mathbb{Q}$). Usually, their
results (and often also their proofs) nevertheless are just as valid over
arbitrary $\mathbf{k}$. We see no reason to restrict our generality here.}
Every tensor sign $\otimes$ without a subscript should be understood to mean
$\otimes_{\mathbf{k}}$.

Let $x_{1},x_{2},x_{3},\ldots$ be countably many distinct indeterminates. We
let $\operatorname*{Mon}$ be the free abelian monoid on the set $\left\{
x_{1},x_{2},x_{3},\ldots\right\}  $ (written multiplicatively); it consists of
elements of the form $x_{1}^{a_{1}}x_{2}^{a_{2}}x_{3}^{a_{3}}\cdots$ for
finitely supported $\left(  a_{1},a_{2},a_{3},\ldots\right)  \in
\mathbb{N}^{\infty}$ (where \textquotedblleft finitely
supported\textquotedblright\ means that all but finitely many positive
integers $i$ satisfy $a_{i}=0$). A \textit{monomial} will mean an element of
$\operatorname*{Mon}$. Thus, monomials are combinatorial objects (without
coefficients), independent of $\mathbf{k}$.

We consider the $\mathbf{k}$-algebra $\mathbf{k}\left[  \left[  x_{1}%
,x_{2},x_{3},\ldots\right]  \right]  $ of (commutative) power series in
countably many distinct indeterminates $x_{1},x_{2},x_{3},\ldots$ over
$\mathbf{k}$. By abuse of notation, we shall identify every monomial
$x_{1}^{a_{1}}x_{2}^{a_{2}}x_{3}^{a_{3}}\cdots\in\operatorname*{Mon}$ with the
corresponding element $x_{1}^{a_{1}}\cdot x_{2}^{a_{2}}\cdot x_{3}^{a_{3}%
}\cdot\cdots$ of $\mathbf{k}\left[  \left[  x_{1},x_{2},x_{3},\ldots\right]
\right]  $ when necessary (e.g., when we speak of the sum of two monomials or
when we multiply a monomial with an element of $\mathbf{k}$); however,
monomials don't live in $\mathbf{k}\left[  \left[  x_{1},x_{2},x_{3}%
,\ldots\right]  \right]  $ per se\footnote{This is a technicality. Indeed, the
monomials $1$ and $x_{1}$ are distinct, but the corresponding elements $1$ and
$x_{1}$ of $\mathbf{k}\left[  \left[  x_{1},x_{2},x_{3},\ldots\right]
\right]  $ are identical when $\mathbf{k}=0$. So we could not regard the
monomials as lying in $\mathbf{k}\left[  \left[  x_{1},x_{2},x_{3}%
,\ldots\right]  \right]  $ by default.}.

The $\mathbf{k}$-algebra $\mathbf{k}\left[  \left[  x_{1},x_{2},x_{3}%
,\ldots\right]  \right]  $ is a topological $\mathbf{k}$-algebra; its topology
is the product topology\footnote{More precisely, this topology is defined as
follows (see also \cite[proof of Corollary 2.6.11]{Reiner}):
\par
We endow the ring $\mathbf{k}$ with the discrete topology. To define a
topology on the $\mathbf{k}$-algebra $\mathbf{k}\left[  \left[  x_{1}%
,x_{2},x_{3},\ldots\right]  \right]  $, we (temporarily) regard every power
series in $\mathbf{k}\left[  \left[  x_{1},x_{2},x_{3},\ldots\right]  \right]
$ as the family of its coefficients. Thus, $\mathbf{k}\left[  \left[
x_{1},x_{2},x_{3},\ldots\right]  \right]  $ becomes a product of infinitely
many copies of $\mathbf{k}$ (one for each monomial). This allows us to define
a product topology on $\mathbf{k}\left[  \left[  x_{1},x_{2},x_{3}%
,\ldots\right]  \right]  $. This product topology is the topology that we will
be using whenever we make statements about convergence in $\mathbf{k}\left[
\left[  x_{1},x_{2},x_{3},\ldots\right]  \right]  $ or write down infinite
sums of power series. A sequence $\left(  a_{n}\right)  _{n\in\mathbb{N}}$ of
power series converges to a power series $a$ with respect to this topology if
and only if for every monomial $\mathfrak{m}$, all sufficiently high
$n\in\mathbb{N}$ satisfy%
\[
\left(  \text{the coefficient of }\mathfrak{m}\text{ in }a_{n}\right)
=\left(  \text{the coefficient of }\mathfrak{m}\text{ in }a\right)  .
\]
\par
Note that this is \textbf{not} the topology obtained by taking the completion
of $\mathbf{k}\left[  x_{1},x_{2},x_{3},\ldots\right]  $ with respect to the
standard grading (in which all $x_{i}$ have degree $1$). Indeed, this
completion is not even the whole $\mathbf{k}\left[  \left[  x_{1},x_{2}%
,x_{3},\ldots\right]  \right]  $.}. The polynomial ring $\mathbf{k}\left[
x_{1},x_{2},x_{3},\ldots\right]  $ is a dense subset of $\mathbf{k}\left[
\left[  x_{1},x_{2},x_{3},\ldots\right]  \right]  $ with respect to this
topology. This allows us to prove certain identities in the $\mathbf{k}%
$-algebra $\mathbf{k}\left[  \left[  x_{1},x_{2},x_{3},\ldots\right]  \right]
$ (such as the associativity of multiplication, just to give a stupid example)
by first proving them in $\mathbf{k}\left[  x_{1},x_{2},x_{3},\ldots\right]  $
(that is, for polynomials), and then arguing that they follow by density in
the topological space $\mathbf{k}\left[  \left[  x_{1},x_{2},x_{3}%
,\ldots\right]  \right]  $.

If $\mathfrak{m}$ is a monomial, then $\operatorname*{Supp}\mathfrak{m}$ will
denote the subset
\[
\left\{  i\in\left\{  1,2,3,\ldots\right\}  \ \mid\ \text{the exponent with
which }x_{i}\text{ occurs in }\mathfrak{m}\text{ is }>0\right\}
\]
of $\left\{  1,2,3,\ldots\right\}  $; this subset is finite. The
\textit{degree} $\deg\mathfrak{m}$ of a monomial $\mathfrak{m}=x_{1}^{a_{1}%
}x_{2}^{a_{2}}x_{3}^{a_{3}}\cdots$ is defined to be $a_{1}+a_{2}+a_{3}%
+\cdots\in\mathbb{N}$.

A power series $P\in\mathbf{k}\left[  \left[  x_{1},x_{2},x_{3},\ldots\right]
\right]  $ is said to be \textit{bounded-degree} if there exists an
$N\in\mathbb{N}$ such that every monomial of degree $>N$ appears with
coefficient $0$ in $P$. Let $\mathbf{k}\left[  \left[  x_{1},x_{2}%
,x_{3},\ldots\right]  \right]  _{\operatorname*{bdd}}$ denote the $\mathbf{k}%
$-subalgebra of $\mathbf{k}\left[  \left[  x_{1},x_{2},x_{3},\ldots\right]
\right]  $ formed by the bounded-degree power series in $\mathbf{k}\left[
\left[  x_{1},x_{2},x_{3},\ldots\right]  \right]  $.

The $\mathbf{k}$\textit{-algebra of symmetric functions} over $\mathbf{k}$ is
defined as the $\mathbf{k}$-subalgebra of $\mathbf{k}\left[  \left[
x_{1},x_{2},x_{3},\ldots\right]  \right]  _{\operatorname*{bdd}}$ consisting
of all bounded-degree power series which are invariant under any permutation
of the indeterminates. This $\mathbf{k}$-subalgebra is denoted by
$\operatorname*{Sym}$. (Notice that $\operatorname*{Sym}$ is denoted $\Lambda$
in \cite{Reiner}.) As a $\mathbf{k}$-module, $\operatorname*{Sym}$ is known to
have several bases, such as the basis of complete homogeneous symmetric
functions $\left(  h_{\lambda}\right)  $ and that of the Schur functions
$\left(  s_{\lambda}\right)  $, both indexed by the integer partitions.

Two monomials $\mathfrak{m}$ and $\mathfrak{n}$ are said to be
\textit{pack-equivalent} if they have the form $\mathfrak{m}=x_{i_{1}}%
^{\alpha_{1}}x_{i_{2}}^{\alpha_{2}}\cdots x_{i_{\ell}}^{\alpha_{\ell}}$ and
$\mathfrak{n}=x_{j_{1}}^{\alpha_{1}}x_{j_{2}}^{\alpha_{2}}\cdots x_{j_{\ell}%
}^{\alpha_{\ell}}$ for some $\ell\in\mathbb{N}$, some positive integers
$\alpha_{1}$, $\alpha_{2}$, $\ldots$, $\alpha_{\ell}$, some positive integers
$i_{1}$, $i_{2}$, $\ldots$, $i_{\ell}$ satisfying $i_{1}<i_{2}<\cdots<i_{\ell
}$, and some positive integers $j_{1}$, $j_{2}$, $\ldots$, $j_{\ell}$
satisfying $j_{1}<j_{2}<\cdots<j_{\ell}$\ \ \ \ \footnote{For instance, the
monomial $x_{1}^{4} x_{2}^{2} x_{3} x_{7}^{6}$ is pack-equivalent to
$x_{2}^{4} x_{4}^{2} x_{5} x_{6}^{6}$, but not to $x_{2}^{2} x_{1}^{4} x_{3}
x_{7}^{6}$.}. A power series $P\in\mathbf{k}\left[  \left[  x_{1},x_{2}%
,x_{3},\ldots\right]  \right]  $ is said to be \textit{quasisymmetric} if any
two pack-equivalent monomials have equal coefficients in $P$. The $\mathbf{k}%
$\textit{-algebra of quasisymmetric functions} over $\mathbf{k}$ is defined as
the $\mathbf{k}$-subalgebra of $\mathbf{k}\left[  \left[  x_{1},x_{2}%
,x_{3},\ldots\right]  \right]  _{\operatorname*{bdd}}$ consisting of all
bounded-degree power series which are quasisymmetric. It is clear that
$\operatorname*{Sym}\subseteq\operatorname*{QSym}$.

For every composition $\alpha=\left(  \alpha_{1},\alpha_{2},\ldots
,\alpha_{\ell}\right)  $, the \textit{monomial quasisymmetric function}
$M_{\alpha}$ is defined by%
\[
M_{\alpha}=\sum_{1\leq i_{1}<i_{2}<\cdots<i_{\ell}}x_{i_{1}}^{\alpha_{1}%
}x_{i_{2}}^{\alpha_{2}}\cdots x_{i_{\ell}}^{\alpha_{\ell}}\in\mathbf{k}\left[
\left[  x_{1},x_{2},x_{3},\ldots\right]  \right]  _{\operatorname*{bdd}}.
\]
One easily sees that $M_{\alpha}\in\operatorname*{QSym}$ for every $\alpha
\in\operatorname*{Comp}$. It is well-known that $\left(  M_{\alpha}\right)
_{\alpha\in\operatorname*{Comp}}$ is a basis of the $\mathbf{k}$-module
$\operatorname*{QSym}$; this is the so-called \textit{monomial basis} of
$\operatorname*{QSym}$. Other bases of $\operatorname*{QSym}$ exist as well,
some of which we are going to encounter below.

It is well-known that the $\mathbf{k}$-algebras $\operatorname*{Sym}$ and
$\operatorname*{QSym}$ can be canonically endowed with Hopf algebra structures
such that $\operatorname*{Sym}$ is a Hopf subalgebra of $\operatorname*{QSym}%
$. We refer to \cite[Chapters 4 and 6]{HGK} and \cite[Chapters 2 and
5]{Reiner} for the definitions of these structures (and for a definition of
the notion of a Hopf algebra); at this point, let us merely state a few
properties. The comultiplication $\Delta:\operatorname*{QSym}\rightarrow
\operatorname*{QSym}\otimes\operatorname*{QSym}$ of $\operatorname*{QSym}$
satisfies%
\[
\Delta\left(  M_{\alpha}\right)  =\sum_{i=0}^{\ell}M_{\left(  \alpha
_{1},\alpha_{2},\ldots,\alpha_{i}\right)  }\otimes M_{\left(  \alpha
_{i+1},\alpha_{i+2},\ldots,\alpha_{\ell}\right)  }%
\]
for every $\alpha=\left(  \alpha_{1},\alpha_{2},\ldots,\alpha_{\ell}\right)
\in\operatorname*{Comp}$. The counit $\varepsilon:\operatorname*{QSym}%
\rightarrow\mathbf{k}$ of $\operatorname*{QSym}$ satisfies $\varepsilon\left(
M_{\alpha}\right)  =
\begin{cases}
1, & \text{if }\alpha=\varnothing;\\
0, & \text{if }\alpha\neq\varnothing
\end{cases}
\quad$ for every $\alpha\in\operatorname*{Comp}$.

We shall always use the notation $\Delta$ for the comultiplication of a Hopf
algebra, the notation $\varepsilon$ for the counit of a Hopf algebra, and the
notation $S$ for the antipode of a Hopf algebra. Occasionally we shall use
\textit{Sweedler's notation} for working with coproducts of elements of a Hopf
algebra\footnote{In a nutshell, Sweedler's notation (or, more precisely, the
special case of Sweedler's notation that we will use) consists in writing
$\sum_{\left(  c\right)  }c_{\left(  1\right)  }\otimes c_{\left(  2\right)
}$ for the tensor $\Delta\left(  c\right)  \in C\otimes C$, where $c$ is an
element of a $\mathbf{k}$-coalgebra $C$. The sum $\sum_{\left(  c\right)
}c_{\left(  1\right)  }\otimes c_{\left(  2\right)  }$ symbolizes a
representation of the tensor $\Delta\left(  c\right)  $ as a sum $\sum
_{i=1}^{N}c_{1,i}\otimes c_{2,i}$ of pure tensors; it allows us to manipulate
$\Delta\left(  c\right)  $ without having to explicitly introduce the $N$ and
the $c_{1,i}$ and the $c_{2,i}$. For instance, if $f:C\rightarrow\mathbf{k}$
is a $\mathbf{k}$-linear map, then we can write $\sum_{\left(  c\right)
}f\left(  c_{\left(  1\right)  }\right)  c_{\left(  2\right)  }$ for
$\sum_{i=1}^{N}f\left(  c_{1,i}\right)  c_{2,i}$. Of course, we need to be
careful not to use Sweedler's notation for terms which do depend on the
specific choice of the $N$ and the $c_{1,i}$ and the $c_{2,i}$; for instance,
we must not write $\sum_{\left(  c\right)  }c_{\left(  1\right)  }%
^{2}c_{\left(  2\right)  }$.}.

If $\alpha=\left(  \alpha_{1},\alpha_{2},\ldots,\alpha_{\ell}\right)  $ is a
composition of an $n\in\mathbb{N}$, then we define a subset $D\left(
\alpha\right)  $ of $\left\{  1,2,\ldots,n-1\right\}  $ by%
\[
D\left(  \alpha\right)  =\left\{  \alpha_{1},\alpha_{1}+\alpha_{2},\alpha
_{1}+\alpha_{2}+\alpha_{3},\ldots,\alpha_{1}+\alpha_{2}+\cdots+\alpha_{\ell
-1}\right\}  .
\]
This subset $D\left(  \alpha\right)  $ is called the \textit{set of partial
sums} of the composition $\alpha$; see \cite[Definition 5.1.10]{Reiner} for
its further properties. Most importantly, a composition $\alpha$ of size $n$
can be uniquely reconstructed from $n$ and $D\left(  \alpha\right)  $.

If $\alpha=\left(  \alpha_{1},\alpha_{2},\ldots,\alpha_{\ell}\right)  $ is a
composition of an $n\in\mathbb{N}$, then the \textit{fundamental
quasisymmetric function }$F_{\alpha}\in\mathbf{k}\left[  \left[  x_{1}%
,x_{2},x_{3},\ldots\right]  \right]  _{\operatorname*{bdd}}$ can be defined by%
\begin{equation}
F_{\alpha}=\sum_{\substack{i_{1}\leq i_{2}\leq\cdots\leq i_{n};\\i_{j}%
<i_{j+1}\text{ if }j\in D\left(  \alpha\right)  }}x_{i_{1}}x_{i_{2}}\cdots
x_{i_{n}}. \label{eq.F.def}%
\end{equation}
(This is only one of several possible definitions of $F_{\alpha}$. In
\cite[Definition 5.2.4]{Reiner}, the power series $F_{\alpha}$ is denoted by
$L_{\alpha}$ and defined differently; but \cite[Proposition 5.2.9]{Reiner}
proves the equivalence of this definition with ours.\footnote{In fact,
\cite[(5.2.3)]{Reiner} is exactly our equality (\ref{eq.F.def}).}) One can
easily see that $F_{\alpha}\in\operatorname*{QSym}$ for every $\alpha
\in\operatorname*{Comp}$. The family $\left(  F_{\alpha}\right)  _{\alpha
\in\operatorname*{Comp}}$ is a basis of the $\mathbf{k}$-module
$\operatorname*{QSym}$ as well; it is called the \textit{fundamental basis} of
$\operatorname*{QSym}$.

\begin{noncompile}
We use the notation $\left[  \mathcal{A}\right]  $ for $%
\begin{cases}
1, & \text{if }\mathcal{A}\text{ is true;}\\
0, & \text{if }\mathcal{A}\text{ is false}%
\end{cases}
\quad$ whenever $\mathcal{A}$ is a statement. [Currently we don't use this
notation in this note.]
\end{noncompile}

\section{\label{sect.dendri}Restricted-product operations}

We shall now define two binary operations on $\mathbf{k}\left[  \left[
x_{1},x_{2},x_{3},\ldots\right]  \right]  $.

\begin{definition}
\label{def.Dless}We define a binary operation $\left.  \prec\right.
:\mathbf{k}\left[  \left[  x_{1},x_{2},x_{3},\ldots\right]  \right]
\times\mathbf{k}\left[  \left[  x_{1},x_{2},x_{3},\ldots\right]  \right]
\rightarrow\mathbf{k}\left[  \left[  x_{1},x_{2},x_{3},\ldots\right]  \right]
$ (written in infix notation\footnotemark) by the requirements that it be
$\mathbf{k}$-bilinear and continuous with respect to the topology on
$\mathbf{k}\left[  \left[  x_{1},x_{2},x_{3},\ldots\right]  \right]  $ and
that it satisfy%
\begin{equation}
\mathfrak{m}\left.  \prec\right.  \mathfrak{n}=
\begin{cases}
\mathfrak{m}\cdot\mathfrak{n}, & \text{if }\min\left(  \operatorname*{Supp}%
\mathfrak{m}\right)  <\min\left(  \operatorname*{Supp}\mathfrak{n}\right)  ;\\
0, & \text{if }\min\left(  \operatorname*{Supp}\mathfrak{m}\right)  \geq
\min\left(  \operatorname*{Supp}\mathfrak{n}\right)
\end{cases}
\label{eq.def.Dless.monomial}%
\end{equation}
for any two monomials $\mathfrak{m}$ and $\mathfrak{n}$.
\end{definition}

\footnotetext{By this we mean that we write $a\left.  \prec\right.  b$ instead
of $\left.  \prec\right.  \left(  a,b\right)  $.}Some clarifications are in
order. First, we are using $\left.  \prec\right.  $ as an operation symbol
(rather than as a relation symbol as it is commonly used)\footnote{Of course,
the symbol has been chosen because it is reminiscent of the smaller symbol in
\textquotedblleft$\min\left(  \operatorname*{Supp}\mathfrak{m}\right)
<\min\left(  \operatorname*{Supp}\mathfrak{n}\right)  $\textquotedblright.}.
Second, we consider $\min\varnothing$ to be $\infty$, and this symbol $\infty$
is understood to be greater than every integer\footnote{but not greater than
itself}. Hence, $\mathfrak{m}\left.  \prec\right.  1=\mathfrak{m}$ for every
nonconstant monomial $\mathfrak{m}$, and $1\left.  \prec\right.
\mathfrak{m}=0$ for every monomial $\mathfrak{m}$.

Let us first see why the operation $\left.  \prec\right.  $ in Definition
\ref{def.Dless} is well-defined. Recall that the topology on $\mathbf{k}%
\left[  \left[  x_{1},x_{2},x_{3},\ldots\right]  \right]  $ is the product
topology. Hence, if $\left.  \prec\right.  $ is to be $\mathbf{k}$-bilinear
and continuous with respect to it, we must have%
\[
\left(  \sum_{\mathfrak{m}\in\operatorname*{Mon}}\lambda_{\mathfrak{m}%
}\mathfrak{m}\right)  \left.  \prec\right.  \left(  \sum_{\mathfrak{n}%
\in\operatorname*{Mon}}\mu_{\mathfrak{n}}\mathfrak{n}\right)  =\sum
_{\mathfrak{m}\in\operatorname*{Mon}}\ \ \sum_{\mathfrak{n}\in
\operatorname*{Mon}}\lambda_{\mathfrak{m}}\mu_{\mathfrak{n}}\mathfrak{m}%
\left.  \prec\right.  \mathfrak{n}%
\]
for any families $\left(  \lambda_{\mathfrak{m}}\right)  _{\mathfrak{m}%
\in\operatorname*{Mon}}\in\mathbf{k}^{\operatorname*{Mon}}$ and $\left(
\mu_{\mathfrak{n}}\right)  _{\mathfrak{n}\in\operatorname*{Mon}}\in
\mathbf{k}^{\operatorname*{Mon}}$ of scalars. Combined with
(\ref{eq.def.Dless.monomial}), this uniquely determines $\left.  \prec\right.
$. Therefore, the binary operation $\left.  \prec\right.  $ satisfying the
conditions of Definition \ref{def.Dless} is unique (if it exists). But it also
exists, because if we define a binary operation $\left.  \prec\right.  $ on
$\mathbf{k}\left[  \left[  x_{1},x_{2},x_{3},\ldots\right]  \right]  $ by the
explicit formula%
\begin{align*}
&  \left(  \sum_{\mathfrak{m}\in\operatorname*{Mon}}\lambda_{\mathfrak{m}%
}\mathfrak{m}\right)  \left.  \prec\right.  \left(  \sum_{\mathfrak{n}%
\in\operatorname*{Mon}}\mu_{\mathfrak{n}}\mathfrak{n}\right)  =\sum
_{\substack{\left(  \mathfrak{m},\mathfrak{n}\right)  \in\operatorname*{Mon}%
\times\operatorname*{Mon}\text{;}\\\min\left(  \operatorname*{Supp}%
\mathfrak{m}\right)  <\min\left(  \operatorname*{Supp}\mathfrak{n}\right)
}}\lambda_{\mathfrak{m}}\mu_{\mathfrak{n}}\mathfrak{mn}\\
&  \ \ \ \ \ \ \ \ \ \ \text{for all }\left(  \lambda_{\mathfrak{m}}\right)
_{\mathfrak{m}\in\operatorname*{Mon}}\in\mathbf{k}^{\operatorname*{Mon}}\text{
and }\left(  \mu_{\mathfrak{n}}\right)  _{\mathfrak{n}\in\operatorname*{Mon}%
}\in\mathbf{k}^{\operatorname*{Mon}},
\end{align*}
then it clearly satisfies the conditions of Definition \ref{def.Dless} (and is well-defined).

The operation $\left.  \prec\right.  $ is not associative; however, it is part
of what is called a \textit{dendriform algebra} structure on $\mathbf{k}%
\left[  \left[  x_{1},x_{2},x_{3},\ldots\right]  \right]  $ (and on
$\operatorname*{QSym}$, as we shall see below). The following remark (which
will not be used until Section \ref{sect.WQSym}, and thus can be skipped by a
reader not familiar with dendriform algebras) provides some details:

\begin{remark}
\label{rmk.dendri}Let us define another binary operation $\left.
\succeq\right.  $ on $\mathbf{k}\left[  \left[  x_{1},x_{2},x_{3}%
,\ldots\right]  \right]  $ similarly to $\left.  \prec\right.  $ except that
we set%
\[
\mathfrak{m}\left.  \succeq\right.  \mathfrak{n}=
\begin{cases}
\mathfrak{m}\cdot\mathfrak{n}, & \text{if }\min\left(  \operatorname*{Supp}%
\mathfrak{m}\right)  \geq\min\left(  \operatorname*{Supp}\mathfrak{n}\right)
;\\
0, & \text{if }\min\left(  \operatorname*{Supp}\mathfrak{m}\right)
<\min\left(  \operatorname*{Supp}\mathfrak{n}\right)
\end{cases}
\quad.
\]
Then, the structure $\left(  \mathbf{k}\left[  \left[  x_{1},x_{2}%
,x_{3},\ldots\right]  \right]  ,\left.  \prec\right.  , \left.  \succeq
\right.  \right)  $ is a dendriform algebra augmented to satisfy
\cite[(15)]{EbrFar08}. In particular, any three elements $a$, $b$ and $c$ of
$\mathbf{k}\left[  \left[  x_{1},x_{2},x_{3},\ldots\right]  \right]  $ satisfy%
\begin{align*}
a\left.  \prec\right.  b+a \left.  \succeq\right.  b  &  =ab;\\
\left(  a\left.  \prec\right.  b\right)  \left.  \prec\right.  c  &  =a\left.
\prec\right.  \left(  bc\right)  ;\\
\left(  a\left.  \succeq\right.  b\right)  \left.  \prec\right.  c  &
=a\left.  \succeq\right.  \left(  b\left.  \prec\right.  c\right)  ;\\
a\left.  \succeq\right.  \left(  b\left.  \succeq\right.  c\right)   &
=\left(  ab\right)  \left.  \succeq\right.  c.
\end{align*}

\end{remark}

Now, we introduce another binary operation.

\begin{definition}
\label{def.belgthor}We define a binary operation $\bel :\mathbf{k}\left[
\left[  x_{1},x_{2},x_{3},\ldots\right]  \right]  \times\mathbf{k}\left[
\left[  x_{1},x_{2},x_{3},\ldots\right]  \right]  \rightarrow\mathbf{k}\left[
\left[  x_{1},x_{2},x_{3},\ldots\right]  \right]  $ (written in infix
notation) by the requirements that it be $\mathbf{k}$-bilinear and continuous
with respect to the topology on $\mathbf{k}\left[  \left[  x_{1},x_{2}%
,x_{3},\ldots\right]  \right]  $ and that it satisfy%
\[
\mathfrak{m} \bel \mathfrak{n}=
\begin{cases}
\mathfrak{m}\cdot\mathfrak{n}, & \text{if }\max\left(  \operatorname*{Supp}%
\mathfrak{m}\right)  \leq\min\left(  \operatorname*{Supp}\mathfrak{n}\right)
;\\
0, & \text{if }\max\left(  \operatorname*{Supp}\mathfrak{m}\right)
>\min\left(  \operatorname*{Supp}\mathfrak{n}\right)
\end{cases}
\]
for any two monomials $\mathfrak{m}$ and $\mathfrak{n}$.
\end{definition}

Here, $\max\varnothing$ is understood as $0$. The well-definedness of the
operation $\bel $ in Definition \ref{def.belgthor} is proven in the same way
as that of the operation $\left.  \prec\right.  $.

Let us make a simple observation which will not be used until Section
\ref{sect.WQSym}, but provides some context:

\begin{proposition}
\label{prop.belgthor.assoc}The binary operation $\bel $ is associative. It is
also unital (with $1$ serving as the unity).
\end{proposition}

\begin{vershort}
\begin{proof}
[Proof of Proposition \ref{prop.belgthor.assoc}.]We shall only sketch the
proof; see the detailed version for more details.

In order to show that $\bel $ is associative, it suffices to prove that
$\left(  \mathfrak{m} \bel \mathfrak{n}\right)  \bel \mathfrak{p}=\mathfrak{m}
\bel \left(  \mathfrak{n} \bel \mathfrak{p}\right)  $ for any three monomials
$\mathfrak{m}$, $\mathfrak{n}$ and $\mathfrak{p}$ (since $\bel $ is bilinear).
But this follows from observing that both $\left(  \mathfrak{m}
\bel \mathfrak{n}\right)  \bel \mathfrak{p}$ and $\mathfrak{m} \bel \left(
\mathfrak{n} \bel \mathfrak{p}\right)  $ are equal to $\mathfrak{mnp}$ if the
three inequalities $\max\left(  \operatorname*{Supp}\mathfrak{m}\right)
\leq\min\left(  \operatorname*{Supp}\mathfrak{n}\right)  $ and $\max\left(
\operatorname*{Supp}\mathfrak{m}\right)  \leq\min\left(  \operatorname*{Supp}%
\mathfrak{p}\right)  $ and $\max\left(  \operatorname*{Supp}\mathfrak{n}%
\right)  \leq\min\left(  \operatorname*{Supp}\mathfrak{p}\right)  $ hold, and
equal to $0$ otherwise.

The proof of the unitality of $\bel $ is similar.
\end{proof}
\end{vershort}

\begin{verlong}
\begin{proof}
[Proof of Proposition \ref{prop.belgthor.assoc}.]Let us first show that
$\bel $ is associative.

In order to show this, we must prove that
\begin{equation}
\left(  a \bel b\right)  \bel c=a \bel \left(  b \bel c\right)
\label{pf.prop.belgthor.assoc.assoc}%
\end{equation}
for any three elements $a$, $b$ and $c$ of $\mathbf{k}\left[  \left[
x_{1},x_{2},x_{3},\ldots\right]  \right]  $.

But if $\mathfrak{m}$, $\mathfrak{n}$ and $\mathfrak{p}$ are three monomials,
then the definition of $\bel$ readily shows that%
\begin{align*}
\left(  \mathfrak{m}\bel\mathfrak{n}\right)  \bel\mathfrak{p}  &  =%
\begin{cases}
\mathfrak{mnp}, & \text{if }\max\left(  \operatorname*{Supp}\mathfrak{m}%
\right)  \leq\min\left(  \operatorname*{Supp}\mathfrak{n}\right) \\
& \ \ \ \ \ \ \ \ \ \ \ \text{ and }\max\left(  \operatorname*{Supp}\left(
\mathfrak{mn}\right)  \right)  \leq\min\left(  \operatorname*{Supp}%
\mathfrak{p}\right)  ;\\
0, & \text{otherwise}%
\end{cases}
\\
&  =%
\begin{cases}
\mathfrak{mnp}, & \text{if }\max\left(  \operatorname*{Supp}\mathfrak{m}%
\right)  \leq\min\left(  \operatorname*{Supp}\mathfrak{n}\right) \\
& \ \ \ \ \ \ \ \ \ \ \ \text{ and }\max\left(  \left(  \operatorname*{Supp}%
\mathfrak{m}\right)  \cup\left(  \operatorname*{Supp}\mathfrak{n}\right)
\right)  \leq\min\left(  \operatorname*{Supp}\mathfrak{p}\right)  ;\\
0, & \text{otherwise}%
\end{cases}
\\
&  \ \ \ \ \ \ \ \ \ \ \left(  \text{since }\operatorname*{Supp}\left(
\mathfrak{mn}\right)  =\left(  \operatorname*{Supp}\mathfrak{m}\right)
\cup\left(  \operatorname*{Supp}\mathfrak{n}\right)  \right)
\end{align*}
and%
\begin{align*}
\mathfrak{m}\bel\left(  \mathfrak{n}\bel\mathfrak{p}\right)   &  =%
\begin{cases}
\mathfrak{mnp}, & \text{if }\max\left(  \operatorname*{Supp}\mathfrak{n}%
\right)  \leq\min\left(  \operatorname*{Supp}\mathfrak{p}\right) \\
& \ \ \ \ \ \ \ \ \ \ \ \text{ and }\max\left(  \operatorname*{Supp}%
\mathfrak{m}\right)  \leq\min\left(  \operatorname*{Supp}\left(
\mathfrak{np}\right)  \right)  ;\\
0, & \text{otherwise}%
\end{cases}
\\
&  =%
\begin{cases}
\mathfrak{mnp}, & \text{if }\max\left(  \operatorname*{Supp}\mathfrak{n}%
\right)  \leq\min\left(  \operatorname*{Supp}\mathfrak{p}\right) \\
& \ \ \ \ \ \ \ \ \ \ \ \text{ and }\max\left(  \operatorname*{Supp}%
\mathfrak{m}\right)  \leq\min\left(  \left(  \operatorname*{Supp}%
\mathfrak{n}\right)  \cup\left(  \operatorname*{Supp}\mathfrak{p}\right)
\right)  ;\\
0, & \text{otherwise}%
\end{cases}
\\
&  \ \ \ \ \ \ \ \ \ \ \left(  \text{since }\operatorname*{Supp}\left(
\mathfrak{np}\right)  =\left(  \operatorname*{Supp}\mathfrak{n}\right)
\cup\left(  \operatorname*{Supp}\mathfrak{p}\right)  \right)  ;
\end{align*}
thus, $\left(  \mathfrak{m}\bel\mathfrak{n}\right)  \bel\mathfrak{p}%
=\mathfrak{m}\bel\left(  \mathfrak{n}\bel\mathfrak{p}\right)  $ (since it is
straightforward to check that the condition \newline$\left(  \max\left(
\operatorname*{Supp}\mathfrak{m}\right)  \leq\min\left(  \operatorname*{Supp}%
\mathfrak{n}\right)  \text{ and }\max\left(  \left(  \operatorname*{Supp}%
\mathfrak{m}\right)  \cup\left(  \operatorname*{Supp}\mathfrak{n}\right)
\right)  \leq\min\left(  \operatorname*{Supp}\mathfrak{p}\right)  \right)  $
is equivalent to the condition \newline$\left(  \max\left(
\operatorname*{Supp}\mathfrak{n}\right)  \leq\min\left(  \operatorname*{Supp}%
\mathfrak{p}\right)  \text{ and }\max\left(  \operatorname*{Supp}%
\mathfrak{m}\right)  \leq\min\left(  \left(  \operatorname*{Supp}%
\mathfrak{n}\right)  \cup\left(  \operatorname*{Supp}\mathfrak{p}\right)
\right)  \right)  $\ \ \ \ \footnote{Indeed, both conditions are equivalent to
\newline$\left(  \max\left(  \operatorname*{Supp}\mathfrak{m}\right)  \leq
\min\left(  \operatorname*{Supp}\mathfrak{n}\right)  \text{ and }\max\left(
\operatorname*{Supp}\mathfrak{m}\right)  \leq\min\left(  \operatorname*{Supp}%
\mathfrak{p}\right)  \text{ and }\max\left(  \operatorname*{Supp}%
\mathfrak{n}\right)  \leq\min\left(  \operatorname*{Supp}\mathfrak{p}\right)
\right)  $.}). In other words, the equality
(\ref{pf.prop.belgthor.assoc.assoc}) holds when $a$, $b$ and $c$ are
monomials. Thus, this equality also holds whenever $a$, $b$ and $c$ are
polynomials (since it is $\mathbf{k}$-linear in $a$, $b$ and $c$), and
consequently also holds whenever $a$, $b$ and $c$ are power series (since it
is continuous in $a$, $b$ and $c$). This proves that $\bel$ is associative.

The proof of the fact that $\bel $ is unital (with unity $1$) is similar and
left to the reader. Proposition \ref{prop.belgthor.assoc} is thus shown.
\end{proof}
\end{verlong}

Here is another property of $\bel $ that will not be used until Section
\ref{sect.WQSym}:

\begin{proposition}
\label{prop.QSym.closed}Every $a\in\operatorname*{QSym}$ and $b\in
\operatorname*{QSym}$ satisfy $a\left.  \prec\right.  b\in\operatorname*{QSym}%
$ and $a \bel b\in\operatorname*{QSym}$.
\end{proposition}

\begin{vershort}
For example, we can explicitly describe the operation $\bel $ on the monomial
basis $\left(  M_{\gamma}\right)  _{\gamma\in\operatorname*{Comp}}$ of
$\operatorname*{QSym}$. Namely, any two nonempty compositions $\alpha$ and
$\beta$ satisfy $M_{\alpha} \bel M_{\beta}=M_{\left[  \alpha,\beta\right]
}+M_{\alpha\odot\beta}$, where $\left[  \alpha,\beta\right]  $ and
$\alpha\odot\beta$ are two compositions defined by%
\begin{align*}
\left[  \left(  \alpha_{1},\alpha_{2},\ldots,\alpha_{\ell}\right)  ,\left(
\beta_{1},\beta_{2},\ldots,\beta_{m}\right)  \right]   &  =\left(  \alpha
_{1},\alpha_{2},\ldots,\alpha_{\ell},\beta_{1},\beta_{2},\ldots,\beta
_{m}\right)  ;\\
\left(  \alpha_{1},\alpha_{2},\ldots,\alpha_{\ell}\right)  \odot\left(
\beta_{1},\beta_{2},\ldots,\beta_{m}\right)   &  =\left(  \alpha_{1}%
,\alpha_{2},\ldots,\alpha_{\ell-1},\alpha_{\ell}+\beta_{1},\beta_{2},\beta
_{3},\ldots,\beta_{m}\right)  .
\end{align*}
If one of $\alpha$ and $\beta$ is empty, then $M_{\alpha} \bel M_{\beta
}=M_{\left[  \alpha,\beta\right]  }$.
\end{vershort}

\begin{verlong}
For example, we can explicitly describe the operation $\bel $ on the monomial
basis $\left(  M_{\gamma}\right)  _{\gamma\in\operatorname*{Comp}}$ of
$\operatorname*{QSym}$. Namely, any two nonempty compositions $\alpha$ and
$\beta$ satisfy $M_{\alpha} \bel M_{\beta}=M_{\left[  \alpha,\beta\right]
}+M_{\alpha\odot\beta}$, where $\left[  \alpha,\beta\right]  $ and
$\alpha\odot\beta$ are two compositions defined by%
\begin{align*}
\left[  \left(  \alpha_{1},\alpha_{2},\ldots,\alpha_{\ell}\right)  ,\left(
\beta_{1},\beta_{2},\ldots,\beta_{m}\right)  \right]   &  =\left(  \alpha
_{1},\alpha_{2},\ldots,\alpha_{\ell},\beta_{1},\beta_{2},\ldots,\beta
_{m}\right)  ;\\
\left(  \alpha_{1},\alpha_{2},\ldots,\alpha_{\ell}\right)  \odot\left(
\beta_{1},\beta_{2},\ldots,\beta_{m}\right)   &  =\left(  \alpha_{1}%
,\alpha_{2},\ldots,\alpha_{\ell-1},\alpha_{\ell}+\beta_{1},\beta_{2},\beta
_{3},\ldots,\beta_{m}\right)  .
\end{align*}
\footnote{What we call $\left[  \alpha,\beta\right]  $ is denoted by
$\alpha\cdot\beta$ in \cite[before Proposition 5.1.7]{Reiner}.} If one of
$\alpha$ and $\beta$ is empty, then $M_{\alpha} \bel M_{\beta}=M_{\left[
\alpha,\beta\right]  }$.
\end{verlong}

Proposition \ref{prop.QSym.closed} can reasonably be called obvious; the below
proof owes its length mainly to the difficulty of formalizing the intuition.

\begin{proof}
[Proof of Proposition \ref{prop.QSym.closed}.]We shall first introduce a few
more notations.

\begin{vershort}
If $\mathfrak{m}$ is a monomial, then the \textit{Parikh composition} of
$\mathfrak{m}$ is defined as follows: Write $\mathfrak{m}$ in the form
$\mathfrak{m}=x_{i_{1}}^{\alpha_{1}}x_{i_{2}}^{\alpha_{2}}\cdots x_{i_{\ell}%
}^{\alpha_{\ell}}$ for some $\ell\in\mathbb{N}$, some positive integers
$\alpha_{1}$, $\alpha_{2}$, $\ldots$, $\alpha_{\ell}$, and some positive
integers $i_{1}$, $i_{2}$, $\ldots$, $i_{\ell}$ satisfying $i_{1}<i_{2}%
<\cdots<i_{\ell}$. Notice that this way of writing $\mathfrak{m}$ is unique.
Then, the Parikh composition of $\mathfrak{m}$ is defined to be the
composition $\left(  \alpha_{1},\alpha_{2},\ldots,\alpha_{\ell}\right)  $.
\end{vershort}

\begin{verlong}
If $\mathfrak{m}$ is a monomial, then the \textit{Parikh composition} of
$\mathfrak{m}$ is defined as follows: Write $\mathfrak{m}$ in the form
$\mathfrak{m}=x_{i_{1}}^{\alpha_{1}}x_{i_{2}}^{\alpha_{2}}\cdots x_{i_{\ell}%
}^{\alpha_{\ell}}$ for some $\ell\in\mathbb{N}$, some positive integers
$\alpha_{1}$, $\alpha_{2}$, $\ldots$, $\alpha_{\ell}$, and some positive
integers $i_{1}$, $i_{2}$, $\ldots$, $i_{\ell}$ satisfying $i_{1}<i_{2}%
<\cdots<i_{\ell}$. (Notice that this way of writing $\mathfrak{m}$ is unique.)
Then, the Parikh composition of $\mathfrak{m}$ is defined to be the
composition $\left(  \alpha_{1},\alpha_{2},\ldots,\alpha_{\ell}\right)  $.
\end{verlong}

We denote by $\operatorname*{Parikh}\mathfrak{m}$ the Parikh composition of a
monomial $\mathfrak{m}$. Now, it is easy to see that the definition of a
monomial quasisymmetric function $M_{\alpha}$ can be rewritten as follows: For
every $\alpha\in\operatorname*{Comp}$, we have%
\begin{equation}
M_{\alpha}=\sum_{\substack{\mathfrak{m}\in\operatorname*{Mon}%
;\\\operatorname*{Parikh}\mathfrak{m}=\alpha}}\mathfrak{m}.
\label{pf.prop.QSym.closed.Malpha-via-Parikh}%
\end{equation}
(Indeed, for any given composition $\alpha=\left(  \alpha_{1},\alpha
_{2},\ldots,\alpha_{\ell}\right)  $, the monomials $\mathfrak{m}$ satisfying
$\operatorname*{Parikh}\mathfrak{m}=\alpha$ are precisely the monomials of the
form $x_{i_{1}}^{\alpha_{1}}x_{i_{2}}^{\alpha_{2}}\cdots x_{i_{\ell}}%
^{\alpha_{\ell}}$ with $i_{1}$, $i_{2}$, $\ldots$, $i_{\ell}$ being positive
integers satisfying $i_{1}<i_{2}<\cdots<i_{\ell}$.)

Now, pack-equivalent monomials can be characterized as follows: Two monomials
$\mathfrak{m}$ and $\mathfrak{n}$ are pack-equivalent if and only if they have
the same Parikh composition.

Now, we come to the proof of Proposition \ref{prop.QSym.closed}.

Let us first fix two compositions $\alpha$ and $\beta$. We shall prove that
$M_{\alpha}\left.  \prec\right.  M_{\beta}\in\operatorname*{QSym}$.

Write the compositions $\alpha$ and $\beta$ as $\alpha=\left(  \alpha
_{1},\alpha_{2},\ldots,\alpha_{\ell}\right)  $ and $\beta=\left(  \beta
_{1},\beta_{2},\ldots,\beta_{m}\right)  $. Let $\mathcal{S}_{0}$ denote the
$\ell$-element set $\left\{  0\right\}  \times\left\{  1,2,\ldots
,\ell\right\}  $. Let $\mathcal{S}_{1}$ denote the $m$-element set $\left\{
1\right\}  \times\left\{  1,2,\ldots,m\right\}  $. Let $\mathcal{S}$ denote
the $\left(  \ell+m\right)  $-element set $\mathcal{S}_{0}\cup\mathcal{S}_{1}%
$. Let $\operatorname*{inc}\nolimits_{0}:\left\{  1,2,\ldots,\ell\right\}
\rightarrow\mathcal{S}$ be the map which sends every $p\in\left\{
1,2,\ldots,\ell\right\}  $ to $\left(  0,p\right)  \in\mathcal{S}_{0}%
\subseteq\mathcal{S}$. Let $\operatorname*{inc}\nolimits_{1}:\left\{
1,2,\ldots,m\right\}  \rightarrow\mathcal{S}$ be the map which sends every
$q\in\left\{  1,2,\ldots,m\right\}  $ to $\left(  1,q\right)  \in
\mathcal{S}_{1}\subseteq\mathcal{S}$. Define a map $\rho:\mathcal{S}%
\rightarrow\left\{  1,2,3,\ldots\right\}  $ by setting%
\begin{align*}
\rho\left(  0,p\right)   &  =\alpha_{p}\ \ \ \ \ \ \ \ \ \ \text{for all }%
p\in\left\{  1,2,\ldots,\ell\right\}  ;\\
\rho\left(  1,q\right)   &  =\beta_{q}\ \ \ \ \ \ \ \ \ \ \text{for all }%
q\in\left\{  1,2,\ldots,m\right\}  .
\end{align*}

For every composition $\gamma=\left(  \gamma_{1},\gamma_{2},\ldots,\gamma
_{n}\right)  $, we define a $\gamma$\textit{-smap} to be a map $f:\mathcal{S}%
\rightarrow\left\{  1,2,\ldots,n\right\}  $ satisfying the following three properties:

\begin{itemize}
\item The maps $f\circ\operatorname*{inc}\nolimits_{0}$ and $f\circ
\operatorname*{inc}\nolimits_{1}$ are strictly increasing.

\item We have\footnote{Keep in mind that we set $\min\varnothing=\infty$.}
$\min\left(  f\left(  \mathcal{S}_{0}\right)  \right)  <\min\left(  f\left(
\mathcal{S}_{1}\right)  \right)  $.

\item Every $u\in\left\{  1,2,\ldots,n\right\}  $ satisfies%
\[
\sum_{s\in f^{-1}\left(  u\right)  }\rho\left(  s\right)  =\gamma_{u}.
\]

\end{itemize}

These three properties will be called the three \textit{defining properties}
of a $\gamma$-smap.

Now, we make the following claim:

\begin{statement}
\textit{Claim 1:} Let $\mathfrak{q}$ be any monomial. Let $\gamma$ be the
Parikh composition of $\mathfrak{q}$. The coefficient of $\mathfrak{q}$ in
$M_{\alpha}\left.  \prec\right.  M_{\beta}$ equals the number of all $\gamma$-smaps.
\end{statement}

\begin{vershort}
\textit{Proof of Claim 1:} We shall give a brief outline of this proof; for
more details, we refer to the detailed version of this note.

Write the composition $\gamma$ in the form $\gamma=\left(  \gamma_{1}%
,\gamma_{2},\ldots,\gamma_{n}\right)  $. Write the monomial $\mathfrak{q}$ in
the form $\mathfrak{q}=x_{k_{1}}^{\gamma_{1}}x_{k_{2}}^{\gamma_{2}}\cdots
x_{k_{n}}^{\gamma_{n}}$ for some positive integers $k_{1}$, $k_{2}$, $\ldots$,
$k_{n}$ satisfying $k_{1}<k_{2}<\cdots<k_{n}$. (This is possible because
$\left(  \gamma_{1},\gamma_{2},\ldots,\gamma_{n}\right)  =\gamma$ is the
Parikh composition of $\mathfrak{q}$.) Then, $\operatorname*{Supp}%
\mathfrak{q}=\left\{  k_{1},k_{2},\ldots,k_{n}\right\}  $.

From (\ref{pf.prop.QSym.closed.Malpha-via-Parikh}), we get $M_{\alpha}%
=\sum_{\substack{\mathfrak{m}\in\operatorname*{Mon};\\\operatorname*{Parikh}%
\mathfrak{m}=\alpha}}\mathfrak{m}$. Similarly, $M_{\beta}=\sum
_{\substack{\mathfrak{n}\in\operatorname*{Mon};\\\operatorname*{Parikh}%
\mathfrak{n}=\beta}}\mathfrak{n}$. Hence,%
\begin{align*}
&  M_{\alpha}\left.  \prec\right.  M_{\beta}\\
&  =\left(  \sum_{\substack{\mathfrak{m}\in\operatorname*{Mon}%
;\\\operatorname*{Parikh}\mathfrak{m}=\alpha}}\mathfrak{m}\right)  \left.
\prec\right.  \left(  \sum_{\substack{\mathfrak{n}\in\operatorname*{Mon}%
;\\\operatorname*{Parikh}\mathfrak{n}=\beta}}\mathfrak{n}\right)
=\sum_{\substack{\left(  \mathfrak{m},\mathfrak{n}\right)  \in
\operatorname*{Mon}\times\operatorname*{Mon};\\\operatorname*{Parikh}%
\mathfrak{m}=\alpha;\\\operatorname*{Parikh}\mathfrak{n}=\beta;\\\min\left(
\operatorname*{Supp}\mathfrak{m}\right)  <\min\left(  \operatorname*{Supp}%
\mathfrak{n}\right)  }}\mathfrak{mn}%
\end{align*}
(by the explicit formula for $\left.  \prec\right.  $). Thus, the coefficient
of $\mathfrak{q}$ in $M_{\alpha}\left.  \prec\right.  M_{\beta}$ equals the
number of all pairs $\left(  \mathfrak{m},\mathfrak{n}\right)  \in
\operatorname*{Mon}\times\operatorname*{Mon}$ such that
$\operatorname*{Parikh}\mathfrak{m}=\alpha$, $\operatorname*{Parikh}%
\mathfrak{n}=\beta$, $\min\left(  \operatorname*{Supp}\mathfrak{m}\right)
<\min\left(  \operatorname*{Supp}\mathfrak{n}\right)  $ and $\mathfrak{mn}%
=\mathfrak{q}$. These pairs shall be called $\mathfrak{q}$\textit{-spairs}.

Now, we shall construct a bijection $\Phi$ from the set of all $\gamma$-smaps
to the set of all $\mathfrak{q}$-spairs. This is a simple exercise in
re-encoding data, so we leave the details to the reader (they can be found in
the detailed version of this note). Let us just state how the bijection and
its inverse are defined:

\begin{itemize}
\item If $f:\mathcal{S}\rightarrow\left\{  1,2,\ldots,n\right\}  $ is a
$\gamma$-smap, then the $\mathfrak{q}$-spair $\Phi\left(  f\right)  $ is
defined to be $\left(  \prod_{p=1}^{\ell}x_{k_{f\left(  0,p\right)  }}%
^{\alpha_{p}},\prod_{q=1}^{m}x_{k_{f\left(  1,q\right)  }}^{\beta_{q}}\right)
$.

\item If $\left(  \mathfrak{m},\mathfrak{n}\right)  $ is a $\mathfrak{q}%
$-spair, then the $\gamma$-smap $\Phi^{-1}\left(  \mathfrak{m},\mathfrak{n}%
\right)  $ is defined as follows: Write the monomial $\mathfrak{m}$ in the
form $\mathfrak{m}=x_{k_{u_{1}}}^{\alpha_{1}}x_{k_{u_{2}}}^{\alpha_{2}}\cdots
x_{k_{u_{\ell}}}^{\alpha_{\ell}}$ for some elements $1\leq u_{1}<u_{2}%
<\cdots<u_{\ell}\leq n$. (This is possible since $\operatorname*{Supp}%
\mathfrak{m}\subseteq\operatorname*{Supp}\mathfrak{q}=\left\{  k_{1}%
,k_{2},\ldots,k_{n}\right\}  $ and $\operatorname*{Parikh}\mathfrak{m}=\alpha
$.) Similarly, write the monomial $\mathfrak{n}$ in the form $\mathfrak{n}%
=x_{k_{v_{1}}}^{\beta_{1}}x_{k_{v_{2}}}^{\beta_{2}}\cdots x_{k_{v_{m}}}%
^{\beta_{m}}$ for some elements $1\leq v_{1}<v_{2}<\cdots<v_{m}\leq n$. Now,
the $\gamma$-smap $\Phi^{-1}\left(  \mathfrak{m},\mathfrak{n}\right)  $ is
defined as the map $f:\mathcal{S}\rightarrow\left\{  1,2,\ldots,n\right\}  $
which sends every $\left(  0,p\right)  $ to $u_{p}$ and every $\left(
1,q\right)  $ to $v_{q}$.
\end{itemize}

This bijection $\Phi$ shows that the number of all $\mathfrak{q}$-spairs
equals the number of all $\gamma$-smaps. Since the coefficient of
$\mathfrak{q}$ in $M_{\alpha}\left.  \prec\right.  M_{\beta}$ equals the
former number, it thus must equal the latter number. This proves Claim 1.
\end{vershort}

\begin{verlong}
\textit{Proof of Claim 1:} Write the composition $\gamma$ in the form
$\gamma=\left(  \gamma_{1},\gamma_{2},\ldots,\gamma_{n}\right)  $. Write the
monomial $\mathfrak{q}$ in the form $\mathfrak{q}=x_{k_{1}}^{\gamma_{1}%
}x_{k_{2}}^{\gamma_{2}}\cdots x_{k_{n}}^{\gamma_{n}}$ for some positive
integers $k_{1}$, $k_{2}$, $\ldots$, $k_{n}$ satisfying $k_{1}<k_{2}%
<\cdots<k_{n}$. (This is possible because $\left(  \gamma_{1},\gamma
_{2},\ldots,\gamma_{n}\right)  =\gamma$ is the Parikh composition of
$\mathfrak{q}$.) Then, $\operatorname*{Supp}\mathfrak{q}=\left\{  k_{1}%
,k_{2},\ldots,k_{n}\right\}  $.

From (\ref{pf.prop.QSym.closed.Malpha-via-Parikh}), we get $M_{\alpha}%
=\sum_{\substack{\mathfrak{m}\in\operatorname*{Mon};\\\operatorname*{Parikh}%
\mathfrak{m}=\alpha}}\mathfrak{m}$. Similarly, $M_{\beta}=\sum
_{\substack{\mathfrak{n}\in\operatorname*{Mon};\\\operatorname*{Parikh}%
\mathfrak{n}=\beta}}\mathfrak{n}$. Hence,%
\begin{align*}
&  M_{\alpha}\left.  \prec\right.  M_{\beta}\\
&  =\left(  \sum_{\substack{\mathfrak{m}\in\operatorname*{Mon}%
;\\\operatorname*{Parikh}\mathfrak{m}=\alpha}}\mathfrak{m}\right)  \left.
\prec\right.  \left(  \sum_{\substack{\mathfrak{n}\in\operatorname*{Mon}%
;\\\operatorname*{Parikh}\mathfrak{n}=\beta}}\mathfrak{n}\right) \\
&  =\sum_{\substack{\mathfrak{m}\in\operatorname*{Mon}%
;\\\operatorname*{Parikh}\mathfrak{m}=\alpha}}\ \ \sum_{\substack{\mathfrak{n}%
\in\operatorname*{Mon};\\\operatorname*{Parikh}\mathfrak{n}=\beta
}}\underbrace{\mathfrak{m}\left.  \prec\right.  \mathfrak{n}}_{\substack{=%
\begin{cases}
\mathfrak{mn}, & \text{if }\min\left(  \operatorname*{Supp}\mathfrak{m}%
\right)  <\min\left(  \operatorname*{Supp}\mathfrak{n}\right)  ;\\
0, & \text{if }\min\left(  \operatorname*{Supp}\mathfrak{m}\right)  \geq
\min\left(  \operatorname*{Supp}\mathfrak{n}\right)
\end{cases}
\\\text{(by the definition of }\left.  \prec\right.  \text{ on monomials)}}}\\
&  \ \ \ \ \ \ \ \ \ \ \left(  \text{since the operation }\left.
\prec\right.  \text{ is }\mathbf{k}\text{-bilinear and continuous}\right) \\
&  =\sum_{\substack{\mathfrak{m}\in\operatorname*{Mon}%
;\\\operatorname*{Parikh}\mathfrak{m}=\alpha}}\ \ \sum_{\substack{\mathfrak{n}%
\in\operatorname*{Mon};\\\operatorname*{Parikh}\mathfrak{n}=\beta}}%
\begin{cases}
\mathfrak{mn}, & \text{if }\min\left(  \operatorname*{Supp}\mathfrak{m}%
\right)  <\min\left(  \operatorname*{Supp}\mathfrak{n}\right)  ;\\
0, & \text{if }\min\left(  \operatorname*{Supp}\mathfrak{m}\right)  \geq
\min\left(  \operatorname*{Supp}\mathfrak{n}\right)
\end{cases}
\\
&  =\sum_{\substack{\left(  \mathfrak{m},\mathfrak{n}\right)  \in
\operatorname*{Mon}\times\operatorname*{Mon};\\\operatorname*{Parikh}%
\mathfrak{m}=\alpha;\\\operatorname*{Parikh}\mathfrak{n}=\beta;\\\min\left(
\operatorname*{Supp}\mathfrak{m}\right)  <\min\left(  \operatorname*{Supp}%
\mathfrak{n}\right)  }}\mathfrak{mn}.
\end{align*}
Thus, the coefficient of $\mathfrak{q}$ in $M_{\alpha}\left.  \prec\right.
M_{\beta}$ equals the number of all pairs $\left(  \mathfrak{m},\mathfrak{n}%
\right)  \in\operatorname*{Mon}\times\operatorname*{Mon}$ such that
$\operatorname*{Parikh}\mathfrak{m}=\alpha$, $\operatorname*{Parikh}%
\mathfrak{n}=\beta$, $\min\left(  \operatorname*{Supp}\mathfrak{m}\right)
<\min\left(  \operatorname*{Supp}\mathfrak{n}\right)  $ and $\mathfrak{mn}%
=\mathfrak{q}$. These pairs shall be called \textit{spairs}. (The concept of a
spair depends on $\mathfrak{q}$; we nevertheless omit $\mathfrak{q}$ from the
notation, since we regard $\mathfrak{q}$ as fixed.)

Now, we shall construct a bijection between the $\gamma$-smaps and the spairs.

Indeed, we first define a map $\Phi$ from the set of $\gamma$-smaps to the set
of spairs as follows: Let $f:\mathcal{S}\rightarrow\left\{  1,2,\ldots
,n\right\}  $ be a $\gamma$-smap. Then, $\Phi\left(  f\right)  $ is defined to
be the spair%
\[
\left(  \prod_{p=1}^{\ell}x_{k_{f\left(  0,p\right)  }}^{\alpha_{p}}%
,\prod_{q=1}^{m}x_{k_{f\left(  1,q\right)  }}^{\beta_{q}}\right)  .
\]
\footnote{This is a well-defined spair, for the following reasons:
\par
\begin{itemize}
\item The first defining property of a $\gamma$-smap can be rewritten as
\textquotedblleft$f\left(  0,1\right)  <f\left(  0,2\right)  <\cdots<f\left(
0,\ell\right)  $ and $f\left(  1,1\right)  <f\left(  1,2\right)
<\cdots<f\left(  1,m\right)  $\textquotedblright. Combined with $k_{1}%
<k_{2}<\cdots<k_{n}$, this shows that $k_{f\left(  0,1\right)  }<k_{f\left(
0,2\right)  }<\cdots<k_{f\left(  0,\ell\right)  }$ and $k_{f\left(
1,1\right)  }<k_{f\left(  1,2\right)  }<\cdots<k_{f\left(  1,m\right)  }$.
Hence, $\operatorname*{Parikh}\left(  \prod_{p=1}^{\ell}x_{k_{f\left(
0,p\right)  }}^{\alpha_{p}}\right)  =\alpha$ and $\operatorname*{Parikh}%
\left(  \prod_{q=1}^{m}x_{k_{f\left(  1,q\right)  }}^{\beta_{q}}\right)
=\beta$.
\par
\item Set $k_{\infty}=\infty$. Thus, $k_{1}<k_{2}<\cdots<k_{n}<k_{\infty}$.
But the second defining property of a $\gamma$-smap shows that $\min\left(
f\left(  \mathcal{S}_{0}\right)  \right)  <\min\left(  f\left(  \mathcal{S}%
_{1}\right)  \right)  $, so that $k_{\min\left(  f\left(  \mathcal{S}%
_{0}\right)  \right)  }<k_{\min\left(  f\left(  \mathcal{S}_{1}\right)
\right)  }$ (since $k_{1}<k_{2}<\cdots<k_{n}<k_{\infty}$). But
$\operatorname*{Supp}\left(  \prod_{p=1}^{\ell}x_{k_{f\left(  0,p\right)  }%
}^{\alpha_{p}}\right)  =\left\{  k_{f\left(  s\right)  }\ \mid\ s\in
\mathcal{S}_{0}\right\}  $ and thus $\min\left(  \operatorname*{Supp}\left(
\prod_{p=1}^{\ell}x_{k_{f\left(  0,p\right)  }}^{\alpha_{p}}\right)  \right)
=\min\left\{  k_{f\left(  s\right)  }\ \mid\ s\in\mathcal{S}_{0}\right\}
=k_{\min\left(  f\left(  \mathcal{S}_{0}\right)  \right)  }$ (since
$k_{1}<k_{2}<\cdots<k_{n}$ and $k_{\infty}=\infty$). Similarly, $\min\left(
\operatorname*{Supp}\left(  \prod_{q=1}^{m}x_{k_{f\left(  1,q\right)  }%
}^{\beta_{q}}\right)  \right)  =k_{\min\left(  f\left(  \mathcal{S}%
_{1}\right)  \right)  }$. Hence,%
\[
\min\left(  \operatorname*{Supp}\left(  \prod_{p=1}^{\ell}x_{k_{f\left(
0,p\right)  }}^{\alpha_{p}}\right)  \right)  =k_{\min\left(  f\left(
\mathcal{S}_{0}\right)  \right)  }<k_{\min\left(  f\left(  \mathcal{S}%
_{1}\right)  \right)  }=\min\left(  \operatorname*{Supp}\left(  \prod
_{q=1}^{m}x_{k_{f\left(  1,q\right)  }}^{\beta_{q}}\right)  \right)  .
\]
\par
\item The third defining property of a $\gamma$-smap shows that $\sum_{s\in
f^{-1}\left(  u\right)  }\rho\left(  s\right)  =\gamma_{u}$ for every
$u\in\left\{  1,2,\ldots,n\right\}  $. Now, every $p\in\left\{  1,2,\ldots
,\ell\right\}  $ satisfies $\alpha_{p}=\rho\left(  0,p\right)  $. Hence,
$\prod_{p=1}^{\ell}x_{k_{f\left(  0,p\right)  }}^{\alpha_{p}}=\prod
_{p=1}^{\ell}x_{k_{f\left(  0,p\right)  }}^{\rho\left(  0,p\right)  }%
=\prod_{s\in\mathcal{S}_{0}}x_{k_{f\left(  s\right)  }}^{\rho\left(  s\right)
}$. Similarly, $\prod_{q=1}^{m}x_{k_{f\left(  1,q\right)  }}^{\beta_{q}}%
=\prod_{s\in\mathcal{S}_{1}}x_{k_{f\left(  s\right)  }}^{\rho\left(  s\right)
}$. Multiplying these two identities, we obtain%
\begin{align*}
\left(  \prod_{p=1}^{\ell}x_{k_{f\left(  0,p\right)  }}^{\alpha_{p}}\right)
\left(  \prod_{q=1}^{m}x_{k_{f\left(  1,q\right)  }}^{\beta_{q}}\right)   &
=\left(  \prod_{s\in\mathcal{S}_{0}}x_{k_{f\left(  s\right)  }}^{\rho\left(
s\right)  }\right)  \left(  \prod_{s\in\mathcal{S}_{1}}x_{k_{f\left(
s\right)  }}^{\rho\left(  s\right)  }\right)  =\prod_{s\in\mathcal{S}%
}x_{k_{f\left(  s\right)  }}^{\rho\left(  s\right)  }=\prod_{u=1}^{n}%
\prod_{s\in f^{-1}\left(  u\right)  }\underbrace{x_{k_{f\left(  s\right)  }%
}^{\rho\left(  s\right)  }}_{\substack{=x_{k_{u}}^{\rho\left(  s\right)
}\\\text{(since }f\left(  s\right)  =u\text{)}}}\\
&  =\prod_{u=1}^{n}\underbrace{\prod_{s\in f^{-1}\left(  u\right)  }x_{k_{u}%
}^{\rho\left(  s\right)  }}_{\substack{=x_{k_{u}}^{\gamma_{u}}\\\text{(since
}\sum_{s\in f^{-1}\left(  u\right)  }\rho\left(  s\right)  =\gamma_{u}%
\text{)}}}=\prod_{u=1}^{n}x_{k_{u}}^{\gamma_{u}}=x_{k_{1}}^{\gamma_{1}%
}x_{k_{2}}^{\gamma_{2}}\cdots x_{k_{n}}^{\gamma_{n}}=\mathfrak{q}.
\end{align*}
\end{itemize}
}

Conversely, we define a map $\Psi$ from the set of spairs to the set of
$\gamma$-smaps as follows: Let $\left(  \mathfrak{m},\mathfrak{n}\right)  $ be
a spair. Then, we write the monomial $\mathfrak{m}$ in the form $\mathfrak{m}%
=x_{i_{1}}^{\alpha_{1}}x_{i_{2}}^{\alpha_{2}}\cdots x_{i_{\ell}}^{\alpha
_{\ell}}$ for some positive integers $i_{1}$, $i_{2}$, $\ldots$, $i_{\ell}$
satisfying $i_{1}<i_{2}<\cdots<i_{\ell}$ (this is possible since
$\operatorname*{Parikh}\mathfrak{m}=\alpha$), and we write the monomial
$\mathfrak{n}$ in the form $\mathfrak{n}=x_{j_{1}}^{\beta_{1}}x_{j_{2}}%
^{\beta_{2}}\cdots x_{j_{m}}^{\beta_{m}}$ for some positive integers $j_{1}$,
$j_{2}$, $\ldots$, $j_{m}$ satisfying $j_{1}<j_{2}<\cdots<j_{m}$ (this is
possible since $\operatorname*{Parikh}\mathfrak{n}=\beta$). Of course,
$\operatorname*{Supp}\mathfrak{m}=\left\{  i_{1},i_{2},\ldots,i_{\ell
}\right\}  $ and $\operatorname*{Supp}\mathfrak{n}=\left\{  j_{1},j_{2}%
,\ldots,j_{m}\right\}  $, so that $\min\left\{  i_{1},i_{2},\ldots,i_{\ell
}\right\}  <\min\left\{  j_{1},j_{2},\ldots,j_{m}\right\}  $ (since
$\min\left(  \operatorname*{Supp}\mathfrak{m}\right)  <\min\left(
\operatorname*{Supp}\mathfrak{n}\right)  $).

Now, we define a map $f:\mathcal{S}\rightarrow\left\{  1,2,\ldots,n\right\}  $
as follows:

\begin{itemize}
\item For every $p\in\left\{  1,2,\ldots,\ell\right\}  $, we let $f\left(
0,p\right)  $ be the unique $r\in\left\{  1,2,\ldots,n\right\}  $ such that
$i_{p}=k_{r}$.\ \ \ \ \footnote{To prove that this is well-defined, we need to
show that this $r$ exists and is unique. The uniqueness of $r$ is obvious
(since $k_{1}<k_{2}<\cdots<k_{n}$). To prove its existence, we notice that
$i_{p}\in\operatorname*{Supp}\mathfrak{m}$ (since $\mathfrak{m}=x_{i_{1}%
}^{\alpha_{1}}x_{i_{2}}^{\alpha_{2}}\cdots x_{i_{\ell}}^{\alpha_{\ell}}$ and
$\alpha_{p}>0$) and thus $i_{p}\in\operatorname*{Supp}\mathfrak{m}%
\subseteq\operatorname*{Supp}\underbrace{\left(  \mathfrak{mn}\right)
}_{=\mathfrak{q}}=\operatorname*{Supp}\mathfrak{q}=\left\{  k_{1},k_{2}%
,\ldots,k_{n}\right\}  $.}

\item For every $q\in\left\{  1,2,\ldots,m\right\}  $, we let $f\left(
1,q\right)  $ be the unique $r\in\left\{  1,2,\ldots,n\right\}  $ such that
$j_{q}=k_{r}$.\ \ \ \ \footnote{This is again well-defined, for similar
reasons as the $r$ in the definition of $f\left(  0,p\right)  $.}
\end{itemize}

It is now straightforward to show that $f$ is a $\gamma$%
-smap.\footnote{Indeed:
\par
\begin{itemize}
\item The first defining property of a $\gamma$-smap holds. (\textit{Proof:}
Let us show that $f\circ\operatorname*{inc}\nolimits_{0}$ is strictly
increasing (the proof for $f\circ\operatorname*{inc}\nolimits_{1}$ is
similar). Assume it is not. Then there exist some $p,p^{\prime}\in\left\{
1,2,\ldots,\ell\right\}  $ satisfying $p<p^{\prime}$ and $\left(
f\circ\operatorname*{inc}\nolimits_{0}\right)  \left(  p\right)  \geq\left(
f\circ\operatorname*{inc}\nolimits_{0}\right)  \left(  p^{\prime}\right)  $.
Consider these $p,p^{\prime}$. We have $p<p^{\prime}$, and therefore
$i_{p}<i_{p^{\prime}}$ (since $i_{1}<i_{2}<\cdots<i_{\ell}$). But $\left(
f\circ\operatorname*{inc}\nolimits_{0}\right)  \left(  p\right)  \geq\left(
f\circ\operatorname*{inc}\nolimits_{0}\right)  \left(  p^{\prime}\right)  $,
and thus $k_{\left(  f\circ\operatorname*{inc}\nolimits_{0}\right)  \left(
p\right)  }\geq k_{\left(  f\circ\operatorname*{inc}\nolimits_{0}\right)
\left(  p^{\prime}\right)  }$ (since $k_{1}<k_{2}<\cdots<k_{n}$). Since
$k_{\left(  f\circ\operatorname*{inc}\nolimits_{0}\right)  \left(  p\right)
}=k_{f\left(  0,p\right)  }=i_{p}$ (by the definition of $f\left(  0,p\right)
$) and similarly $k_{\left(  f\circ\operatorname*{inc}\nolimits_{0}\right)
\left(  p^{\prime}\right)  }=i_{p^{\prime}}$, this rewrites as $i_{p}\geq
i_{p^{\prime}}$. This contradicts $i_{p}<i_{p^{\prime}}$. This contradiction
completes the proof.)
\par
\item The second defining property of a $\gamma$-smap holds. (\textit{Proof:}
We WLOG assume that $\ell$ and $m$ are positive, since the other case is
straightforward. We have $i_{1}<i_{2}<\cdots<i_{\ell}$. In other words,
$k_{f\left(  0,1\right)  }<k_{f\left(  0,2\right)  }<\cdots<k_{f\left(
0,\ell\right)  }$ (since $k_{f\left(  0,p\right)  }=i_{p}$ for every
$p\in\left\{  1,2,\ldots,\ell\right\}  $). Hence, $f\left(  0,1\right)
<f\left(  0,2\right)  <\cdots<f\left(  0,\ell\right)  $ (since $k_{1}%
<k_{2}<\cdots<k_{n}$). Hence, $\min\left(  f\left(  \mathcal{S}_{0}\right)
\right)  =f\left(  0,1\right)  $. Similarly, $\min\left(  f\left(
\mathcal{S}_{1}\right)  \right)  =f\left(  1,1\right)  $. But from
$i_{1}<i_{2}<\cdots<i_{\ell}$, we obtain $i_{1}=\min\left\{  i_{1}%
,i_{2},\ldots,i_{\ell}\right\}  $; similarly, $j_{1}=\min\left\{  j_{1}%
,j_{2},\ldots,j_{m}\right\}  $. Hence, $k_{f\left(  0,1\right)  }=i_{1}%
=\min\left\{  i_{1},i_{2},\ldots,i_{\ell}\right\}  <\min\left\{  j_{1}%
,j_{2},\ldots,j_{m}\right\}  =j_{1}=k_{f\left(  1,1\right)  }$, so that
$f\left(  0,1\right)  <f\left(  1,1\right)  $ (since $k_{1}<k_{2}<\cdots
<k_{n}$). Hence, $\min\left(  f\left(  \mathcal{S}_{0}\right)  \right)
=f\left(  0,1\right)  <f\left(  1,1\right)  =\min\left(  f\left(
\mathcal{S}_{1}\right)  \right)  $, qed.)
\par
\item The third defining property of a $\gamma$-smap holds. (\textit{Proof:}
We have%
\[
\mathfrak{m}=x_{i_{1}}^{\alpha_{1}}x_{i_{2}}^{\alpha_{2}}\cdots x_{i_{\ell}%
}^{\alpha_{\ell}}=\prod_{p=1}^{\ell}\underbrace{x_{i_{p}}^{\alpha_{p}}%
}_{\substack{=x_{k_{f\left(  0,p\right)  }}^{\rho\left(  0,p\right)
}\\\text{(since }\alpha_{p}=\rho\left(  0,p\right)  \\\text{and }%
i_{p}=k_{f\left(  0,p\right)  }\text{)}}}=\prod_{p=1}^{\ell}x_{k_{f\left(
0,p\right)  }}^{\rho\left(  0,p\right)  }=\prod_{s\in\mathcal{S}_{0}%
}x_{k_{f\left(  s\right)  }}^{\rho\left(  s\right)  }%
\]
and similarly $\mathfrak{n}=\prod_{s\in\mathcal{S}_{1}}x_{k_{f\left(
s\right)  }}^{\rho\left(  s\right)  }$. Hence,%
\[
\mathfrak{mn}=\left(  \prod_{s\in\mathcal{S}_{0}}x_{k_{f\left(  s\right)  }%
}^{\rho\left(  s\right)  }\right)  \left(  \prod_{s\in\mathcal{S}_{1}%
}x_{k_{f\left(  s\right)  }}^{\rho\left(  s\right)  }\right)  =\prod
_{s\in\mathcal{S}}x_{k_{f\left(  s\right)  }}^{\rho\left(  s\right)  }
\qquad\qquad\left(  \text{since } \mathcal{S} = \mathcal{S}_{0} \cup
\mathcal{S}_{1} \text{ and } \mathcal{S}_{0} \cap\mathcal{S}_{1} =
\varnothing\right)  .
\]
Thus, $\prod_{s\in\mathcal{S}}x_{k_{f\left(  s\right)  }}^{\rho\left(
s\right)  }=\mathfrak{mn}=\mathfrak{q}=x_{k_{1}}^{\gamma_{1}}x_{k_{2}}%
^{\gamma_{2}}\cdots x_{k_{n}}^{\gamma_{n}}$. Now, for any $u\in\left\{
1,2,\ldots,n\right\}  $, the exponent of $x_{k_{u}}$ on the left hand side of
this equality is $\sum_{s\in f^{-1}\left(  u\right)  }\rho\left(  s\right)  $
(since $k_{1}<k_{2}<\cdots<k_{n}$), whereas the exponent of $x_{k_{u}}$ on the
right hand side is $\gamma_{u}$. Comparing these coefficients, we find
$\sum_{s\in f^{-1}\left(  u\right)  }\rho\left(  s\right)  =\gamma_{u}$.)
\end{itemize}
} We define $\Psi\left(  \mathfrak{m},\mathfrak{n}\right)  $ to be this
$\gamma$-smap $f$.

We thus have defined a map $\Phi$ from the set of $\gamma$-smaps to the set of
spairs, and a map $\Psi$ from the set of spairs to the set of $\gamma$-smaps.
It is straightforward to see that these two maps $\Phi$ and $\Psi$ are
mutually inverse, and thus $\Phi$ is a bijection. We thus have found a
bijection between the set of $\gamma$-smaps and the set of spairs.
Consequently, the number of all $\gamma$-smaps equals the number of all spairs.

Now, recall that the coefficient of $\mathfrak{q}$ in $M_{\alpha}\left.
\prec\right.  M_{\beta}$ equals the number of all spairs. Hence, the
coefficient of $\mathfrak{q}$ in $M_{\alpha}\left.  \prec\right.  M_{\beta}$
equals the number of all $\gamma$-smaps (since the number of all $\gamma
$-smaps equals the number of all spairs). In other words, Claim 1 is proven.
\end{verlong}

Claim 1 shows that the coefficient of a monomial $\mathfrak{q}$ in $M_{\alpha
}\left.  \prec\right.  M_{\beta}$ depends not on $\mathfrak{q}$ but only on
the Parikh composition of $\mathfrak{q}$. Thus, any two pack-equivalent
monomials have equal coefficients in $M_{\alpha}\left.  \prec\right.
M_{\beta}$ (since any two pack-equivalent monomials have the same Parikh
composition). In other words, the power series $M_{\alpha}\left.
\prec\right.  M_{\beta}$ is quasisymmetric. Since $M_{\alpha}\left.
\prec\right.  M_{\beta}\in\mathbf{k}\left[  \left[  x_{1},x_{2},x_{3}%
,\ldots\right]  \right]  _{\operatorname*{bdd}}$, this yields that $M_{\alpha
}\left.  \prec\right.  M_{\beta}\in\operatorname*{QSym}$.

[At this point, let us remark that we can give an explicit formula for
$M_{\alpha}\left.  \prec\right.  M_{\beta}$: Namely,%
\begin{equation}
M_{\alpha}\left.  \prec\right.  M_{\beta}=\sum_{\gamma\in\operatorname*{Comp}%
}\mathfrak{s}_{\alpha,\beta}^{\gamma}M_{\gamma},
\label{pf.QSym.closed.eq.rmk.smaps.1}%
\end{equation}
where $\mathfrak{s}_{\alpha,\beta}^{\gamma}$ is the number of all $\gamma
$-smaps. Indeed, for every monomial $\mathfrak{q}$, the coefficient of
$\mathfrak{q}$ on the left-hand side of (\ref{pf.QSym.closed.eq.rmk.smaps.1})
equals $\mathfrak{s}_{\alpha,\beta}^{\gamma}$ where $\gamma$ is the Parikh
composition of $\mathfrak{q}$ (because of Claim 1), whereas the coefficient of
$\mathfrak{q}$ on the right-hand side of (\ref{pf.QSym.closed.eq.rmk.smaps.1})
also equals $\mathfrak{s}_{\alpha,\beta}^{\gamma}$ (for obvious reasons).
Hence, every monomial has equal coefficients on the two sides of
(\ref{pf.QSym.closed.eq.rmk.smaps.1}), and so
(\ref{pf.QSym.closed.eq.rmk.smaps.1}) holds. Of course,
(\ref{pf.QSym.closed.eq.rmk.smaps.1}) again proves that $M_{\alpha}\left.
\prec\right.  M_{\beta}\in\operatorname*{QSym}$, since the sum $\sum
_{\gamma\in\operatorname*{Comp}}\mathfrak{s}_{\alpha,\beta}^{\gamma}M_{\gamma
}$ has only finitely many nonzero addends (indeed, $\gamma$-smaps can only
exist if $\left\vert \gamma\right\vert \leq\left\vert \alpha\right\vert
+\left\vert \beta\right\vert $).]

Now, let us forget that we fixed $\alpha$ and $\beta$. We thus have shown that
every two compositions $\alpha$ and $\beta$ satisfy $M_{\alpha}\left.
\prec\right.  M_{\beta}\in\operatorname*{QSym}$.

\begin{vershort}
Since $\left(  M_{\alpha}\right)  _{\alpha\in\operatorname*{Comp}}$ is a basis
of $\operatorname*{QSym}$ (and since $\left.  \prec\right.  $ is $\mathbf{k}%
$-bilinear), this shows that $a\left.  \prec\right.  b\in\operatorname*{QSym}$
for every $a\in\operatorname*{QSym}$ and $b\in\operatorname*{QSym}$. The proof
of $a \bel b\in\operatorname*{QSym}$ is similar\footnote{Alternatively, of
course, $a \bel b\in\operatorname*{QSym}$ can be checked using the formula
$M_{\alpha} \bel M_{\beta}=M_{\left[  \alpha,\beta\right]  }+M_{\alpha
\odot\beta}$ (which is easily proven). However, there is no such simple proof
for $a\left.  \prec\right.  b\in\operatorname*{QSym}$.}.
\end{vershort}

\begin{verlong}
Now, let $a\in\operatorname*{QSym}$ and $b\in\operatorname*{QSym}$. We shall
only prove that $a\left.  \prec\right.  b\in\operatorname*{QSym}$ (since the
proof of $a \bel b\in\operatorname*{QSym}$ is very
similar\footnote{Alternatively, of course, $a \bel b\in\operatorname*{QSym}$
can be checked using the formula $M_{\alpha} \bel M_{\beta}=M_{\left[
\alpha,\beta\right]  }+M_{\alpha\odot\beta}$ (which is easily proven).
However, there is no such simple proof for $a\left.  \prec\right.
b\in\operatorname*{QSym}$.}).

The statement that we need to prove ($a\left.  \prec\right.  b\in
\operatorname*{QSym}$) is $\mathbf{k}$-linear in each of $a$ and $b$. Hence,
we can WLOG assume that both $a$ and $b$ are elements of the monomial basis of
$\operatorname*{QSym}$. Assume this. Thus, $a=M_{\alpha}$ and $b=M_{\beta}$
for some compositions $\alpha$ and $\beta$. Consider these $\alpha$ and
$\beta$. Now, as we know, $M_{\alpha}\left.  \prec\right.  M_{\beta}%
\in\operatorname*{QSym}$, so that $\underbrace{a}_{=M_{\alpha}}\left.
\prec\right.  \underbrace{b}_{=M_{\beta}}=M_{\alpha}\left.  \prec\right.
M_{\beta}\in\operatorname*{QSym}$. This completes our proof of Proposition
\ref{prop.QSym.closed}.
\end{verlong}
\end{proof}

\Needspace{6pc}

\begin{remark}
\label{rmk.QSym.closed}The proof of Proposition \ref{prop.QSym.closed} given
above actually yields a combinatorial formula for $M_{\alpha}\left.
\prec\right.  M_{\beta}$ whenever $\alpha$ and $\beta$ are two compositions.
Namely, let $\alpha$ and $\beta$ be two compositions. Then,%
\begin{equation}
M_{\alpha}\left.  \prec\right.  M_{\beta}=\sum_{\gamma\in\operatorname*{Comp}%
}\mathfrak{s}_{\alpha,\beta}^{\gamma}M_{\gamma}, \label{eq.rmk.smaps.1}%
\end{equation}
where $\mathfrak{s}_{\alpha,\beta}^{\gamma}$ is the number of all smaps
$\left(  \alpha,\beta\right)  \rightarrow\gamma$. Here a \textit{smap}
$\left(  \alpha,\beta\right)  \rightarrow\gamma$ means what was called a
$\gamma$-smap in the above proof of Proposition \ref{prop.QSym.closed}.

This is similar to the well-known formula for $M_{\alpha}M_{\beta}$ (see, for
example, \cite[Proposition 5.1.3]{Reiner}) which (translated into our
language) states that
\begin{equation}
M_{\alpha}M_{\beta}=\sum_{\gamma\in\operatorname*{Comp}}\mathfrak{t}%
_{\alpha,\beta}^{\gamma}M_{\gamma}, \label{eq.rmk.smaps.2}%
\end{equation}
where $\mathfrak{t}_{\alpha,\beta}^{\gamma}$ is the number of all overlapping
shuffles $\left(  \alpha,\beta\right)  \rightarrow\gamma$. Here, the
\textit{overlapping shuffles} $\left(  \alpha,\beta\right)  \rightarrow\gamma$
are defined in the same way as the $\gamma$-smaps, with the only difference
that the second of the three properties that define a $\gamma$-smap (namely,
the property $\min\left(  f\left(  \mathcal{S}_{0}\right)  \right)
<\min\left(  f\left(  \mathcal{S}_{1}\right)  \right)  $) is omitted. Needless
to say, (\ref{eq.rmk.smaps.2}) can be proven similarly to our proof of
(\ref{eq.rmk.smaps.1}) above.
\end{remark}

Here is a somewhat nontrivial property of $\bel $ and $\left.  \prec\right.  $:

\begin{theorem}
\label{thm.beldend}Let $S$ denote the antipode of the Hopf algebra
$\operatorname*{QSym}$. Let us use Sweedler's notation $\sum_{\left(
b\right)  }b_{\left(  1\right)  }\otimes b_{\left(  2\right)  }$ for
$\Delta\left(  b\right)  $, where $b$ is any element of $\operatorname*{QSym}%
$. Then,%
\[
\sum_{\left(  b\right)  }\left(  S\left(  b_{\left(  1\right)  }\right)
\bel
a\right)  b_{\left(  2\right)  }=a\left.  \prec\right.  b
\]
for any $a\in\mathbf{k}\left[  \left[  x_{1},x_{2},x_{3},\ldots\right]
\right]  $ and $b\in\operatorname*{QSym}$ such that the constant term of $a$
is $0$.
\end{theorem}

\begin{proof}
[Proof of Theorem \ref{thm.beldend}.]Let $a\in\mathbf{k}\left[  \left[
x_{1},x_{2},x_{3},\ldots\right]  \right]  $ be such that the constant term of
$a$ is $0$. We can WLOG assume that $a$ is a monomial (because all operations
in sight are $\mathbf{k}$-linear and continuous). So assume this. That is,
$a=\mathfrak{n}$ for some monomial $\mathfrak{n}$. Consider this
$\mathfrak{n}$. This monomial $\mathfrak{n}=a$ is nonconstant (since the
constant term of $a$ is $0$). Hence, $\operatorname*{Supp}\mathfrak{n}%
\neq\varnothing$.

Let $k=\min\left(  \operatorname*{Supp}\mathfrak{n}\right)  $. Thus,
$k\in\left\{  1,2,3,\ldots\right\}  $ (since $\operatorname*{Supp}%
\mathfrak{n}\neq\varnothing$).

\begin{verlong}
Every composition $\alpha=\left(  \alpha_{1},\alpha_{2},\ldots,\alpha_{\ell
}\right)  $ satisfies%
\begin{equation}
a\left.  \prec\right.  M_{\alpha}=\left(  \sum_{k<i_{1}<i_{2}<\cdots<i_{\ell}%
}x_{i_{1}}^{\alpha_{1}}x_{i_{2}}^{\alpha_{2}}\cdots x_{i_{\ell}}^{\alpha
_{\ell}}\right)  \cdot a \label{pf.thm.beldend.a-Dless-Malpha}%
\end{equation}
\footnote{\textit{Proof of (\ref{pf.thm.beldend.a-Dless-Malpha}):} Let
$\alpha=\left(  \alpha_{1},\alpha_{2},\ldots,\alpha_{\ell}\right)  $ be a
composition. The definition of $M_{\alpha}$ yields $M_{\alpha}=\sum_{1\leq
i_{1}<i_{2}<\cdots<i_{\ell}}x_{i_{1}}^{\alpha_{1}}x_{i_{2}}^{\alpha_{2}}\cdots
x_{i_{\ell}}^{\alpha_{\ell}}$. Combined with $a=\mathfrak{n}$, this yields%
\begin{align*}
a\left.  \prec\right.  M_{\alpha}  &  =\mathfrak{n}\left.  \prec\right.
\left(  \sum_{1\leq i_{1}<i_{2}<\cdots<i_{\ell}}x_{i_{1}}^{\alpha_{1}}%
x_{i_{2}}^{\alpha_{2}}\cdots x_{i_{\ell}}^{\alpha_{\ell}}\right) \\
&  =\sum_{1\leq i_{1}<i_{2}<\cdots<i_{\ell}}\underbrace{\mathfrak{n}\left.
\prec\right.  \left(  x_{i_{1}}^{\alpha_{1}}x_{i_{2}}^{\alpha_{2}}\cdots
x_{i_{\ell}}^{\alpha_{\ell}}\right)  }_{\substack{=%
\begin{cases}
\mathfrak{n}\cdot x_{i_{1}}^{\alpha_{1}}x_{i_{2}}^{\alpha_{2}}\cdots
x_{i_{\ell}}^{\alpha_{\ell}}, & \text{if }\min\left(  \operatorname*{Supp}%
\mathfrak{n}\right)  <\min\left\{  i_{1},i_{2},\ldots,i_{\ell}\right\}  ;\\
0, & \text{if }\min\left(  \operatorname*{Supp}\mathfrak{n}\right)  \geq
\min\left\{  i_{1},i_{2},\ldots,i_{\ell}\right\}
\end{cases}
\\\text{(by the definition of }\left.  \prec\right.  \text{ on monomials)}}}\\
&  \ \ \ \ \ \ \ \ \ \ \left(  \text{since }\left.  \prec\right.  \text{ is
}\mathbf{k}\text{-bilinear and continuous}\right) \\
&  =\sum_{1\leq i_{1}<i_{2}<\cdots<i_{\ell}}%
\begin{cases}
\mathfrak{n}\cdot x_{i_{1}}^{\alpha_{1}}x_{i_{2}}^{\alpha_{2}}\cdots
x_{i_{\ell}}^{\alpha_{\ell}}, & \text{if }\min\left(  \operatorname*{Supp}%
\mathfrak{n}\right)  <\min\left\{  i_{1},i_{2},\ldots,i_{\ell}\right\}  ;\\
0, & \text{if }\min\left(  \operatorname*{Supp}\mathfrak{n}\right)  \geq
\min\left\{  i_{1},i_{2},\ldots,i_{\ell}\right\}
\end{cases}
\\
&  =\underbrace{\sum_{\substack{1\leq i_{1}<i_{2}<\cdots<i_{\ell}%
;\\\min\left(  \operatorname*{Supp}\mathfrak{n}\right)  <\min\left\{
i_{1},i_{2},\ldots,i_{\ell}\right\}  }}}_{\substack{=\sum_{\min\left(
\operatorname*{Supp}\mathfrak{n}\right)  <i_{1}<i_{2}<\cdots<i_{\ell}}%
\\=\sum_{k<i_{1}<i_{2}<\cdots<i_{\ell}}\\\text{(since }\min\left(
\operatorname*{Supp}\mathfrak{n}\right)  =k\text{)}}}\underbrace{\mathfrak{n}%
}_{=a}\cdot x_{i_{1}}^{\alpha_{1}}x_{i_{2}}^{\alpha_{2}}\cdots x_{i_{\ell}%
}^{\alpha_{\ell}}=\sum_{k<i_{1}<i_{2}<\cdots<i_{\ell}}a\cdot x_{i_{1}}%
^{\alpha_{1}}x_{i_{2}}^{\alpha_{2}}\cdots x_{i_{\ell}}^{\alpha_{\ell}}\\
&  =\left(  \sum_{k<i_{1}<i_{2}<\cdots<i_{\ell}}x_{i_{1}}^{\alpha_{1}}%
x_{i_{2}}^{\alpha_{2}}\cdots x_{i_{\ell}}^{\alpha_{\ell}}\right)  \cdot a.
\end{align*}
This proves (\ref{pf.thm.beldend.a-Dless-Malpha}).}.

Let us define a map $\mathfrak{B}_{k}:\mathbf{k}\left[  \left[  x_{1}%
,x_{2},x_{3},\ldots\right]  \right]  \rightarrow\mathbf{k}\left[  \left[
x_{1},x_{2},x_{3},\ldots\right]  \right]  $ by%
\[
\mathfrak{B}_{k}\left(  p\right)  =p\left(  x_{1},x_{2},\ldots,x_{k}%
,0,0,0,\ldots\right)  \ \ \ \ \ \ \ \ \ \ \text{for every }p\in\mathbf{k}%
\left[  \left[  x_{1},x_{2},x_{3},\ldots\right]  \right]  .
\]
Then, $\mathfrak{B}_{k}$ is an evaluation map (in an appropriate sense) and
thus a continuous $\mathbf{k}$-algebra homomorphism. Any monomial
$\mathfrak{m}$ satisfies%
\begin{equation}
\mathfrak{B}_{k}\left(  \mathfrak{m}\right)  =%
\begin{cases}
\mathfrak{m}, & \text{if }\max\left(  \operatorname*{Supp}\mathfrak{m}\right)
\leq k;\\
0, & \text{if }\max\left(  \operatorname*{Supp}\mathfrak{m}\right)  >k
\end{cases}
\label{pf.thm.beldend.monom}%
\end{equation}
\footnote{\textit{Proof.} Let $\mathfrak{m}$ be a monomial. Then,%
\begin{align*}
\mathfrak{B}_{k}\left(  \mathfrak{m}\right)   &  =\mathfrak{m}\left(
x_{1},x_{2},\ldots,x_{k},0,0,0,\ldots\right)  \ \ \ \ \ \ \ \ \ \ \left(
\text{by the definition of }\mathfrak{B}_{k}\right) \\
&  =\left(  \text{the result of replacing the indeterminates }x_{k+1}%
,x_{k+2},x_{k+3},\ldots\text{ by }0\text{ in }\mathfrak{m}\right) \\
&  =%
\begin{cases}
\mathfrak{m}, & \text{if none of the indeterminates }x_{k+1},x_{k+2}%
,x_{k+3},\ldots\text{ appears in }\mathfrak{m}\text{;}\\
0, & \text{if some of the indeterminates }x_{k+1},x_{k+2},x_{k+3},\ldots\text{
appear in }\mathfrak{m}%
\end{cases}
\\
&  =%
\begin{cases}
\mathfrak{m}, & \text{if }\max\left(  \operatorname*{Supp}\mathfrak{m}\right)
\leq k;\\
0, & \text{if }\max\left(  \operatorname*{Supp}\mathfrak{m}\right)  >k
\end{cases}
\end{align*}
(because none of the indeterminates $x_{k+1},x_{k+2},x_{k+3},\ldots$ appears
in $\mathfrak{m}$ if and only if $\max\left(  \operatorname*{Supp}%
\mathfrak{m}\right)  \leq k$). This proves (\ref{pf.thm.beldend.monom}).}. Any
$p\in\mathbf{k}\left[  \left[  x_{1},x_{2},x_{3},\ldots\right]  \right]  $
satisfies%
\begin{equation}
p\bel a=a\cdot\mathfrak{B}_{k}\left(  p\right)  \label{pf.thm.beldend.B}%
\end{equation}
\footnote{\textit{Proof of (\ref{pf.thm.beldend.B}):} Fix $p\in\mathbf{k}%
\left[  \left[  x_{1},x_{2},x_{3},\ldots\right]  \right]  $. Since the
equality (\ref{pf.thm.beldend.B}) is $\mathbf{k}$-linear and continuous in
$p$, we can WLOG assume that $p$ is a monomial. Assume this. Hence,
$p=\mathfrak{m}$ for some monomial $\mathfrak{m}$. Consider this
$\mathfrak{m}$. We have%
\begin{equation}
\mathfrak{B}_{k}\left(  \underbrace{p}_{=\mathfrak{m}}\right)  =\mathfrak{B}%
_{k}\left(  \mathfrak{m}\right)  =%
\begin{cases}
\mathfrak{m}, & \text{if }\max\left(  \operatorname*{Supp}\mathfrak{m}\right)
\leq k;\\
0, & \text{if }\max\left(  \operatorname*{Supp}\mathfrak{m}\right)  >k
\end{cases}
\label{pf.thm.beldend.B.pf.2}%
\end{equation}
(by (\ref{pf.thm.beldend.monom})). Now,%
\begin{align*}
\underbrace{p}_{=\mathfrak{m}}\bel\underbrace{a}_{=\mathfrak{n}}  &
=\mathfrak{m}\bel\mathfrak{n}=%
\begin{cases}
\mathfrak{m}\cdot\mathfrak{n}, & \text{if }\max\left(  \operatorname*{Supp}%
\mathfrak{m}\right)  \leq\min\left(  \operatorname*{Supp}\mathfrak{n}\right)
;\\
0, & \text{if }\max\left(  \operatorname*{Supp}\mathfrak{m}\right)
>\min\left(  \operatorname*{Supp}\mathfrak{n}\right)
\end{cases}
\\
&  \ \ \ \ \ \ \ \ \ \ \left(  \text{by the definition of }\bel\right) \\
&  =%
\begin{cases}
\mathfrak{m}\cdot\mathfrak{n}, & \text{if }\max\left(  \operatorname*{Supp}%
\mathfrak{m}\right)  \leq k;\\
0, & \text{if }\max\left(  \operatorname*{Supp}\mathfrak{m}\right)  >k
\end{cases}
\ \ \ \ \ \ \ \ \ \ \left(  \text{since }\min\left(  \operatorname*{Supp}%
\mathfrak{n}\right)  =k\right) \\
&  =\underbrace{\mathfrak{n}}_{=a}\cdot\underbrace{%
\begin{cases}
\mathfrak{m}, & \text{if }\max\left(  \operatorname*{Supp}\mathfrak{m}\right)
\leq k;\\
0, & \text{if }\max\left(  \operatorname*{Supp}\mathfrak{m}\right)  >k
\end{cases}
}_{\substack{=\mathfrak{B}_{k}\left(  p\right)  \\\text{(by
(\ref{pf.thm.beldend.B.pf.2}))}}}\\
&  =a\cdot\mathfrak{B}_{k}\left(  p\right)  .
\end{align*}
This proves (\ref{pf.thm.beldend.B}).}. Also, every composition $\alpha
=\left(  \alpha_{1},\alpha_{2},\ldots,\alpha_{\ell}\right)  $ satisfies%
\begin{equation}
\mathfrak{B}_{k}\left(  M_{\alpha}\right)  =\sum_{1\leq i_{1}<i_{2}%
<\cdots<i_{\ell}\leq k}x_{i_{1}}^{\alpha_{1}}x_{i_{2}}^{\alpha_{2}}\cdots
x_{i_{\ell}}^{\alpha_{\ell}} \label{pf.thm.beldend.Malpha-bel-a}%
\end{equation}
\footnote{\textit{Proof of (\ref{pf.thm.beldend.Malpha-bel-a}):} Let
$\alpha=\left(  \alpha_{1},\alpha_{2},\ldots,\alpha_{\ell}\right)  $ be a
composition. The definition of $M_{\alpha}$ yields $M_{\alpha}=\sum_{1\leq
i_{1}<i_{2}<\cdots<i_{\ell}}x_{i_{1}}^{\alpha_{1}}x_{i_{2}}^{\alpha_{2}}\cdots
x_{i_{\ell}}^{\alpha_{\ell}}$. Applying the map $\mathfrak{B}_{k}$ to both
sides of this equality, we obtain%
\begin{align*}
\mathfrak{B}_{k}\left(  M_{\alpha}\right)   &  =\mathfrak{B}_{k}\left(
\sum_{1\leq i_{1}<i_{2}<\cdots<i_{\ell}}x_{i_{1}}^{\alpha_{1}}x_{i_{2}%
}^{\alpha_{2}}\cdots x_{i_{\ell}}^{\alpha_{\ell}}\right) \\
&  =\sum_{1\leq i_{1}<i_{2}<\cdots<i_{\ell}}\underbrace{\mathfrak{B}%
_{k}\left(  x_{i_{1}}^{\alpha_{1}}x_{i_{2}}^{\alpha_{2}}\cdots x_{i_{\ell}%
}^{\alpha_{\ell}}\right)  }_{\substack{=%
\begin{cases}
x_{i_{1}}^{\alpha_{1}}x_{i_{2}}^{\alpha_{2}}\cdots x_{i_{\ell}}^{\alpha_{\ell
}}, & \text{if }\max\left(  \operatorname*{Supp}\left(  x_{i_{1}}^{\alpha_{1}%
}x_{i_{2}}^{\alpha_{2}}\cdots x_{i_{\ell}}^{\alpha_{\ell}}\right)  \right)
\leq k;\\
0, & \text{if }\max\left(  \operatorname*{Supp}\left(  x_{i_{1}}^{\alpha_{1}%
}x_{i_{2}}^{\alpha_{2}}\cdots x_{i_{\ell}}^{\alpha_{\ell}}\right)  \right)  >k
\end{cases}
\\\text{(by (\ref{pf.thm.beldend.monom}), applied to }\mathfrak{m}=x_{i_{1}%
}^{\alpha_{1}}x_{i_{2}}^{\alpha_{2}}\cdots x_{i_{\ell}}^{\alpha_{\ell}%
}\text{)}}}\\
&  \ \ \ \ \ \ \ \ \ \ \left(  \text{since }\mathfrak{B}_{k}\text{ is
}\mathbf{k}\text{-linear and continuous}\right) \\
&  =\sum_{1\leq i_{1}<i_{2}<\cdots<i_{\ell}}\underbrace{%
\begin{cases}
x_{i_{1}}^{\alpha_{1}}x_{i_{2}}^{\alpha_{2}}\cdots x_{i_{\ell}}^{\alpha_{\ell
}}, & \text{if }\max\left(  \operatorname*{Supp}\left(  x_{i_{1}}^{\alpha_{1}%
}x_{i_{2}}^{\alpha_{2}}\cdots x_{i_{\ell}}^{\alpha_{\ell}}\right)  \right)
\leq k;\\
0, & \text{if }\max\left(  \operatorname*{Supp}\left(  x_{i_{1}}^{\alpha_{1}%
}x_{i_{2}}^{\alpha_{2}}\cdots x_{i_{\ell}}^{\alpha_{\ell}}\right)  \right)  >k
\end{cases}
}_{\substack{=%
\begin{cases}
x_{i_{1}}^{\alpha_{1}}x_{i_{2}}^{\alpha_{2}}\cdots x_{i_{\ell}}^{\alpha_{\ell
}}, & \text{if }\max\left\{  i_{1},i_{2},\ldots,i_{\ell}\right\}  \leq k;\\
0, & \text{if }\max\left\{  i_{1},i_{2},\ldots,i_{\ell}\right\}  >k
\end{cases}
\\\text{(since }\operatorname*{Supp}\left(  x_{i_{1}}^{\alpha_{1}}x_{i_{2}%
}^{\alpha_{2}}\cdots x_{i_{\ell}}^{\alpha_{\ell}}\right)  =\left\{
i_{1},i_{2},\ldots,i_{\ell}\right\}  \text{)}}}\\
&  =\sum_{1\leq i_{1}<i_{2}<\cdots<i_{\ell}}%
\begin{cases}
x_{i_{1}}^{\alpha_{1}}x_{i_{2}}^{\alpha_{2}}\cdots x_{i_{\ell}}^{\alpha_{\ell
}}, & \text{if }\max\left\{  i_{1},i_{2},\ldots,i_{\ell}\right\}  \leq k;\\
0, & \text{if }\max\left\{  i_{1},i_{2},\ldots,i_{\ell}\right\}  >k
\end{cases}
\\
&  =\underbrace{\sum_{\substack{1\leq i_{1}<i_{2}<\cdots<i_{\ell}%
;\\\max\left\{  i_{1},i_{2},\ldots,i_{\ell}\right\}  \leq k}}}_{=\sum_{1\leq
i_{1}<i_{2}<\cdots<i_{\ell}\leq k}}x_{i_{1}}^{\alpha_{1}}x_{i_{2}}^{\alpha
_{2}}\cdots x_{i_{\ell}}^{\alpha_{\ell}}=\sum_{1\leq i_{1}<i_{2}%
<\cdots<i_{\ell}\leq k}x_{i_{1}}^{\alpha_{1}}x_{i_{2}}^{\alpha_{2}}\cdots
x_{i_{\ell}}^{\alpha_{\ell}}.
\end{align*}
This proves (\ref{pf.thm.beldend.Malpha-bel-a}).}.

Recall that $\infty$ is regarded as an object which is greater than every
integer. We shall also use an additional symbol $\infty+1$, which is
understood to be greater than every element of $\left\{  1,2,3,\ldots\right\}
\cup\left\{  \infty\right\}  $.

We shall use one further obvious observation: If $i_{1},i_{2},\ldots,i_{\ell}$
are some positive integers satisfying $i_{1}<i_{2}<\cdots<i_{\ell}$, then%
\begin{equation}
\text{there exists exactly one }j\in\left\{  0,1,\ldots,\ell\right\}  \text{
satisfying }i_{j}\leq k<i_{j+1} \label{pf.thm.beldend.exactlyonej}%
\end{equation}
(where $i_{0}$ is to be understood as $1$, and $i_{\ell+1}$ as $\infty+1$).

Let us now notice that every $f\in\operatorname*{QSym}$ satisfies%
\begin{equation}
af=\sum_{\left(  f\right)  }\mathfrak{B}_{k}\left(  f_{\left(  1\right)
}\right)  \left(  a\left.  \prec\right.  f_{\left(  2\right)  }\right)  .
\label{pf.thm.beldend.mainlem}%
\end{equation}

\textit{Proof of (\ref{pf.thm.beldend.mainlem}):} Both sides of the equality
(\ref{pf.thm.beldend.mainlem}) are $\mathbf{k}$-linear in $f$. Hence, it is
enough to check (\ref{pf.thm.beldend.mainlem}) on the basis $\left(
M_{\gamma}\right)  _{\gamma\in\operatorname*{Comp}}$ of $\operatorname*{QSym}%
$, that is, to prove that (\ref{pf.thm.beldend.mainlem}) holds whenever
$f=M_{\gamma}$ for some $\gamma\in\operatorname*{Comp}$. In other words, it is
enough to show that
\[
aM_{\gamma}=\sum_{\left(  M_{\gamma}\right)  }\mathfrak{B}_{k}\left(  \left(
M_{\gamma}\right)  _{\left(  1\right)  }\right)  \cdot\left(  a\left.
\prec\right.  \left(  M_{\gamma}\right)  _{\left(  2\right)  }\right)
\ \ \ \ \ \ \ \ \ \ \text{for every }\gamma\in\operatorname*{Comp}.
\]
But this is easily done: Let $\gamma\in\operatorname*{Comp}$. Write $\gamma$
in the form $\gamma=\left(  \gamma_{1},\gamma_{2},\ldots,\gamma_{\ell}\right)
$. Then,%
\begin{align*}
&  \sum_{\left(  M_{\gamma}\right)  }\mathfrak{B}_{k}\left(  \left(
M_{\gamma}\right)  _{\left(  1\right)  }\right)  \cdot\left(  a\left.
\prec\right.  \left(  M_{\gamma}\right)  _{\left(  2\right)  }\right) \\
&  =\sum_{j=0}^{\ell}\underbrace{\mathfrak{B}_{k}\left(  M_{\left(  \gamma
_{1},\gamma_{2},\ldots,\gamma_{j}\right)  }\right)  }_{\substack{=\sum_{1\leq
i_{1}<i_{2}<\cdots<i_{j}\leq k}x_{i_{1}}^{\gamma_{1}}x_{i_{2}}^{\gamma_{2}%
}\cdots x_{i_{j}}^{\gamma_{j}}\\\text{(by (\ref{pf.thm.beldend.Malpha-bel-a}%
))}}}\cdot\underbrace{\left(  a\left.  \prec\right.  M_{\left(  \gamma
_{j+1},\gamma_{j+2},\ldots,\gamma_{\ell}\right)  }\right)  }%
_{\substack{=\left(  \sum_{k<i_{1}<i_{2}<\cdots<i_{\ell-j}}x_{i_{1}}%
^{\gamma_{j+1}}x_{i_{2}}^{\gamma_{j+2}}\cdots x_{i_{\ell-j}}^{\gamma_{\ell}%
}\right)  \cdot a\\\text{(by (\ref{pf.thm.beldend.a-Dless-Malpha}))}}}\\
&  \ \ \ \ \ \ \ \ \ \ \left(  \text{since }\sum_{\left(  M_{\gamma}\right)
}\left(  M_{\gamma}\right)  _{\left(  1\right)  }\otimes\left(  M_{\gamma
}\right)  _{\left(  2\right)  }=\Delta\left(  M_{\gamma}\right)  =\sum
_{j=0}^{\ell}M_{\left(  \gamma_{1},\gamma_{2},\ldots,\gamma_{j}\right)
}\otimes M_{\left(  \gamma_{j+1},\gamma_{j+2},\ldots,\gamma_{\ell}\right)
}\right) \\
&  =\sum_{j=0}^{\ell}\left(  \sum_{1\leq i_{1}<i_{2}<\cdots<i_{j}\leq
k}x_{i_{1}}^{\gamma_{1}}x_{i_{2}}^{\gamma_{2}}\cdots x_{i_{j}}^{\gamma_{j}%
}\right)  \underbrace{\left(  \sum_{k<i_{1}<i_{2}<\cdots<i_{\ell-j}}x_{i_{1}%
}^{\gamma_{j+1}}x_{i_{2}}^{\gamma_{j+2}}\cdots x_{i_{\ell-j}}^{\gamma_{\ell}%
}\right)  }_{\substack{=\sum_{k<i_{j+1}<i_{j+2}<\cdots<i_{\ell}}x_{i_{j+1}%
}^{\gamma_{j+1}}x_{i_{j+2}}^{\gamma_{j+2}}\cdots x_{i_{\ell}}^{\gamma_{\ell}%
}\\\text{(here, we have renamed the summation index}\\\left(  i_{1}%
,i_{2},\ldots,i_{\ell-j}\right)  \text{ as }\left(  i_{j+1},i_{j+2}%
,\ldots,i_{\ell}\right)  \text{)}}}\cdot a\\
&  =\sum_{j=0}^{\ell}\left(  \sum_{1\leq i_{1}<i_{2}<\cdots<i_{j}\leq
k}x_{i_{1}}^{\gamma_{1}}x_{i_{2}}^{\gamma_{2}}\cdots x_{i_{j}}^{\gamma_{j}%
}\right)  \left(  \sum_{k<i_{j+1}<i_{j+2}<\cdots<i_{\ell}}x_{i_{j+1}}%
^{\gamma_{j+1}}x_{i_{j+2}}^{\gamma_{j+2}}\cdots x_{i_{\ell}}^{\gamma_{\ell}%
}\right)  \cdot a\\
&  =\underbrace{\sum_{j=0}^{\ell}\ \ \sum_{1\leq i_{1}<i_{2}<\cdots<i_{j}\leq
k}\ \ \sum_{k<i_{j+1}<i_{j+2}<\cdots<i_{\ell}}}_{\substack{=\sum_{1\leq
i_{1}<i_{2}<\cdots<i_{\ell}}\ \ \sum_{\substack{j\in\left\{  0,1,\ldots
,\ell\right\}  ;\\i_{j}\leq k<i_{j+1}}}\\\text{(where }i_{0}\text{ is to be
understood as }1\text{, and }i_{\ell+1}\text{ as }\infty+1\text{)}%
}}\underbrace{\left(  x_{i_{1}}^{\gamma_{1}}x_{i_{2}}^{\gamma_{2}}\cdots
x_{i_{j}}^{\gamma_{j}}\right)  \left(  x_{i_{j+1}}^{\gamma_{j+1}}x_{i_{j+2}%
}^{\gamma_{j+2}}\cdots x_{i_{\ell}}^{\gamma_{\ell}}\right)  }_{=x_{i_{1}%
}^{\gamma_{1}}x_{i_{2}}^{\gamma_{2}}\cdots x_{i_{\ell}}^{\gamma_{\ell}}}%
\cdot\,a\\
&  =\sum_{1\leq i_{1}<i_{2}<\cdots<i_{\ell}}\ \ \underbrace{\sum
_{\substack{j\in\left\{  0,1,\ldots,\ell\right\}  ;\\i_{j}\leq k<i_{j+1}%
}}x_{i_{1}}^{\gamma_{1}}x_{i_{2}}^{\gamma_{2}}\cdots x_{i_{\ell}}%
^{\gamma_{\ell}}}_{\substack{\text{this sum has precisely one addend}%
\\\text{(because of (\ref{pf.thm.beldend.exactlyonej})),}\\\text{and thus
equals }x_{i_{1}}^{\gamma_{1}}x_{i_{2}}^{\gamma_{2}}\cdots x_{i_{\ell}%
}^{\gamma_{\ell}}}}\cdot\,a\\
&  =\underbrace{\sum_{1\leq i_{1}<i_{2}<\cdots<i_{\ell}}x_{i_{1}}^{\gamma_{1}%
}x_{i_{2}}^{\gamma_{2}}\cdots x_{i_{\ell}}^{\gamma_{\ell}}}_{=M_{\gamma}}%
\cdot\,a\\
&  =M_{\gamma}\cdot a=aM_{\gamma},
\end{align*}
qed. Thus, (\ref{pf.thm.beldend.mainlem}) is proven.

Now, every $b\in\operatorname*{QSym}$ satisfies%
\begin{align*}
&  \sum_{\left(  b\right)  }\underbrace{\left(  S\left(  b_{\left(  1\right)
}\right)  \bel a\right)  }_{\substack{=a\cdot\mathfrak{B}_{k}\left(  S\left(
b_{\left(  1\right)  }\right)  \right)  \\\text{(by (\ref{pf.thm.beldend.B}),
applied to }p=S\left(  b_{\left(  1\right)  }\right)  \text{)}}}b_{\left(
2\right)  }\\
&  =\sum_{\left(  b\right)  }a\cdot\mathfrak{B}_{k}\left(  S\left(  b_{\left(
1\right)  }\right)  \right)  b_{\left(  2\right)  }=\sum_{\left(  b\right)
}\mathfrak{B}_{k}\left(  S\left(  b_{\left(  1\right)  }\right)  \right)
\cdot\underbrace{ab_{\left(  2\right)  }}_{\substack{=\sum_{\left(  b_{\left(
2\right)  }\right)  }\mathfrak{B}_{k}\left(  \left(  b_{\left(  2\right)
}\right)  _{\left(  1\right)  }\right)  \left(  a\left.  \prec\right.  \left(
b_{\left(  2\right)  }\right)  _{\left(  2\right)  }\right)  \\\text{(by
(\ref{pf.thm.beldend.mainlem}), applied to }f=b_{\left(  2\right)  }\text{)}%
}}\\
&  =\sum_{\left(  b\right)  }\mathfrak{B}_{k}\left(  S\left(  b_{\left(
1\right)  }\right)  \right)  \left(  \sum_{\left(  b_{\left(  2\right)
}\right)  }\mathfrak{B}_{k}\left(  \left(  b_{\left(  2\right)  }\right)
_{\left(  1\right)  }\right)  \left(  a\left.  \prec\right.  \left(
b_{\left(  2\right)  }\right)  _{\left(  2\right)  }\right)  \right) \\
&  =\sum_{\left(  b\right)  }\sum_{\left(  b_{\left(  2\right)  }\right)
}\mathfrak{B}_{k}\left(  S\left(  b_{\left(  1\right)  }\right)  \right)
\mathfrak{B}_{k}\left(  \left(  b_{\left(  2\right)  }\right)  _{\left(
1\right)  }\right)  \left(  a\left.  \prec\right.  \left(  b_{\left(
2\right)  }\right)  _{\left(  2\right)  }\right) \\
&  =\sum_{\left(  b\right)  }\underbrace{\sum_{\left(  b_{\left(  1\right)
}\right)  }\mathfrak{B}_{k}\left(  S\left(  \left(  b_{\left(  1\right)
}\right)  _{\left(  1\right)  }\right)  \right)  \mathfrak{B}_{k}\left(
\left(  b_{\left(  1\right)  }\right)  _{\left(  2\right)  }\right)
}_{\substack{=\mathfrak{B}_{k}\left(  \sum_{\left(  b_{\left(  1\right)
}\right)  }S\left(  \left(  b_{\left(  1\right)  }\right)  _{\left(  1\right)
}\right)  \cdot\left(  b_{\left(  1\right)  }\right)  _{\left(  2\right)
}\right)  \\\text{(since }\mathfrak{B}_{k}\text{ is a }\mathbf{k}%
\text{-algebra homomorphism)}}}\left(  a\left.  \prec\right.  b_{\left(
2\right)  }\right) \\
&  \ \ \ \ \ \ \ \ \ \ \left(
\begin{array}
[c]{c}%
\text{since the coassociativity of }\Delta\text{ yields}\\
\sum_{\left(  b\right)  }\sum_{\left(  b_{\left(  2\right)  }\right)
}b_{\left(  1\right)  }\otimes\left(  b_{\left(  2\right)  }\right)  _{\left(
1\right)  }\otimes\left(  b_{\left(  2\right)  }\right)  _{\left(  2\right)
}=\sum_{\left(  b\right)  }\sum_{\left(  b_{\left(  1\right)  }\right)
}\left(  b_{\left(  1\right)  }\right)  _{\left(  1\right)  }\otimes\left(
b_{\left(  1\right)  }\right)  _{\left(  2\right)  }\otimes b_{\left(
2\right)  }%
\end{array}
\right) \\
&  =\sum_{\left(  b\right)  }\mathfrak{B}_{k}\left(  \underbrace{\sum_{\left(
b_{\left(  1\right)  }\right)  }S\left(  \left(  b_{\left(  1\right)
}\right)  _{\left(  1\right)  }\right)  \left(  b_{\left(  1\right)  }\right)
_{\left(  2\right)  }}_{\substack{=\varepsilon\left(  b_{\left(  1\right)
}\right)  \\\text{(by one of the defining equations of the antipode)}%
}}\right)  \left(  a\left.  \prec\right.  b_{\left(  2\right)  }\right) \\
&  =\sum_{\left(  b\right)  }\underbrace{\mathfrak{B}_{k}\left(
\varepsilon\left(  b_{\left(  1\right)  }\right)  \right)  }%
_{\substack{=\varepsilon\left(  b_{\left(  1\right)  }\right)  \\\text{(since
}\mathfrak{B}_{k}\text{ is a }\mathbf{k}\text{-algebra}\\\text{homomorphism,
and}\\\varepsilon\left(  b_{\left(  1\right)  }\right)  \in\mathbf{k}\text{ is
a scalar)}}}\left(  a\left.  \prec\right.  b_{\left(  2\right)  }\right)
=\sum_{\left(  b\right)  }\varepsilon\left(  b_{\left(  1\right)  }\right)
\cdot\left(  a\left.  \prec\right.  b_{\left(  2\right)  }\right) \\
&  =\sum_{\left(  b\right)  }a\left.  \prec\right.  \left(  \varepsilon\left(
b_{\left(  1\right)  }\right)  b_{\left(  2\right)  }\right)  =a\left.
\prec\right.  \underbrace{\left(  \sum_{\left(  b\right)  }\varepsilon\left(
b_{\left(  1\right)  }\right)  b_{\left(  2\right)  }\right)  }_{=b}=a\left.
\prec\right.  b.
\end{align*}

\end{verlong}

\begin{vershort}
Using the definitions of $\left.  \prec\right.  $ and $M_{\alpha}$ (and
recalling that $a=\mathfrak{n}$ has $\min\left(  \operatorname*{Supp}%
\mathfrak{n}\right)  =k$), it is now straightforward to check that every
composition $\alpha=\left(  \alpha_{1},\alpha_{2},\ldots,\alpha_{\ell}\right)
$ satisfies%
\begin{equation}
a\left.  \prec\right.  M_{\alpha}=\left(  \sum_{k<i_{1}<i_{2}<\cdots<i_{\ell}%
}x_{i_{1}}^{\alpha_{1}}x_{i_{2}}^{\alpha_{2}}\cdots x_{i_{\ell}}^{\alpha
_{\ell}}\right)  \cdot a. \label{pf.thm.beldend.short.a-Dless-Malpha}%
\end{equation}

Let us define a map $\mathfrak{B}_{k}:\mathbf{k}\left[  \left[  x_{1}%
,x_{2},x_{3},\ldots\right]  \right]  \rightarrow\mathbf{k}\left[  \left[
x_{1},x_{2},x_{3},\ldots\right]  \right]  $ by%
\[
\mathfrak{B}_{k}\left(  p\right)  =p\left(  x_{1},x_{2},\ldots,x_{k}%
,0,0,0,\ldots\right)  \ \ \ \ \ \ \ \ \ \ \text{for every }p\in\mathbf{k}%
\left[  \left[  x_{1},x_{2},x_{3},\ldots\right]  \right]  .
\]
Then, $\mathfrak{B}_{k}$ is an evaluation map (in an appropriate sense) and
thus a continuous $\mathbf{k}$-algebra homomorphism. Clearly, any monomial
$\mathfrak{m}$ satisfies%
\begin{equation}
\mathfrak{B}_{k}\left(  \mathfrak{m}\right)  =%
\begin{cases}
\mathfrak{m}, & \text{if }\max\left(  \operatorname*{Supp}\mathfrak{m}\right)
\leq k;\\
0, & \text{if }\max\left(  \operatorname*{Supp}\mathfrak{m}\right)  >k
\end{cases}
\quad. \label{pf.thm.beldend.short.monom}%
\end{equation}
Using this (and the definition of $\bel$), we see that any $p\in
\mathbf{k}\left[  \left[  x_{1},x_{2},x_{3},\ldots\right]  \right]  $
satisfies%
\begin{equation}
p\bel a=a\cdot\mathfrak{B}_{k}\left(  p\right)  \label{pf.thm.beldend.short.B}%
\end{equation}
(indeed, this is trivial to check for $p$ being a monomial, and thus follows
by linearity and continuity for all $p$). Also, every composition
$\alpha=\left(  \alpha_{1},\alpha_{2},\ldots,\alpha_{\ell}\right)  $ satisfies%
\begin{equation}
\mathfrak{B}_{k}\left(  M_{\alpha}\right)  =\sum_{1\leq i_{1}<i_{2}%
<\cdots<i_{\ell}\leq k}x_{i_{1}}^{\alpha_{1}}x_{i_{2}}^{\alpha_{2}}\cdots
x_{i_{\ell}}^{\alpha_{\ell}} \label{pf.thm.beldend.short.Malpha-bel-a}%
\end{equation}
(as follows easily from the definitions of $\mathfrak{B}_{k}$ and $M_{\alpha}$).

Let us now notice that every $f\in\operatorname*{QSym}$ satisfies%
\begin{equation}
af=\sum_{\left(  f\right)  }\mathfrak{B}_{k}\left(  f_{\left(  1\right)
}\right)  \left(  a\left.  \prec\right.  f_{\left(  2\right)  }\right)  .
\label{pf.thm.beldend.short.mainlem}%
\end{equation}

\textit{Proof of (\ref{pf.thm.beldend.short.mainlem}):} Both sides of the
equality (\ref{pf.thm.beldend.short.mainlem}) are $\mathbf{k}$-linear in $f$.
Hence, it is enough to check (\ref{pf.thm.beldend.short.mainlem}) on the basis
$\left(  M_{\gamma}\right)  _{\gamma\in\operatorname*{Comp}}$ of
$\operatorname*{QSym}$, that is, to prove that
(\ref{pf.thm.beldend.short.mainlem}) holds whenever $f=M_{\gamma}$ for some
$\gamma\in\operatorname*{Comp}$. In other words, it is enough to show that
\[
aM_{\gamma}=\sum_{\left(  M_{\gamma}\right)  }\mathfrak{B}_{k}\left(  \left(
M_{\gamma}\right)  _{\left(  1\right)  }\right)  \cdot\left(  a\left.
\prec\right.  \left(  M_{\gamma}\right)  _{\left(  2\right)  }\right)
\ \ \ \ \ \ \ \ \ \ \text{for every }\gamma\in\operatorname*{Comp}.
\]
But this is easily done: Let $\gamma\in\operatorname*{Comp}$. Write $\gamma$
in the form $\gamma=\left(  \gamma_{1},\gamma_{2},\ldots,\gamma_{\ell}\right)
$. Then,%
\begin{align*}
&  \sum_{\left(  M_{\gamma}\right)  }\mathfrak{B}_{k}\left(  \left(
M_{\gamma}\right)  _{\left(  1\right)  }\right)  \cdot\left(  a\left.
\prec\right.  \left(  M_{\gamma}\right)  _{\left(  2\right)  }\right)  \\
&  =\sum_{j=0}^{\ell}\underbrace{\mathfrak{B}_{k}\left(  M_{\left(  \gamma
_{1},\gamma_{2},\ldots,\gamma_{j}\right)  }\right)  }_{\substack{=\sum_{1\leq
i_{1}<i_{2}<\cdots<i_{j}\leq k}x_{i_{1}}^{\gamma_{1}}x_{i_{2}}^{\gamma_{2}%
}\cdots x_{i_{j}}^{\gamma_{j}}\\\text{(by
(\ref{pf.thm.beldend.short.Malpha-bel-a}))}}}\cdot\underbrace{\left(  a\left.
\prec\right.  M_{\left(  \gamma_{j+1},\gamma_{j+2},\ldots,\gamma_{\ell
}\right)  }\right)  }_{\substack{=\left(  \sum_{k<i_{1}<i_{2}<\cdots
<i_{\ell-j}}x_{i_{1}}^{\gamma_{j+1}}x_{i_{2}}^{\gamma_{j+2}}\cdots
x_{i_{\ell-j}}^{\gamma_{\ell}}\right)  \cdot a\\\text{(by
(\ref{pf.thm.beldend.short.a-Dless-Malpha}))}}}\\
&  \ \ \ \ \ \ \ \ \ \ \left(  \text{since }\sum_{\left(  M_{\gamma}\right)
}\left(  M_{\gamma}\right)  _{\left(  1\right)  }\otimes\left(  M_{\gamma
}\right)  _{\left(  2\right)  }=\Delta\left(  M_{\gamma}\right)  =\sum
_{j=0}^{\ell}M_{\left(  \gamma_{1},\gamma_{2},\ldots,\gamma_{j}\right)
}\otimes M_{\left(  \gamma_{j+1},\gamma_{j+2},\ldots,\gamma_{\ell}\right)
}\right)  \\
&  =\sum_{j=0}^{\ell}\left(  \sum_{1\leq i_{1}<i_{2}<\cdots<i_{j}\leq
k}x_{i_{1}}^{\gamma_{1}}x_{i_{2}}^{\gamma_{2}}\cdots x_{i_{j}}^{\gamma_{j}%
}\right)  \underbrace{\left(  \sum_{k<i_{1}<i_{2}<\cdots<i_{\ell-j}}x_{i_{1}%
}^{\gamma_{j+1}}x_{i_{2}}^{\gamma_{j+2}}\cdots x_{i_{\ell-j}}^{\gamma_{\ell}%
}\right)  }_{\substack{=\sum_{k<i_{j+1}<i_{j+2}<\cdots<i_{\ell}}x_{i_{j+1}%
}^{\gamma_{j+1}}x_{i_{j+2}}^{\gamma_{j+2}}\cdots x_{i_{\ell}}^{\gamma_{\ell}}%
}}\cdot\,a\\
&  =\sum_{j=0}^{\ell}\left(  \sum_{1\leq i_{1}<i_{2}<\cdots<i_{j}\leq
k}x_{i_{1}}^{\gamma_{1}}x_{i_{2}}^{\gamma_{2}}\cdots x_{i_{j}}^{\gamma_{j}%
}\right)  \left(  \sum_{k<i_{j+1}<i_{j+2}<\cdots<i_{\ell}}x_{i_{j+1}}%
^{\gamma_{j+1}}x_{i_{j+2}}^{\gamma_{j+2}}\cdots x_{i_{\ell}}^{\gamma_{\ell}%
}\right)  \cdot a\\
&  =\underbrace{\sum_{j=0}^{\ell}\ \ \sum_{1\leq i_{1}<i_{2}<\cdots<i_{j}\leq
k}\ \ \sum_{k<i_{j+1}<i_{j+2}<\cdots<i_{\ell}}}_{\substack{=\sum_{1\leq
i_{1}<i_{2}<\cdots<i_{\ell}}\ \ \sum_{\substack{j\in\left\{  0,1,\ldots
,\ell\right\}  ;\\i_{j}\leq k<i_{j+1}}}\\\text{(where }i_{0}\text{ is to be
understood as }1\text{, and }i_{\ell+1}\text{ as }\infty\text{)}%
}}\underbrace{\left(  x_{i_{1}}^{\gamma_{1}}x_{i_{2}}^{\gamma_{2}}\cdots
x_{i_{j}}^{\gamma_{j}}\right)  \left(  x_{i_{j+1}}^{\gamma_{j+1}}x_{i_{j+2}%
}^{\gamma_{j+2}}\cdots x_{i_{\ell}}^{\gamma_{\ell}}\right)  }_{=x_{i_{1}%
}^{\gamma_{1}}x_{i_{2}}^{\gamma_{2}}\cdots x_{i_{\ell}}^{\gamma_{\ell}}}%
\cdot\,a\\
&  =\sum_{1\leq i_{1}<i_{2}<\cdots<i_{\ell}}\ \ \underbrace{\sum
_{\substack{j\in\left\{  0,1,\ldots,\ell\right\}  ;\\i_{j}\leq k<i_{j+1}%
}}x_{i_{1}}^{\gamma_{1}}x_{i_{2}}^{\gamma_{2}}\cdots x_{i_{\ell}}%
^{\gamma_{\ell}}}_{\substack{\text{this sum has precisely one addend,}%
\\\text{and thus equals }x_{i_{1}}^{\gamma_{1}}x_{i_{2}}^{\gamma_{2}}\cdots
x_{i_{\ell}}^{\gamma_{\ell}}}}\cdot\,a\\
&  =\underbrace{\sum_{1\leq i_{1}<i_{2}<\cdots<i_{\ell}}x_{i_{1}}^{\gamma_{1}%
}x_{i_{2}}^{\gamma_{2}}\cdots x_{i_{\ell}}^{\gamma_{\ell}}}_{=M_{\gamma}}%
\cdot\,a\\
&  =M_{\gamma}\cdot a=aM_{\gamma},
\end{align*}
qed. Thus, (\ref{pf.thm.beldend.short.mainlem}) is proven.

Now, every $b\in\operatorname*{QSym}$ satisfies%
\begin{align*}
&  \sum_{\left(  b\right)  }\underbrace{\left(  S\left(  b_{\left(  1\right)
}\right)  \bel a\right)  }_{\substack{=a\cdot\mathfrak{B}_{k}\left(  S\left(
b_{\left(  1\right)  }\right)  \right)  \\\text{(by
(\ref{pf.thm.beldend.short.B}), applied to }p=S\left(  b_{\left(  1\right)
}\right)  \text{)}}}b_{\left(  2\right)  }\\
&  =\sum_{\left(  b\right)  }a\cdot\mathfrak{B}_{k}\left(  S\left(  b_{\left(
1\right)  }\right)  \right)  b_{\left(  2\right)  }=\sum_{\left(  b\right)
}\mathfrak{B}_{k}\left(  S\left(  b_{\left(  1\right)  }\right)  \right)
\cdot\underbrace{ab_{\left(  2\right)  }}_{\substack{=\sum_{\left(  b_{\left(
2\right)  }\right)  }\mathfrak{B}_{k}\left(  \left(  b_{\left(  2\right)
}\right)  _{\left(  1\right)  }\right)  \left(  a\left.  \prec\right.  \left(
b_{\left(  2\right)  }\right)  _{\left(  2\right)  }\right)  \\\text{(by
(\ref{pf.thm.beldend.short.mainlem}), applied to }f=b_{\left(  2\right)
}\text{)}}}\\
&  =\sum_{\left(  b\right)  }\mathfrak{B}_{k}\left(  S\left(  b_{\left(
1\right)  }\right)  \right)  \left(  \sum_{\left(  b_{\left(  2\right)
}\right)  }\mathfrak{B}_{k}\left(  \left(  b_{\left(  2\right)  }\right)
_{\left(  1\right)  }\right)  \left(  a\left.  \prec\right.  \left(
b_{\left(  2\right)  }\right)  _{\left(  2\right)  }\right)  \right) \\
&  =\sum_{\left(  b\right)  }\sum_{\left(  b_{\left(  2\right)  }\right)
}\mathfrak{B}_{k}\left(  S\left(  b_{\left(  1\right)  }\right)  \right)
\mathfrak{B}_{k}\left(  \left(  b_{\left(  2\right)  }\right)  _{\left(
1\right)  }\right)  \left(  a\left.  \prec\right.  \left(  b_{\left(
2\right)  }\right)  _{\left(  2\right)  }\right) \\
&  =\sum_{\left(  b\right)  }\underbrace{\sum_{\left(  b_{\left(  1\right)
}\right)  }\mathfrak{B}_{k}\left(  S\left(  \left(  b_{\left(  1\right)
}\right)  _{\left(  1\right)  }\right)  \right)  \mathfrak{B}_{k}\left(
\left(  b_{\left(  1\right)  }\right)  _{\left(  2\right)  }\right)
}_{\substack{=\mathfrak{B}_{k}\left(  \sum_{\left(  b_{\left(  1\right)
}\right)  }S\left(  \left(  b_{\left(  1\right)  }\right)  _{\left(  1\right)
}\right)  \cdot\left(  b_{\left(  1\right)  }\right)  _{\left(  2\right)
}\right)  \\\text{(since }\mathfrak{B}_{k}\text{ is a }\mathbf{k}%
\text{-algebra homomorphism)}}}\left(  a\left.  \prec\right.  b_{\left(
2\right)  }\right) \\
&  \ \ \ \ \ \ \ \ \ \ \left(
\begin{array}
[c]{c}%
\text{since the coassociativity of }\Delta\text{ yields}\\
\sum_{\left(  b\right)  }\sum_{\left(  b_{\left(  2\right)  }\right)
}b_{\left(  1\right)  }\otimes\left(  b_{\left(  2\right)  }\right)  _{\left(
1\right)  }\otimes\left(  b_{\left(  2\right)  }\right)  _{\left(  2\right)
}=\sum_{\left(  b\right)  }\sum_{\left(  b_{\left(  1\right)  }\right)
}\left(  b_{\left(  1\right)  }\right)  _{\left(  1\right)  }\otimes\left(
b_{\left(  1\right)  }\right)  _{\left(  2\right)  }\otimes b_{\left(
2\right)  }%
\end{array}
\right) \\
&  =\sum_{\left(  b\right)  }\mathfrak{B}_{k}\left(  \underbrace{\sum_{\left(
b_{\left(  1\right)  }\right)  }S\left(  \left(  b_{\left(  1\right)
}\right)  _{\left(  1\right)  }\right)  \left(  b_{\left(  1\right)  }\right)
_{\left(  2\right)  }}_{\substack{=\varepsilon\left(  b_{\left(  1\right)
}\right)  \\\text{(by one of the defining equations of the antipode)}%
}}\right)  \left(  a\left.  \prec\right.  b_{\left(  2\right)  }\right) \\
&  =\sum_{\left(  b\right)  }\underbrace{\mathfrak{B}_{k}\left(
\varepsilon\left(  b_{\left(  1\right)  }\right)  \right)  }%
_{\substack{=\varepsilon\left(  b_{\left(  1\right)  }\right)  \\\text{(since
}\mathfrak{B}_{k}\text{ is a }\mathbf{k}\text{-algebra}\\\text{homomorphism,
and}\\\varepsilon\left(  b_{\left(  1\right)  }\right)  \in\mathbf{k}\text{ is
a scalar)}}}\left(  a\left.  \prec\right.  b_{\left(  2\right)  }\right)
=\sum_{\left(  b\right)  }\varepsilon\left(  b_{\left(  1\right)  }\right)
\cdot\left(  a\left.  \prec\right.  b_{\left(  2\right)  }\right) \\
&  =\sum_{\left(  b\right)  }a\left.  \prec\right.  \left(  \varepsilon\left(
b_{\left(  1\right)  }\right)  b_{\left(  2\right)  }\right)  =a\left.
\prec\right.  \underbrace{\left(  \sum_{\left(  b\right)  }\varepsilon\left(
b_{\left(  1\right)  }\right)  b_{\left(  2\right)  }\right)  }_{=b}=a\left.
\prec\right.  b.
\end{align*}

\end{vershort}

This proves Theorem \ref{thm.beldend}.
\end{proof}

\begin{noncompile}
Here is another proof of Theorem \ref{thm.beldend}, which is more or less the
same as the above but using a slightly different language.

\begin{proof}
[Second proof of Theorem \ref{thm.beldend} (sketched).]We can WLOG assume that
$a$ is a monomial (because of the $\mathbf{k}$-linearity and continuity of all
operations in sight). So assume this. Then, $a$ is a nonconstant monomial
(since $a$ has constant term $0$), so that $\operatorname*{Supp}%
a\neq\varnothing$. Let $i=\min\left(  \operatorname*{Supp}a\right)  $. Then,
it is easy to see that every $f\in\operatorname*{QSym}$ satisfies $f\bel
a=f\left(  x_{1},x_{2},\ldots,x_{i}\right)  \cdot a$ (where we evaluate
quasisymmetric functions on variables in the usual way) and $a\left.
\prec\right.  f=a\cdot f\left(  x_{i+1},x_{i+2},x_{i+3},\ldots\right)  $.
Hence, the equality in question, $\sum_{\left(  b\right)  }\left(  S\left(
b_{\left(  1\right)  }\right)  \bel a\right)  b_{\left(  2\right)  }=a\left.
\prec\right.  b$, rewrites as%
\[
\sum_{\left(  b\right)  }\left(  S\left(  b_{\left(  1\right)  }\right)
\right)  \left(  x_{1},x_{2},\ldots,x_{i}\right)  \cdot a\cdot b_{\left(
2\right)  }=a\cdot b\left(  x_{i+1},x_{i+2},x_{i+3},\ldots\right)  .
\]
This will immediately follow if we can show%
\begin{equation}
\sum_{\left(  b\right)  }\left(  S\left(  b_{\left(  1\right)  }\right)
\right)  \left(  x_{1},x_{2},\ldots,x_{i}\right)  \cdot b_{\left(  2\right)
}=b\left(  x_{i+1},x_{i+2},x_{i+3},\ldots\right)  . \label{pf.beldend.2}%
\end{equation}
So we need to prove (\ref{pf.beldend.2}).

Every $f\in\operatorname*{QSym}$ satisfies%
\begin{equation}
f=\sum_{\left(  f\right)  }f_{\left(  1\right)  }\left(  x_{1},x_{2}%
,\ldots,x_{i}\right)  \cdot f_{\left(  2\right)  }\left(  x_{i+1}%
,x_{i+2},x_{i+3},\ldots\right)  . \label{pf.beldend.4}%
\end{equation}
\footnote{In fact, this follows from the fact that the ordered alphabet
$\left(  x_{1},x_{2},x_{3},\ldots\right)  $ can be regarded as the sum of the
ordered alphabets $\left(  x_{1},x_{2},\ldots,x_{i}\right)  $ and $\left(
x_{i+1},x_{i+2},x_{i+3},\ldots\right)  $. But there is an alternative proof,
which is probably faster done than the use of alphabets is justified:
\par
It is clearly enough to check (\ref{pf.beldend.4}) on the basis $\left(
M_{\gamma}\right)  _{\gamma\in\operatorname*{Comp}}$ of $\operatorname*{QSym}%
$, that is, to prove that (\ref{pf.beldend.4}) holds whenever $f=M_{\gamma}$
for some $\gamma\in\operatorname*{Comp}$. In other words, it is enough to show
that
\[
M_{\gamma}=\sum_{\left(  M_{\gamma}\right)  }\left(  M_{\gamma}\right)
_{\left(  1\right)  }\left(  x_{1},x_{2},\ldots,x_{i}\right)  \cdot\left(
M_{\gamma}\right)  _{\left(  2\right)  }\left(  x_{i+1},x_{i+2},x_{i+3}%
,\ldots\right)  \ \ \ \ \ \ \ \ \ \ \text{for every }\gamma\in
\operatorname*{Comp}.
\]
But this is easily done: Let $\gamma\in\operatorname*{Comp}$. Write $\gamma$
in the form $\gamma=\left(  \gamma_{1},\gamma_{2},\ldots,\gamma_{\ell}\right)
$. Then,%
\begin{align*}
&  \sum_{\left(  M_{\gamma}\right)  }\left(  M_{\gamma}\right)  _{\left(
1\right)  }\left(  x_{1},x_{2},\ldots,x_{i}\right)  \cdot\left(  M_{\gamma
}\right)  _{\left(  2\right)  }\left(  x_{i+1},x_{i+2},x_{i+3},\ldots\right)
\\
&  =\sum_{j=0}^{\ell}\underbrace{M_{\left(  \gamma_{1},\gamma_{2}%
,\ldots,\gamma_{j}\right)  }\left(  x_{1},x_{2},\ldots,x_{i}\right)  }%
_{=\sum_{1\leq u_{1}<u_{2}<\cdots<u_{j}\leq i}x_{u_{1}}^{\gamma_{1}}x_{u_{2}%
}^{\gamma_{2}}\cdots x_{u_{j}}^{\gamma_{j}}}\cdot\underbrace{M_{\left(
\gamma_{j+1},\gamma_{j+2},\ldots,\gamma_{\ell}\right)  }\left(  x_{i+1}%
,x_{i+2},x_{i+3},\ldots\right)  }_{\substack{=\sum_{1\leq u_{1}<u_{2}%
<\cdots<u_{\ell-j}}x_{i+u_{1}}^{\gamma_{j+1}}x_{i+u_{2}}^{\gamma_{j+2}}\cdots
x_{i+u_{\ell-j}}^{\gamma_{\ell}}\\=\sum_{i+1\leq u_{1}<u_{2}<\cdots<u_{\ell
-j}}x_{u_{1}}^{\gamma_{j+1}}x_{u_{2}}^{\gamma_{j+2}}\cdots x_{u_{\ell-j}%
}^{\gamma_{\ell}}\\=\sum_{i+1\leq u_{j+1}<u_{j+2}<\cdots<u_{\ell}}x_{u_{j+1}%
}^{\gamma_{j+1}}x_{u_{j+2}}^{\gamma_{j+2}}\cdots x_{u_{\ell}}^{\gamma_{\ell}}%
}}\\
&  \ \ \ \ \ \ \ \ \ \ \left(  \text{since }\Delta\left(  M_{\gamma}\right)
=\sum_{j=0}^{\ell}M_{\left(  \gamma_{1},\gamma_{2},\ldots,\gamma_{j}\right)
}\otimes M_{\left(  \gamma_{j+1},\gamma_{j+2},\ldots,\gamma_{\ell}\right)
}\right) \\
&  =\sum_{j=0}^{\ell}\left(  \sum_{1\leq u_{1}<u_{2}<\cdots<u_{j}\leq
i}x_{u_{1}}^{\gamma_{1}}x_{u_{2}}^{\gamma_{2}}\cdots x_{u_{j}}^{\gamma_{j}%
}\right)  \left(  \sum_{i+1\leq u_{j+1}<u_{j+2}<\cdots<u_{\ell}}x_{u_{j+1}%
}^{\gamma_{j+1}}x_{u_{j+2}}^{\gamma_{j+2}}\cdots x_{u_{\ell}}^{\gamma_{\ell}%
}\right) \\
&  =\underbrace{\sum_{j=0}^{\ell}\ \ \sum_{1\leq u_{1}<u_{2}<\cdots<u_{j}\leq
i}\ \ \sum_{i+1\leq u_{j+1}<u_{j+2}<\cdots<u_{\ell}}}_{=\sum_{1\leq
u_{1}<u_{2}<\cdots<u_{\ell}}\ \ \sum_{\substack{j\in\left\{  0,1,\ldots
,\ell\right\}  ;\\u_{j}\leq i<u_{j+1}}}}\left(  x_{u_{1}}^{\gamma_{1}}%
x_{u_{2}}^{\gamma_{2}}\cdots x_{u_{j}}^{\gamma_{j}}\right)  \left(
x_{u_{j+1}}^{\gamma_{j+1}}x_{u_{j+2}}^{\gamma_{j+2}}\cdots x_{u_{\ell}%
}^{\gamma_{\ell}}\right) \\
&  =\sum_{1\leq u_{1}<u_{2}<\cdots<u_{\ell}}\ \ \sum_{\substack{j\in\left\{
0,1,\ldots,\ell\right\}  ;\\u_{j}\leq i<u_{j+1}}}\underbrace{\left(  x_{u_{1}%
}^{\gamma_{1}}x_{u_{2}}^{\gamma_{2}}\cdots x_{u_{j}}^{\gamma_{j}}\right)
\left(  x_{u_{j+1}}^{\gamma_{j+1}}x_{u_{j+2}}^{\gamma_{j+2}}\cdots x_{u_{\ell
}}^{\gamma_{\ell}}\right)  }_{=x_{u_{1}}^{\gamma_{1}}x_{u_{2}}^{\gamma_{2}%
}\cdots x_{u_{\ell}}^{\gamma_{\ell}}}\\
&  =\sum_{1\leq u_{1}<u_{2}<\cdots<u_{\ell}}\ \ \underbrace{\sum
_{\substack{j\in\left\{  0,1,\ldots,\ell\right\}  ;\\u_{j}\leq i<u_{j+1}%
}}x_{u_{1}}^{\gamma_{1}}x_{u_{2}}^{\gamma_{2}}\cdots x_{u_{\ell}}%
^{\gamma_{\ell}}}_{\text{this sum has precisely one addend}}=\sum_{1\leq
u_{1}<u_{2}<\cdots<u_{\ell}}x_{u_{1}}^{\gamma_{1}}x_{u_{2}}^{\gamma_{2}}\cdots
x_{u_{\ell}}^{\gamma_{\ell}}=M_{\gamma},
\end{align*}
qed.}

Now, (\ref{pf.beldend.2}) becomes%
\begin{align*}
&  \sum_{\left(  b\right)  }\left(  S\left(  b_{\left(  1\right)  }\right)
\right)  \left(  x_{1},x_{2},\ldots,x_{i}\right)  \cdot\underbrace{b_{\left(
2\right)  }}_{\substack{=\sum_{\left(  b_{\left(  2\right)  }\right)  }\left(
b_{\left(  2\right)  }\right)  _{\left(  1\right)  }\left(  x_{1},x_{2}%
,\ldots,x_{i}\right)  \cdot\left(  b_{\left(  2\right)  }\right)  _{\left(
2\right)  }\left(  x_{i+1},x_{i+2},x_{i+3},\ldots\right)  \\\text{(by
(\ref{pf.beldend.4}))}}}\\
&  =\sum_{\left(  b\right)  }\sum_{\left(  b_{\left(  2\right)  }\right)
}\left(  S\left(  b_{\left(  1\right)  }\right)  \right)  \left(  x_{1}%
,x_{2},\ldots,x_{i}\right)  \cdot\left(  b_{\left(  2\right)  }\right)
_{\left(  1\right)  }\left(  x_{1},x_{2},\ldots,x_{i}\right)  \cdot\left(
b_{\left(  2\right)  }\right)  _{\left(  2\right)  }\left(  x_{i+1}%
,x_{i+2},x_{i+3},\ldots\right) \\
&  =\sum_{\left(  b\right)  }\underbrace{\sum_{\left(  b_{\left(  1\right)
}\right)  }\left(  S\left(  \left(  b_{\left(  1\right)  }\right)  _{\left(
1\right)  }\right)  \right)  \left(  x_{1},x_{2},\ldots,x_{i}\right)
\cdot\left(  b_{\left(  1\right)  }\right)  _{\left(  2\right)  }\left(
x_{1},x_{2},\ldots,x_{i}\right)  }_{=\left(  \sum_{\left(  b_{\left(
1\right)  }\right)  }S\left(  \left(  b_{\left(  1\right)  }\right)  _{\left(
1\right)  }\right)  \cdot\left(  b_{\left(  1\right)  }\right)  _{\left(
2\right)  }\right)  \left(  x_{1},x_{2},\ldots,x_{i}\right)  }\cdot
\,b_{\left(  2\right)  }\left(  x_{i+1},x_{i+2},x_{i+3},\ldots\right) \\
&  \ \ \ \ \ \ \ \ \ \ \left(  \text{by the coassociativity of }\Delta\right)
\\
&  =\sum_{\left(  b\right)  }\underbrace{\left(  \sum_{\left(  b_{\left(
1\right)  }\right)  }S\left(  \left(  b_{\left(  1\right)  }\right)  _{\left(
1\right)  }\right)  \cdot\left(  b_{\left(  1\right)  }\right)  _{\left(
2\right)  }\right)  }_{\substack{=\varepsilon\left(  b_{\left(  1\right)
}\right)  \\\text{(by one of the defining equations of the antipode)}}}\left(
x_{1},x_{2},\ldots,x_{i}\right)  \cdot b_{\left(  2\right)  }\left(
x_{i+1},x_{i+2},x_{i+3},\ldots\right) \\
&  =\sum_{\left(  b\right)  }\underbrace{\left(  \varepsilon\left(  b_{\left(
1\right)  }\right)  \right)  \left(  x_{1},x_{2},\ldots,x_{i}\right)
}_{=\varepsilon\left(  b_{\left(  1\right)  }\right)  }\cdot b_{\left(
2\right)  }\left(  x_{i+1},x_{i+2},x_{i+3},\ldots\right) \\
&  =\sum_{\left(  b\right)  }\varepsilon\left(  b_{\left(  1\right)  }\right)
\cdot b_{\left(  2\right)  }\left(  x_{i+1},x_{i+2},x_{i+3},\ldots\right)
=\underbrace{\left(  \sum_{\left(  b\right)  }\varepsilon\left(  b_{\left(
1\right)  }\right)  \cdot b_{\left(  2\right)  }\right)  }%
_{\substack{=b\\\text{(by the coalgebra axioms)}}}\left(  x_{i+1}%
,x_{i+2},x_{i+3},\ldots\right) \\
&  =b\left(  x_{i+1},x_{i+2},x_{i+3},\ldots\right)  .
\end{align*}
This proves (\ref{pf.beldend.2}) and thus Theorem \ref{thm.beldend}. (A
different proof might appear in \cite{Gri-gammapart}.)
\end{proof}
\end{noncompile}

Let us connect the $\bel $ operation with the fundamental basis of
$\operatorname*{QSym}$:

\begin{proposition}
\label{prop.bel.F}For any two compositions $\alpha$ and $\beta$, define a
composition $\alpha\odot\beta$ as follows:

\begin{itemize}
\item If $\alpha$ is empty, then set $\alpha\odot\beta=\beta$.

\item Otherwise, if $\beta$ is empty, then set $\alpha\odot\beta=\alpha$.

\item Otherwise, define $\alpha\odot\beta$ as $\left(  \alpha_{1},\alpha
_{2},\ldots,\alpha_{\ell-1},\alpha_{\ell}+\beta_{1},\beta_{2},\beta_{3}%
,\ldots,\beta_{m}\right)  $, where $\alpha$ is written as $\alpha=\left(
\alpha_{1},\alpha_{2},\ldots,\alpha_{\ell}\right)  $ and where $\beta$ is
written as $\beta=\left(  \beta_{1},\beta_{2},\ldots,\beta_{m}\right)  $.
\end{itemize}

Then, any two compositions $\alpha$ and $\beta$ satisfy%
\[
F_{\alpha}\bel F_{\beta}=F_{\alpha\odot\beta}.
\]

\end{proposition}

\begin{verlong}
Our proof of this proposition will rely on the following lemma:

\begin{lemma}
\label{lem.D(a(.)b)}If $G$ is a set of integers and $r$ is an integer, then we
let $G+r$ denote the set $\left\{  g+r\ \mid\ g\in G\right\}  $ of integers.

Let $p\in\mathbb{N}$ and $q\in\mathbb{N}$. Let $\alpha$ be a composition of
$p$. Let $\beta$ be a composition of $q$. Consider the composition
$\alpha\odot\beta$ defined in Proposition \ref{prop.bel.F}.

\textbf{(a)} Then, $\alpha\odot\beta$ is a composition of $p+q$ satisfying
$D\left(  \alpha\odot\beta\right)  =D\left(  \alpha\right)  \cup\left(
D\left(  \beta\right)  +p\right)  $.

\textbf{(b)} Also, define a composition $\left[  \alpha,\beta\right]  $ by
$\left[  \alpha,\beta\right]  =\left(  \alpha_{1},\alpha_{2},\ldots
,\alpha_{\ell},\beta_{1},\beta_{2},\ldots,\beta_{m}\right)  $, where $\alpha$
and $\beta$ are written in the forms $\alpha=\left(  \alpha_{1},\alpha
_{2},\ldots,\alpha_{\ell}\right)  $ and $\beta=\left(  \beta_{1},\beta
_{2},\ldots,\beta_{m}\right)  $. Assume that $p > 0$ and $q > 0$. Then,
$\left[  \alpha,\beta\right]  $ is a composition of $p+q$ satisfying $D\left(
\left[  \alpha,\beta\right]  \right)  =D\left(  \alpha\right)  \cup\left\{
p\right\}  \cup\left(  D\left(  \beta\right)  +p\right)  $.
\end{lemma}

(Actually, part \textbf{(b)} of this lemma will not be used until much later,
but part \textbf{(a)} will be used soon.)

\begin{proof}
[Proof of Lemma \ref{lem.D(a(.)b)}.]Write $\alpha$ in the form $\alpha=\left(
\alpha_{1},\alpha_{2},\ldots,\alpha_{\ell}\right)  $. Thus, $\left\vert
\alpha\right\vert =\alpha_{1}+\alpha_{2}+\cdots+\alpha_{\ell}$, so that
$\alpha_{1}+\alpha_{2}+\cdots+\alpha_{\ell}=\left\vert \alpha\right\vert =p$
(since $\alpha$ is a composition of $p$).

Write $\beta$ in the form $\beta=\left(  \beta_{1},\beta_{2},\ldots,\beta
_{m}\right)  $. Thus, $\left\vert \beta\right\vert =\beta_{1}+\beta_{2}%
+\cdots+\beta_{m}$, so that $\beta_{1}+\beta_{2}+\cdots+\beta_{m}=\left\vert
\beta\right\vert =q$ (since $\beta$ is a composition of $q$).

We have $\beta=\left(  \beta_{1},\beta_{2},\ldots,\beta_{m}\right)  $, and
thus%
\begin{align*}
D\left(  \beta\right)   &  =\left\{  \beta_{1},\beta_{1}+\beta_{2},\beta
_{1}+\beta_{2}+\beta_{3},\ldots,\beta_{1}+\beta_{2}+\cdots+\beta_{m-1}\right\}
\\
&  \ \ \ \ \ \ \ \ \ \ \left(  \text{by the definition of }D\left(
\beta\right)  \right) \\
&  =\left\{  \beta_{1}+\beta_{2}+\cdots+\beta_{j}\ \mid\ j\in\left\{
1,2,\ldots,m-1\right\}  \right\}  .
\end{align*}
Also, $\alpha=\left(  \alpha_{1},\alpha_{2},\ldots,\alpha_{\ell}\right)  $,
and thus%
\begin{align*}
D\left(  \alpha\right)   &  =\left\{  \alpha_{1},\alpha_{1}+\alpha_{2}%
,\alpha_{1}+\alpha_{2}+\alpha_{3},\ldots,\alpha_{1}+\alpha_{2}+\cdots
+\alpha_{\ell-1}\right\} \\
&  \ \ \ \ \ \ \ \ \ \ \left(  \text{by the definition of }D\left(
\alpha\right)  \right) \\
&  =\left\{  \alpha_{1}+\alpha_{2}+\cdots+\alpha_{i}\ \mid\ i\in\left\{
1,2,\ldots,\ell-1\right\}  \right\}  .
\end{align*}

\textbf{(a)} If $\alpha$ or $\beta$ is empty, then Lemma \ref{lem.D(a(.)b)}
\textbf{(a)} holds for obvious reasons (because of the definition of
$\alpha\odot\beta$ in this case). Thus, we WLOG assume that neither $\alpha$
nor $\beta$ is empty.

We have $\alpha\odot\beta=\left(  \alpha_{1},\alpha_{2},\ldots,\alpha_{\ell
-1},\alpha_{\ell}+\beta_{1},\beta_{2},\beta_{3},\ldots,\beta_{m}\right)  $ (by
the definition of $\alpha\odot\beta$) and thus%
\begin{align*}
\left\vert \alpha\odot\beta\right\vert  &  =\alpha_{1}+\alpha_{2}%
+\cdots+\alpha_{\ell-1}+\left(  \alpha_{\ell}+\beta_{1}\right)  +\beta
_{2}+\beta_{3}+\cdots+\beta_{m}\\
&  =\underbrace{\left(  \alpha_{1}+\alpha_{2}+\cdots+\alpha_{\ell}\right)
}_{=p}+\underbrace{\left(  \beta_{1}+\beta_{2}+\cdots+\beta_{m}\right)  }%
_{=q}=p+q.
\end{align*}
Thus, $\alpha\odot\beta$ is a composition of $p+q$. Hence, it remains to show
that $D\left(  \alpha\odot\beta\right)  =D\left(  \alpha\right)  \cup\left(
D\left(  \beta\right)  +p\right)  $.

Now, $\alpha\odot\beta=\left(  \alpha_{1},\alpha_{2},\ldots,\alpha_{\ell
-1},\alpha_{\ell}+\beta_{1},\beta_{2},\beta_{3},\ldots,\beta_{m}\right)  $, so
that%
\begin{align*}
&  D\left(  \alpha\odot\beta\right) \\
&  =\left\{  \alpha_{1},\alpha_{1}+\alpha_{2},\alpha_{1}+\alpha_{2}+\alpha
_{3},\ldots,\alpha_{1}+\alpha_{2}+\cdots+\alpha_{\ell-1},\right. \\
&  \ \ \ \ \ \ \ \ \ \ \left.  \alpha_{1}+\alpha_{2}+\cdots+\alpha_{\ell
-1}+\left(  \alpha_{\ell}+\beta_{1}\right)  ,\alpha_{1}+\alpha_{2}%
+\cdots+\alpha_{\ell-1}+\left(  \alpha_{\ell}+\beta_{1}\right)  +\beta
_{2},\right. \\
&  \ \ \ \ \ \ \ \ \ \ \left.  \alpha_{1}+\alpha_{2}+\cdots+\alpha_{\ell
-1}+\left(  \alpha_{\ell}+\beta_{1}\right)  +\beta_{2}+\beta_{3}%
,\ldots,\right. \\
&  \ \ \ \ \ \ \ \ \ \ \left.  \alpha_{1}+\alpha_{2}+\cdots+\alpha_{\ell
-1}+\left(  \alpha_{\ell}+\beta_{1}\right)  +\beta_{2}+\beta_{3}+\cdots
+\beta_{m-1}\right\} \\
&  \ \ \ \ \ \ \ \ \ \ \left(  \text{by the definition of }D\left(
\alpha\odot\beta\right)  \right) \\
&  =\underbrace{\left\{  \alpha_{1}+\alpha_{2}+\cdots+\alpha_{i}\ \mid
\ i\in\left\{  1,2,\ldots,\ell-1\right\}  \right\}  }_{=D\left(
\alpha\right)  }\\
&  \ \ \ \ \ \ \ \ \ \ \cup\left\{  \underbrace{\alpha_{1}+\alpha_{2}%
+\cdots+\alpha_{\ell-1}+\left(  \alpha_{\ell}+\beta_{1}\right)  +\beta
_{2}+\beta_{3}+\cdots+\beta_{j}}_{\substack{=\left(  \alpha_{1}+\alpha
_{2}+\cdots+\alpha_{\ell}\right)  +\left(  \beta_{1}+\beta_{2}+\cdots
+\beta_{j}\right)  \\=\left(  \beta_{1}+\beta_{2}+\cdots+\beta_{j}\right)
+\left(  \alpha_{1}+\alpha_{2}+\cdots+\alpha_{\ell}\right)  }}\ \mid
\ j\in\left\{  1,2,\ldots,m-1\right\}  \right\} \\
&  =D\left(  \alpha\right)  \cup\left\{  \left(  \beta_{1}+\beta_{2}%
+\cdots+\beta_{j}\right)  +\underbrace{\left(  \alpha_{1}+\alpha_{2}%
+\cdots+\alpha_{\ell}\right)  }_{=p}\ \mid\ j\in\left\{  1,2,\ldots
,m-1\right\}  \right\} \\
&  =D\left(  \alpha\right)  \cup\underbrace{\left\{  \left(  \beta_{1}%
+\beta_{2}+\cdots+\beta_{j}\right)  +p\ \mid\ j\in\left\{  1,2,\ldots
,m-1\right\}  \right\}  }_{=\left\{  \beta_{1}+\beta_{2}+\cdots+\beta
_{j}\ \mid\ j\in\left\{  1,2,\ldots,m-1\right\}  \right\}  +p}\\
&  =D\left(  \alpha\right)  \cup\left(  \underbrace{\left\{  \beta_{1}%
+\beta_{2}+\cdots+\beta_{j}\ \mid\ j\in\left\{  1,2,\ldots,m-1\right\}
\right\}  }_{=D\left(  \beta\right)  }+p\right) \\
&  =D\left(  \alpha\right)  \cup\left(  D\left(  \beta\right)  +p\right)  .
\end{align*}
This completes the proof of Lemma \ref{lem.D(a(.)b)} \textbf{(a)}.

\textbf{(b)} We have $p > 0$. Thus, the composition $\alpha$ is nonempty
(since $\alpha$ is a composition of $p$). In other words, the composition
$\left(  \alpha_{1}, \alpha_{2}, \ldots, \alpha_{\ell}\right)  $ is nonempty
(since $\alpha= \left(  \alpha_{1}, \alpha_{2}, \ldots, \alpha_{\ell}\right)
$). Hence, $\ell> 0$.

We have $q > 0$. Thus, the composition $\beta$ is nonempty (since $\beta$ is a
composition of $q$). In other words, the composition $\left(  \beta_{1},
\beta_{2}, \ldots, \beta_{m}\right)  $ is nonempty (since $\beta= \left(
\beta_{1}, \beta_{2}, \ldots, \beta_{m}\right)  $). Hence, $m > 0$.

We have $\left[  \alpha,\beta\right]  =\left(  \alpha_{1},\alpha_{2}%
,\ldots,\alpha_{\ell},\beta_{1},\beta_{2},\ldots,\beta_{m}\right)  $ (by the
definition of $\left[  \alpha,\beta\right]  $) and thus%
\begin{align*}
\left\vert \left[  \alpha,\beta\right]  \right\vert  &  =\alpha_{1}+\alpha
_{2}+\cdots+\alpha_{\ell}+\beta_{1}+\beta_{2}+\cdots+\beta_{m}\\
&  =\underbrace{\left(  \alpha_{1}+\alpha_{2}+\cdots+\alpha_{\ell}\right)
}_{=p}+\underbrace{\left(  \beta_{1}+\beta_{2}+\cdots+\beta_{m}\right)  }%
_{=q}=p+q.
\end{align*}
Thus, $\left[  \alpha,\beta\right]  $ is a composition of $p+q$. Hence, it
remains to show that $D\left(  \left[  \alpha,\beta\right]  \right)  =D\left(
\alpha\right)  \cup\left\{  p\right\}  \cup\left(  D\left(  \beta\right)
+p\right)  $.

Now, $\left[  \alpha,\beta\right]  =\left(  \alpha_{1},\alpha_{2}%
,\ldots,\alpha_{\ell},\beta_{1},\beta_{2},\ldots,\beta_{m}\right)  $, so that%
\begin{align*}
&  D\left(  \left[  \alpha,\beta\right]  \right) \\
&  =\left\{  \alpha_{1},\alpha_{1}+\alpha_{2},\alpha_{1}+\alpha_{2}+\alpha
_{3},\ldots,\alpha_{1}+\alpha_{2}+\cdots+\alpha_{\ell-1},\right. \\
&  \ \ \ \ \ \ \ \ \ \ \left.  \alpha_{1}+\alpha_{2}+\cdots+\alpha_{\ell
-1}+\alpha_{\ell},\alpha_{1}+\alpha_{2}+\cdots+\alpha_{\ell-1}+\alpha_{\ell
}+\beta_{1},\right. \\
&  \ \ \ \ \ \ \ \ \ \ \left.  \alpha_{1}+\alpha_{2}+\cdots+\alpha_{\ell
-1}+\alpha_{\ell}+\beta_{1}+\beta_{2},\ldots,\right. \\
&  \ \ \ \ \ \ \ \ \ \ \left.  \alpha_{1}+\alpha_{2}+\cdots+\alpha_{\ell
-1}+\alpha_{\ell}+\beta_{1}+\beta_{2}+\cdots+\beta_{m-1}\right\} \\
&  \ \ \ \ \ \ \ \ \ \ \left(  \text{by the definition of }D\left(  \left[
\alpha,\beta\right]  \right)  \right) \\
&  =\underbrace{\left\{  \alpha_{1}+\alpha_{2}+\cdots+\alpha_{i}\ \mid
\ i\in\left\{  1,2,\ldots,\ell-1\right\}  \right\}  }_{=D\left(
\alpha\right)  }\\
&  \ \ \ \ \ \ \ \ \ \ \cup\left\{  \underbrace{\alpha_{1}+\alpha_{2}%
+\cdots+\alpha_{\ell}}_{=p}\right\} \\
&  \ \ \ \ \ \ \ \ \ \ \cup\left\{  \underbrace{\alpha_{1}+\alpha_{2}%
+\cdots+\alpha_{\ell-1}+\alpha_{\ell}+\beta_{1}+\beta_{2}+\cdots+\beta_{j}%
}_{\substack{=\left(  \alpha_{1}+\alpha_{2}+\cdots+\alpha_{\ell}\right)
+\left(  \beta_{1}+\beta_{2}+\cdots+\beta_{j}\right)  \\=\left(  \beta
_{1}+\beta_{2}+\cdots+\beta_{j}\right)  +\left(  \alpha_{1}+\alpha_{2}%
+\cdots+\alpha_{\ell}\right)  }}\ \mid\ j\in\left\{  1,2,\ldots,m-1\right\}
\right\} \\
&  =D\left(  \alpha\right)  \cup\left\{  p\right\}  \cup\left\{  \left(
\beta_{1}+\beta_{2}+\cdots+\beta_{j}\right)  +\underbrace{\left(  \alpha
_{1}+\alpha_{2}+\cdots+\alpha_{\ell}\right)  }_{=p}\ \mid\ j\in\left\{
1,2,\ldots,m-1\right\}  \right\} \\
&  =D\left(  \alpha\right)  \cup\left\{  p\right\}  \cup\underbrace{\left\{
\left(  \beta_{1}+\beta_{2}+\cdots+\beta_{j}\right)  +p\ \mid\ j\in\left\{
1,2,\ldots,m-1\right\}  \right\}  }_{=\left\{  \beta_{1}+\beta_{2}%
+\cdots+\beta_{j}\ \mid\ j\in\left\{  1,2,\ldots,m-1\right\}  \right\}  +p}\\
&  =D\left(  \alpha\right)  \cup\left\{  p\right\}  \cup\left(
\underbrace{\left\{  \beta_{1}+\beta_{2}+\cdots+\beta_{j}\ \mid\ j\in\left\{
1,2,\ldots,m-1\right\}  \right\}  }_{=D\left(  \beta\right)  }+\,p\right) \\
&  =D\left(  \alpha\right)  \cup\left\{  p\right\}  \cup\left(  D\left(
\beta\right)  +p\right)  .
\end{align*}
This completes the proof of Lemma \ref{lem.D(a(.)b)} \textbf{(b)}.
\end{proof}
\end{verlong}

\begin{proof}
[Proof of Proposition \ref{prop.bel.F}.]If either $\alpha$ or $\beta$ is
empty, then this is obvious (since $\bel $ is unital with $1$ as its unity,
and since $F_{\varnothing}=1$). So let us WLOG assume that neither is. Write
$\alpha$ as $\alpha=\left(  \alpha_{1},\alpha_{2},\ldots,\alpha_{\ell}\right)
$, and write $\beta$ as $\beta=\left(  \beta_{1},\beta_{2},\ldots,\beta
_{m}\right)  $. Thus, $\ell$ and $m$ are positive (since $\alpha$ and $\beta$
are nonempty).

Let $p=\left\vert \alpha\right\vert $ and $q=\left\vert \beta\right\vert $.
Thus, $p$ and $q$ are positive (since $\alpha$ and $\beta$ are nonempty).
Recall that we use the notation $D\left(  \alpha\right)  $ for the set of
partial sums of a composition $\alpha$. If $G$ is a set of integers and $r$ is
an integer, then we let $G+r$ denote the set $\left\{  g+r\ \mid\ g\in
G\right\}  $ of integers.

\begin{verlong}
Lemma \ref{lem.D(a(.)b)} \textbf{(a)} shows that $\alpha\odot\beta$ is a
composition of $p+q$ satisfying $D\left(  \alpha\odot\beta\right)  =D\left(
\alpha\right)  \cup\left(  D\left(  \beta\right)  +p\right)  $.
\end{verlong}

Applying (\ref{eq.F.def}) to $p$ instead of $n$, we obtain%
\begin{equation}
F_{\alpha}=\sum_{\substack{i_{1}\leq i_{2}\leq\cdots\leq i_{p};\\i_{j}%
<i_{j+1}\text{ if }j\in D\left(  \alpha\right)  }}x_{i_{1}}x_{i_{2}}\cdots
x_{i_{p}}. \label{pf.bel.F.1}%
\end{equation}
Applying (\ref{eq.F.def}) to $q$ and $\beta$ instead of $n$ and $\alpha$, we
obtain%
\[
F_{\beta}=\sum_{\substack{i_{1}\leq i_{2}\leq\cdots\leq i_{q};\\i_{j}%
<i_{j+1}\text{ if }j\in D\left(  \beta\right)  }}x_{i_{1}}x_{i_{2}}\cdots
x_{i_{q}}=\sum_{\substack{i_{p+1}\leq i_{p+2}\leq\cdots\leq i_{p+q}%
;\\i_{j}<i_{j+1}\text{ if }j\in D\left(  \beta\right)  +p}}x_{i_{p+1}%
}x_{i_{p+2}}\cdots x_{i_{p+q}}%
\]
(here, we renamed the summation index $\left(  i_{1},i_{2},\ldots
,i_{q}\right)  $ as $\left(  i_{p+1},i_{p+2},\ldots,i_{p+q}\right)  $). This,
together with (\ref{pf.bel.F.1}), yields%
\begin{align}
&  F_{\alpha}\bel F_{\beta}\nonumber\\
&  =\left(  \sum_{\substack{i_{1}\leq i_{2}\leq\cdots\leq i_{p};\\i_{j}%
<i_{j+1}\text{ if }j\in D\left(  \alpha\right)  }}x_{i_{1}}x_{i_{2}}\cdots
x_{i_{p}}\right)  \bel\left(  \sum_{\substack{i_{p+1}\leq i_{p+2}\leq
\cdots\leq i_{p+q};\\i_{j}<i_{j+1}\text{ if }j\in D\left(  \beta\right)
+p}}x_{i_{p+1}}x_{i_{p+2}}\cdots x_{i_{p+q}}\right) \nonumber\\
&  =\sum_{\substack{i_{1}\leq i_{2}\leq\cdots\leq i_{p};\\i_{j}<i_{j+1}\text{
if }j\in D\left(  \alpha\right)  }}\ \ \sum_{\substack{i_{p+1}\leq i_{p+2}%
\leq\cdots\leq i_{p+q};\\i_{j}<i_{j+1}\text{ if }j\in D\left(  \beta\right)
+p}}\underbrace{\left(  x_{i_{1}}x_{i_{2}}\cdots x_{i_{p}}\right)  \bel\left(
x_{i_{p+1}}x_{i_{p+2}}\cdots x_{i_{p+q}}\right)  }_{\substack{=%
\begin{cases}
x_{i_{1}}x_{i_{2}}\cdots x_{i_{p}}x_{i_{p+1}}x_{i_{p+2}}\cdots x_{i_{p+q}}, &
\text{if }i_{p}\leq i_{p+1};\\
0, & \text{if }i_{p}>i_{p+1}%
\end{cases}
\\\text{(by the definition of }\bel\text{ on monomials)}}}\nonumber\\
&  =\sum_{\substack{i_{1}\leq i_{2}\leq\cdots\leq i_{p};\\i_{j}<i_{j+1}\text{
if }j\in D\left(  \alpha\right)  }}\ \ \sum_{\substack{i_{p+1}\leq i_{p+2}%
\leq\cdots\leq i_{p+q};\\i_{j}<i_{j+1}\text{ if }j\in D\left(  \beta\right)
+p}}%
\begin{cases}
x_{i_{1}}x_{i_{2}}\cdots x_{i_{p}}x_{i_{p+1}}x_{i_{p+2}}\cdots x_{i_{p+q}}, &
\text{if }i_{p}\leq i_{p+1};\\
0, & \text{if }i_{p}>i_{p+1}%
\end{cases}
\nonumber\\
&  =\underbrace{\sum_{\substack{i_{1}\leq i_{2}\leq\cdots\leq i_{p}%
;\\i_{j}<i_{j+1}\text{ if }j\in D\left(  \alpha\right)  ;\\i_{p+1}\leq
i_{p+2}\leq\cdots\leq i_{p+q};\\i_{j}<i_{j+1}\text{ if }j\in D\left(
\beta\right)  +p;\\i_{p}\leq i_{p+1}}}}_{=\sum_{\substack{i_{1}\leq i_{2}%
\leq\cdots\leq i_{p+q};\\i_{j}<i_{j+1}\text{ if }j\in D\left(  \alpha\right)
\cup\left(  D\left(  \beta\right)  +p\right)  }}}\underbrace{x_{i_{1}}%
x_{i_{2}}\cdots x_{i_{p}}x_{i_{p+1}}x_{i_{p+2}}\cdots x_{i_{p+q}}}_{=x_{i_{1}%
}x_{i_{2}}\cdots x_{i_{p+q}}}\nonumber\\
&  =\sum_{\substack{i_{1}\leq i_{2}\leq\cdots\leq i_{p+q};\\i_{j}%
<i_{j+1}\text{ if }j\in D\left(  \alpha\right)  \cup\left(  D\left(
\beta\right)  +p\right)  }}x_{i_{1}}x_{i_{2}}\cdots x_{i_{p+q}}.
\label{pf.bel.F.4}%
\end{align}

On the other hand, $\alpha\odot\beta$ is a composition of $p+q$ satisfying
$D\left(  \alpha\odot\beta\right)  =D\left(  \alpha\right)  \cup\left(
D\left(  \beta\right)  +p\right)  $. Thus, (\ref{eq.F.def}) (applied to
$\alpha\odot\beta$ and $p+q$ instead of $\alpha$ and $n$) yields%
\[
F_{\alpha\odot\beta}=\sum_{\substack{i_{1}\leq i_{2}\leq\cdots\leq
i_{p+q};\\i_{j}<i_{j+1}\text{ if }j\in D\left(  \alpha\odot\beta\right)
}}x_{i_{1}}x_{i_{2}}\cdots x_{i_{p+q}}=\sum_{\substack{i_{1}\leq i_{2}%
\leq\cdots\leq i_{p+q};\\i_{j}<i_{j+1}\text{ if }j\in D\left(  \alpha\right)
\cup\left(  D\left(  \beta\right)  +p\right)  }}x_{i_{1}}x_{i_{2}}\cdots
x_{i_{p+q}}%
\]
(since $D\left(  \alpha\odot\beta\right)  =D\left(  \alpha\right)  \cup\left(
D\left(  \beta\right)  +p\right)  $). Compared with (\ref{pf.bel.F.4}), this
yields $F_{\alpha} \bel F_{\beta}=F_{\alpha\odot\beta}$. This proves
Proposition \ref{prop.bel.F}.
\end{proof}

For our goals, we need a certain particular case of Proposition
\ref{prop.bel.F}. Namely, let us recall that for every $m\in\mathbb{N}$, the
$m$\textit{-th complete homogeneous symmetric function} $h_{m}$ is defined as
the element $\sum\limits_{1\leq i_{1}\leq i_{2}\leq\cdots\leq i_{m}}x_{i_{1}%
}x_{i_{2}}\cdots x_{i_{m}}$ of $\operatorname*{Sym}$. It is easy to see that
$h_{m}=F_{\left(  m\right)  }$ for every positive integer $m$. From this, we obtain:

\begin{corollary}
\label{cor.bel.F.hm}For any two compositions $\alpha$ and $\beta$, define a
composition $\alpha\odot\beta$ as in Proposition \ref{prop.bel.F}. Then, every
composition $\alpha$ and every positive integer $m$ satisfy%
\begin{equation}
F_{\alpha\odot\left(  m\right)  }=F_{\alpha} \bel h_{m}. \label{pf.hmDless.2}%
\end{equation}

\end{corollary}

\begin{verlong}
\begin{proof}
[Proof of Corollary \ref{cor.bel.F.hm}.]Let $\alpha$ be a composition. Let $m$
be a positive integer. Recall that $h_{m}=F_{\left(  m\right)  }$. Proposition
\ref{prop.bel.F} yields $F_{\alpha} \bel F_{\left(  m\right)  }=F_{\alpha
\odot\left(  m\right)  }$. Hence, $F_{\alpha\odot\left(  m\right)  }%
=F_{\alpha} \bel \underbrace{F_{\left(  m\right)  }}_{=h_{m}}=F_{\alpha}
\bel h_{m}$. This proves Corollary \ref{cor.bel.F.hm}.
\end{proof}
\end{verlong}

\begin{vershort}
\begin{remark}
We can also define a binary operation $\tvi
:\mathbf{k}\left[  \left[  x_{1},x_{2},x_{3},\ldots\right]  \right]
\times\mathbf{k}\left[  \left[  x_{1},x_{2},x_{3},\ldots\right]  \right]
\rightarrow\mathbf{k}\left[  \left[  x_{1},x_{2},x_{3},\ldots\right]  \right]
$ (written in infix notation) by the requirements that it be $\mathbf{k}%
$-bilinear and continuous with respect to the topology on $\mathbf{k}\left[
\left[  x_{1},x_{2},x_{3},\ldots\right]  \right]  $ and that it satisfy%
\[
\mathfrak{m}\tvi \mathfrak{n}=
\begin{cases}
\mathfrak{m}\cdot\mathfrak{n}, & \text{if }\max\left(  \operatorname*{Supp}%
\mathfrak{m}\right)  <\min\left(  \operatorname*{Supp}\mathfrak{n}\right)  ;\\
0, & \text{if }\max\left(  \operatorname*{Supp}\mathfrak{m}\right)  \geq
\min\left(  \operatorname*{Supp}\mathfrak{n}\right)
\end{cases}
\]
for any two monomials $\mathfrak{m}$ and $\mathfrak{n}$. (Recall that
$\max\varnothing=0$ and $\min\varnothing=\infty$.)

This operation $\tvi $ shares some of the properties of $\bel $ (in
particular, it is associative and has neutral element $1$); an analogue of
Theorem \ref{thm.beldend} says that
\[
\sum_{\left(  b\right)  }\left(  S\left(  b_{\left(  1\right)  }\right)
\tvi a\right)  b_{\left(  2\right)  }=a\left.  \preceq\right.  b
\]
for any $a\in\mathbf{k}\left[  \left[  x_{1},x_{2},x_{3},\ldots\right]
\right]  $ and $b\in\operatorname*{QSym}$, where $a\left.  \preceq\right.  b$
stands for $b\left.  \succeq\right.  a$. (Of course, we could also define
$\left.  \preceq\right.  $ by changing the \textquotedblleft$<$%
\textquotedblright\ into a \textquotedblleft$\leq$\textquotedblright\ and the
\textquotedblleft$\geq$\textquotedblright\ into a \textquotedblleft%
$>$\textquotedblright\ in the definition of $\left.  \prec\right.  $.)
\end{remark}
\end{vershort}

\begin{verlong}
For the sake of completeness (or, rather, in order not to lose old writing),
let me write down the definitions of some more operations on $\mathbf{k}%
\left[  \left[  x_{1},x_{2},x_{3},\ldots\right]  \right]  $.

\begin{definition}
We define a binary operation $\left.  \preceq\right.  :\mathbf{k}\left[
\left[  x_{1},x_{2},x_{3},\ldots\right]  \right]  \times\mathbf{k}\left[
\left[  x_{1},x_{2},x_{3},\ldots\right]  \right]  \rightarrow\mathbf{k}\left[
\left[  x_{1},x_{2},x_{3},\ldots\right]  \right]  $ (written in infix
notation) by the requirements that it be $\mathbf{k}$-bilinear and continuous
with respect to the topology on $\mathbf{k}\left[  \left[  x_{1},x_{2}%
,x_{3},\ldots\right]  \right]  $ and that it satisfy%
\[
\mathfrak{m}\left.  \preceq\right.  \mathfrak{n}=
\begin{cases}
\mathfrak{m}\cdot\mathfrak{n}, & \text{if }\min\left(  \operatorname*{Supp}%
\mathfrak{m}\right)  \leq\min\left(  \operatorname*{Supp}\mathfrak{n}\right)
;\\
0, & \text{if }\min\left(  \operatorname*{Supp}\mathfrak{m}\right)
>\min\left(  \operatorname*{Supp}\mathfrak{n}\right)
\end{cases}
\]
for any two monomials $\mathfrak{m}$ and $\mathfrak{n}$.
\end{definition}

Here, all the remarks we made after Definition \ref{def.Dless} apply. In
particular, $\min\varnothing=\infty$, and we are using $\left.  \preceq
\right.  $ as an operation symbol.

We have $\mathfrak{m}\left.  \preceq\right.  1=\mathfrak{m}$ for every
monomial $\mathfrak{m}$, and $1\left.  \preceq\right.  \mathfrak{m}=0$ for
every nonconstant monomial $\mathfrak{m}$.

\begin{remark}
The operation $\left.  \preceq\right.  $ is part of a dendriform algebra
structure on $\mathbf{k}\left[  \left[  x_{1},x_{2},x_{3},\ldots\right]
\right]  $ (and on $\operatorname*{QSym}$). More precisely, if we define
another binary operation $\left.  \succ\right.  $ on $\mathbf{k}\left[
\left[  x_{1},x_{2},x_{3},\ldots\right]  \right]  $ similarly to $\left.
\preceq\right.  $ except that we set%
\[
\mathfrak{m}\left.  \succ\right.  \mathfrak{n}=
\begin{cases}
\mathfrak{m}\cdot\mathfrak{n}, & \text{if }\min\left(  \operatorname*{Supp}%
\mathfrak{m}\right)  >\min\left(  \operatorname*{Supp}\mathfrak{n}\right)  ;\\
0, & \text{if }\min\left(  \operatorname*{Supp}\mathfrak{m}\right)  \leq
\min\left(  \operatorname*{Supp}\mathfrak{n}\right)
\end{cases}
\quad,
\]
then the structure $\left(  \mathbf{k}\left[  \left[  x_{1},x_{2},x_{3}%
,\ldots\right]  \right]  ,\left.  \preceq\right.  ,\left.  \succ\right.
\right)  $ is a dendriform algebra augmented to satisfy \cite[(15)]{EbrFar08}.
In particular, any three elements $a$, $b$ and $c$ of $\mathbf{k}\left[
\left[  x_{1},x_{2},x_{3},\ldots\right]  \right]  $ satisfy%
\begin{align*}
a\left.  \preceq\right.  b+a\left.  \succ\right.  b  &  =ab;\\
\left(  a\left.  \preceq\right.  b\right)  \left.  \preceq\right.  c  &
=a\left.  \preceq\right.  \left(  bc\right)  ;\\
\left(  a\left.  \succ\right.  b\right)  \left.  \preceq\right.  c  &
=a\left.  \succ\right.  \left(  b\left.  \preceq\right.  c\right)  ;\\
a\left.  \succ\right.  \left(  b\left.  \succ\right.  c\right)   &  =\left(
ab\right)  \left.  \succ\right.  c.
\end{align*}

\end{remark}

And here is an analogue of $\bel $ which has \textit{mostly} similar properties:

\begin{definition}
We define a binary operation $\tvi :\mathbf{k}\left[  \left[  x_{1}%
,x_{2},x_{3},\ldots\right]  \right]  \times\mathbf{k}\left[  \left[
x_{1},x_{2},x_{3},\ldots\right]  \right]  \rightarrow\mathbf{k}\left[  \left[
x_{1},x_{2},x_{3},\ldots\right]  \right]  $ (written in infix notation) by the
requirements that it be $\mathbf{k}$-bilinear and continuous with respect to
the topology on $\mathbf{k}\left[  \left[  x_{1},x_{2},x_{3},\ldots\right]
\right]  $ and that it satisfy%
\[
\mathfrak{m} \tvi \mathfrak{n}=
\begin{cases}
\mathfrak{m}\cdot\mathfrak{n}, & \text{if }\max\left(  \operatorname*{Supp}%
\mathfrak{m}\right)  <\min\left(  \operatorname*{Supp}\mathfrak{n}\right)  ;\\
0, & \text{if }\max\left(  \operatorname*{Supp}\mathfrak{m}\right)  \geq
\min\left(  \operatorname*{Supp}\mathfrak{n}\right)
\end{cases}
\]
for any two monomials $\mathfrak{m}$ and $\mathfrak{n}$.
\end{definition}

Here, again, $\max\varnothing$ is understood as $0$. The binary operation
$\tvi$ is associative. It is also unital (with $1$ serving as the unity).

\begin{proposition}
Every $a\in\operatorname*{QSym}$ and $b\in\operatorname*{QSym}$ satisfy
$a\left.  \preceq\right.  b\in\operatorname*{QSym}$ and $a \tvi b\in
\operatorname*{QSym}$.
\end{proposition}

For example, any two compositions $\alpha$ and $\beta$ satisfy $M_{\alpha
}\tvi
M_{\beta}=M_{\left[  \alpha,\beta\right]  }$, where $\left[  \alpha
,\beta\right]  $ denotes the concatenation of $\alpha$ and $\beta$ (defined
by
\[
\left[  \left(  \alpha_{1},\alpha_{2},\ldots,\alpha_{\ell}\right)  ,\left(
\beta_{1},\beta_{2},\ldots,\beta_{m}\right)  \right]  =\left(  \alpha
_{1},\alpha_{2},\ldots,\alpha_{\ell},\beta_{1},\beta_{2},\ldots,\beta
_{m}\right)
\]
). (Recall that $\left(  M_{\gamma}\right)  _{\gamma\in\operatorname*{Comp}}$
is the monomial basis of $\operatorname*{QSym}$.)

\begin{theorem}
\label{thm.tvidend}Let $S$ denote the antipode of the Hopf algebra
$\operatorname*{QSym}$. Let us use Sweedler's notation $\sum_{\left(
b\right)  }b_{\left(  1\right)  }\otimes b_{\left(  2\right)  }$ for
$\Delta\left(  b\right)  $, where $b$ is any element of $\operatorname*{QSym}%
$. Then,%
\[
\sum_{\left(  b\right)  }\left(  S\left(  b_{\left(  1\right)  }\right)
\tvi a\right)  b_{\left(  2\right)  }=a\left.  \preceq\right.  b
\]
for any $a\in\mathbf{k}\left[  \left[  x_{1},x_{2},x_{3},\ldots\right]
\right]  $ and $b\in\operatorname*{QSym}$.
\end{theorem}

\begin{proof}
[Proof of Theorem \ref{thm.tvidend}.]The following proof is mostly analogous
to the proof of Theorem \ref{thm.beldend}.

Let $a\in\mathbf{k}\left[  \left[  x_{1},x_{2},x_{3},\ldots\right]  \right]
$. We can WLOG assume that $a$ is a monomial (because all operations in sight
are $\mathbf{k}$-linear and continuous). So assume this. That is,
$a=\mathfrak{n}$ for some monomial $\mathfrak{n}$. Consider this
$\mathfrak{n}$. Let $k=\min\left(  \operatorname*{Supp}\mathfrak{n}\right)  $.
Notice that $k\in\left\{  1,2,3,\ldots\right\}  \cup\left\{  \infty\right\}  $.

(Some remarks about $\infty$ are in order. We use $\infty$ as an object which
is greater than every integer. We will use summation signs like $\sum_{1\leq
i_{1}<i_{2}<\cdots<i_{\ell}<k}$ and $\sum_{k\leq i_{1}<i_{2}<\cdots<i_{\ell}}$
in the following. Both of these summation signs range over $\left(
i_{1},i_{2},\ldots,i_{\ell}\right)  \in\left\{  1,2,3,\ldots\right\}  ^{\ell}$
satisfying certain conditions ($1\leq i_{1}<i_{2}<\cdots<i_{\ell}<k$ in the
first case, and $k\leq i_{1}<i_{2}<\cdots<i_{\ell}$ in the second case). In
particular, none of the $i_{1},i_{2},\ldots,i_{\ell}$ is allowed to be
$\infty$ (unlike $k$). So the summation $\sum_{1\leq i_{1}<i_{2}%
<\cdots<i_{\ell}<k}$ is identical to $\sum_{1\leq i_{1}<i_{2}<\cdots<i_{\ell}%
}$ when $k=\infty$, whereas the summation $\sum_{k\leq i_{1}<i_{2}%
<\cdots<i_{\ell}}$ is empty when $k=\infty$ unless $\ell=0$. (If $\ell=0$,
then the summation $\sum_{k\leq i_{1}<i_{2}<\cdots<i_{\ell}}$ ranges over the
empty $0$-tuple, no matter what $k$ is.))

Every composition $\alpha=\left(  \alpha_{1},\alpha_{2},\ldots,\alpha_{\ell
}\right)  $ satisfies%
\begin{equation}
a\left.  \preceq\right.  M_{\alpha}=\left(  \sum_{k\leq i_{1}<i_{2}%
<\cdots<i_{\ell}}x_{i_{1}}^{\alpha_{1}}x_{i_{2}}^{\alpha_{2}}\cdots
x_{i_{\ell}}^{\alpha_{\ell}}\right)  \cdot a
\label{pf.thm.tvidend.a-Dless-Malpha}%
\end{equation}
\footnote{\textit{Proof of (\ref{pf.thm.tvidend.a-Dless-Malpha}):} Let
$\alpha=\left(  \alpha_{1},\alpha_{2},\ldots,\alpha_{\ell}\right)  $ be a
composition. The definition of $M_{\alpha}$ yields $M_{\alpha}=\sum_{1\leq
i_{1}<i_{2}<\cdots<i_{\ell}}x_{i_{1}}^{\alpha_{1}}x_{i_{2}}^{\alpha_{2}}\cdots
x_{i_{\ell}}^{\alpha_{\ell}}$. Combined with $a=\mathfrak{n}$, this yields%
\begin{align*}
a\left.  \preceq\right.  M_{\alpha}  &  =\mathfrak{n}\left.  \preceq\right.
\left(  \sum_{1\leq i_{1}<i_{2}<\cdots<i_{\ell}}x_{i_{1}}^{\alpha_{1}}%
x_{i_{2}}^{\alpha_{2}}\cdots x_{i_{\ell}}^{\alpha_{\ell}}\right) \\
&  =\sum_{1\leq i_{1}<i_{2}<\cdots<i_{\ell}}\underbrace{\mathfrak{n}\left.
\preceq\right.  \left(  x_{i_{1}}^{\alpha_{1}}x_{i_{2}}^{\alpha_{2}}\cdots
x_{i_{\ell}}^{\alpha_{\ell}}\right)  }_{\substack{=
\begin{cases}
\mathfrak{n}\cdot x_{i_{1}}^{\alpha_{1}}x_{i_{2}}^{\alpha_{2}}\cdots
x_{i_{\ell}}^{\alpha_{\ell}}, & \text{if }\min\left(  \operatorname*{Supp}%
\mathfrak{n}\right)  \leq\min\left\{  i_{1},i_{2},\ldots,i_{\ell}\right\}  ;\\
0, & \text{if }\min\left(  \operatorname*{Supp}\mathfrak{n}\right)
>\min\left\{  i_{1},i_{2},\ldots,i_{\ell}\right\}
\end{cases}
\\\text{(by the definition of }\left.  \preceq\right.  \text{ on monomials)}%
}}\\
&  \ \ \ \ \ \ \ \ \ \ \left(  \text{since }\left.  \preceq\right.  \text{ is
}\mathbf{k}\text{-bilinear and continuous}\right) \\
&  =\sum_{1\leq i_{1}<i_{2}<\cdots<i_{\ell}}
\begin{cases}
\mathfrak{n}\cdot x_{i_{1}}^{\alpha_{1}}x_{i_{2}}^{\alpha_{2}}\cdots
x_{i_{\ell}}^{\alpha_{\ell}}, & \text{if }\min\left(  \operatorname*{Supp}%
\mathfrak{n}\right)  \leq\min\left\{  i_{1},i_{2},\ldots,i_{\ell}\right\}  ;\\
0, & \text{if }\min\left(  \operatorname*{Supp}\mathfrak{n}\right)
>\min\left\{  i_{1},i_{2},\ldots,i_{\ell}\right\}
\end{cases}
\\
&  =\underbrace{\sum_{\substack{1\leq i_{1}<i_{2}<\cdots<i_{\ell}%
;\\\min\left(  \operatorname*{Supp}\mathfrak{n}\right)  \leq\min\left\{
i_{1},i_{2},\ldots,i_{\ell}\right\}  }}}_{\substack{=\sum_{\min\left(
\operatorname*{Supp}\mathfrak{n}\right)  \leq i_{1}<i_{2}<\cdots<i_{\ell}%
}\\=\sum_{k\leq i_{1}<i_{2}<\cdots<i_{\ell}}\\\text{(since }\min\left(
\operatorname*{Supp}\mathfrak{n}\right)  =k\text{)}}}\underbrace{\mathfrak{n}%
}_{=a}\cdot x_{i_{1}}^{\alpha_{1}}x_{i_{2}}^{\alpha_{2}}\cdots x_{i_{\ell}%
}^{\alpha_{\ell}}=\sum_{k\leq i_{1}<i_{2}<\cdots<i_{\ell}}a\cdot x_{i_{1}%
}^{\alpha_{1}}x_{i_{2}}^{\alpha_{2}}\cdots x_{i_{\ell}}^{\alpha_{\ell}}\\
&  =\left(  \sum_{k\leq i_{1}<i_{2}<\cdots<i_{\ell}}x_{i_{1}}^{\alpha_{1}%
}x_{i_{2}}^{\alpha_{2}}\cdots x_{i_{\ell}}^{\alpha_{\ell}}\right)  \cdot a.
\end{align*}
This proves (\ref{pf.thm.tvidend.a-Dless-Malpha}).}.

Let us define a map $\mathfrak{B}_{k}^{\prime}:\mathbf{k}\left[  \left[
x_{1},x_{2},x_{3},\ldots\right]  \right]  \rightarrow\mathbf{k}\left[  \left[
x_{1},x_{2},x_{3},\ldots\right]  \right]  $ by%
\[
\mathfrak{B}_{k}^{\prime}\left(  p\right)  =p\left(  x_{1},x_{2}%
,\ldots,x_{k-1},0,0,0,\ldots\right)  \ \ \ \ \ \ \ \ \ \ \text{for every }%
p\in\mathbf{k}\left[  \left[  x_{1},x_{2},x_{3},\ldots\right]  \right]
\]
(where $p\left(  x_{1},x_{2},\ldots,x_{k-1},0,0,0,\ldots\right)  $ has to be
understood as $p\left(  x_{1},x_{2},x_{3},\ldots\right)  =p$ when $k=\infty$).
Then, $\mathfrak{B}_{k}^{\prime}$ is an evaluation map (in an appropriate
sense) and thus a continuous $\mathbf{k}$-algebra homomorphism. Any monomial
$\mathfrak{m}$ satisfies%
\begin{equation}
\mathfrak{B}_{k}^{\prime}\left(  \mathfrak{m}\right)  =
\begin{cases}
\mathfrak{m}, & \text{if }\max\left(  \operatorname*{Supp}\mathfrak{m}\right)
<k;\\
0, & \text{if }\max\left(  \operatorname*{Supp}\mathfrak{m}\right)  \geq k
\end{cases}
\label{pf.thm.tvidend.monom}%
\end{equation}
\footnote{\textit{Proof.} Let $\mathfrak{m}$ be a monomial. Then,%
\begin{align*}
\mathfrak{B}_{k}^{\prime}\left(  \mathfrak{m}\right)   &  =\mathfrak{m}\left(
x_{1},x_{2},\ldots,x_{k-1},0,0,0,\ldots\right)  \ \ \ \ \ \ \ \ \ \ \left(
\text{by the definition of }\mathfrak{B}_{k}^{\prime}\right) \\
&  =\left(  \text{the result of replacing the indeterminates }x_{k}%
,x_{k+1},x_{k+2},\ldots\text{ by }0\text{ in }\mathfrak{m}\right) \\
&  =
\begin{cases}
\mathfrak{m}, & \text{if none of the indeterminates }x_{k},x_{k+1}%
,x_{k+2},\ldots\text{ appears in }\mathfrak{m}\text{;}\\
0, & \text{if some of the indeterminates }x_{k},x_{k+1},x_{k+2},\ldots\text{
appear in }\mathfrak{m}%
\end{cases}
\\
&  =
\begin{cases}
\mathfrak{m}, & \text{if }\max\left(  \operatorname*{Supp}\mathfrak{m}\right)
<k;\\
0, & \text{if }\max\left(  \operatorname*{Supp}\mathfrak{m}\right)  \geq k
\end{cases}
\end{align*}
(because none of the indeterminates $x_{k},x_{k+1},x_{k+2},\ldots$ appears in
$\mathfrak{m}$ if and only if $\max\left(  \operatorname*{Supp}\mathfrak{m}%
\right)  <k$). This proves (\ref{pf.thm.tvidend.monom}).}. Any $p\in
\mathbf{k}\left[  \left[  x_{1},x_{2},x_{3},\ldots\right]  \right]  $
satisfies%
\begin{equation}
p \tvi a=a\cdot\mathfrak{B}_{k}^{\prime}\left(  p\right)
\label{pf.thm.tvidend.B}%
\end{equation}
\footnote{\textit{Proof of (\ref{pf.thm.tvidend.B}):} Fix $p\in\mathbf{k}%
\left[  \left[  x_{1},x_{2},x_{3},\ldots\right]  \right]  $. Since the
equality (\ref{pf.thm.tvidend.B}) is $\mathbf{k}$-linear and continuous in
$p$, we can WLOG assume that $p$ is a monomial. Assume this. Hence,
$p=\mathfrak{m}$ for some monomial $\mathfrak{m}$. Consider this
$\mathfrak{m}$. We have%
\begin{equation}
\mathfrak{B}_{k}^{\prime}\left(  \underbrace{p}_{=\mathfrak{m}}\right)
=\mathfrak{B}_{k}^{\prime}\left(  \mathfrak{m}\right)  =
\begin{cases}
\mathfrak{m}, & \text{if }\max\left(  \operatorname*{Supp}\mathfrak{m}\right)
<k;\\
0, & \text{if }\max\left(  \operatorname*{Supp}\mathfrak{m}\right)  \geq k
\end{cases}
\label{pf.thm.tvidend.B.pf.2}%
\end{equation}
(by (\ref{pf.thm.tvidend.monom})). Now,%
\begin{align*}
\underbrace{p}_{=\mathfrak{m}} \tvi \underbrace{a}_{=\mathfrak{n}}  &
=\mathfrak{m} \tvi \mathfrak{n}=
\begin{cases}
\mathfrak{m}\cdot\mathfrak{n}, & \text{if }\max\left(  \operatorname*{Supp}%
\mathfrak{m}\right)  <\min\left(  \operatorname*{Supp}\mathfrak{n}\right)  ;\\
0, & \text{if }\max\left(  \operatorname*{Supp}\mathfrak{m}\right)  \geq
\min\left(  \operatorname*{Supp}\mathfrak{n}\right)
\end{cases}
\\
&  \ \ \ \ \ \ \ \ \ \ \left(  \text{by the definition of } \tvi \right) \\
&  =
\begin{cases}
\mathfrak{m}\cdot\mathfrak{n}, & \text{if }\max\left(  \operatorname*{Supp}%
\mathfrak{m}\right)  <k;\\
0, & \text{if }\max\left(  \operatorname*{Supp}\mathfrak{m}\right)  \geq k
\end{cases}
\ \ \ \ \ \ \ \ \ \ \left(  \text{since }\min\left(  \operatorname*{Supp}%
\mathfrak{n}\right)  =k\right) \\
&  =\underbrace{\mathfrak{n}}_{=a}\cdot\underbrace{
\begin{cases}
\mathfrak{m}, & \text{if }\max\left(  \operatorname*{Supp}\mathfrak{m}\right)
<k;\\
0, & \text{if }\max\left(  \operatorname*{Supp}\mathfrak{m}\right)  \geq k
\end{cases}
}_{\substack{=\mathfrak{B}_{k}^{\prime}\left(  p\right)  \\\text{(by
(\ref{pf.thm.tvidend.B.pf.2}))}}}\\
&  =a\cdot\mathfrak{B}_{k}^{\prime}\left(  p\right)  .
\end{align*}
This proves (\ref{pf.thm.tvidend.B}).}. Also, every composition $\alpha
=\left(  \alpha_{1},\alpha_{2},\ldots,\alpha_{\ell}\right)  $ satisfies%
\begin{equation}
\mathfrak{B}_{k}^{\prime}\left(  M_{\alpha}\right)  =\sum_{1\leq i_{1}%
<i_{2}<\cdots<i_{\ell}<k}x_{i_{1}}^{\alpha_{1}}x_{i_{2}}^{\alpha_{2}}\cdots
x_{i_{\ell}}^{\alpha_{\ell}} \label{pf.thm.tvidend.Malpha-bel-a}%
\end{equation}
(where the sum $\sum_{1\leq i_{1}<i_{2}<\cdots<i_{\ell}<k}x_{i_{1}}%
^{\alpha_{1}}x_{i_{2}}^{\alpha_{2}}\cdots x_{i_{\ell}}^{\alpha_{\ell}}$ has to
be interpreted as being equal to $1$, rather than being empty, when $\ell
=0$)\ \ \ \ \footnote{\textit{Proof of (\ref{pf.thm.tvidend.Malpha-bel-a}):}
Let $\alpha=\left(  \alpha_{1},\alpha_{2},\ldots,\alpha_{\ell}\right)  $ be a
composition. The definition of $M_{\alpha}$ yields $M_{\alpha}=\sum_{1\leq
i_{1}<i_{2}<\cdots<i_{\ell}}x_{i_{1}}^{\alpha_{1}}x_{i_{2}}^{\alpha_{2}}\cdots
x_{i_{\ell}}^{\alpha_{\ell}}$. Applying the map $\mathfrak{B}_{k}^{\prime}$ to
both sides of this equality, we obtain%
\begin{align*}
\mathfrak{B}_{k}^{\prime}\left(  M_{\alpha}\right)   &  =\mathfrak{B}%
_{k}^{\prime}\left(  \sum_{1\leq i_{1}<i_{2}<\cdots<i_{\ell}}x_{i_{1}}%
^{\alpha_{1}}x_{i_{2}}^{\alpha_{2}}\cdots x_{i_{\ell}}^{\alpha_{\ell}}\right)
\\
&  =\sum_{1\leq i_{1}<i_{2}<\cdots<i_{\ell}}\underbrace{\mathfrak{B}%
_{k}^{\prime}\left(  x_{i_{1}}^{\alpha_{1}}x_{i_{2}}^{\alpha_{2}}\cdots
x_{i_{\ell}}^{\alpha_{\ell}}\right)  }_{\substack{=
\begin{cases}
x_{i_{1}}^{\alpha_{1}}x_{i_{2}}^{\alpha_{2}}\cdots x_{i_{\ell}}^{\alpha_{\ell
}}, & \text{if }\max\left(  \operatorname*{Supp}\left(  x_{i_{1}}^{\alpha_{1}%
}x_{i_{2}}^{\alpha_{2}}\cdots x_{i_{\ell}}^{\alpha_{\ell}}\right)  \right)
<k;\\
0, & \text{if }\max\left(  \operatorname*{Supp}\left(  x_{i_{1}}^{\alpha_{1}%
}x_{i_{2}}^{\alpha_{2}}\cdots x_{i_{\ell}}^{\alpha_{\ell}}\right)  \right)
\geq k
\end{cases}
\\\text{(by (\ref{pf.thm.tvidend.monom}), applied to }\mathfrak{m}=x_{i_{1}%
}^{\alpha_{1}}x_{i_{2}}^{\alpha_{2}}\cdots x_{i_{\ell}}^{\alpha_{\ell}%
}\text{)}}}\\
&  \ \ \ \ \ \ \ \ \ \ \left(  \text{since }\mathfrak{B}_{k}^{\prime}\text{ is
}\mathbf{k}\text{-linear and continuous}\right) \\
&  =\sum_{1\leq i_{1}<i_{2}<\cdots<i_{\ell}}\underbrace{
\begin{cases}
x_{i_{1}}^{\alpha_{1}}x_{i_{2}}^{\alpha_{2}}\cdots x_{i_{\ell}}^{\alpha_{\ell
}}, & \text{if }\max\left(  \operatorname*{Supp}\left(  x_{i_{1}}^{\alpha_{1}%
}x_{i_{2}}^{\alpha_{2}}\cdots x_{i_{\ell}}^{\alpha_{\ell}}\right)  \right)
<k;\\
0, & \text{if }\max\left(  \operatorname*{Supp}\left(  x_{i_{1}}^{\alpha_{1}%
}x_{i_{2}}^{\alpha_{2}}\cdots x_{i_{\ell}}^{\alpha_{\ell}}\right)  \right)
\geq k
\end{cases}
}_{\substack{=
\begin{cases}
x_{i_{1}}^{\alpha_{1}}x_{i_{2}}^{\alpha_{2}}\cdots x_{i_{\ell}}^{\alpha_{\ell
}}, & \text{if }\max\left\{  i_{1},i_{2},\ldots,i_{\ell}\right\}  <k;\\
0, & \text{if }\max\left\{  i_{1},i_{2},\ldots,i_{\ell}\right\}  \geq k
\end{cases}
\\\text{(since }\operatorname*{Supp}\left(  x_{i_{1}}^{\alpha_{1}}x_{i_{2}%
}^{\alpha_{2}}\cdots x_{i_{\ell}}^{\alpha_{\ell}}\right)  =\left\{
i_{1},i_{2},\ldots,i_{\ell}\right\}  \text{)}}}\\
&  =\sum_{1\leq i_{1}<i_{2}<\cdots<i_{\ell}}
\begin{cases}
x_{i_{1}}^{\alpha_{1}}x_{i_{2}}^{\alpha_{2}}\cdots x_{i_{\ell}}^{\alpha_{\ell
}}, & \text{if }\max\left\{  i_{1},i_{2},\ldots,i_{\ell}\right\}  <k;\\
0, & \text{if }\max\left\{  i_{1},i_{2},\ldots,i_{\ell}\right\}  \geq k
\end{cases}
\\
&  =\underbrace{\sum_{\substack{1\leq i_{1}<i_{2}<\cdots<i_{\ell}%
;\\\max\left\{  i_{1},i_{2},\ldots,i_{\ell}\right\}  <k}}}_{=\sum_{1\leq
i_{1}<i_{2}<\cdots<i_{\ell}<k}}x_{i_{1}}^{\alpha_{1}}x_{i_{2}}^{\alpha_{2}%
}\cdots x_{i_{\ell}}^{\alpha_{\ell}}=\sum_{1\leq i_{1}<i_{2}<\cdots<i_{\ell
}<k}x_{i_{1}}^{\alpha_{1}}x_{i_{2}}^{\alpha_{2}}\cdots x_{i_{\ell}}%
^{\alpha_{\ell}}.
\end{align*}
This proves (\ref{pf.thm.tvidend.Malpha-bel-a}).}.

We shall use one further obvious observation: If $i_{1},i_{2},\ldots,i_{\ell}$
are some positive integers satisfying $i_{1}<i_{2}<\cdots<i_{\ell}$, then%
\begin{equation}
\text{there exists exactly one }j\in\left\{  0,1,\ldots,\ell\right\}  \text{
satisfying }i_{j}<k\leq i_{j+1} \label{pf.thm.tvidend.exactlyonej}%
\end{equation}
(where $i_{0}$ is to be understood as $0$, and $i_{\ell+1}$ as $\infty$).

Let us now notice that every $f\in\operatorname*{QSym}$ satisfies%
\begin{equation}
af=\sum_{\left(  f\right)  }\mathfrak{B}_{k}^{\prime}\left(  f_{\left(
1\right)  }\right)  \left(  a\left.  \preceq\right.  f_{\left(  2\right)
}\right)  . \label{pf.thm.tvidend.mainlem}%
\end{equation}

\textit{Proof of (\ref{pf.thm.tvidend.mainlem}):} Both sides of the equality
(\ref{pf.thm.tvidend.mainlem}) are $\mathbf{k}$-linear in $f$. Hence, it is
enough to check (\ref{pf.thm.tvidend.mainlem}) on the basis $\left(
M_{\gamma}\right)  _{\gamma\in\operatorname*{Comp}}$ of $\operatorname*{QSym}%
$, that is, to prove that (\ref{pf.thm.tvidend.mainlem}) holds whenever
$f=M_{\gamma}$ for some $\gamma\in\operatorname*{Comp}$. In other words, it is
enough to show that
\[
aM_{\gamma}=\sum_{\left(  M_{\gamma}\right)  }\mathfrak{B}_{k}^{\prime}\left(
\left(  M_{\gamma}\right)  _{\left(  1\right)  }\right)  \cdot\left(  a\left.
\preceq\right.  \left(  M_{\gamma}\right)  _{\left(  2\right)  }\right)
\ \ \ \ \ \ \ \ \ \ \text{for every }\gamma\in\operatorname*{Comp}.
\]
But this is easily done: Let $\gamma\in\operatorname*{Comp}$. Write $\gamma$
in the form $\gamma=\left(  \gamma_{1},\gamma_{2},\ldots,\gamma_{\ell}\right)
$. Then,%
\begin{align*}
&  \sum_{\left(  M_{\gamma}\right)  }\mathfrak{B}_{k}^{\prime}\left(  \left(
M_{\gamma}\right)  _{\left(  1\right)  }\right)  \cdot\left(  a\left.
\preceq\right.  \left(  M_{\gamma}\right)  _{\left(  2\right)  }\right) \\
&  =\sum_{j=0}^{\ell}\underbrace{\mathfrak{B}_{k}^{\prime}\left(  M_{\left(
\gamma_{1},\gamma_{2},\ldots,\gamma_{j}\right)  }\right)  }_{\substack{=\sum
_{1\leq i_{1}<i_{2}<\cdots<i_{j}<k}x_{i_{1}}^{\gamma_{1}}x_{i_{2}}^{\gamma
_{2}}\cdots x_{i_{j}}^{\gamma_{j}}\\\text{(by
(\ref{pf.thm.tvidend.Malpha-bel-a}))}}}\cdot\underbrace{\left(  a\left.
\preceq\right.  M_{\left(  \gamma_{j+1},\gamma_{j+2},\ldots,\gamma_{\ell
}\right)  }\right)  }_{\substack{=\left(  \sum_{k\leq i_{1}<i_{2}%
<\cdots<i_{\ell-j}}x_{i_{1}}^{\gamma_{j+1}}x_{i_{2}}^{\gamma_{j+2}}\cdots
x_{i_{\ell-j}}^{\gamma_{\ell}}\right)  \cdot a\\\text{(by
(\ref{pf.thm.tvidend.a-Dless-Malpha}))}}}\\
&  \ \ \ \ \ \ \ \ \ \ \left(  \text{since }\sum_{\left(  M_{\gamma}\right)
}\left(  M_{\gamma}\right)  _{\left(  1\right)  }\otimes\left(  M_{\gamma
}\right)  _{\left(  2\right)  }=\Delta\left(  M_{\gamma}\right)  =\sum
_{j=0}^{\ell}M_{\left(  \gamma_{1},\gamma_{2},\ldots,\gamma_{j}\right)
}\otimes M_{\left(  \gamma_{j+1},\gamma_{j+2},\ldots,\gamma_{\ell}\right)
}\right) \\
&  =\sum_{j=0}^{\ell}\left(  \sum_{1\leq i_{1}<i_{2}<\cdots<i_{j}<k}x_{i_{1}%
}^{\gamma_{1}}x_{i_{2}}^{\gamma_{2}}\cdots x_{i_{j}}^{\gamma_{j}}\right)
\underbrace{\left(  \sum_{k\leq i_{1}<i_{2}<\cdots<i_{\ell-j}}x_{i_{1}%
}^{\gamma_{j+1}}x_{i_{2}}^{\gamma_{j+2}}\cdots x_{i_{\ell-j}}^{\gamma_{\ell}%
}\right)  }_{\substack{=\sum_{k\leq i_{j+1}<i_{j+2}<\cdots<i_{\ell}}%
x_{i_{j+1}}^{\gamma_{j+1}}x_{i_{j+2}}^{\gamma_{j+2}}\cdots x_{i_{\ell}%
}^{\gamma_{\ell}}\\\text{(here, we have renamed the summation index}\\\left(
i_{1},i_{2},\ldots,i_{\ell-j}\right)  \text{ as }\left(  i_{j+1}%
,i_{j+2},\ldots,i_{\ell}\right)  \text{)}}}\cdot\,a\\
&  =\sum_{j=0}^{\ell}\left(  \sum_{1\leq i_{1}<i_{2}<\cdots<i_{j}<k}x_{i_{1}%
}^{\gamma_{1}}x_{i_{2}}^{\gamma_{2}}\cdots x_{i_{j}}^{\gamma_{j}}\right)
\left(  \sum_{k\leq i_{j+1}<i_{j+2}<\cdots<i_{\ell}}x_{i_{j+1}}^{\gamma_{j+1}%
}x_{i_{j+2}}^{\gamma_{j+2}}\cdots x_{i_{\ell}}^{\gamma_{\ell}}\right)  \cdot
a\\
&  =\underbrace{\sum_{j=0}^{\ell}\ \ \sum_{1\leq i_{1}<i_{2}<\cdots<i_{j}%
<k}\ \ \sum_{k\leq i_{j+1}<i_{j+2}<\cdots<i_{\ell}}}_{\substack{=\sum_{1\leq
i_{1}<i_{2}<\cdots<i_{\ell}}\ \ \sum_{\substack{j\in\left\{  0,1,\ldots
,\ell\right\}  ;\\i_{j}<k\leq i_{j+1}}}\\\text{(where }i_{0}\text{ is to be
understood as }0\text{, and }i_{\ell+1}\text{ as }\infty\text{)}%
}}\underbrace{\left(  x_{i_{1}}^{\gamma_{1}}x_{i_{2}}^{\gamma_{2}}\cdots
x_{i_{j}}^{\gamma_{j}}\right)  \left(  x_{i_{j+1}}^{\gamma_{j+1}}x_{i_{j+2}%
}^{\gamma_{j+2}}\cdots x_{i_{\ell}}^{\gamma_{\ell}}\right)  }_{=x_{i_{1}%
}^{\gamma_{1}}x_{i_{2}}^{\gamma_{2}}\cdots x_{i_{\ell}}^{\gamma_{\ell}}}%
\cdot\,a\\
&  =\sum_{1\leq i_{1}<i_{2}<\cdots<i_{\ell}}\ \ \underbrace{\sum
_{\substack{j\in\left\{  0,1,\ldots,\ell\right\}  ;\\i_{j}<k\leq i_{j+1}%
}}x_{i_{1}}^{\gamma_{1}}x_{i_{2}}^{\gamma_{2}}\cdots x_{i_{\ell}}%
^{\gamma_{\ell}}}_{\substack{\text{this sum has precisely one addend}%
\\\text{(because of (\ref{pf.thm.tvidend.exactlyonej})),}\\\text{and thus
equals }x_{i_{1}}^{\gamma_{1}}x_{i_{2}}^{\gamma_{2}}\cdots x_{i_{\ell}%
}^{\gamma_{\ell}}}}\cdot\,a\\
&  =\underbrace{\sum_{1\leq i_{1}<i_{2}<\cdots<i_{\ell}}x_{i_{1}}^{\gamma_{1}%
}x_{i_{2}}^{\gamma_{2}}\cdots x_{i_{\ell}}^{\gamma_{\ell}}}_{=M_{\gamma}}%
\cdot\,a\\
&  =M_{\gamma}\cdot a=aM_{\gamma},
\end{align*}
qed. Thus, (\ref{pf.thm.tvidend.mainlem}) is proven.

Now, every $b\in\operatorname*{QSym}$ satisfies%
\begin{align*}
&  \sum_{\left(  b\right)  }\underbrace{\left(  S\left(  b_{\left(  1\right)
}\right)  \tvi a\right)  }_{\substack{=a\cdot\mathfrak{B}_{k}^{\prime}\left(
S\left(  b_{\left(  1\right)  }\right)  \right)  \\\text{(by
(\ref{pf.thm.tvidend.B}), applied to }p=S\left(  b_{\left(  1\right)
}\right)  \text{)}}}b_{\left(  2\right)  }\\
&  =\sum_{\left(  b\right)  }a\cdot\mathfrak{B}_{k}^{\prime}\left(  S\left(
b_{\left(  1\right)  }\right)  \right)  b_{\left(  2\right)  }=\sum_{\left(
b\right)  }\mathfrak{B}_{k}^{\prime}\left(  S\left(  b_{\left(  1\right)
}\right)  \right)  \cdot\underbrace{ab_{\left(  2\right)  }}_{\substack{=\sum
_{\left(  b_{\left(  2\right)  }\right)  }\mathfrak{B}_{k}^{\prime}\left(
\left(  b_{\left(  2\right)  }\right)  _{\left(  1\right)  }\right)  \left(
a\left.  \preceq\right.  \left(  b_{\left(  2\right)  }\right)  _{\left(
2\right)  }\right)  \\\text{(by (\ref{pf.thm.tvidend.mainlem}), applied to
}f=b_{\left(  2\right)  }\text{)}}}\\
&  =\sum_{\left(  b\right)  }\mathfrak{B}_{k}^{\prime}\left(  S\left(
b_{\left(  1\right)  }\right)  \right)  \left(  \sum_{\left(  b_{\left(
2\right)  }\right)  }\mathfrak{B}_{k}^{\prime}\left(  \left(  b_{\left(
2\right)  }\right)  _{\left(  1\right)  }\right)  \left(  a\left.
\preceq\right.  \left(  b_{\left(  2\right)  }\right)  _{\left(  2\right)
}\right)  \right) \\
&  =\sum_{\left(  b\right)  }\sum_{\left(  b_{\left(  2\right)  }\right)
}\mathfrak{B}_{k}^{\prime}\left(  S\left(  b_{\left(  1\right)  }\right)
\right)  \mathfrak{B}_{k}^{\prime}\left(  \left(  b_{\left(  2\right)
}\right)  _{\left(  1\right)  }\right)  \left(  a\left.  \preceq\right.
\left(  b_{\left(  2\right)  }\right)  _{\left(  2\right)  }\right) \\
&  =\sum_{\left(  b\right)  }\underbrace{\sum_{\left(  b_{\left(  1\right)
}\right)  }\mathfrak{B}_{k}^{\prime}\left(  S\left(  \left(  b_{\left(
1\right)  }\right)  _{\left(  1\right)  }\right)  \right)  \mathfrak{B}%
_{k}^{\prime}\left(  \left(  b_{\left(  1\right)  }\right)  _{\left(
2\right)  }\right)  }_{\substack{=\mathfrak{B}_{k}^{\prime}\left(
\sum_{\left(  b_{\left(  1\right)  }\right)  }S\left(  \left(  b_{\left(
1\right)  }\right)  _{\left(  1\right)  }\right)  \cdot\left(  b_{\left(
1\right)  }\right)  _{\left(  2\right)  }\right)  \\\text{(since }%
\mathfrak{B}_{k}^{\prime}\text{ is a }\mathbf{k}\text{-algebra homomorphism)}%
}}\left(  a\left.  \preceq\right.  b_{\left(  2\right)  }\right) \\
&  \ \ \ \ \ \ \ \ \ \ \left(
\begin{array}
[c]{c}%
\text{since the coassociativity of }\Delta\text{ yields}\\
\sum_{\left(  b\right)  }\sum_{\left(  b_{\left(  2\right)  }\right)
}b_{\left(  1\right)  }\otimes\left(  b_{\left(  2\right)  }\right)  _{\left(
1\right)  }\otimes\left(  b_{\left(  2\right)  }\right)  _{\left(  2\right)
}=\sum_{\left(  b\right)  }\sum_{\left(  b_{\left(  1\right)  }\right)
}\left(  b_{\left(  1\right)  }\right)  _{\left(  1\right)  }\otimes\left(
b_{\left(  1\right)  }\right)  _{\left(  2\right)  }\otimes b_{\left(
2\right)  }%
\end{array}
\right) \\
&  =\sum_{\left(  b\right)  }\mathfrak{B}_{k}^{\prime}\left(  \underbrace{\sum
_{\left(  b_{\left(  1\right)  }\right)  }S\left(  \left(  b_{\left(
1\right)  }\right)  _{\left(  1\right)  }\right)  \left(  b_{\left(  1\right)
}\right)  _{\left(  2\right)  }}_{\substack{=\varepsilon\left(  b_{\left(
1\right)  }\right)  \\\text{(by one of the defining equations of the
antipode)}}}\right)  \left(  a\left.  \preceq\right.  b_{\left(  2\right)
}\right) \\
&  =\sum_{\left(  b\right)  }\underbrace{\mathfrak{B}_{k}^{\prime}\left(
\varepsilon\left(  b_{\left(  1\right)  }\right)  \right)  }%
_{\substack{=\varepsilon\left(  b_{\left(  1\right)  }\right)  \\\text{(since
}\mathfrak{B}_{k}^{\prime}\text{ is a }\mathbf{k}\text{-algebra}%
\\\text{homomorphism, and}\\\varepsilon\left(  b_{\left(  1\right)  }\right)
\in\mathbf{k}\text{ is a scalar)}}}\left(  a\left.  \preceq\right.  b_{\left(
2\right)  }\right)  =\sum_{\left(  b\right)  }\varepsilon\left(  b_{\left(
1\right)  }\right)  \cdot\left(  a\left.  \preceq\right.  b_{\left(  2\right)
}\right) \\
&  =\sum_{\left(  b\right)  }a\left.  \preceq\right.  \left(  \varepsilon
\left(  b_{\left(  1\right)  }\right)  b_{\left(  2\right)  }\right)
=a\left.  \preceq\right.  \underbrace{\left(  \sum_{\left(  b\right)
}\varepsilon\left(  b_{\left(  1\right)  }\right)  b_{\left(  2\right)
}\right)  }_{=b}=a\left.  \preceq\right.  b.
\end{align*}

This proves Theorem \ref{thm.tvidend}.
\end{proof}

\begin{noncompile}
Here is another proof of Theorem \ref{thm.tvidend}, which is more or less the
same as the above but using a slightly different language.

\begin{proof}
[Second proof of Theorem \ref{thm.tvidend} (sketched).]We can WLOG assume that
$a$ is a monomial (because of the $\mathbf{k}$-linearity and continuity of all
operations in sight). So assume this. Let $i=\min\left(  \operatorname*{Supp}%
a\right)  $. (Note that this $i$ can be $\infty$.) Then, it is easy to see
that every $f\in\operatorname*{QSym}$ satisfies $f\tvi a=f\left(  x_{1}%
,x_{2},\ldots,x_{i-1}\right)  \cdot a$ (where we evaluate quasisymmetric
functions on variables in the usual way, and for $i=\infty$ we simply regard
$f\left(  x_{1},x_{2},\ldots,x_{i-1}\right)  $ as meaning $f\left(
x_{1},x_{2},x_{3},\ldots\right)  =f$) and $a\left.  \preceq\right.  f=a\cdot
f\left(  x_{i},x_{i+1},x_{i+2},\ldots\right)  $ (which is supposed to mean
$\varepsilon\left(  f\right)  $ if $i=\infty$). Hence, the equality in
question, $\sum_{\left(  b\right)  }\left(  S\left(  b_{\left(  1\right)
}\right)  \tvi a\right)  b_{\left(  2\right)  }=a\left.  \preceq\right.  b$,
rewrites as%
\[
\sum_{\left(  b\right)  }\left(  S\left(  b_{\left(  1\right)  }\right)
\right)  \left(  x_{1},x_{2},\ldots,x_{i-1}\right)  \cdot a\cdot b_{\left(
2\right)  }=a\cdot b\left(  x_{i},x_{i+1},x_{i+2},\ldots\right)  .
\]
This will immediately follow if we can show that%
\begin{equation}
\sum_{\left(  b\right)  }\left(  S\left(  b_{\left(  1\right)  }\right)
\right)  \left(  x_{1},x_{2},\ldots,x_{i-1}\right)  \cdot b_{\left(  2\right)
}=b\left(  x_{i},x_{i+1},x_{i+2},\ldots\right)  . \label{pf.tvidend.2}%
\end{equation}
So we need to prove (\ref{pf.tvidend.2}).

Every $f\in\operatorname*{QSym}$ satisfies%
\begin{equation}
f=\sum_{\left(  f\right)  }f_{\left(  1\right)  }\left(  x_{1},x_{2}%
,\ldots,x_{i-1}\right)  \cdot f_{\left(  2\right)  }\left(  x_{i}%
,x_{i+1},x_{i+2},\ldots\right)  . \label{pf.tvidend.4}%
\end{equation}
\footnote{In fact, this follows from the fact that the ordered alphabet
$\left(  x_{1},x_{2},x_{3},\ldots\right)  $ can be regarded as the sum of the
ordered alphabets $\left(  x_{1},x_{2},\ldots,x_{i-1}\right)  $ and $\left(
x_{i},x_{i+1},x_{i+2},\ldots\right)  $. Or, more down-to-earthly, it can be
checked on the basis $\left(  M_{\gamma}\right)  _{\gamma\in
\operatorname*{Comp}}$ of $\operatorname*{QSym}$ in a few lines.} Now,
(\ref{pf.tvidend.2}) becomes%
\begin{align*}
&  \sum_{\left(  b\right)  }\left(  S\left(  b_{\left(  1\right)  }\right)
\right)  \left(  x_{1},x_{2},\ldots,x_{i-1}\right)  \cdot
\underbrace{b_{\left(  2\right)  }}_{\substack{=\sum_{\left(  b_{\left(
2\right)  }\right)  }\left(  b_{\left(  2\right)  }\right)  _{\left(
1\right)  }\left(  x_{1},x_{2},\ldots,x_{i-1}\right)  \cdot\left(  b_{\left(
2\right)  }\right)  _{\left(  2\right)  }\left(  x_{i},x_{i+1},x_{i+2}%
,\ldots\right)  \\\text{(by (\ref{pf.tvidend.4}))}}}\\
&  =\sum_{\left(  b\right)  }\sum_{\left(  b_{\left(  2\right)  }\right)
}\left(  S\left(  b_{\left(  1\right)  }\right)  \right)  \left(  x_{1}%
,x_{2},\ldots,x_{i-1}\right)  \cdot\left(  b_{\left(  2\right)  }\right)
_{\left(  1\right)  }\left(  x_{1},x_{2},\ldots,x_{i-1}\right)  \cdot\left(
b_{\left(  2\right)  }\right)  _{\left(  2\right)  }\left(  x_{i}%
,x_{i+1},x_{i+2},\ldots\right) \\
&  =\sum_{\left(  b\right)  }\underbrace{\sum_{\left(  b_{\left(  1\right)
}\right)  }\left(  S\left(  \left(  b_{\left(  1\right)  }\right)  _{\left(
1\right)  }\right)  \right)  \left(  x_{1},x_{2},\ldots,x_{i-1}\right)
\cdot\left(  b_{\left(  1\right)  }\right)  _{\left(  2\right)  }\left(
x_{1},x_{2},\ldots,x_{i-1}\right)  }_{=\left(  \sum_{\left(  b_{\left(
1\right)  }\right)  }S\left(  \left(  b_{\left(  1\right)  }\right)  _{\left(
1\right)  }\right)  \cdot\left(  b_{\left(  1\right)  }\right)  _{\left(
2\right)  }\right)  \left(  x_{1},x_{2},\ldots,x_{i-1}\right)  }\cdot
b_{\left(  2\right)  }\left(  x_{i},x_{i+1},x_{i+2},\ldots\right) \\
&  \ \ \ \ \ \ \ \ \ \ \left(  \text{by the coassociativity of }\Delta\right)
\\
&  =\sum_{\left(  b\right)  }\underbrace{\left(  \sum_{\left(  b_{\left(
1\right)  }\right)  }S\left(  \left(  b_{\left(  1\right)  }\right)  _{\left(
1\right)  }\right)  \cdot\left(  b_{\left(  1\right)  }\right)  _{\left(
2\right)  }\right)  }_{\substack{=\varepsilon\left(  b_{\left(  1\right)
}\right)  \\\text{(by one of the defining equations of the antipode)}}}\left(
x_{1},x_{2},\ldots,x_{i-1}\right)  \cdot b_{\left(  2\right)  }\left(
x_{i},x_{i+1},x_{i+2},\ldots\right) \\
&  =\sum_{\left(  b\right)  }\underbrace{\left(  \varepsilon\left(  b_{\left(
1\right)  }\right)  \right)  \left(  x_{1},x_{2},\ldots,x_{i-1}\right)
}_{=\varepsilon\left(  b_{\left(  1\right)  }\right)  }\cdot b_{\left(
2\right)  }\left(  x_{i},x_{i+1},x_{i+2},\ldots\right) \\
&  =\sum_{\left(  b\right)  }\varepsilon\left(  b_{\left(  1\right)  }\right)
\cdot b_{\left(  2\right)  }\left(  x_{i},x_{i+1},x_{i+2},\ldots\right)
=\underbrace{\left(  \sum_{\left(  b\right)  }\varepsilon\left(  b_{\left(
1\right)  }\right)  \cdot b_{\left(  2\right)  }\right)  }%
_{\substack{=b\\\text{(by the coalgebra axioms)}}}\left(  x_{i},x_{i+1}%
,x_{i+2},\ldots\right) \\
&  =b\left(  x_{i},x_{i+1},x_{i+2},\ldots\right)  .
\end{align*}
This proves (\ref{pf.tvidend.2}) and thus Theorem \ref{thm.tvidend}. (A
different proof might appear in \cite{Gri-gammapart}.)
\end{proof}
\end{noncompile}
\end{verlong}

\section{\label{sect.dualimmac}Dual immaculate functions and the operation
$\left.  \prec\right.  $}

We will now study the dual immaculate functions defined in \cite{BBSSZ}.
However, instead of defining them as was done in \cite[Section 3.7]{BBSSZ}, we
shall give a different (but equivalent) definition. First, we introduce
immaculate tableaux (which we define as in \cite[Definition 3.9]{BBSSZ}),
which are an analogue of the well-known semistandard Young tableaux (also
known as \textquotedblleft column-strict tableaux\textquotedblright%
)\footnote{See, e.g., \cite[Chapter 7]{Stanley-EC2} for a study of
semistandard Young tableaux. We will not use them in this note; however, our
terminology for immaculate tableaux will imitate some of the classical
terminology defined for semistandard Young tableaux.}:

\Needspace{15cm}

\begin{definition}
\label{def.immactab}Let $\alpha=\left(  \alpha_{1},\alpha_{2},\ldots
,\alpha_{\ell}\right)  $ be a composition.

\textbf{(a)} The \textit{Young diagram} of $\alpha$ will mean the subset
$\left\{  \left(  i,j\right)  \in\mathbb{Z}^{2}\ \mid\ 1\leq i\leq\ell;\ 1\leq
j\leq\alpha_{i}\right\}  $ of $\mathbb{Z}^{2}$. It is denoted by $Y\left(
\alpha\right)  $.

\textbf{(b)} An \textit{immaculate tableau of shape }$\alpha$ will mean a map
$T:Y\left(  \alpha\right)  \rightarrow\left\{  1,2,3,\ldots\right\}  $ which
satisfies the following two axioms:

\begin{enumerate}
\item We have $T\left(  i,1\right)  <T\left(  j,1\right)  $ for any integers
$i$ and $j$ satisfying $1\leq i<j\leq\ell$.

\item We have $T\left(  i,u\right)  \leq T\left(  i,v\right)  $ for any
integers $i$, $u$ and $v$ satisfying $1\leq i\leq\ell$ and $1\leq
u<v\leq\alpha_{i}$.
\end{enumerate}

The \textit{entries} of an immaculate tableau $T$ mean the images of elements
of $Y\left(  \alpha\right)  $ under $T$.

We will use the same graphical representation of immaculate tableaux
(analogous to the \textquotedblleft English notation\textquotedblright\ for
semistandard Young tableaux) that was used in \cite{BBSSZ}: An immaculate
tableau $T$ of shape $\alpha=\left(  \alpha_{1},\alpha_{2},\ldots,\alpha
_{\ell}\right)  $ is represented as a table whose rows are left-aligned (but
can have different lengths), and whose $i$-th row (counted from top) has
$\alpha_{i}$ boxes, which are respectively filled with the entries $T\left(
i,1\right)  $, $T\left(  i,2\right)  $, $\ldots$, $T\left(  i,\alpha
_{i}\right)  $ (from left to right). For example, an immaculate tableau $T$ of
shape $\left(  3,1,2\right)  $ is represented by the picture%
\[
\ytableausetup{boxsize=2em} \begin{ytableau}
a_{1,1} & a_{1,2} & a_{1,3} \\
a_{2,1} \\
a_{3,1} & a_{3,2}
\end{ytableau}
\quad,
\]
where $a_{i,j}=T\left(  i,j\right)  $ for every $\left(  i,j\right)  \in
Y\left(  \left(  3,1,2\right)  \right)  $. Thus, the first of the above two
axioms for an immaculate tableau $T$ says that the entries of $T$ are strictly
increasing down the first column of $Y\left(  \alpha\right)  $, whereas the
second of the above two axioms says that the entries of $T$ are weakly
increasing along each row of $Y\left(  \alpha\right)  $.

\textbf{(c)} Let $\beta=\left(  \beta_{1},\beta_{2},\ldots,\beta_{k}\right)  $
be a composition of $\left\vert \alpha\right\vert $. An immaculate tableau $T$
of shape $\alpha$ is said to have \textit{content }$\beta$ if every
$j\in\left\{  1,2,3,\ldots\right\}  $ satisfies%
\[
\left\vert T^{-1}\left(  j\right)  \right\vert =
\begin{cases}
\beta_{j}, & \text{if }j\leq k;\\
0, & \text{if }j>k
\end{cases}
\quad.
\]
Notice that not every immaculate tableau has a content (with this definition),
because we only allow compositions as contents. More precisely, if $T$ is an
immaculate tableau of shape $\alpha$, then there exists a composition $\beta$
such that $T$ has content $\beta$ if and only if there exists a $k\in
\mathbb{N}$ such that $T\left(  Y\left(  \alpha\right)  \right)  =\left\{
1,2,\ldots,k\right\}  $.

\textbf{(d)} Let $\beta$ be a composition of $\left\vert \alpha\right\vert $.
Then, $K_{\alpha,\beta}$ denotes the number of immaculate tableaux of shape
$\alpha$ and content $\beta$.
\end{definition}

\begin{noncompile}
If $\alpha$ is a composition and $T$ is an immaculate tableau of shape
$\alpha$, then we refer to the image of $\left(  i,j\right)  \in Y\left(
\alpha\right)  $ under $T$ as the \textit{entry}\ of the cell $\left(
i,j\right)  $ in the tableau $T$.
\end{noncompile}

\begin{vershort}
For future reference, let us notice that if $\alpha$ is a composition, if $T$
is an immaculate tableau of shape $\alpha$, and if $\left(  i,j\right)  \in
Y\left(  \alpha\right)  $ is such that $i>1$, then%
\begin{equation}
T\left(  1,1\right)  <T\left(  i,1\right)  \leq T\left(  i,j\right)  .
\label{eq.immactab.min.2.short}%
\end{equation}

\end{vershort}

\begin{verlong}
For future reference, let us notice that if $\alpha$ is a nonempty composition
and if $T$ is an immaculate tableau of shape $\alpha$, then%
\begin{equation}
\text{the smallest entry of }T\text{ is }T\left(  1,1\right)
\label{eq.immactab.min}%
\end{equation}
(because every $\left(  i,j\right)  \in Y\left(  \alpha\right)  $ satisfies
$T\left(  1,1\right)  \leq T\left(  i,1\right)  \leq T\left(  i,j\right)  $).
Moreover, if $\alpha$ is a composition, if $T$ is an immaculate tableau of
shape $\alpha$, and if $\left(  i,j\right)  \in Y\left(  \alpha\right)  $ is
such that $i>1$, then%
\begin{equation}
T\left(  1,1\right)  <T\left(  i,1\right)  \leq T\left(  i,j\right)  .
\label{eq.immactab.min.2}%
\end{equation}

\end{verlong}

\begin{definition}
\label{def.dualimmac}Let $\alpha$ be a composition. The \textit{dual
immaculate function }$\mathfrak{S}_{\alpha}^{\ast}$ corresponding to $\alpha$
is defined as the quasisymmetric function%
\[
\sum_{\beta\models\left\vert \alpha\right\vert }K_{\alpha,\beta}M_{\beta}.
\]

\end{definition}

This definition is not identical to the definition of $\mathfrak{S}_{\alpha
}^{\ast}$ used in \cite{BBSSZ}, but it is equivalent to it, as the following
proposition shows.

\begin{proposition}
\label{prop.dualimmac.equiv}Definition \ref{def.dualimmac} is equivalent to
the definition of $\mathfrak{S}_{\alpha}^{\ast}$ used in \cite{BBSSZ}.
\end{proposition}

\begin{vershort}
\begin{proof}
[Proof of Proposition \ref{prop.dualimmac.equiv}.]Let $\leq_{\ell}$ denote the
lexicographic order on compositions. Let $\alpha$ be a composition. From
\cite[Proposition 3.15 (2)]{BBSSZ}, we know that $K_{\alpha,\beta}=0$ for
every $\beta\models\left\vert \alpha\right\vert $ which does not satisfy
$\beta\leq_{\ell}\alpha$. Hence, in the sum $\sum_{\beta\models\left\vert
\alpha\right\vert }K_{\alpha,\beta}M_{\beta}$, only the compositions $\beta$
satisfying $\beta\leq_{\ell}\alpha$ contribute nonzero addends. Consequently,
$\sum_{\beta\models\left\vert \alpha\right\vert }K_{\alpha,\beta}M_{\beta
}=\sum_{\substack{\beta\models\left\vert \alpha\right\vert ;\\\beta\leq_{\ell
}\alpha}}K_{\alpha,\beta}M_{\beta}$. The left hand side of this equality is
$\mathfrak{S}_{\alpha}^{\ast}$ according to our definition, whereas the right
hand side is $\mathfrak{S}_{\alpha}^{\ast}$ as defined in \cite{BBSSZ} (by
\cite[Proposition 3.36]{BBSSZ}). Hence, the two definitions are equivalent.
\end{proof}
\end{vershort}

\begin{verlong}
\begin{proof}
[Proof of Proposition \ref{prop.dualimmac.equiv}.]Let $\leq_{\ell}$ denote the
lexicographic order on compositions.

Let $\alpha$ be a composition. Then, \cite[Proposition 3.36]{BBSSZ} yields the
following:%
\[
\left(  \text{the dual immaculate function }\mathfrak{S}_{\alpha}^{\ast}\text{
as defined in \cite{BBSSZ}}\right)  =\sum_{\substack{\beta\models\left\vert
\alpha\right\vert ;\\\beta\leq_{\ell}\alpha}}K_{\alpha,\beta}M_{\beta}.
\]
Compared with%
\begin{align*}
&  \left(  \text{the dual immaculate function }\mathfrak{S}_{\alpha}^{\ast
}\text{ as defined in Definition \ref{def.dualimmac}}\right) \\
&  =\sum_{\beta\models\left\vert \alpha\right\vert }K_{\alpha,\beta}M_{\beta
}=\sum_{\substack{\beta\models\left\vert \alpha\right\vert ;\\\beta\leq_{\ell
}\alpha}}K_{\alpha,\beta}M_{\beta}+\sum_{\substack{\beta\models\left\vert
\alpha\right\vert ;\\\text{not }\beta\leq_{\ell}\alpha}}\underbrace{K_{\alpha
,\beta}}_{\substack{=0\\\text{(by \cite[Proposition 3.15 (2)]{BBSSZ})}%
}}M_{\beta}\\
&  =\sum_{\substack{\beta\models\left\vert \alpha\right\vert ;\\\beta
\leq_{\ell}\alpha}}K_{\alpha,\beta}M_{\beta}+\underbrace{\sum_{\substack{\beta
\models\left\vert \alpha\right\vert ;\\\text{not }\beta\leq_{\ell}\alpha
}}0M_{\beta}}_{=0}=\sum_{\substack{\beta\models\left\vert \alpha\right\vert
;\\\beta\leq_{\ell}\alpha}}K_{\alpha,\beta}M_{\beta},
\end{align*}
this yields%
\begin{align*}
&  \left(  \text{the dual immaculate function }\mathfrak{S}_{\alpha}^{\ast
}\text{ as defined in \cite{BBSSZ}}\right) \\
&  =\left(  \text{the dual immaculate function }\mathfrak{S}_{\alpha}^{\ast
}\text{ as defined in Definition \ref{def.dualimmac}}\right)  .
\end{align*}
Hence, Definition \ref{def.dualimmac} is equivalent to the definition of
$\mathfrak{S}_{\alpha}^{\ast}$ used in \cite{BBSSZ}. This proves Proposition
\ref{prop.dualimmac.equiv}.
\end{proof}
\end{verlong}

It is helpful to think of dual immaculate functions as analogues of Schur
functions obtained by replacing semistandard Young tableaux by immaculate
tableaux. Definition \ref{def.dualimmac} is the analogue of the well-known
formula $s_{\lambda}=\sum_{\mu\vdash\left\vert \lambda\right\vert }%
k_{\lambda,\mu}m_{\mu}$ for any partition $\lambda$, where $s_{\lambda}$
denotes the Schur function corresponding to $\lambda$, where $m_{\mu}$ denotes
the monomial symmetric function corresponding to the partition $\mu$, and
where $k_{\lambda,\mu}$ is the $\left(  \lambda,\mu\right)  $-th Kostka number
(i.e., the number of semistandard Young tableaux of shape $\lambda$ and
content $\mu$). The following formula for the $\mathfrak{S}_{\alpha}^{\ast}$
(known to the authors of \cite{BBSSZ} but not explicitly stated in their work)
should not come as a surprise:

\begin{proposition}
\label{prop.dualImm}Let $\alpha$ be a composition. Then,%
\[
\mathfrak{S}_{\alpha}^{\ast}=\sum_{\substack{T\text{ is an immaculate}%
\\\text{tableau of shape }\alpha}}\mathbf{x}_{T}.
\]
Here, $\mathbf{x}_{T}$ is defined as $\prod_{\left(  i,j\right)  \in Y\left(
\alpha\right)  }x_{T\left(  i,j\right)  }$ when $T$ is an immaculate tableau
of shape $\alpha$.
\end{proposition}

Before we prove this proposition, let us state a fundamental and simple lemma:

\begin{lemma}
\label{lem.dualImm.trivia}\textbf{(a)} If $I$ is a finite subset of $\left\{
1,2,3,\ldots\right\}  $, then there exists a unique strictly increasing
bijection $\left\{  1,2,\ldots,\left\vert I\right\vert \right\}  \rightarrow
I$. Let us denote this bijection by $r_{I}$. Its inverse $r_{I}^{-1}$ is
obviously again a strictly increasing bijection.

Now, let $\alpha$ be a composition.

\textbf{(b)} If $T$ is an immaculate tableau of shape $\alpha$, then
$r_{T\left(  Y\left(  \alpha\right)  \right)  }^{-1}\circ T$ (remember that
immaculate tableaux are maps from $Y\left(  \alpha\right)  $ to $\left\{
1,2,3,\ldots\right\}  $) is an immaculate tableau of shape $\alpha$ as well,
and has the additional property that there exists a unique composition $\beta$
of $\left\vert \alpha\right\vert $ such that $r_{T\left(  Y\left(
\alpha\right)  \right)  }^{-1}\circ T$ has content $\beta$.

\textbf{(c)} Let $Q$ be an immaculate tableau of shape $\alpha$. Let $\beta$
be a composition of $\left\vert \alpha\right\vert $ such that $Q$ has content
$\beta$. Then,%
\begin{equation}
M_{\beta}=\sum_{\substack{T\text{ is an immaculate}\\\text{tableau of shape
}\alpha;\\r_{T\left(  Y\left(  \alpha\right)  \right)  }^{-1}\circ
T=Q}}\mathbf{x}_{T}. \label{pf.dualImm.M}%
\end{equation}

\end{lemma}

\begin{proof}
[Proof of Lemma \ref{lem.dualImm.trivia}.]\textbf{(a)} Lemma
\ref{lem.dualImm.trivia} \textbf{(a)} is obvious.

\textbf{(b)} Let $T$ be an immaculate tableau of shape $\alpha$. Then,
$r_{T\left(  Y\left(  \alpha\right)  \right)  }^{-1}\circ T$ is an immaculate
tableau of shape $\alpha$ as well\footnote{This is because the map
$r_{T\left(  Y\left(  \alpha\right)  \right)  }^{-1}$ is strictly increasing,
and the inequality conditions which decide whether a map $Y\left(
\alpha\right)  \rightarrow\left\{  1,2,3,\ldots\right\}  $ is an immaculate
tableau of shape $\alpha$ are preserved under composition with a strictly
increasing map.}. Let $R=r_{T\left(  Y\left(  \alpha\right)  \right)  }%
^{-1}\circ T:Y\left(  \alpha\right)  \rightarrow\left\{  1,2,\ldots,\left\vert
T\left(  Y\left(  \alpha\right)  \right)  \right\vert \right\}  $. Then,%
\begin{align*}
\underbrace{R}_{=r_{T\left(  Y\left(  \alpha\right)  \right)  }^{-1}\circ
T}\left(  Y\left(  \alpha\right)  \right)   &  =\left(  r_{T\left(  Y\left(
\alpha\right)  \right)  }^{-1}\circ T\right)  \left(  Y\left(  \alpha\right)
\right) \\
&  =r_{T\left(  Y\left(  \alpha\right)  \right)  }^{-1}\left(  T\left(
Y\left(  \alpha\right)  \right)  \right)  =\left\{  1,2,\ldots,\left\vert
T\left(  Y\left(  \alpha\right)  \right)  \right\vert \right\}  .
\end{align*}
Hence, $\left(  \left\vert R^{-1}\left(  1\right)  \right\vert ,\left\vert
R^{-1}\left(  2\right)  \right\vert ,\ldots,\left\vert R^{-1}\left(
\left\vert T\left(  Y\left(  \alpha\right)  \right)  \right\vert \right)
\right\vert \right)  $ is a composition. Therefore, there exists a unique
composition $\beta$ of $\left\vert \alpha\right\vert $ such that $R$ has
content $\beta$ (namely, $\beta=\left(  \left\vert R^{-1}\left(  1\right)
\right\vert ,\left\vert R^{-1}\left(  2\right)  \right\vert ,\ldots,\left\vert
R^{-1}\left(  \left\vert T\left(  Y\left(  \alpha\right)  \right)  \right\vert
\right)  \right\vert \right)  $). In other words, there exists a unique
composition $\beta$ of $\left\vert \alpha\right\vert $ such that $r_{T\left(
Y\left(  \alpha\right)  \right)  }^{-1}\circ T$ has content $\beta$ (since
$R=r_{T\left(  Y\left(  \alpha\right)  \right)  }^{-1}\circ T$). This
completes the proof of Lemma \ref{lem.dualImm.trivia} \textbf{(b)}.

\textbf{(c)} If $T$ is a map $Y\left(  \alpha\right)  \rightarrow\left\{
1,2,3,\ldots\right\}  $ satisfying $r_{T\left(  Y\left(  \alpha\right)
\right)  }^{-1}\circ T=Q$, then $T$ is automatically an immaculate tableau of
shape $\alpha$\ \ \ \ \footnote{\textit{Proof.} Let $T$ be a map $Y\left(
\alpha\right)  \rightarrow\left\{  1,2,3,\ldots\right\}  $ satisfying
$r_{T\left(  Y\left(  \alpha\right)  \right)  }^{-1}\circ T=Q$. Thus,
$T=r_{T\left(  Y\left(  \alpha\right)  \right)  }\circ Q$. Since $Q$ is an
immaculate tableau of shape $\alpha$, this shows that $T$ is an immaculate
tableau of shape $\alpha$ (since the map $r_{T\left(  Y\left(  \alpha\right)
\right)  }$ is strictly increasing, and the inequality conditions which decide
whether a map $Y\left(  \alpha\right)  \rightarrow\left\{  1,2,3,\ldots
\right\}  $ is an immaculate tableau of shape $\alpha$ are preserved under
composition with a strictly increasing map).}. Hence, the summation sign
\textquotedblleft$\sum_{\substack{T\text{ is an immaculate}\\\text{tableau of
shape }\alpha;\\r_{T\left(  Y\left(  \alpha\right)  \right)  }^{-1}\circ
T=Q}}$\textquotedblright\ on the right hand side of (\ref{pf.dualImm.M}) can
be replaced by \textquotedblleft$\sum_{\substack{T:Y\left(  \alpha\right)
\rightarrow\left\{  1,2,3,\ldots\right\}  ;\\r_{T\left(  Y\left(
\alpha\right)  \right)  }^{-1}\circ T=Q}}$\textquotedblright. Hence,%
\[
\sum_{\substack{T\text{ is an immaculate}\\\text{tableau of shape }%
\alpha;\\r_{T\left(  Y\left(  \alpha\right)  \right)  }^{-1}\circ
T=Q}}\mathbf{x}_{T}=\sum_{\substack{T:Y\left(  \alpha\right)  \rightarrow
\left\{  1,2,3,\ldots\right\}  ;\\r_{T\left(  Y\left(  \alpha\right)  \right)
}^{-1}\circ T=Q}}\mathbf{x}_{T}.
\]

Now, let us write the composition $\beta$ in the form $\left(  \beta_{1}%
,\beta_{2},\ldots,\beta_{\ell}\right)  $. Then, we have%
\begin{equation}
\left\vert Q^{-1}\left(  k\right)  \right\vert =%
\begin{cases}
\beta_{k}, & \text{if }k\leq\ell;\\
0, & \text{if }k>\ell
\end{cases}
\ \ \ \ \ \ \ \ \ \ \text{for every positive integer }k
\label{pf.dualImm.fn.ct}%
\end{equation}
(since $Q$ has content $\beta$). Hence, $Q\left(  Y\left(  \alpha\right)
\right)  =\left\{  1,2,\ldots,\ell\right\}  $. As a consequence, the maps
$T:Y\left(  \alpha\right)  \rightarrow\left\{  1,2,3,\ldots\right\}  $
satisfying $r_{T\left(  Y\left(  \alpha\right)  \right)  }^{-1}\circ T=Q$ are
in 1-to-1 correspondence with the $\ell$-element subsets of $\left\{
1,2,3,\ldots\right\}  $ (the correspondence sends a map $T$ to the $\ell
$-element subset $T\left(  Y\left(  \alpha\right)  \right)  $, and the inverse
correspondence sends an $\ell$-element subset $I$ to the map $r_{I}\circ Q$).
But these latter subsets, in turn, are in 1-to-1 correspondence with the
strictly increasing length-$\ell$ sequences $\left(  i_{1}<i_{2}%
<\cdots<i_{\ell}\right)  $ of positive integers (the correspondence sends a
subset $G$ to the sequence $\left(  r_{G}\left(  1\right)  ,r_{G}\left(
2\right)  ,\ldots,r_{G}\left(  \ell\right)  \right)  $; of course, this latter
sequence is just the list of all elements of $G$ in increasing order).
Composing these two 1-to-1 correspondences, we conclude that the maps
$T:Y\left(  \alpha\right)  \rightarrow\left\{  1,2,3,\ldots\right\}  $
satisfying $r_{T\left(  Y\left(  \alpha\right)  \right)  }^{-1}\circ T=Q$ are
in 1-to-1 correspondence with the strictly increasing length-$\ell$ sequences
$\left(  i_{1}<i_{2}<\cdots<i_{\ell}\right)  $ of positive integers (the
correspondence sends a map $T$ to the sequence $\left(  r_{T\left(  Y\left(
\alpha\right)  \right)  }\left(  1\right)  ,r_{T\left(  Y\left(
\alpha\right)  \right)  }\left(  2\right)  ,\ldots,r_{T\left(  Y\left(
\alpha\right)  \right)  }\left(  \ell\right)  \right)  $), and this
correspondence has the property that $\mathbf{x}_{T}=x_{i_{1}}^{\beta_{1}%
}x_{i_{2}}^{\beta_{2}}\cdots x_{i_{\ell}}^{\beta_{\ell}}$ whenever some map
$T$ gets sent to some sequence $\left(  i_{1}<i_{2}<\cdots<i_{\ell}\right)  $
(because if some map $T$ gets sent to some sequence $\left(  i_{1}%
<i_{2}<\cdots<i_{\ell}\right)  $, then $\left(  i_{1},i_{2},\ldots,i_{\ell
}\right)  =\left(  r_{T\left(  Y\left(  \alpha\right)  \right)  }\left(
1\right)  ,r_{T\left(  Y\left(  \alpha\right)  \right)  }\left(  2\right)
,\ldots,r_{T\left(  Y\left(  \alpha\right)  \right)  }\left(  \ell\right)
\right)  $, so that every $k\in\left\{  1,2,\ldots,\ell\right\}  $ satisfies
$i_{k}=r_{T\left(  Y\left(  \alpha\right)  \right)  }\left(  k\right)  $, and
now we have%
\begin{align*}
\mathbf{x}_{T}  &  =\prod_{\left(  i,j\right)  \in Y\left(  \alpha\right)
}x_{T\left(  i,j\right)  }=\prod_{k=1}^{\ell}\ \ \prod_{\substack{\left(
i,j\right)  \in Y\left(  \alpha\right)  ;\\Q\left(  i,j\right)  =k}%
}\underbrace{x_{T\left(  i,j\right)  }}_{\substack{=x_{r_{T\left(  Y\left(
\alpha\right)  \right)  }\left(  Q\left(  i,j\right)  \right)  }\\\text{(since
}T\left(  i,j\right)  =r_{T\left(  Y\left(  \alpha\right)  \right)  }\left(
Q\left(  i,j\right)  \right)  \\\text{(because }r_{T\left(  Y\left(
\alpha\right)  \right)  }^{-1}\circ T=Q\\\text{and thus }T=r_{T\left(
Y\left(  \alpha\right)  \right)  }\circ Q\text{))}}}\\
&  \ \ \ \ \ \ \ \ \ \ \left(  \text{since }Q\left(  Y\left(  \alpha\right)
\right)  =\left\{  1,2,\ldots,\ell\right\}  \right) \\
&  =\prod_{k=1}^{\ell}\ \ \underbrace{\prod_{\substack{\left(  i,j\right)  \in
Y\left(  \alpha\right)  ;\\Q\left(  i,j\right)  =k}}}_{=\prod_{\left(
i,j\right)  \in Q^{-1}\left(  k\right)  }}\underbrace{x_{r_{T\left(  Y\left(
\alpha\right)  \right)  }\left(  Q\left(  i,j\right)  \right)  }%
}_{\substack{=x_{r_{T\left(  Y\left(  \alpha\right)  \right)  }\left(
k\right)  }\\\text{(since }Q\left(  i,j\right)  =k\text{)}}}\\
&  =\prod_{k=1}^{\ell}\ \ \underbrace{\prod_{\left(  i,j\right)  \in
Q^{-1}\left(  k\right)  }x_{r_{T\left(  Y\left(  \alpha\right)  \right)
}\left(  k\right)  }}_{\substack{=x_{r_{T\left(  Y\left(  \alpha\right)
\right)  }\left(  k\right)  }^{\left\vert Q^{-1}\left(  k\right)  \right\vert
}=x_{i_{k}}^{\left\vert Q^{-1}\left(  k\right)  \right\vert }\\\text{(since
}r_{T\left(  Y\left(  \alpha\right)  \right)  }\left(  k\right)
=i_{k}\text{)}}}=\prod_{k=1}^{\ell}\underbrace{x_{i_{k}}^{\left\vert
Q^{-1}\left(  k\right)  \right\vert }}_{\substack{=x_{i_{k}}^{\beta_{k}%
}\\\text{(since }\left\vert Q^{-1}\left(  k\right)  \right\vert =\beta
_{k}\\\text{(by (\ref{pf.dualImm.fn.ct})))}}}=\prod_{k=1}^{\ell}x_{i_{k}%
}^{\beta_{k}}=x_{i_{1}}^{\beta_{1}}x_{i_{2}}^{\beta_{2}}\cdots x_{i_{\ell}%
}^{\beta_{\ell}}%
\end{align*}
). Hence,
\[
\sum_{\substack{T:Y\left(  \alpha\right)  \rightarrow\left\{  1,2,3,\ldots
\right\}  ;\\r_{T\left(  Y\left(  \alpha\right)  \right)  }^{-1}\circ
T=Q}}\mathbf{x}_{T}=\sum_{1\leq i_{1}<i_{2}<\cdots<i_{\ell}}x_{i_{1}}%
^{\beta_{1}}x_{i_{2}}^{\beta_{2}}\cdots x_{i_{\ell}}^{\beta_{\ell}}=M_{\beta}%
\]
(by the definition of $M_{\beta}$). Altogether, we thus have%
\[
\sum_{\substack{T\text{ is an immaculate}\\\text{tableau of shape }%
\alpha;\\r_{T\left(  Y\left(  \alpha\right)  \right)  }^{-1}\circ
T=Q}}\mathbf{x}_{T}=\sum_{\substack{T:Y\left(  \alpha\right)  \rightarrow
\left\{  1,2,3,\ldots\right\}  ;\\r_{T\left(  Y\left(  \alpha\right)  \right)
}^{-1}\circ T=Q}}\mathbf{x}_{T}=M_{\beta}.
\]
This proves Lemma \ref{lem.dualImm.trivia} \textbf{(c)}.
\end{proof}

\begin{proof}
[Proof of Proposition \ref{prop.dualImm}.]For every finite subset $I$ of
$\left\{  1,2,3,\ldots\right\}  $, we shall use the notation $r_{I}$
introduced in Lemma \ref{lem.dualImm.trivia} \textbf{(a)}. Recall Lemma
\ref{lem.dualImm.trivia} \textbf{(b)}; it says that if $T$ is an immaculate
tableau of shape $\alpha$, then $r_{T\left(  Y\left(  \alpha\right)  \right)
}^{-1}\circ T$ is an immaculate tableau of shape $\alpha$ as well, and has the
additional property that there exists a unique composition $\beta$ of
$\left\vert \alpha\right\vert $ such that $r_{T\left(  Y\left(  \alpha\right)
\right)  }^{-1}\circ T$ has content $\beta$.

Now,
\begin{align}
\mathfrak{S}_{\alpha}^{\ast}  & =\sum_{\beta\models\left\vert \alpha
\right\vert }\underbrace{K_{\alpha,\beta}M_{\beta}}_{\substack{=\sum
_{\substack{Q\text{ is an immaculate}\\\text{tableau of shape }\alpha
\\\text{and content }\beta}}M_{\beta}\\\text{(by the definition of }%
K_{\alpha,\beta}\text{)}}}\nonumber\\
& =\sum_{\beta\models\left\vert \alpha\right\vert }\ \ \sum_{\substack{Q\text{
is an immaculate}\\\text{tableau of shape }\alpha\\\text{and content }\beta
}}M_{\beta}.\label{pf.dualImm.1}%
\end{align}

But (\ref{pf.dualImm.M}) shows that every composition $\beta$ of $\left\vert
\alpha\right\vert $ satisfies%
\[
\sum_{\substack{Q\text{ is an immaculate}\\\text{tableau of shape }%
\alpha\\\text{and content }\beta}}M_{\beta}=\sum_{\substack{Q\text{ is an
immaculate}\\\text{tableau of shape }\alpha\\\text{and content }\beta
}}\ \ \sum_{\substack{T\text{ is an immaculate}\\\text{tableau of shape
}\alpha;\\r_{T\left(  Y\left(  \alpha\right)  \right)  }^{-1}\circ
T=Q}}\mathbf{x}_{T}=\sum_{\substack{T\text{ is an immaculate}\\\text{tableau
of shape }\alpha\\\text{such that }r_{T\left(  Y\left(  \alpha\right)
\right)  }^{-1}\circ T\\\text{has content }\beta}}\mathbf{x}_{T}%
\]
(because for every immaculate tableau $T$ of shape $\alpha$, the map
$r_{T\left(  Y\left(  \alpha\right)  \right)  }^{-1}\circ T$ is an immaculate
tableau of shape $\alpha$ as well). Substituting this into (\ref{pf.dualImm.1}%
), we obtain%
\begin{align*}
\mathfrak{S}_{\alpha}^{\ast}  &  =\sum_{\beta\models\left\vert \alpha
\right\vert }\ \ \underbrace{\sum_{\substack{Q\text{ is an immaculate}%
\\\text{tableau of shape }\alpha\\\text{and content }\beta}}M_{\beta}}%
_{=\sum_{\substack{T\text{ is an immaculate}\\\text{tableau of shape }%
\alpha\\\text{such that }r_{T\left(  Y\left(  \alpha\right)  \right)  }%
^{-1}\circ T\\\text{has content }\beta}}\mathbf{x}_{T}}=\sum_{\beta
\models\left\vert \alpha\right\vert }\ \ \sum_{\substack{T\text{ is an
immaculate}\\\text{tableau of shape }\alpha\\\text{such that }r_{T\left(
Y\left(  \alpha\right)  \right)  }^{-1}\circ T\\\text{has content }\beta
}}\mathbf{x}_{T}\\
&  =\sum_{\substack{T\text{ is an immaculate}\\\text{tableau of shape }\alpha
}}\mathbf{x}_{T}%
\end{align*}
(because for every immaculate tableau $T$ of shape $\alpha$, there exists a
unique composition $\beta$ of $\left\vert \alpha\right\vert $ such that
$r_{T\left(  Y\left(  \alpha\right)  \right)  }^{-1}\circ T$ has content
$\beta$), whence Proposition \ref{prop.dualImm} follows.
\end{proof}

\begin{corollary}
\label{cor.dualImm.dend}Let $\alpha=\left(  \alpha_{1},\alpha_{2}%
,\ldots,\alpha_{\ell}\right)  $ be a composition with $\ell>0$. Let
$\overline{\alpha}$ denote the composition $\left(  \alpha_{2},\alpha
_{3},\ldots,\alpha_{\ell}\right)  $ of $\left\vert \alpha\right\vert
-\alpha_{1}$. Then,%
\[
\mathfrak{S}_{\alpha}^{\ast}=h_{\alpha_{1}}\left.  \prec\right.
\mathfrak{S}_{\overline{\alpha}}^{\ast}.
\]
Here, $h_{n}$ denotes the $n$-th complete homogeneous symmetric function for
every $n\in\mathbb{N}$.
\end{corollary}

\begin{proof}
[Proof of Corollary \ref{cor.dualImm.dend}.]Proposition \ref{prop.dualImm}
shows that
\begin{equation}
\mathfrak{S}_{\alpha}^{\ast}=\sum_{\substack{T\text{ is an immaculate}%
\\\text{tableau of shape }\alpha}}\mathbf{x}_{T}=\sum_{\substack{Q\text{ is an
immaculate}\\\text{tableau of shape }\alpha}}\mathbf{x}_{Q}
\label{pf.cor.dualImm.dend.LHS}%
\end{equation}
(here, we have renamed the summation index $T$ as $Q$).

Let $n=\alpha_{1}$. If $i_{1},i_{2},\ldots,i_{n}$ are positive integers
satisfying $i_{1}\leq i_{2}\leq\cdots\leq i_{n}$, and if $T$ is an immaculate
tableau of shape $\overline{\alpha}$, then%
\begin{align}
&  \left(  x_{i_{1}}x_{i_{2}}\cdots x_{i_{n}}\right)  \left.  \prec\right.
\mathbf{x}_{T}\nonumber\\
&  =
\begin{cases}
x_{i_{1}}x_{i_{2}}\cdots x_{i_{n}}\mathbf{x}_{T}, & \text{if } \min\left(
\operatorname*{Supp}\left(  x_{i_{1}}x_{i_{2}}\cdots x_{i_{n}}\right)
\right)  <\min\left(  \operatorname*{Supp}\left(  \mathbf{x}_{T}\right)
\right)  ;\\
0, & \text{if }\min\left(  \operatorname*{Supp}\left(  x_{i_{1}}x_{i_{2}%
}\cdots x_{i_{n}}\right)  \right)  \geq\min\left(  \operatorname*{Supp}\left(
\mathbf{x}_{T}\right)  \right)
\end{cases}
\nonumber\\
&  \ \ \ \ \ \ \ \ \ \ \left(  \text{by the definition of }\left.
\prec\right.  \text{ on monomials}\right) \nonumber\\
&  =
\begin{cases}
x_{i_{1}}x_{i_{2}}\cdots x_{i_{n}}\mathbf{x}_{T}, & \text{if } i_{1}%
<\min\left(  T\left(  Y\left(  \overline{\alpha}\right)  \right)  \right)  ;\\
0, & \text{if }i_{1}\geq\min\left(  T\left(  Y\left(  \overline{\alpha
}\right)  \right)  \right)
\end{cases}
\label{pf.cor.dualImm.dend.1}\\
&  \ \ \ \ \ \ \ \ \ \ \left(  \text{since }\min\left(  \operatorname*{Supp}%
\left(  x_{i_{1}}x_{i_{2}}\cdots x_{i_{n}}\right)  \right)  =i_{1}\text{ and
}\operatorname*{Supp}\left(  \mathbf{x}_{T}\right)  =T\left(  Y\left(
\overline{\alpha}\right)  \right)  \right)  .\nonumber
\end{align}

\begin{vershort}
But from $n=\alpha_{1}$, we obtain $h_{n}=h_{\alpha_{1}}$, so that
$h_{\alpha_{1}}=h_{n}=\sum\limits_{i_{1}\leq i_{2}\leq\cdots\leq i_{n}%
}x_{i_{1}}x_{i_{2}}\cdots x_{i_{n}}$ and $\mathfrak{S}_{\overline{\alpha}%
}^{\ast}=\sum_{\substack{T\text{ is an immaculate}\\\text{tableau of shape
}\overline{\alpha}}}\mathbf{x}_{T}$ (by Proposition \ref{prop.dualImm}).
Hence,%
\begin{align}
&  h_{\alpha_{1}}\left.  \prec\right.  \mathfrak{S}_{\overline{\alpha}}^{\ast
}\nonumber\\
&  =\left(  \sum\limits_{i_{1}\leq i_{2}\leq\cdots\leq i_{n}}x_{i_{1}}%
x_{i_{2}}\cdots x_{i_{n}}\right)  \left.  \prec\right.  \left(  \sum
_{\substack{T\text{ is an immaculate}\\\text{tableau of shape }\overline
{\alpha}}}\mathbf{x}_{T}\right)  \nonumber\\
&  =\sum\limits_{i_{1}\leq i_{2}\leq\cdots\leq i_{n}}\ \ \sum
_{\substack{T\text{ is an immaculate}\\\text{tableau of shape }\overline
{\alpha}}}\left(  x_{i_{1}}x_{i_{2}}\cdots x_{i_{n}}\right)  \left.
\prec\right.  \mathbf{x}_{T}\nonumber\\
&  =\sum\limits_{\substack{i_{1}\leq i_{2}\leq\cdots\leq i_{n};\\T\text{ is an
immaculate}\\\text{tableau of shape }\overline{\alpha};\\i_{1}<\min\left(
T\left(  Y\left(  \overline{\alpha}\right)  \right)  \right)  }}x_{i_{1}%
}x_{i_{2}}\cdots x_{i_{n}}\mathbf{x}_{T}\ \ \ \ \ \ \ \ \ \ \left(  \text{by
(\ref{pf.cor.dualImm.dend.1})}\right)  .\label{pf.cor.dualImm.dend.short.RHS}%
\end{align}
We need to check that this equals $\mathfrak{S}_{\alpha}^{\ast}=\sum
_{\substack{Q\text{ is an immaculate}\\\text{tableau of shape }\alpha
}}\mathbf{x}_{Q}$.
\end{vershort}

\begin{verlong}
But from $n=\alpha_{1}$, we obtain $h_{n}=h_{\alpha_{1}}$, so that
$h_{\alpha_{1}}=h_{n}=\sum\limits_{i_{1}\leq i_{2}\leq\cdots\leq i_{n}%
}x_{i_{1}}x_{i_{2}}\cdots x_{i_{n}}$ and $\mathfrak{S}_{\overline{\alpha}%
}^{\ast}=\sum_{\substack{T\text{ is an immaculate}\\\text{tableau of shape
}\overline{\alpha}}}\mathbf{x}_{T}$ (by Proposition \ref{prop.dualImm}).
Hence,%
\begin{align}
&  h_{\alpha_{1}}\left.  \prec\right.  \mathfrak{S}_{\overline{\alpha}}^{\ast
}\nonumber\\
&  =\left(  \sum\limits_{i_{1}\leq i_{2}\leq\cdots\leq i_{n}}x_{i_{1}}%
x_{i_{2}}\cdots x_{i_{n}}\right)  \left.  \prec\right.  \left(  \sum
_{\substack{T\text{ is an immaculate}\\\text{tableau of shape }\overline
{\alpha}}}\mathbf{x}_{T}\right) \nonumber\\
&  =\sum\limits_{i_{1}\leq i_{2}\leq\cdots\leq i_{n}}\ \ \sum
_{\substack{T\text{ is an immaculate}\\\text{tableau of shape }\overline
{\alpha}}}\underbrace{\left(  x_{i_{1}}x_{i_{2}}\cdots x_{i_{n}}\right)
\left.  \prec\right.  \mathbf{x}_{T}}_{\substack{=%
\begin{cases}
x_{i_{1}}x_{i_{2}}\cdots x_{i_{n}}\mathbf{x}_{T}, & \text{if }i_{1}%
<\min\left(  T\left(  Y\left(  \overline{\alpha}\right)  \right)  \right)  ;\\
0, & \text{if }i_{1}\geq\min\left(  T\left(  Y\left(  \overline{\alpha
}\right)  \right)  \right)
\end{cases}
\\\text{(by (\ref{pf.cor.dualImm.dend.1}))}}}\nonumber\\
&  =\sum\limits_{i_{1}\leq i_{2}\leq\cdots\leq i_{n}}\ \ \sum
_{\substack{T\text{ is an immaculate}\\\text{tableau of shape }\overline
{\alpha}}}%
\begin{cases}
x_{i_{1}}x_{i_{2}}\cdots x_{i_{n}}\mathbf{x}_{T}, & \text{if }i_{1}%
<\min\left(  T\left(  Y\left(  \overline{\alpha}\right)  \right)  \right)  ;\\
0, & \text{if }i_{1}\geq\min\left(  T\left(  Y\left(  \overline{\alpha
}\right)  \right)  \right)
\end{cases}
\nonumber\\
&  =\sum\limits_{\substack{i_{1}\leq i_{2}\leq\cdots\leq i_{n};\\T\text{ is an
immaculate}\\\text{tableau of shape }\overline{\alpha};\\i_{1}<\min\left(
T\left(  Y\left(  \overline{\alpha}\right)  \right)  \right)  }}x_{i_{1}%
}x_{i_{2}}\cdots x_{i_{n}}\mathbf{x}_{T}. \label{pf.cor.dualImm.dend.RHS}%
\end{align}
We need to check that this equals $\mathfrak{S}_{\alpha}^{\ast}=\sum
_{\substack{Q\text{ is an immaculate}\\\text{tableau of shape }\alpha
}}\mathbf{x}_{Q}$.
\end{verlong}

Now, let us define a map $\Phi$ from:

\begin{itemize}
\item the set of all pairs $\left(  \left(  i_{1},i_{2},\ldots,i_{n}\right)
,T\right)  $, where $i_{1}$, $i_{2}$, $\ldots$, $i_{n}$ are positive integers
satisfying $i_{1}\leq i_{2}\leq\cdots\leq i_{n}$, and where $T$ is an
immaculate tableau of shape $\overline{\alpha}$ satisfying $i_{1}<\min\left(
T\left(  Y\left(  \overline{\alpha}\right)  \right)  \right)  $
\end{itemize}

to:

\begin{itemize}
\item the set of all immaculate tableaux of shape $\alpha$.
\end{itemize}

Namely, we define the image of a pair $\left(  \left(  i_{1},i_{2}%
,\ldots,i_{n}\right)  ,T\right)  $ under $\Phi$ to be the immaculate tableau
obtained by adding a new row, filled with the entries $i_{1},i_{2}%
,\ldots,i_{n}$ (from left to right), to the top\footnote{Here, we are using
the graphical representation of immaculate tableaux introduced in Definition
\ref{def.immactab}.} of the tableau $T$ \ \ \ \ \footnote{Formally speaking,
this means that the image of $\left(  \left(  i_{1},i_{2},\ldots,i_{n}\right)
,T\right)  $ is the map $Y\left(  \alpha\right)  \rightarrow\left\{
1,2,3,\ldots\right\}  $ which sends every $\left(  u,v\right)  \in Y\left(
\alpha\right)  $ to $%
\begin{cases}
i_{v}, & \text{if }u=1;\\
T\left(  u-1,v\right)  , & \text{if }u\neq1
\end{cases}
\quad$. Proving that this map is an immaculate tableau is easy.}.

\begin{vershort}
This map $\Phi$ is a bijection\footnote{\textit{Proof.} The injectivity of the
map $\Phi$ is obvious. Its surjectivity follows from the observation that if
$Q$ is an immaculate tableau of shape $\alpha$, then the first entry of its
top row is smaller than the smallest entry of the immaculate tableau formed by
all other rows of $Q$. (This is a consequence of
(\ref{eq.immactab.min.2.short}), applied to $Q$ instead of $T$.)}, and has the
property that if $Q$ denotes the image of a pair $\left(  \left(  i_{1}%
,i_{2},\ldots,i_{n}\right)  ,T\right)  $ under the bijection $\Phi$, then
$\mathbf{x}_{Q}=x_{i_{1}}x_{i_{2}}\cdots x_{i_{n}}\mathbf{x}_{T}$. Hence,%
\[
\sum_{\substack{Q\text{ is an immaculate}\\\text{tableau of shape }\alpha
}}\mathbf{x}_{Q}=\sum\limits_{\substack{i_{1}\leq i_{2}\leq\cdots\leq
i_{n};\\T\text{ is an immaculate}\\\text{tableau of shape }\overline{\alpha
};\\i_{1}<\min\left(  T\left(  Y\left(  \overline{\alpha}\right)  \right)
\right)  }}x_{i_{1}}x_{i_{2}}\cdots x_{i_{n}}\mathbf{x}_{T}.
\]
In light of (\ref{pf.cor.dualImm.dend.LHS}) and
(\ref{pf.cor.dualImm.dend.short.RHS}), this rewrites as $\mathfrak{S}_{\alpha
}^{\ast}=h_{\alpha_{1}}\left.  \prec\right.  \mathfrak{S}_{\overline{\alpha}%
}^{\ast}$.
\end{vershort}

\begin{verlong}
This map $\Phi$ is a bijection\footnote{\textit{Proof.} The injectivity of the
map $\Phi$ is obvious. Its surjectivity follows from the observation that if
$Q$ is an immaculate tableau of shape $\alpha$, then the first entry of its
top row is smaller than the smallest entry of the immaculate tableau formed by
all other rows of $Q$. (This is a consequence of (\ref{eq.immactab.min.2}),
applied to $Q$ instead of $T$.)}, and has the property that if $Q$ denotes the
image of a pair $\left(  \left(  i_{1},i_{2},\ldots,i_{n}\right)  ,T\right)  $
under the bijection $\Phi$, then $\mathbf{x}_{Q}=x_{i_{1}}x_{i_{2}}\cdots
x_{i_{n}}\mathbf{x}_{T}$. Hence,%
\[
\sum_{\substack{Q\text{ is an immaculate}\\\text{tableau of shape }\alpha
}}\mathbf{x}_{Q}=\sum\limits_{\substack{i_{1}\leq i_{2}\leq\cdots\leq
i_{n};\\T\text{ is an immaculate}\\\text{tableau of shape }\overline{\alpha
};\\i_{1}<\min\left(  T\left(  Y\left(  \overline{\alpha}\right)  \right)
\right)  }}x_{i_{1}}x_{i_{2}}\cdots x_{i_{n}}\mathbf{x}_{T}.
\]
In light of (\ref{pf.cor.dualImm.dend.LHS}) and (\ref{pf.cor.dualImm.dend.RHS}%
), this rewrites as $\mathfrak{S}_{\alpha}^{\ast}=h_{\alpha_{1}}\left.
\prec\right.  \mathfrak{S}_{\overline{\alpha}}^{\ast}$. So Corollary
\ref{cor.dualImm.dend} is proven.
\end{verlong}
\end{proof}

\begin{corollary}
\label{cor.dualImm.dend.explicit}Let $\alpha=\left(  \alpha_{1},\alpha
_{2},\ldots,\alpha_{\ell}\right)  $ be a composition. Then,%
\[
\mathfrak{S}_{\alpha}^{\ast}=h_{\alpha_{1}}\left.  \prec\right.  \left(
h_{\alpha_{2}}\left.  \prec\right.  \left(  \cdots\left.  \prec\right.
\left(  h_{\alpha_{\ell}}\left.  \prec\right.  1\right)  \cdots\right)
\right)  .
\]

\end{corollary}

\begin{verlong}
\begin{proof}
[Proof of Corollary \ref{cor.dualImm.dend.explicit}.]We prove Corollary
\ref{cor.dualImm.dend.explicit} by induction over $\ell$:

\textit{Induction base:} If $\ell=0$, then $\alpha=\varnothing$ and thus
$\mathfrak{S}_{\alpha}^{\ast}=\mathfrak{S}_{\varnothing}^{\ast}=1$. But if
$\ell=0$, then we also have $h_{\alpha_{1}}\left.  \prec\right.  \left(
h_{\alpha_{2}}\left.  \prec\right.  \left(  \cdots\left.  \prec\right.
\left(  h_{\alpha_{\ell}}\left.  \prec\right.  1\right)  \cdots\right)
\right)  =1$. Hence, if $\ell=0$, then $\mathfrak{S}_{\alpha}^{\ast
}=1=h_{\alpha_{1}}\left.  \prec\right.  \left(  h_{\alpha_{2}}\left.
\prec\right.  \left(  \cdots\left.  \prec\right.  \left(  h_{\alpha_{\ell}%
}\left.  \prec\right.  1\right)  \cdots\right)  \right)  $. Thus, Corollary
\ref{cor.dualImm.dend.explicit} is proven when $\ell=0$. The induction base is complete.

\textit{Induction step:} Let $L$ be a positive integer. Assume that Corollary
\ref{cor.dualImm.dend.explicit} holds for $\ell=L-1$. We now need to prove
that Corollary \ref{cor.dualImm.dend.explicit} holds for $\ell=L$.

So let $\alpha=\left(  \alpha_{1},\alpha_{2},\ldots,\alpha_{\ell}\right)  $ be
a composition with $\ell=L$. Then, $\ell=L>0$. Now, let $\overline{\alpha}$
denote the composition $\left(  \alpha_{2},\alpha_{3},\ldots,\alpha_{\ell
}\right)  $ of $\left\vert \alpha\right\vert -\alpha_{1}$. Then, Corollary
\ref{cor.dualImm.dend} yields $\mathfrak{S}_{\alpha}^{\ast}=h_{\alpha_{1}%
}\left.  \prec\right.  \mathfrak{S}_{\overline{\alpha}}^{\ast}$. But by our
induction hypothesis, we can apply Corollary \ref{cor.dualImm.dend.explicit}
to $\overline{\alpha}=\left(  \alpha_{2},\alpha_{3},\ldots,\alpha_{\ell
}\right)  $ instead of $\alpha=\left(  \alpha_{1},\alpha_{2},\ldots
,\alpha_{\ell}\right)  $ (since $\ell-1=L-1$). As a result, we obtain
$\mathfrak{S}_{\overline{\alpha}}^{\ast}=h_{\alpha_{2}}\left.  \prec\right.
\left(  h_{\alpha_{3}}\left.  \prec\right.  \left(  \cdots\left.
\prec\right.  \left(  h_{\alpha_{\ell}}\left.  \prec\right.  1\right)
\cdots\right)  \right)  $. Hence,%
\begin{align*}
\mathfrak{S}_{\alpha}^{\ast}  &  =h_{\alpha_{1}}\left.  \prec\right.
\underbrace{\mathfrak{S}_{\overline{\alpha}}^{\ast}}_{=h_{\alpha_{2}}\left.
\prec\right.  \left(  h_{\alpha_{3}}\left.  \prec\right.  \left(
\cdots\left.  \prec\right.  \left(  h_{\alpha_{\ell}}\left.  \prec\right.
1\right)  \cdots\right)  \right)  }=h_{\alpha_{1}}\left.  \prec\right.
\left(  h_{\alpha_{2}}\left.  \prec\right.  \left(  h_{\alpha_{3}}\left.
\prec\right.  \left(  \cdots\left.  \prec\right.  \left(  h_{\alpha_{\ell}%
}\left.  \prec\right.  1\right)  \cdots\right)  \right)  \right) \\
&  =h_{\alpha_{1}}\left.  \prec\right.  \left(  h_{\alpha_{2}}\left.
\prec\right.  \left(  \cdots\left.  \prec\right.  \left(  h_{\alpha_{\ell}%
}\left.  \prec\right.  1\right)  \cdots\right)  \right)  .
\end{align*}
Now, let us forget that we fixed $\alpha$. We thus have shown that
\newline$\mathfrak{S}_{\alpha}^{\ast}=h_{\alpha_{1}}\left.  \prec\right.
\left(  h_{\alpha_{2}}\left.  \prec\right.  \left(  \cdots\left.
\prec\right.  \left(  h_{\alpha_{\ell}}\left.  \prec\right.  1\right)
\cdots\right)  \right)  $ for every composition $\alpha=\left(  \alpha
_{1},\alpha_{2},\ldots,\alpha_{\ell}\right)  $ which satisfies $\ell=L$. In
other words, Corollary \ref{cor.dualImm.dend.explicit} holds for $\ell=L$.
This completes the induction step. The induction proof of Corollary
\ref{cor.dualImm.dend.explicit} is thus complete.
\end{proof}
\end{verlong}

\begin{vershort}
\begin{proof}
[Proof of Corollary \ref{cor.dualImm.dend.explicit}.]This follows by induction
from Corollary \ref{cor.dualImm.dend} (since $\mathfrak{S}_{\varnothing}%
^{\ast}=1$).
\end{proof}
\end{vershort}

\section{\label{sect.zabrocki}An alternative description of $h_{m}\left.
\prec\right.  $}

In this section, we shall also use the Hopf algebra of \textit{noncommutative
symmetric functions}. This Hopf algebra (a noncommutative one, for a change)
is denoted by $\operatorname*{NSym}$ and has been discussed in \cite[Section
5.4]{Reiner} and \cite[Chapter 6]{HGK}; all we need to know about it are the
following properties:

\begin{itemize}
\item There is a nondegenerate pairing between $\operatorname*{NSym}$ and
$\operatorname*{QSym}$, that is, a nondegenerate $\mathbf{k}$-bilinear form
$\operatorname*{NSym}\times\operatorname*{QSym}\rightarrow\mathbf{k}$. We
shall denote this bilinear form by $\left(  \cdot,\cdot\right)  $. This
$\mathbf{k}$-bilinear form is a Hopf algebra pairing, i.e., it satisfies%
\begin{align}
\left(  ab,c\right)   &  =\sum_{\left(  c\right)  }\left(  a,c_{\left(
1\right)  }\right)  \left(  b,c_{\left(  2\right)  }\right)  \label{eq.(ab,c)}%
\\
&  \ \ \ \ \ \ \ \ \ \ \text{for all }a\in\operatorname*{NSym}\text{, }%
b\in\operatorname*{NSym}\text{ and }c\in\operatorname*{QSym};\nonumber
\end{align}%
\[
\left(  1,c\right)  =\varepsilon\left(  c\right)
\ \ \ \ \ \ \ \ \ \ \text{for all }c\in\operatorname*{QSym};
\]%
\begin{align*}
\sum_{\left(  a\right)  }\left(  a_{\left(  1\right)  },b\right)  \left(
a_{\left(  2\right)  },c\right)   &  =\left(  a,bc\right) \\
\ \ \ \ \ \ \ \ \ \ \text{for all }a  &  \in\operatorname*{NSym}\text{, }%
b\in\operatorname*{QSym}\text{ and }c\in\operatorname*{QSym};
\end{align*}%
\[
\left(  a,1\right)  =\varepsilon\left(  a\right)
\ \ \ \ \ \ \ \ \ \ \text{for all }a\in\operatorname*{NSym};
\]%
\[
\left(  S\left(  a\right)  ,b\right)  =\left(  a,S\left(  b\right)  \right)
\ \ \ \ \ \ \ \ \ \ \text{for all }a\in\operatorname*{NSym}\text{ and }%
b\in\operatorname*{QSym}%
\]
(where we use Sweedler's notation).

\item There is a basis of the $\mathbf{k}$-module $\operatorname*{NSym}$ which
is dual to the fundamental basis $\left(  F_{\alpha}\right)  _{\alpha
\in\operatorname*{Comp}}$ of $\operatorname*{QSym}$ with respect to the
bilinear form $\left(  \cdot,\cdot\right)  $. This basis is called the
\textit{ribbon basis} and will be denoted by $\left(  R_{\alpha}\right)
_{\alpha\in\operatorname*{Comp}}$.
\end{itemize}

Both of these properties are immediate consequences of the definitions of
$\operatorname*{NSym}$ and of $\left(  R_{\alpha}\right)  _{\alpha
\in\operatorname*{Comp}}$ given in \cite[Section 5.4]{Reiner} (although other
sources define these objects differently, and then the properties no longer
are immediate). The notations we are using here are the same as the ones used
in \cite[Section 5.4]{Reiner} (except that \cite[Section 5.4]{Reiner} calls
$L_{\alpha}$ what we denote by $F_{\alpha}$), and only slightly differ from
those in \cite{BBSSZ} (namely, \cite{BBSSZ} denotes the pairing $\left(
\cdot,\cdot\right)  $ by $\left\langle \cdot,\cdot\right\rangle $ instead).

We need some more definitions. For any $g\in\operatorname*{NSym}$, let
$\operatorname{L}_{g}:\operatorname*{NSym}\rightarrow\operatorname*{NSym}$
denote the left multiplication by $g$ on $\operatorname*{NSym}$ (that is, the
$\mathbf{k}$-linear map $\operatorname{NSym}\rightarrow\operatorname{NSym}%
,\ f\mapsto gf$). For any $g\in\operatorname{NSym}$, let $g^{\perp
}:\operatorname*{QSym}\rightarrow\operatorname*{QSym}$ be the $\mathbf{k}%
$-linear map adjoint to $\operatorname{L}_{g}:\operatorname*{NSym}%
\rightarrow\operatorname*{NSym}$ with respect to the pairing $\left(
\cdot,\cdot\right)  $ between $\operatorname*{NSym}$ and $\operatorname*{QSym}%
$. Thus, for any $g\in\operatorname*{NSym}$, $a\in\operatorname*{NSym}$ and
$c\in\operatorname*{QSym}$, we have%
\begin{equation}
\left(  a,g^{\perp}c\right)  =\left(  \underbrace{\operatorname*{L}%
\nolimits_{g}a}_{=ga},c\right)  =\left(  ga,c\right)  .
\label{pf.lem.adjoint-g.pre}%
\end{equation}
The following fact is well-known (and also is an easy formal consequence of
the definition of $g^{\perp}$ and of (\ref{eq.(ab,c)})):

\begin{lemma}
\label{lem.adjoint-g}Every $g\in\operatorname*{NSym}$ and $f\in
\operatorname*{QSym}$ satisfy%
\begin{equation}
g^{\perp}f=\sum_{\left(  f\right)  }\left(  g,f_{\left(  1\right)  }\right)
f_{\left(  2\right)  }. \label{pf.hmDless.1}%
\end{equation}

\end{lemma}

\begin{vershort}
\begin{proof}
[Proof of Lemma \ref{lem.adjoint-g}.]See the detailed version of this note.
\end{proof}
\end{vershort}

\begin{verlong}
\begin{proof}
[Proof of Lemma \ref{lem.adjoint-g}.]Let $g\in\operatorname*{NSym}$ and
$f\in\operatorname*{QSym}$. For every $a\in\operatorname*{NSym}$, we have%
\begin{align*}
\left(  a,g^{\perp}f\right)   &  =\left(  \underbrace{\operatorname*{L}%
\nolimits_{g}a}_{\substack{=ga\\\text{(by the definition of }\operatorname*{L}%
\nolimits_{g}\text{)}}},f\right)  \ \ \ \ \ \ \ \ \ \ \left(
\begin{array}
[c]{c}%
\text{since the map }g^{\perp}\text{ is adjoint to }\operatorname*{L}%
\nolimits_{g}\\
\text{with respect to the pairing }\left(  \cdot,\cdot\right)
\end{array}
\right) \\
&  =\left(  ga,f\right)  =\sum_{\left(  f\right)  }\left(  g,f_{\left(
1\right)  }\right)  \left(  a,f_{\left(  2\right)  }\right)
\ \ \ \ \ \ \ \ \ \ \left(
\begin{array}
[c]{c}%
\text{by (\ref{eq.(ab,c)}), applied to }g\text{, }a\text{ and }f\\
\text{instead of }a\text{, }b\text{ and }c
\end{array}
\right) \\
&  =\left(  a,\sum_{\left(  f\right)  }\left(  g,f_{\left(  1\right)
}\right)  f_{\left(  2\right)  }\right)  \ \ \ \ \ \ \ \ \ \ \left(
\text{since the pairing }\left(  \cdot,\cdot\right)  \text{ is }%
\mathbf{k}\text{-bilinear}\right)  .
\end{align*}
Since the pairing $\left(  \cdot,\cdot\right)  $ is nondegenerate, this
entails that $g^{\perp}f=\sum_{\left(  f\right)  }\left(  g,f_{\left(
1\right)  }\right)  f_{\left(  2\right)  }$. This proves Lemma
\ref{lem.adjoint-g}.
\end{proof}
\end{verlong}

For any composition $\alpha$, we define a composition $\omega\left(
\alpha\right)  $ as follows: Let $n=\left\vert \alpha\right\vert $, and write
$\alpha$ as $\alpha=\left(  \alpha_{1},\alpha_{2},\ldots,\alpha_{\ell}\right)
$. Let $\operatorname*{rev}\alpha$ denote the composition $\left(
\alpha_{\ell},\alpha_{\ell-1},\ldots,\alpha_{1}\right)  $ of $n$. Then,
$\omega\left(  \alpha\right)  $ shall be the unique composition $\beta$ of $n$
which satisfies $D\left(  \beta\right)  =\left\{  1,2,\ldots,n-1\right\}
\setminus D\left(  \operatorname*{rev}\alpha\right)  $. (This definition is
identical with that in \cite[Definition 5.2.14]{Reiner}. Some authors denote
$\omega\left(  \alpha\right)  $ by $\alpha^{\prime}$ instead.) We notice that
$\omega\left(  \omega\left(  \alpha\right)  \right)  =\alpha$ for any
composition $\alpha$.

\begin{verlong}
Here is a simple property of the composition $\omega\left(  \alpha\right)  $
that will later be used:

\begin{proposition}
\label{prop.omega.odot}\textbf{(a)} We have $\omega\left(  \left[
\alpha,\beta\right]  \right)  =\omega\left(  \beta\right)  \odot\omega\left(
\alpha\right)  $ for any two compositions $\alpha$ and $\beta$.

\textbf{(b)} We have $\omega\left(  \alpha\odot\beta\right)  =\left[
\omega\left(  \beta\right)  ,\omega\left(  \alpha\right)  \right]  $ for any
two compositions $\alpha$ and $\beta$.

\textbf{(c)} We have $\omega\left(  \omega\left(  \gamma\right)  \right)
=\gamma$ for every composition $\gamma$.
\end{proposition}

\begin{proof}
[Proof of Proposition \ref{prop.omega.odot}.]For any composition $\alpha$, we
define a composition $\operatorname*{rev}\alpha$ as follows: Let $n=\left\vert
\alpha\right\vert $, and write $\alpha$ as $\alpha=\left(  \alpha_{1}%
,\alpha_{2},\ldots,\alpha_{\ell}\right)  $. Let $\operatorname*{rev}\alpha$
denote the composition $\left(  \alpha_{\ell},\alpha_{\ell-1},\ldots
,\alpha_{1}\right)  $ of $n$. (This definition of $\operatorname*{rev}\alpha$
is the same as the one we gave above during the definition of $\omega\left(
\alpha\right)  $.) Clearly,%
\begin{equation}
\left\vert \operatorname*{rev}\gamma\right\vert =\left\vert \gamma\right\vert
\ \ \ \ \ \ \ \ \ \ \text{for any composition }\gamma.
\label{pf.prop.omega.odot.rev0}%
\end{equation}

It is easy to see that%
\begin{align}
\operatorname*{rev}\left(  \left[  \alpha,\beta\right]  \right)   &  =\left[
\operatorname*{rev}\beta,\operatorname*{rev}\alpha\right]
\ \ \ \ \ \ \ \ \ \ \text{and}\label{pf.prop.omega.odot.rev1}\\
\operatorname*{rev}\left(  \alpha\odot\beta\right)   &  =\left(
\operatorname*{rev}\beta\right)  \odot\left(  \operatorname*{rev}%
\alpha\right)  \label{pf.prop.omega.odot.rev2}%
\end{align}
for any two compositions $\alpha$ and $\beta$.

Recall that a composition $\gamma$ of a nonnegative integer $n$ is uniquely
determined by the set $D\left(  \gamma\right)  $ and the number $n$. Thus, if
$\gamma_{1}$ and $\gamma_{2}$ are two compositions of one and the same
nonnegative integer $n$ satisfying $D\left(  \gamma_{1}\right)  =D\left(
\gamma_{2}\right)  $, then%
\begin{equation}
\gamma_{1}=\gamma_{2}. \label{pf.prop.omega.odot.gamma1=2}%
\end{equation}

For every composition $\gamma$, we define a composition $\rho\left(
\gamma\right)  $ as follows: Let $n=\left\vert \gamma\right\vert $. Let
$\rho\left(  \gamma\right)  $ be the unique composition $\beta$ of $n$ which
satisfies $D\left(  \beta\right)  =\left\{  1,2,\ldots,n-1\right\}  \setminus
D\left(  \gamma\right)  $. (This is well-defined, because for every subset $T$
of $\left\{  1,2,\ldots,n-1\right\}  $, there exists a unique composition
$\tau$ of $n$ which satisfies $D\left(  \tau\right)  =T$.) Notice that%
\begin{equation}
\left\vert \rho\left(  \gamma\right)  \right\vert =\left\vert \gamma
\right\vert \ \ \ \ \ \ \ \ \ \ \text{for any composition }\gamma.
\label{pf.prop.omega.odot.rho1}%
\end{equation}
Also, if $n\in\mathbb{N}$, and if $\gamma$ is a composition of $n$, then%
\begin{equation}
D\left(  \rho\left(  \gamma\right)  \right)  =\left\{  1,2,\ldots,n-1\right\}
\setminus D\left(  \gamma\right)  \label{pf.prop.omega.odot.c.Drho}%
\end{equation}
\footnote{\textit{Proof of (\ref{pf.prop.omega.odot.c.Drho}):} Let
$n\in\mathbb{N}$. Let $\gamma$ be a composition of $n$. Then, $\rho\left(
\gamma\right)  $ is the unique composition $\beta$ of $n$ which satisfies
$D\left(  \beta\right)  =\left\{  1,2,\ldots,n-1\right\}  \setminus D\left(
\gamma\right)  $ (because this is how $\rho\left(  \gamma\right)  $ was
defined). Thus, $\rho\left(  \gamma\right)  $ is a composition of $n$ which
satisfies $D\left(  \rho\left(  \gamma\right)  \right)  =\left\{
1,2,\ldots,n-1\right\}  \setminus D\left(  \gamma\right)  $. This proves
(\ref{pf.prop.omega.odot.c.Drho}).}.

Notice also that%
\begin{equation}
\omega\left(  \alpha\right)  =\rho\left(  \operatorname*{rev}\alpha\right)
\ \ \ \ \ \ \ \ \ \ \text{for any composition }\alpha
\label{pf.prop.omega.odot.omega1}%
\end{equation}
\footnote{\textit{Proof of (\ref{pf.prop.omega.odot.omega1}):} Let $\alpha$ be
a composition. Let $n=\left\vert \alpha\right\vert $. Thus, $\alpha$ is a
composition of $n$. Hence, $\omega\left(  \alpha\right)  $ is a composition of
$n$ as well. Also, $\operatorname*{rev}\alpha$ is a composition of $n$. Now,
the definition of $\rho\left(  \operatorname*{rev}\alpha\right)  $ shows that
$\rho\left(  \operatorname*{rev}\alpha\right)  $ is the unique composition
$\beta$ of $n$ which satisfies $D\left(  \beta\right)  =\left\{
1,2,\ldots,n-1\right\}  \setminus D\left(  \operatorname*{rev}\alpha\right)
$. Hence, $\rho\left(  \operatorname*{rev}\alpha\right)  $ is a composition of
$n$ and satisfies $D\left(  \rho\left(  \operatorname*{rev}\alpha\right)
\right)  =\left\{  1,2,\ldots,n-1\right\}  \setminus D\left(
\operatorname*{rev}\alpha\right)  $.
\par
On the other hand, $\omega\left(  \alpha\right)  $ is the unique composition
$\beta$ of $n$ which satisfies $D\left(  \beta\right)  =\left\{
1,2,\ldots,n-1\right\}  \setminus D\left(  \operatorname*{rev}\alpha\right)  $
(by the definition of $\omega\left(  \alpha\right)  $). Thus, $\omega\left(
\alpha\right)  $ is a composition of $n$ and satisfies $D\left(  \omega\left(
\alpha\right)  \right)  =\left\{  1,2,\ldots,n-1\right\}  \setminus D\left(
\operatorname*{rev}\alpha\right)  $.
\par
Hence,%
\[
D\left(  \rho\left(  \operatorname*{rev}\alpha\right)  \right)  =\left\{
1,2,\ldots,n-1\right\}  \setminus D\left(  \operatorname*{rev}\alpha\right)
=D\left(  \omega\left(  \alpha\right)  \right)  .
\]
Applying (\ref{pf.prop.omega.odot.gamma1=2}) to $\gamma_{1}=\rho\left(
\operatorname*{rev}\alpha\right)  $ and $\gamma_{2}=\omega\left(
\alpha\right)  $, we therefore obtain $\rho\left(  \operatorname*{rev}%
\alpha\right)  =\omega\left(  \alpha\right)  $. Qed.}.

Now, we shall prove that%
\begin{equation}
\rho\left(  \left[  \alpha,\beta\right]  \right)  =\rho\left(  \alpha\right)
\odot\rho\left(  \beta\right)  \label{pf.prop.omega.odot.1}%
\end{equation}
for any two compositions $\alpha$ and $\beta$.

\textit{Proof of (\ref{pf.prop.omega.odot.1}):} Let $\alpha$ and $\beta$ be
two compositions. Let $p=\left\vert \alpha\right\vert $ and $q=\left\vert
\beta\right\vert $; thus, $\alpha$ and $\beta$ are compositions of $p$ and
$q$, respectively. We WLOG assume that both compositions $\alpha$ and $\beta$
are nonempty (since otherwise, (\ref{pf.prop.omega.odot.1}) is fairly
obvious). The composition $\alpha$ is a composition of $p$. Thus, $p > 0$
(since $\alpha$ is nonempty). Similarly, $q > 0$.

Hence, $\left[  \alpha,\beta\right]  $ is a composition of $p+q$ satisfying
$D\left(  \left[  \alpha,\beta\right]  \right)  =D\left(  \alpha\right)
\cup\left\{  p\right\}  \cup\left(  D\left(  \beta\right)  +p\right)  $ (by
Lemma \ref{lem.D(a(.)b)} \textbf{(b)}). The definition of $\rho\left(  \left[
\alpha,\beta\right]  \right)  $ thus yields%
\begin{align}
D\left(  \rho\left(  \left[  \alpha,\beta\right]  \right)  \right)   &
=\left\{  1,2,\ldots,p+q-1\right\}  \setminus\underbrace{D\left(  \left[
\alpha,\beta\right]  \right)  }_{=D\left(  \alpha\right)  \cup\left\{
p\right\}  \cup\left(  D\left(  \beta\right)  +p\right)  }\nonumber\\
&  =\left\{  1,2,\ldots,p+q-1\right\}  \setminus\left(  \left\{  p\right\}
\cup D\left(  \alpha\right)  \cup\left(  D\left(  \beta\right)  +p\right)
\right)  . \label{pf.prop.omega.odot.1.pf.1}%
\end{align}

Applying (\ref{pf.prop.omega.odot.rho1}) to $\gamma=\alpha$, we obtain
$\left\vert \rho\left(  \alpha\right)  \right\vert =\left\vert \alpha
\right\vert =p$. Thus, $\rho\left(  \alpha\right)  $ is a composition of $p$.
Similarly, $\rho\left(  \beta\right)  $ is a composition of $q$. Thus, Lemma
\ref{lem.D(a(.)b)} \textbf{(a)} (applied to $\rho\left(  \alpha\right)  $ and
$\rho\left(  \beta\right)  $ instead of $\alpha$ and $\beta$) shows that
$\rho\left(  \alpha\right)  \odot\rho\left(  \beta\right)  $ is a composition
of $p+q$ satisfying $D\left(  \rho\left(  \alpha\right)  \odot\rho\left(
\beta\right)  \right)  =D\left(  \rho\left(  \alpha\right)  \right)
\cup\left(  D\left(  \rho\left(  \beta\right)  \right)  +p\right)  $. Also,
applying (\ref{pf.prop.omega.odot.rho1}) to $\gamma=\left[  \alpha
,\beta\right]  $, we obtain $\left|  \rho\left(  \left[  \alpha,\beta\right]
\right)  \right|  = \left|  \left[  \alpha,\beta\right]  \right|  = p+q$
(since $\left[  \alpha,\beta\right]  $ is a composition of $p+q$). In other
words, $\rho\left(  \left[  \alpha,\beta\right]  \right)  $ is a composition
of $p+q$.

But the definition of $\rho\left(  \alpha\right)  $ shows that $D\left(
\rho\left(  \alpha\right)  \right)  =\left\{  1,2,\ldots,p-1\right\}
\setminus D\left(  \alpha\right)  $. Also, the definition of $\rho\left(
\beta\right)  $ shows that $D\left(  \rho\left(  \beta\right)  \right)
=\left\{  1,2,\ldots,q-1\right\}  \setminus D\left(  \beta\right)  $. Hence,%
\begin{align*}
&  \underbrace{D\left(  \rho\left(  \beta\right)  \right)  }_{=\left\{
1,2,\ldots,q-1\right\}  \setminus D\left(  \beta\right)  }+\,p\\
&  =\left(  \left\{  1,2,\ldots,q-1\right\}  \setminus D\left(  \beta\right)
\right)  +p\\
&  =\underbrace{\left(  \left\{  1,2,\ldots,q-1\right\}  +p\right)
}_{=\left\{  p+1,p+2,\ldots,p+q-1\right\}  }\setminus\left(  D\left(
\beta\right)  +p\right) \\
&  =\left\{  p+1,p+2,\ldots,p+q-1\right\}  \setminus\left(  D\left(
\beta\right)  +p\right)  .
\end{align*}
Also, $D\left(  \beta\right)  \subseteq\left\{  1,2,\ldots,q-1\right\}  $, so
that $D\left(  \beta\right)  +p\subseteq\left\{  1,2,\ldots,q-1\right\}
+p=\left\{  p+1,p+2,\ldots,p+q-1\right\}  $.

Now, it is well-known that if $X$, $Y$, $X^{\prime}$ and $Y^{\prime}$ are four
sets such that $X^{\prime}\subseteq X$, $Y^{\prime}\subseteq Y$ and $X\cap
Y=\varnothing$, then
\begin{equation}
\left(  X\setminus X^{\prime}\right)  \cup\left(  Y\setminus Y^{\prime
}\right)  =\left(  X\cup Y\right)  \setminus\left(  X^{\prime}\cup Y^{\prime
}\right)  . \label{pf.prop.omega.odot.setth}%
\end{equation}
Now,%
\begin{align*}
&  D\left(  \rho\left(  \alpha\right)  \odot\rho\left(  \beta\right)  \right)
\\
&  =\underbrace{D\left(  \rho\left(  \alpha\right)  \right)  }_{=\left\{
1,2,\ldots,p-1\right\}  \setminus D\left(  \alpha\right)  }\cup
\underbrace{\left(  D\left(  \rho\left(  \beta\right)  \right)  +p\right)
}_{=\left\{  p+1,p+2,\ldots,p+q-1\right\}  \setminus\left(  D\left(
\beta\right)  +p\right)  }\\
&  =\left(  \left\{  1,2,\ldots,p-1\right\}  \setminus D\left(  \alpha\right)
\right)  \cup\left(  \left\{  p+1,p+2,\ldots,p+q-1\right\}  \setminus\left(
D\left(  \beta\right)  +p\right)  \right) \\
&  =\underbrace{\left(  \left\{  1,2,\ldots,p-1\right\}  \cup\left\{
p+1,p+2,\ldots,p+q-1\right\}  \right)  }_{=\left\{  1,2,\ldots,p+q-1\right\}
\setminus\left\{  p\right\}  }\\
&  \ \ \ \ \ \ \ \ \ \ \setminus\left(  D\left(  \alpha\right)  \cup\left(
D\left(  \beta\right)  +p\right)  \right) \\
&  \ \ \ \ \ \ \ \ \ \ \left(
\begin{array}
[c]{c}%
\text{by (\ref{pf.prop.omega.odot.setth}), applied to }X=\left\{
1,2,\ldots,p-1\right\}  \text{,}\\
Y=\left\{  p+1,p+2,\ldots,p+q-1\right\}  \text{, }X^{\prime}=D\left(
\alpha\right)  \text{ and }Y^{\prime}=D\left(  \beta\right)  +p
\end{array}
\right) \\
&  =\left(  \left\{  1,2,\ldots,p+q-1\right\}  \setminus\left\{  p\right\}
\right)  \setminus\left(  D\left(  \alpha\right)  \cup\left(  D\left(
\beta\right)  +p\right)  \right) \\
&  =\left\{  1,2,\ldots,p+q-1\right\}  \setminus\underbrace{\left(  \left\{
p\right\}  \cup D\left(  \alpha\right)  \cup\left(  D\left(  \beta\right)
+p\right)  \right)  }_{=D\left(  \alpha\right)  \cup\left\{  p\right\}
\cup\left(  D\left(  \beta\right)  +p\right)  }\\
&  =\left\{  1,2,\ldots,p+q-1\right\}  \setminus\left(  \left\{  p\right\}
\cup D\left(  \alpha\right)  \cup\left(  D\left(  \beta\right)  +p\right)
\right) \\
&  =D\left(  \rho\left(  \left[  \alpha,\beta\right]  \right)  \right)
\ \ \ \ \ \ \ \ \ \ \left(  \text{by (\ref{pf.prop.omega.odot.1.pf.1}%
)}\right)  .
\end{align*}
Thus, (\ref{pf.prop.omega.odot.gamma1=2}) (applied to $n=p+q$, $\gamma
_{1}=\rho\left(  \alpha\right)  \odot\rho\left(  \beta\right)  $ and
$\gamma_{2}=\rho\left(  \left[  \alpha,\beta\right]  \right)  $) shows that
$\rho\left(  \alpha\right)  \odot\rho\left(  \beta\right)  =\rho\left(
\left[  \alpha,\beta\right]  \right)  $. This proves
(\ref{pf.prop.omega.odot.1}).

\textbf{(a)} Let $\alpha$ and $\beta$ be two compositions. Then,
(\ref{pf.prop.omega.odot.omega1}) yields $\omega\left(  \alpha\right)
=\rho\left(  \operatorname*{rev}\alpha\right)  $. Also,
(\ref{pf.prop.omega.odot.omega1}) (applied to $\beta$ instead of $\alpha$)
yields $\omega\left(  \beta\right)  =\rho\left(  \operatorname*{rev}%
\beta\right)  $.

From (\ref{pf.prop.omega.odot.omega1}) (applied to $\left[  \alpha
,\beta\right]  $ instead of $\alpha$), we obtain%
\begin{align*}
\omega\left(  \left[  \alpha,\beta\right]  \right)   &  =\rho\left(
\underbrace{\operatorname*{rev}\left(  \left[  \alpha,\beta\right]  \right)
}_{\substack{=\left[  \operatorname*{rev}\beta,\operatorname*{rev}%
\alpha\right]  \\\text{(by (\ref{pf.prop.omega.odot.rev1}))}}}\right)
=\rho\left(  \left[  \operatorname*{rev}\beta,\operatorname*{rev}%
\alpha\right]  \right) \\
&  =\underbrace{\rho\left(  \operatorname*{rev}\beta\right)  }_{=\omega\left(
\beta\right)  }\odot\underbrace{\rho\left(  \operatorname*{rev}\alpha\right)
}_{=\omega\left(  \alpha\right)  }\ \ \ \ \ \ \ \ \ \ \left(
\begin{array}
[c]{c}%
\text{by (\ref{pf.prop.omega.odot.1}), applied to }\operatorname*{rev}\beta\\
\text{and }\operatorname*{rev}\alpha\text{ instead of }\alpha\text{ and }\beta
\end{array}
\right) \\
&  =\omega\left(  \beta\right)  \odot\omega\left(  \alpha\right)  .
\end{align*}
This proves Proposition \ref{prop.omega.odot} \textbf{(a)}.

\textbf{(c)} First of all, it is clear that%
\begin{equation}
\operatorname*{rev}\left(  \operatorname*{rev}\gamma\right)  =\gamma
\ \ \ \ \ \ \ \ \ \ \text{for every composition }\gamma.
\label{pf.prop.omega.odot.c.rev}%
\end{equation}

Furthermore,%
\begin{equation}
\rho\left(  \rho\left(  \gamma\right)  \right)  =\gamma
\ \ \ \ \ \ \ \ \ \ \text{for every composition }\gamma
\label{pf.prop.omega.odot.c.rho}%
\end{equation}
\footnote{\textit{Proof of (\ref{pf.prop.omega.odot.c.rho}):} Let $\gamma$ be
a composition. Let $n=\left\vert \gamma\right\vert $. Thus, $\gamma$ is a
composition of $n$. The definition of $\rho\left(  \gamma\right)  $ shows that
$\rho\left(  \gamma\right)  $ is the unique composition $\beta$ of $n$ which
satisfies $D\left(  \beta\right)  =\left\{  1,2,\ldots,n-1\right\}  \setminus
D\left(  \gamma\right)  $. Thus, $\rho\left(  \gamma\right)  $ is a
composition of $n$ and satisfies $D\left(  \rho\left(  \gamma\right)  \right)
=\left\{  1,2,\ldots,n-1\right\}  \setminus D\left(  \gamma\right)  $.
\par
Therefore, the definition of $\rho\left(  \rho\left(  \gamma\right)  \right)
$ shows that $\rho\left(  \rho\left(  \gamma\right)  \right)  $ is the unique
composition $\beta$ of $n$ which satisfies $D\left(  \beta\right)  =\left\{
1,2,\ldots,n-1\right\}  \setminus D\left(  \rho\left(  \gamma\right)  \right)
$. Thus, $\rho\left(  \rho\left(  \gamma\right)  \right)  $ is a composition
of $n$ and satisfies $D\left(  \rho\left(  \rho\left(  \gamma\right)  \right)
\right)  =\left\{  1,2,\ldots,n-1\right\}  \setminus D\left(  \rho\left(
\gamma\right)  \right)  $. Hence,%
\begin{align*}
D\left(  \rho\left(  \rho\left(  \gamma\right)  \right)  \right)   &
=\left\{  1,2,\ldots,n-1\right\}  \setminus\underbrace{D\left(  \rho\left(
\gamma\right)  \right)  }_{=\left\{  1,2,\ldots,n-1\right\}  \setminus
D\left(  \gamma\right)  }\\
&  =\left\{  1,2,\ldots,n-1\right\}  \setminus\left(  \left\{  1,2,\ldots
,n-1\right\}  \setminus D\left(  \gamma\right)  \right) \\
&  =D\left(  \gamma\right)  \ \ \ \ \ \ \ \ \ \ \left(  \text{since }D\left(
\gamma\right)  \subseteq\left\{  1,2,\ldots,n-1\right\}  \right)  .
\end{align*}
Hence, (\ref{pf.prop.omega.odot.gamma1=2}) (applied to $\gamma_{1}=\rho\left(
\rho\left(  \gamma\right)  \right)  $ and $\gamma_{2}=\gamma$) shows that
$\rho\left(  \rho\left(  \gamma\right)  \right)  =\gamma$. This proves
(\ref{pf.prop.omega.odot.c.rho}).}.

On the other hand, if $G$ is a set of integers and $r$ is an integer, then we
let $r-G$ denote the set $\left\{  r-g\ \mid\ g\in G\right\}  $ of integers.
Then, for any $n\in\mathbb{N}$ and any composition $\gamma$ of $n$, we have%
\begin{equation}
D\left(  \operatorname*{rev}\gamma\right)  =n-D\left(  \gamma\right)
\label{pf.prop.omega.odot.c.Drev}%
\end{equation}
\footnote{\textit{Proof of (\ref{pf.prop.omega.odot.c.Drev}):} Let
$n\in\mathbb{N}$. Let $\gamma$ be a composition of $n$. Thus, $\gamma$ is a
composition satisfying $\left\vert \gamma\right\vert =n$.
\par
Write $\gamma$ in the form $\gamma=\left(  \gamma_{1},\gamma_{2},\ldots
,\gamma_{\ell}\right)  $. Then, $\operatorname*{rev}\gamma=\left(
\gamma_{\ell},\gamma_{\ell-1},\ldots,\gamma_{1}\right)  $ (by the definition
of $\operatorname*{rev}\gamma$). Also, from $\gamma=\left(  \gamma_{1}%
,\gamma_{2},\ldots,\gamma_{\ell}\right)  $, we obtain $\left\vert
\gamma\right\vert =\gamma_{1}+\gamma_{2}+\cdots+\gamma_{\ell}$, whence
$\gamma_{1}+\gamma_{2}+\cdots+\gamma_{\ell}=\left\vert \gamma\right\vert =n$.
Hence, every $i\in\left\{  1,2,\ldots,\ell-1\right\}  $ satisfies%
\begin{align}
n  &  =\gamma_{1}+\gamma_{2}+\cdots+\gamma_{\ell}=\left(  \gamma_{1}%
+\gamma_{2}+\cdots+\gamma_{i}\right)  +\underbrace{\left(  \gamma_{i+1}%
+\gamma_{i+2}+\cdots+\gamma_{\ell}\right)  }_{=\gamma_{\ell}+\gamma_{\ell
-1}+\cdots+\gamma_{i+1}}\nonumber\\
&  =\left(  \gamma_{1}+\gamma_{2}+\cdots+\gamma_{i}\right)  +\left(
\gamma_{\ell}+\gamma_{\ell-1}+\cdots+\gamma_{i+1}\right)  .
\label{pf.prop.omega.odot.c.Drev.pf.1}%
\end{align}
Also, $\gamma=\left(  \gamma_{1},\gamma_{2},\ldots,\gamma_{\ell}\right)  $, so
that the definition of $D\left(  \gamma\right)  $ yields%
\begin{align}
D\left(  \gamma\right)   &  =\left\{  \gamma_{1},\gamma_{1}+\gamma_{2}%
,\gamma_{1}+\gamma_{2}+\gamma_{3},\ldots,\gamma_{1}+\gamma_{2}+\cdots
+\gamma_{\ell-1}\right\} \nonumber\\
&  =\left\{  \gamma_{1}+\gamma_{2}+\cdots+\gamma_{i}\ \mid\ i\in\left\{
1,2,\ldots,\ell-1\right\}  \right\}  . \label{pf.prop.omega.odot.c.Drev.pf.2}%
\end{align}
\par
But $\operatorname*{rev}\gamma=\left(  \gamma_{\ell},\gamma_{\ell-1}%
,\ldots,\gamma_{1}\right)  $. Hence, the definition of $D\left(
\operatorname*{rev}\gamma\right)  $ yields%
\begin{align*}
D\left(  \operatorname*{rev}\gamma\right)   &  =\left\{  \gamma_{\ell}%
,\gamma_{\ell}+\gamma_{\ell-1},\gamma_{\ell}+\gamma_{\ell-1}+\gamma_{\ell
-2},\ldots,\gamma_{\ell}+\gamma_{\ell-1}+\gamma_{\ell-2}+\cdots+\gamma
_{2}\right\} \\
&  =\left\{  \underbrace{\gamma_{\ell}+\gamma_{\ell-1}+\cdots+\gamma_{i+1}%
}_{\substack{=n-\left(  \gamma_{1}+\gamma_{2}+\cdots+\gamma_{i}\right)
\\\text{(by (\ref{pf.prop.omega.odot.c.Drev.pf.1}))}}}\ \mid\ i\in\left\{
1,2,\ldots,\ell-1\right\}  \right\} \\
&  =\left\{  n-\left(  \gamma_{1}+\gamma_{2}+\cdots+\gamma_{i}\right)
\ \mid\ i\in\left\{  1,2,\ldots,\ell-1\right\}  \right\} \\
&  =n-\underbrace{\left\{  \gamma_{1}+\gamma_{2}+\cdots+\gamma_{i}\ \mid
\ i\in\left\{  1,2,\ldots,\ell-1\right\}  \right\}  }_{\substack{=D\left(
\gamma\right)  \\\text{(by (\ref{pf.prop.omega.odot.c.Drev.pf.2}))}}}\\
&  =n-D\left(  \gamma\right)  .
\end{align*}
This proves (\ref{pf.prop.omega.odot.c.Drev}).}.

Now,%
\begin{equation}
\rho\left(  \operatorname*{rev}\gamma\right)  =\operatorname*{rev}\left(
\rho\left(  \gamma\right)  \right)  \ \ \ \ \ \ \ \ \ \ \text{for every
composition }\gamma\label{pf.prop.omega.odot.c.rhorev}%
\end{equation}
\footnote{\textit{Proof of (\ref{pf.prop.omega.odot.c.rhorev}):} Let $\gamma$
be a composition. Let $n=\left\vert \gamma\right\vert $. Thus, $\gamma$ is a
composition of $n$.
\par
Now, (\ref{pf.prop.omega.odot.rho1}) (applied to $\operatorname*{rev}\gamma$
instead of $\gamma$) yields $\left\vert \rho\left(  \operatorname*{rev}%
\gamma\right)  \right\vert =\left\vert \operatorname*{rev}\gamma\right\vert
=\left\vert \gamma\right\vert $ (by (\ref{pf.prop.omega.odot.rev0})). Also,
(\ref{pf.prop.omega.odot.rev0}) (applied to $\rho\left(  \gamma\right)  $
instead of $\gamma$) yields $\left\vert \operatorname*{rev}\left(  \rho\left(
\gamma\right)  \right)  \right\vert =\left\vert \rho\left(  \gamma\right)
\right\vert =\left\vert \gamma\right\vert $ (by (\ref{pf.prop.omega.odot.rho1}%
)). Now, $\left\vert \rho\left(  \operatorname*{rev}\gamma\right)  \right\vert
=\left\vert \gamma\right\vert =n$, $\left\vert \operatorname*{rev}%
\gamma\right\vert =\left\vert \gamma\right\vert =n$, $\left\vert \rho\left(
\gamma\right)  \right\vert =\left\vert \gamma\right\vert =n$ and $\left\vert
\operatorname*{rev}\left(  \rho\left(  \gamma\right)  \right)  \right\vert
=\left\vert \gamma\right\vert =n$. Hence, all of $\rho\left(
\operatorname*{rev}\gamma\right)  $, $\operatorname*{rev}\gamma$, $\rho\left(
\gamma\right)  $ and $\operatorname*{rev}\left(  \rho\left(  \gamma\right)
\right)  $ are compositions of $n$.
\par
Applying (\ref{pf.prop.omega.odot.c.Drho}) to $\operatorname*{rev}\gamma$
instead of $\gamma$, we obtain%
\begin{align*}
D\left(  \rho\left(  \operatorname*{rev}\gamma\right)  \right)   &
=\underbrace{\left\{  1,2,\ldots,n-1\right\}  }_{=n-\left\{  1,2,\ldots
,n-1\right\}  }\setminus\underbrace{D\left(  \operatorname*{rev}\gamma\right)
}_{\substack{=n-D\left(  \gamma\right)  \\\text{(by
(\ref{pf.prop.omega.odot.c.Drev}))}}}\\
&  =\left(  n-\left\{  1,2,\ldots,n-1\right\}  \right)  \setminus\left(
n-D\left(  \gamma\right)  \right) \\
&  =n-\underbrace{\left(  \left\{  1,2,\ldots,n-1\right\}  \setminus D\left(
\gamma\right)  \right)  }_{\substack{=D\left(  \rho\left(  \gamma\right)
\right)  \\\text{(by (\ref{pf.prop.omega.odot.c.Drho}))}}}\\
&  =n-D\left(  \rho\left(  \gamma\right)  \right)  .
\end{align*}
Comparing this with%
\[
D\left(  \operatorname*{rev}\left(  \rho\left(  \gamma\right)  \right)
\right)  =n-D\left(  \rho\left(  \gamma\right)  \right)
\ \ \ \ \ \ \ \ \ \ \left(  \text{by (\ref{pf.prop.omega.odot.c.Drev}),
applied to }\rho\left(  \gamma\right)  \text{ instead of }\gamma\right)  ,
\]
we obtain $D\left(  \rho\left(  \operatorname*{rev}\gamma\right)  \right)
=D\left(  \operatorname*{rev}\left(  \rho\left(  \gamma\right)  \right)
\right)  $. Hence, (\ref{pf.prop.omega.odot.gamma1=2}) (applied to $\gamma
_{1}=\rho\left(  \operatorname*{rev}\gamma\right)  $ and $\gamma
_{2}=\operatorname*{rev}\left(  \rho\left(  \gamma\right)  \right)  $) yields
$\rho\left(  \operatorname*{rev}\gamma\right)  =\operatorname*{rev}\left(
\rho\left(  \gamma\right)  \right)  $. This proves
(\ref{pf.prop.omega.odot.c.rhorev}).}.

Now, let $\gamma$ be a composition. Then, (\ref{pf.prop.omega.odot.omega1})
(applied to $\alpha=\gamma$) yields $\omega\left(  \gamma\right)  =\rho\left(
\operatorname*{rev}\gamma\right)  =\operatorname*{rev}\left(  \rho\left(
\gamma\right)  \right)  $ (by (\ref{pf.prop.omega.odot.c.rhorev})). But
(\ref{pf.prop.omega.odot.omega1}) (applied to $\alpha=\omega\left(
\gamma\right)  $) yields%
\begin{align*}
\omega\left(  \omega\left(  \gamma\right)  \right)   &  =\rho\left(
\operatorname*{rev}\left(  \underbrace{\omega\left(  \gamma\right)
}_{=\operatorname*{rev}\left(  \rho\left(  \gamma\right)  \right)  }\right)
\right)  =\rho\left(  \underbrace{\operatorname*{rev}\left(
\operatorname*{rev}\left(  \rho\left(  \gamma\right)  \right)  \right)
}_{\substack{=\rho\left(  \gamma\right)  \\\text{(by
(\ref{pf.prop.omega.odot.c.rev}), applied to}\\\rho\left(  \gamma\right)
\text{ instead of }\gamma\text{)}}}\right) \\
&  =\rho\left(  \rho\left(  \gamma\right)  \right)  =\gamma
\ \ \ \ \ \ \ \ \ \ \left(  \text{by (\ref{pf.prop.omega.odot.c.rho})}\right)
.
\end{align*}
This proves Proposition \ref{prop.omega.odot} \textbf{(c)}.

\textbf{(b)} Let $\alpha$ and $\beta$ be two compositions. Then, Proposition
\ref{prop.omega.odot} \textbf{(a)} (applied to $\omega\left(  \beta\right)  $
and $\omega\left(  \alpha\right)  $ instead of $\alpha$ and $\beta$) yields%
\[
\omega\left(  \left[  \omega\left(  \beta\right)  ,\omega\left(
\alpha\right)  \right]  \right)  =\underbrace{\omega\left(  \omega\left(
\alpha\right)  \right)  }_{\substack{=\alpha\\\text{(by Proposition
\ref{prop.omega.odot} \textbf{(c)},}\\\text{applied to }\gamma=\alpha\text{)}%
}}\odot\underbrace{\omega\left(  \omega\left(  \beta\right)  \right)
}_{\substack{=\beta\\\text{(by Proposition \ref{prop.omega.odot}
\textbf{(c)},}\\\text{applied to }\gamma=\beta\text{)}}}=\alpha\odot\beta.
\]
Hence, $\alpha\odot\beta=\omega\left(  \left[  \omega\left(  \beta\right)
,\omega\left(  \alpha\right)  \right]  \right)  $. Applying the map $\omega$
to both sides of this equality, we conclude that%
\[
\omega\left(  \alpha\odot\beta\right)  =\omega\left(  \omega\left(  \left[
\omega\left(  \beta\right)  ,\omega\left(  \alpha\right)  \right]  \right)
\right)  =\left[  \omega\left(  \beta\right)  ,\omega\left(  \alpha\right)
\right]
\]
(by Proposition \ref{prop.omega.odot} \textbf{(c)}, applied to $\gamma=\left[
\omega\left(  \beta\right)  ,\omega\left(  \alpha\right)  \right]  $). This
proves Proposition \ref{prop.omega.odot} \textbf{(b)}.
\end{proof}
\end{verlong}

The notion of $\omega\left(  \alpha\right)  $ gives rise to a simple formula
for the antipode $S$ of the Hopf algebra $\operatorname*{QSym}$ in terms of
its fundamental basis:

\begin{proposition}
\label{prop.S.F}Let $\alpha$ be a composition. Then, $S\left(  F_{\alpha
}\right)  =\left(  -1\right)  ^{\left\vert \alpha\right\vert }F_{\omega\left(
\alpha\right)  }$.
\end{proposition}

This is proven in \cite[Proposition 5.2.15]{Reiner}.

We now state the main result of this note:

\begin{theorem}
\label{thm.hmDless}Let $f\in\operatorname*{QSym}$ and let $m$ be a positive
integer. For any two compositions $\alpha$ and $\beta$, define a composition
$\alpha\odot\beta$ as in Proposition \ref{prop.bel.F}. Then,%
\[
h_{m}\left.  \prec\right.  f=\sum_{\alpha\in\operatorname*{Comp}}\left(
-1\right)  ^{\left\vert \alpha\right\vert }F_{\alpha\odot\left(  m\right)
}R_{\omega\left(  \alpha\right)  }^{\perp}f.
\]
(Here, the sum on the right hand side converges, because all but finitely many
compositions $\alpha$ satisfy $R_{\omega\left(  \alpha\right)  }^{\perp}f=0$
for degree reasons.)
\end{theorem}

The proof is based on the following simple lemma:

\begin{lemma}
\label{lem.hmDless.lem}Let $a\in\operatorname*{QSym}$ and $f\in
\operatorname*{QSym}$ be such that the constant term of $a$ is $0$. Then,%
\[
\sum_{\alpha\in\operatorname*{Comp}}\left(  -1\right)  ^{\left\vert
\alpha\right\vert }\left(  F_{\alpha}\bel a\right)  R_{\omega\left(
\alpha\right)  }^{\perp}f=a\left.  \prec\right.  f.
\]

\end{lemma}

\begin{proof}
[Proof of Lemma \ref{lem.hmDless.lem}.]The basis $\left(  F_{\alpha}\right)
_{\alpha\in\operatorname*{Comp}}$ of $\operatorname*{QSym}$ and the basis
$\left(  R_{\alpha}\right)  _{\alpha\in\operatorname*{Comp}}$ of
$\operatorname*{NSym}$ are dual bases. Thus,%
\begin{equation}
\sum_{\alpha\in\operatorname*{Comp}}F_{\alpha}\left(  R_{\alpha},g\right)
=g\ \ \ \ \ \ \ \ \ \ \text{for every }g\in\operatorname*{QSym}.
\label{pf.lem.hmDless.lem.dualbases}%
\end{equation}

Let us use Sweedler's notation. The map $\operatorname*{Comp}\rightarrow
\operatorname*{Comp},\ \alpha\mapsto\omega\left(  \alpha\right)  $ is a
bijection (since $\omega\left(  \omega\left(  \alpha\right)  \right)  =\alpha$
for any composition $\alpha$). Hence, we can substitute $\omega\left(
\alpha\right)  $ for $\alpha$ in the sum $\sum_{\alpha\in\operatorname*{Comp}%
}\left(  -1\right)  ^{\left\vert \alpha\right\vert }\left(  F_{\alpha}
\bel a\right)  R_{\omega\left(  \alpha\right)  }^{\perp}f$. We thus obtain%
\begin{align*}
&  \sum_{\alpha\in\operatorname*{Comp}}\left(  -1\right)  ^{\left\vert
\alpha\right\vert }\left(  F_{\alpha} \bel a\right)  R_{\omega\left(
\alpha\right)  }^{\perp}f\\
&  =\sum_{\alpha\in\operatorname*{Comp}}\underbrace{\left(  -1\right)
^{\left\vert \omega\left(  \alpha\right)  \right\vert }}_{\substack{=\left(
-1\right)  ^{\left\vert \alpha\right\vert }\\\text{(since }\left\vert
\omega\left(  \alpha\right)  \right\vert =\left\vert \alpha\right\vert
\text{)}}}\left(  F_{\omega\left(  \alpha\right)  } \bel a\right)
\underbrace{R_{\omega\left(  \omega\left(  \alpha\right)  \right)  }^{\perp}%
}_{\substack{=R_{\alpha}^{\perp}\\\text{(since }\omega\left(  \omega\left(
\alpha\right)  \right)  =\alpha\text{)}}}f\\
&  =\sum_{\alpha\in\operatorname*{Comp}}\left(  -1\right)  ^{\left\vert
\alpha\right\vert }\left(  F_{\omega\left(  \alpha\right)  } \bel a\right)
\underbrace{R_{\alpha}^{\perp}f}_{\substack{=\sum_{\left(  f\right)  }\left(
R_{\alpha},f_{\left(  1\right)  }\right)  f_{\left(  2\right)  }\\\text{(by
(\ref{pf.hmDless.1}))}}}\\
&  =\sum_{\alpha\in\operatorname*{Comp}}\left(  -1\right)  ^{\left\vert
\alpha\right\vert }\left(  F_{\omega\left(  \alpha\right)  } \bel a\right)
\sum_{\left(  f\right)  }\left(  R_{\alpha},f_{\left(  1\right)  }\right)
f_{\left(  2\right)  }\\
&  =\sum_{\left(  f\right)  }\sum_{\alpha\in\operatorname*{Comp}}\left(
-1\right)  ^{\left\vert \alpha\right\vert }\left(  F_{\omega\left(
\alpha\right)  } \bel a\right)  \left(  R_{\alpha},f_{\left(  1\right)
}\right)  f_{\left(  2\right)  }\\
&  =\sum_{\left(  f\right)  }\left(  \left(  \sum_{\alpha\in
\operatorname*{Comp}}\underbrace{\left(  -1\right)  ^{\left\vert
\alpha\right\vert }F_{\omega\left(  \alpha\right)  }}_{\substack{=S\left(
F_{\alpha}\right)  \\\text{(by Proposition \ref{prop.S.F})}}}\left(
R_{\alpha},f_{\left(  1\right)  }\right)  \right)  \bel a\right)  f_{\left(
2\right)  }\\
&  =\sum_{\left(  f\right)  }\left(  \left(  \sum_{\alpha\in
\operatorname*{Comp}}S\left(  F_{\alpha}\right)  \left(  R_{\alpha},f_{\left(
1\right)  }\right)  \right)  \bel a\right)  f_{\left(  2\right)  }\\
&  =\sum_{\left(  f\right)  }\left(  S\left(  \underbrace{\sum_{\alpha
\in\operatorname*{Comp}}F_{\alpha}\left(  R_{\alpha},f_{\left(  1\right)
}\right)  }_{\substack{_{\substack{=f_{\left(  1\right)  }}}\\\text{(by
(\ref{pf.lem.hmDless.lem.dualbases}), applied to }g=f_{\left(  1\right)
}\text{)}}}\right)  \bel a\right)  f_{\left(  2\right)  }=\sum_{\left(
f\right)  }\left(  S\left(  f_{\left(  1\right)  }\right)  \bel a\right)
f_{\left(  2\right)  }=a\left.  \prec\right.  f
\end{align*}
(by Theorem \ref{thm.beldend}, applied to $b=f$). This proves Lemma
\ref{lem.hmDless.lem}.
\end{proof}

\begin{proof}
[Proof of Theorem \ref{thm.hmDless}.]The constant term of $h_{m}$ is $0$
(since $m$ is positive); thus, we can apply Lemma \ref{lem.hmDless.lem} to
$a=h_{m}$. Now,%
\[
\sum_{\alpha\in\operatorname*{Comp}}\left(  -1\right)  ^{\left\vert
\alpha\right\vert }\underbrace{F_{\alpha\odot\left(  m\right)  }%
}_{\substack{=F_{\alpha}\bel h_{m}\\\text{(by (\ref{pf.hmDless.2}))}%
}}R_{\omega\left(  \alpha\right)  }^{\perp}f=\sum_{\alpha\in
\operatorname*{Comp}}\left(  -1\right)  ^{\left\vert \alpha\right\vert
}\left(  F_{\alpha}\bel h_{m}\right)  R_{\omega\left(  \alpha\right)  }%
^{\perp}f=h_{m}\left.  \prec\right.  f
\]
(by Lemma \ref{lem.hmDless.lem}, applied to $a=h_{m}$). This proves Theorem
\ref{thm.hmDless}.
\end{proof}

As a consequence, we obtain the following result, conjectured by Mike Zabrocki
(private correspondence):

\begin{corollary}
\label{cor.zabrocki}For every positive integer $m$, define a $\mathbf{k}%
$-linear operator $\mathbf{W}_{m}:\operatorname*{QSym}\rightarrow
\operatorname*{QSym}$ by%
\[
\mathbf{W}_{m}=\sum_{\alpha\in\operatorname*{Comp}}\left(  -1\right)
^{\left\vert \alpha\right\vert }F_{\alpha\odot\left(  m\right)  }%
R_{\omega\left(  \alpha\right)  }^{\perp}%
\]
(where $F_{\alpha\odot\left(  m\right)  }$ means left multiplication by
$F_{\alpha\odot\left(  m\right)  }$). Then, every composition $\alpha=\left(
\alpha_{1},\alpha_{2},\ldots,\alpha_{\ell}\right)  $ satisfies%
\[
\mathfrak{S}_{\alpha}^{\ast}=\left(  \mathbf{W}_{\alpha_{1}}\circ
\mathbf{W}_{\alpha_{2}}\circ\cdots\circ\mathbf{W}_{\alpha_{\ell}}\right)
\left(  1\right)  .
\]

\end{corollary}

\begin{proof}
[Proof of Corollary \ref{cor.zabrocki}.]For every positive integer $m$ and
every $f\in\operatorname*{QSym}$, we have%
\[
\mathbf{W}_{m}f=\sum_{\alpha\in\operatorname*{Comp}}\left(  -1\right)
^{\left\vert \alpha\right\vert }F_{\alpha\odot\left(  m\right)  }%
R_{\omega\left(  \alpha\right)  }^{\perp}f=h_{m}\left.  \prec\right.
f\ \ \ \ \ \ \ \ \ \ \left(  \text{by Theorem \ref{thm.hmDless}}\right)  .
\]
Hence, by induction, for every composition $\alpha=\left(  \alpha_{1}%
,\alpha_{2},\ldots,\alpha_{\ell}\right)  $, we have%
\[
\mathbf{W}_{\alpha_{1}}\left(  \mathbf{W}_{\alpha_{2}}\left(  \cdots\left(
\mathbf{W}_{\alpha_{\ell}}\left(  1\right)  \right)  \cdots\right)  \right)
=h_{\alpha_{1}}\left.  \prec\right.  \left(  h_{\alpha_{2}}\left.
\prec\right.  \left(  \cdots\left.  \prec\right.  \left(  h_{\alpha_{\ell}%
}\left.  \prec\right.  1\right)  \cdots\right)  \right)  =\mathfrak{S}%
_{\alpha}^{\ast}%
\]
(by Corollary \ref{cor.dualImm.dend.explicit}). In other words,%
\[
\mathfrak{S}_{\alpha}^{\ast}=\mathbf{W}_{\alpha_{1}}\left(  \mathbf{W}%
_{\alpha_{2}}\left(  \cdots\left(  \mathbf{W}_{\alpha_{\ell}}\left(  1\right)
\right)  \cdots\right)  \right)  =\left(  \mathbf{W}_{\alpha_{1}}%
\circ\mathbf{W}_{\alpha_{2}}\circ\cdots\circ\mathbf{W}_{\alpha_{\ell}}\right)
\left(  1\right)  .
\]
This proves Corollary \ref{cor.zabrocki}.
\end{proof}

Let us finish this section with two curiosities: two analogues of Theorem
\ref{thm.hmDless}, one of which can be viewed as an \textquotedblleft$m=0$
version\textquotedblright\ and the other as a \textquotedblleft negative $m$
version\textquotedblright. We begin with the \textquotedblleft$m=0$
one\textquotedblright, as it is the easier one to state:

\begin{proposition}
\label{prop.hmDless.analogue0}Let $f\in\operatorname*{QSym}$. Then,%
\[
\varepsilon\left(  f\right)  =\sum_{\alpha\in\operatorname*{Comp}}\left(
-1\right)  ^{\left\vert \alpha\right\vert }F_{\alpha}R_{\omega\left(
\alpha\right)  }^{\perp}f.
\]

\end{proposition}

\begin{vershort}
\begin{proof}
[Proof of Proposition \ref{prop.hmDless.analogue0}.]This proof can be found in
the detailed version of this note; it is similar to the proof of Theorem
\ref{thm.hmDless}.
\end{proof}
\end{vershort}

\begin{verlong}
\begin{proof}
[Proof of Proposition \ref{prop.hmDless.analogue0}.]Let us use Sweedler's
notation. The map
\[
\operatorname*{Comp}\rightarrow\operatorname*{Comp},\ \alpha\mapsto
\omega\left(  \alpha\right)
\]
is a bijection (since $\omega\left(  \omega\left(  \alpha\right)  \right)
=\alpha$ for any composition $\alpha$). Hence, we can substitute
$\omega\left(  \alpha\right)  $ for $\alpha$ in the sum $\sum_{\alpha
\in\operatorname*{Comp}}\left(  -1\right)  ^{\left\vert \alpha\right\vert
}F_{\alpha}R_{\omega\left(  \alpha\right)  }^{\perp}f$. We thus obtain%
\begin{align*}
&  \sum_{\alpha\in\operatorname*{Comp}}\left(  -1\right)  ^{\left\vert
\alpha\right\vert }F_{\alpha}R_{\omega\left(  \alpha\right)  }^{\perp}f\\
&  =\sum_{\alpha\in\operatorname*{Comp}}\underbrace{\left(  -1\right)
^{\left\vert \omega\left(  \alpha\right)  \right\vert }}_{\substack{=\left(
-1\right)  ^{\left\vert \alpha\right\vert }\\\text{(since }\left\vert
\omega\left(  \alpha\right)  \right\vert =\left\vert \alpha\right\vert
\text{)}}}F_{\omega\left(  \alpha\right)  }\underbrace{R_{\omega\left(
\omega\left(  \alpha\right)  \right)  }^{\perp}}_{\substack{=R_{\alpha}%
^{\perp}\\\text{(since }\omega\left(  \omega\left(  \alpha\right)  \right)
=\alpha\text{)}}}f\\
&  =\sum_{\alpha\in\operatorname*{Comp}}\underbrace{\left(  -1\right)
^{\left\vert \alpha\right\vert }F_{\omega\left(  \alpha\right)  }%
}_{\substack{=S\left(  F_{\alpha}\right)  \\\text{(by Proposition
\ref{prop.S.F})}}}\underbrace{R_{\alpha}^{\perp}f}_{\substack{=\sum_{\left(
f\right)  }\left(  R_{\alpha},f_{\left(  1\right)  }\right)  f_{\left(
2\right)  }\\\text{(by (\ref{pf.hmDless.1}))}}}\\
&  =\sum_{\alpha\in\operatorname*{Comp}}S\left(  F_{\alpha}\right)
\sum_{\left(  f\right)  }\left(  R_{\alpha},f_{\left(  1\right)  }\right)
f_{\left(  2\right)  }=\sum_{\alpha\in\operatorname*{Comp}}\sum_{\left(
f\right)  }S\left(  F_{\alpha}\right)  \left(  R_{\alpha},f_{\left(  1\right)
}\right)  f_{\left(  2\right)  }\\
&  =\sum_{\left(  f\right)  }\left(  \sum_{\alpha\in\operatorname*{Comp}%
}S\left(  F_{\alpha}\right)  \left(  R_{\alpha},f_{\left(  1\right)  }\right)
\right)  f_{\left(  2\right)  }=\sum_{\left(  f\right)  }S\left(
\underbrace{\sum_{\alpha\in\operatorname*{Comp}}F_{\alpha}\left(  R_{\alpha
},f_{\left(  1\right)  }\right)  }_{\substack{_{\substack{=f_{\left(
1\right)  }}}\\\text{(by (\ref{pf.lem.hmDless.lem.dualbases}), applied to
}g=f_{\left(  1\right)  }\text{)}}}\right)  f_{\left(  2\right)  }\\
&  =\sum_{\left(  f\right)  }S\left(  f_{\left(  1\right)  }\right)
f_{\left(  2\right)  }=\varepsilon\left(  f\right)
\end{align*}
(by one of the defining properties of the antipode). This proves Proposition
\ref{prop.hmDless.analogue0}.
\end{proof}
\end{verlong}

The \textquotedblleft negative $m$\textquotedblright\ analogue is less
obvious:\footnote{Proposition \ref{prop.hmDless.analogue-} does not literally
involve a negative $m$, but it involves an element $F_{\alpha}^{\setminus m}$
which can be viewed as \textquotedblleft something like $F_{\left(
\alpha\right)  \odot\left(  -m\right)  }$\textquotedblright.}

\begin{proposition}
\label{prop.hmDless.analogue-}Let $f\in\operatorname*{QSym}$ and let $m$ be a
positive integer. For any composition $\alpha=\left(  \alpha_{1},\alpha
_{2},\ldots,\alpha_{\ell}\right)  $, we define an element $F_{\alpha
}^{\setminus m}$ of $\operatorname*{QSym}$ as follows:

\begin{itemize}
\item If $\ell=0$ or $\alpha_{\ell}<m$, then $F_{\alpha}^{\setminus m}=0$.

\item If $\alpha_{\ell}=m$, then $F_{\alpha}^{\setminus m}=F_{\left(
\alpha_{1},\alpha_{2},\ldots,\alpha_{\ell-1}\right)  }$.

\item If $\alpha_{\ell}>m$, then $F_{\alpha}^{\setminus m}=F_{\left(
\alpha_{1},\alpha_{2},\ldots,\alpha_{\ell-1},\alpha_{\ell}-m\right)  }$.
\end{itemize}

(Here, any equality or inequality in which $\alpha_{\ell}$ is mentioned is
understood to include the statement that $\ell>0$.)

Then,%
\[
\left(  -1\right)  ^{m}\sum_{\alpha\in\operatorname*{Comp}}\left(  -1\right)
^{\left\vert \alpha\right\vert }F_{\alpha}^{\setminus m}R_{\omega\left(
\alpha\right)  }^{\perp}f=\varepsilon\left(  R_{\left(  1^{m}\right)  }%
^{\perp}f\right)  .
\]
Here, $\left(  1^{m}\right)  $ denotes the composition $\left(
\underbrace{1,1,\ldots,1}_{m\text{ times}}\right)  $.
\end{proposition}

\begin{vershort}
\begin{proof}
[Proof of Proposition \ref{prop.hmDless.analogue-}.]See the detailed version
of this note.
\end{proof}
\end{vershort}

\begin{verlong}
\begin{proof}
[Proof of Proposition \ref{prop.hmDless.analogue-}.]Let us first make some
auxiliary observations.

Any two elements $a$ and $b$ of $\operatorname*{NSym}$ satisfy
\begin{equation}
\left(  ab\right)  ^{\perp}=b^{\perp}\circ a^{\perp}
\label{pf.prop.hmDless.analogue-.1}%
\end{equation}
\footnote{\textit{Proof of (\ref{pf.prop.hmDless.analogue-.1}):} Let $a$ and
$b$ be two elements of $\operatorname*{NSym}$. Let $c\in\operatorname*{QSym}$.
Then,%
\begin{align*}
\left(  ab\right)  ^{\perp}c  &  =\sum_{\left(  c\right)  }\underbrace{\left(
ab,c_{\left(  1\right)  }\right)  }_{\substack{=\sum_{\left(  c_{\left(
1\right)  }\right)  }\left(  a,\left(  c_{\left(  1\right)  }\right)
_{\left(  1\right)  }\right)  \left(  b,\left(  c_{\left(  1\right)  }\right)
_{\left(  2\right)  }\right)  \\\text{(by (\ref{eq.(ab,c)}), applied to
}c_{\left(  1\right)  }\text{ instead of }c\text{)}}}c_{\left(  2\right)
}\ \ \ \ \ \ \ \ \ \ \left(  \text{by (\ref{pf.hmDless.1}), applied to
}g=ab\text{ and }f=c\right) \\
&  =\sum_{\left(  c\right)  }\sum_{\left(  c_{\left(  1\right)  }\right)
}\left(  a,\left(  c_{\left(  1\right)  }\right)  _{\left(  1\right)
}\right)  \left(  b,\left(  c_{\left(  1\right)  }\right)  _{\left(  2\right)
}\right)  c_{\left(  2\right)  }=\sum_{\left(  c\right)  }\sum_{\left(
c_{\left(  2\right)  }\right)  }\left(  a,c_{\left(  1\right)  }\right)
\left(  b,\left(  c_{\left(  2\right)  }\right)  _{\left(  1\right)  }\right)
\left(  c_{\left(  2\right)  }\right)  _{\left(  2\right)  }\\
&  \ \ \ \ \ \ \ \ \ \ \left(
\begin{array}
[c]{c}%
\text{since the coassociativity of }\Delta\text{ yields}\\
\sum_{\left(  c\right)  }\sum_{\left(  c_{\left(  1\right)  }\right)  }\left(
c_{\left(  1\right)  }\right)  _{\left(  1\right)  }\otimes\left(  c_{\left(
1\right)  }\right)  _{\left(  2\right)  }\otimes c_{\left(  2\right)  }%
=\sum_{\left(  c\right)  }\sum_{\left(  c_{\left(  2\right)  }\right)
}c_{\left(  1\right)  }\otimes\left(  c_{\left(  2\right)  }\right)  _{\left(
1\right)  }\otimes\left(  c_{\left(  2\right)  }\right)  _{\left(  2\right)  }%
\end{array}
\right) \\
&  =\sum_{\left(  c\right)  }\left(  a,c_{\left(  1\right)  }\right)
\sum_{\left(  c_{\left(  2\right)  }\right)  }\left(  b,\left(  c_{\left(
2\right)  }\right)  _{\left(  1\right)  }\right)  \left(  c_{\left(  2\right)
}\right)  _{\left(  2\right)  }.
\end{align*}
Compared with%
\begin{align*}
\left(  b^{\perp}\circ a^{\perp}\right)  \left(  c\right)   &  =b^{\perp
}\left(  \underbrace{a^{\perp}c}_{\substack{=\sum_{\left(  c\right)  }\left(
a,c_{\left(  1\right)  }\right)  c_{\left(  2\right)  }\\\text{(by
(\ref{pf.hmDless.1}), applied to }g=a\text{ and }f=c\text{)}}}\right)
=b^{\perp}\left(  \sum_{\left(  c\right)  }\left(  a,c_{\left(  1\right)
}\right)  c_{\left(  2\right)  }\right) \\
&  =\sum_{\left(  c\right)  }\left(  a,c_{\left(  1\right)  }\right)
\underbrace{b^{\perp}\left(  c_{\left(  2\right)  }\right)  }_{\substack{=\sum
_{\left(  c_{\left(  2\right)  }\right)  }\left(  b,\left(  c_{\left(
2\right)  }\right)  _{\left(  1\right)  }\right)  \left(  c_{\left(  2\right)
}\right)  _{\left(  2\right)  }\\\text{(by (\ref{pf.hmDless.1}), applied to
}g=b\text{ and }f=c_{\left(  2\right)  }\text{)}}}\ \ \ \ \ \ \ \ \ \ \left(
\text{since the map }b^{\perp}\text{ is }\mathbf{k}\text{-linear}\right) \\
&  =\sum_{\left(  c\right)  }\left(  a,c_{\left(  1\right)  }\right)
\sum_{\left(  c_{\left(  2\right)  }\right)  }\left(  b,\left(  c_{\left(
2\right)  }\right)  _{\left(  1\right)  }\right)  \left(  c_{\left(  2\right)
}\right)  _{\left(  2\right)  },
\end{align*}
this yields $\left(  ab\right)  ^{\perp}c=\left(  b^{\perp}\circ a^{\perp
}\right)  \left(  c\right)  $.
\par
Now, let us forget that we fixed $c$. We thus have shown that $\left(
ab\right)  ^{\perp}c=\left(  b^{\perp}\circ a^{\perp}\right)  \left(
c\right)  $ for every $c\in\operatorname*{QSym}$. In other words, $\left(
ab\right)  ^{\perp}=b^{\perp}\circ a^{\perp}$. This proves
(\ref{pf.prop.hmDless.analogue-.1}).}.

For every two compositions $\alpha$ and $\beta$, we define a composition
$\left[  \alpha,\beta\right]  $ by $\left[  \alpha,\beta\right]  =\left(
\alpha_{1},\alpha_{2},\ldots,\alpha_{\ell},\beta_{1},\beta_{2},\ldots
,\beta_{m}\right)  $, where $\alpha$ and $\beta$ are written as $\alpha
=\left(  \alpha_{1},\alpha_{2},\ldots,\alpha_{\ell}\right)  $ and
$\beta=\left(  \beta_{1},\beta_{2},\ldots,\beta_{m}\right)  $. We further
define a composition $\alpha\odot\beta$ as in Proposition \ref{prop.bel.F}.
Then, every two nonempty compositions $\alpha$ and $\beta$ satisfy%
\begin{equation}
R_{\alpha}R_{\beta}=R_{\left[  \alpha,\beta\right]  }+R_{\alpha\odot\beta}.
\label{pf.prop.hmDless.analogue-.2}%
\end{equation}
(This is part of \cite[Theorem 5.4.10(c)]{Reiner}.) Now it is easy to see that%
\begin{equation}
R_{\omega\left(  \left[  \alpha,\left(  m\right)  \right]  \right)
}+R_{\omega\left(  \alpha\odot\left(  m\right)  \right)  }=R_{\left(
1^{m}\right)  }R_{\omega\left(  \alpha\right)  }
\label{pf.prop.hmDless.analogue-.3}%
\end{equation}
for every nonempty composition $\alpha$\ \ \ \ \footnote{\textit{Proof of
(\ref{pf.prop.hmDless.analogue-.3}):} Let $\alpha$ be a nonempty composition.
Proposition \ref{prop.omega.odot} \textbf{(a)} shows that $\omega\left(
\left[  \alpha,\beta\right]  \right)  =\omega\left(  \beta\right)  \odot
\omega\left(  \alpha\right)  $ for every nonempty composition $\beta$.
Applying this to $\beta=\left(  m\right)  $, we obtain $\omega\left(  \left[
\alpha,\left(  m\right)  \right]  \right)  =\underbrace{\omega\left(  \left(
m\right)  \right)  }_{=\left(  1^{m}\right)  }\odot\omega\left(
\alpha\right)  =\left(  1^{m}\right)  \odot\omega\left(  \alpha\right)  $. But
Proposition \ref{prop.omega.odot}\textbf{(b)} shows that $\omega\left(
\alpha\odot\beta\right)  =\left[  \omega\left(  \beta\right)  ,\omega\left(
\alpha\right)  \right]  $ for every nonempty composition $\beta$. Applying
this to $\beta=\left(  m\right)  $, we obtain $\omega\left(  \alpha
\odot\left(  m\right)  \right)  =\left[  \underbrace{\omega\left(  \left(
m\right)  \right)  }_{=\left(  1^{m}\right)  },\omega\left(  \alpha\right)
\right]  =\left[  \left(  1^{m}\right)  ,\omega\left(  \alpha\right)  \right]
$. Now,%
\begin{align*}
R_{\omega\left(  \left[  \alpha,\left(  m\right)  \right]  \right)
}+R_{\omega\left(  \alpha\odot\left(  m\right)  \right)  }  &  =R_{\omega
\left(  \alpha\odot\left(  m\right)  \right)  }+R_{\omega\left(  \left[
\alpha,\left(  m\right)  \right]  \right)  }=R_{\left[  \left(  1^{m}\right)
,\omega\left(  \alpha\right)  \right]  }+R_{\left(  1^{m}\right)  \odot
\omega\left(  \alpha\right)  }\\
&  \ \ \ \ \ \ \ \ \ \ \left(  \text{since }\omega\left(  \alpha\odot\left(
m\right)  \right)  =\left[  \left(  1^{m}\right)  ,\omega\left(
\alpha\right)  \right]  \text{ and }\omega\left(  \left[  \alpha,\left(
m\right)  \right]  \right)  =\left(  1^{m}\right)  \odot\omega\left(
\alpha\right)  \right) \\
&  =R_{\left(  1^{m}\right)  }R_{\omega\left(  \alpha\right)  }%
\end{align*}
(since (\ref{pf.prop.hmDless.analogue-.2}) (applied to $\left(  1^{m}\right)
$ and $\omega\left(  \alpha\right)  $ instead of $\alpha$ and $\beta$) shows
that $R_{\left(  1^{m}\right)  }R_{\omega\left(  \alpha\right)  }=R_{\left[
\left(  1^{m}\right)  ,\omega\left(  \alpha\right)  \right]  }+R_{\left(
1^{m}\right)  \odot\omega\left(  \alpha\right)  }$). This proves
(\ref{pf.prop.hmDless.analogue-.3}).}. Hence, for every nonempty composition
$\alpha$, we have%
\begin{equation}
\left(  \underbrace{R_{\omega\left(  \left[  \alpha,\left(  m\right)  \right]
\right)  }+R_{\omega\left(  \alpha\odot\left(  m\right)  \right)  }%
}_{=R_{\left(  1^{m}\right)  }R_{\omega\left(  \alpha\right)  }}\right)
^{\perp}=\left(  R_{\left(  1^{m}\right)  }R_{\omega\left(  \alpha\right)
}\right)  ^{\perp}=R_{\omega\left(  \alpha\right)  }^{\perp}\circ R_{\left(
1^{m}\right)  }^{\perp} \label{pf.prop.hmDless.analogue-.4}%
\end{equation}
(by (\ref{pf.prop.hmDless.analogue-.1}), applied to $a=R_{\left(
1^{m}\right)  }$ and $b=R_{\omega\left(  \alpha\right)  }$).

We furthermore notice that $\omega\left(  \varnothing\right)  =\varnothing$
and thus $R_{\omega\left(  \varnothing\right)  }^{\perp}=R_{\varnothing
}^{\perp}=\operatorname*{id}$ (since $R_{\varnothing}=1$).

Now,%
\begin{align}
&  \sum_{\substack{\left(  \alpha_{1},\alpha_{2},\ldots,\alpha_{\ell}\right)
\in\operatorname*{Comp};\\\alpha_{\ell}=m}}\underbrace{\left(  -1\right)
^{\left\vert \left(  \alpha_{1},\alpha_{2},\ldots,\alpha_{\ell}\right)
\right\vert }}_{\substack{=\left(  -1\right)  ^{\left\vert \left(  \alpha
_{1},\alpha_{2},\ldots,\alpha_{\ell-1},m\right)  \right\vert }\\\text{(since
}\alpha_{\ell}=m\text{)}}}\underbrace{F_{\left(  \alpha_{1},\alpha_{2}%
,\ldots,\alpha_{\ell}\right)  }^{\setminus m}}_{\substack{=F_{\left(
\alpha_{1},\alpha_{2},\ldots,\alpha_{\ell-1}\right)  }\\\text{(since }%
\alpha_{\ell}=m\text{)}}}\underbrace{R_{\omega\left(  \left(  \alpha
_{1},\alpha_{2},\ldots,\alpha_{\ell}\right)  \right)  }^{\perp}}%
_{\substack{=R_{\omega\left(  \left(  \alpha_{1},\alpha_{2},\ldots
,\alpha_{\ell-1},m\right)  \right)  }^{\perp}\\\text{(since }\alpha_{\ell
}=m\text{)}}}f\nonumber\\
&  =\underbrace{\sum_{\substack{\left(  \alpha_{1},\alpha_{2},\ldots
,\alpha_{\ell}\right)  \in\operatorname*{Comp};\\\alpha_{\ell}=m}}}%
_{=\sum_{\left(  \alpha_{1},\alpha_{2},\ldots,\alpha_{\ell-1}\right)
\in\operatorname*{Comp}}}\underbrace{\left(  -1\right)  ^{\left\vert \left(
\alpha_{1},\alpha_{2},\ldots,\alpha_{\ell-1},m\right)  \right\vert }%
}_{=\left(  -1\right)  ^{\left\vert \left(  \alpha_{1},\alpha_{2}%
,\ldots,\alpha_{\ell-1}\right)  \right\vert +m}}F_{\left(  \alpha_{1}%
,\alpha_{2},\ldots,\alpha_{\ell-1}\right)  }\underbrace{R_{\omega\left(
\left(  \alpha_{1},\alpha_{2},\ldots,\alpha_{\ell-1},m\right)  \right)
}^{\perp}}_{\substack{=R_{\omega\left(  \left[  \left(  \alpha_{1},\alpha
_{2},\ldots,\alpha_{\ell-1}\right)  ,\left(  m\right)  \right]  \right)
}^{\perp}\\\text{(since }\left(  \alpha_{1},\alpha_{2},\ldots,\alpha_{\ell
-1},m\right)  \\=\left[  \left(  \alpha_{1},\alpha_{2},\ldots,\alpha_{\ell
-1}\right)  ,\left(  m\right)  \right]  \text{)}}}f\nonumber\\
&  =\sum_{\left(  \alpha_{1},\alpha_{2},\ldots,\alpha_{\ell-1}\right)
\in\operatorname*{Comp}}\left(  -1\right)  ^{\left\vert \left(  \alpha
_{1},\alpha_{2},\ldots,\alpha_{\ell-1}\right)  \right\vert +m}F_{\left(
\alpha_{1},\alpha_{2},\ldots,\alpha_{\ell-1}\right)  }R_{\omega\left(  \left[
\left(  \alpha_{1},\alpha_{2},\ldots,\alpha_{\ell-1}\right)  ,\left(
m\right)  \right]  \right)  }^{\perp}f\nonumber\\
&  =\sum_{\alpha\in\operatorname*{Comp}}\left(  -1\right)  ^{\left\vert
\alpha\right\vert +m}F_{\alpha}R_{\omega\left(  \left[  \alpha,\left(
m\right)  \right]  \right)  }^{\perp}f\nonumber\\
&  \ \ \ \ \ \ \ \ \ \ \left(  \text{here, we have substituted }\alpha\text{
for }\left(  \alpha_{1},\alpha_{2},\ldots,\alpha_{\ell-1}\right)  \text{ in
the sum}\right) \nonumber\\
&  =\left(  -1\right)  ^{\left\vert \varnothing\right\vert +m}F_{\varnothing
}R_{\omega\left(  \left[  \varnothing,\left(  m\right)  \right]  \right)
}^{\perp}f+\sum_{\substack{\alpha\in\operatorname*{Comp};\\\alpha\text{ is
nonempty}}}\left(  -1\right)  ^{\left\vert \alpha\right\vert +m}F_{\alpha
}R_{\omega\left(  \left[  \alpha,\left(  m\right)  \right]  \right)  }^{\perp
}f \label{pf.prop.hmDless.analogue-.11}%
\end{align}
(here, we have split off the addend for $\alpha=\varnothing$ from the sum). On
the other hand,%
\begin{align}
&  \sum_{\substack{\left(  \alpha_{1},\alpha_{2},\ldots,\alpha_{\ell}\right)
\in\operatorname*{Comp};\\\alpha_{\ell}>m}}\left(  -1\right)  ^{\left\vert
\left(  \alpha_{1},\alpha_{2},\ldots,\alpha_{\ell}\right)  \right\vert
}\underbrace{F_{\left(  \alpha_{1},\alpha_{2},\ldots,\alpha_{\ell}\right)
}^{\setminus m}}_{\substack{=F_{\left(  \alpha_{1},\alpha_{2},\ldots
,\alpha_{\ell-1},\alpha_{\ell}-m\right)  }\\\text{(since }\alpha_{\ell
}>m\text{)}}}R_{\omega\left(  \left(  \alpha_{1},\alpha_{2},\ldots
,\alpha_{\ell}\right)  \right)  }^{\perp}f\nonumber\\
&  =\sum_{\substack{\left(  \alpha_{1},\alpha_{2},\ldots,\alpha_{\ell}\right)
\in\operatorname*{Comp};\\\alpha_{\ell}>m}}\left(  -1\right)  ^{\left\vert
\left(  \alpha_{1},\alpha_{2},\ldots,\alpha_{\ell}\right)  \right\vert
}F_{\left(  \alpha_{1},\alpha_{2},\ldots,\alpha_{\ell-1},\alpha_{\ell
}-m\right)  }R_{\omega\left(  \left(  \alpha_{1},\alpha_{2},\ldots
,\alpha_{\ell}\right)  \right)  }^{\perp}f\nonumber\\
&  =\sum_{\substack{\left(  \alpha_{1},\alpha_{2},\ldots,\alpha_{\ell}\right)
\in\operatorname*{Comp};\\\ell>0}}\underbrace{\left(  -1\right)  ^{\left\vert
\left(  \alpha_{1},\alpha_{2},\ldots,\alpha_{\ell-1},\alpha_{\ell}+m\right)
\right\vert }}_{=\left(  -1\right)  ^{\left\vert \left(  \alpha_{1},\alpha
_{2},\ldots,\alpha_{\ell}\right)  \right\vert +m}}F_{\left(  \alpha_{1}%
,\alpha_{2},\ldots,\alpha_{\ell}\right)  }\nonumber\\
&  \ \ \ \ \ \ \ \ \ \ \underbrace{R_{\omega\left(  \left(  \alpha_{1}%
,\alpha_{2},\ldots,\alpha_{\ell-1},\alpha_{\ell}+m\right)  \right)  }^{\perp}%
}_{\substack{=R_{\omega\left(  \left(  \alpha_{1},\alpha_{2},\ldots
,\alpha_{\ell}\right)  \odot\left(  m\right)  \right)  }^{\perp}\\\text{(since
}\left(  \alpha_{1},\alpha_{2},\ldots,\alpha_{\ell-1},\alpha_{\ell}+m\right)
=\left(  \alpha_{1},\alpha_{2},\ldots,\alpha_{\ell}\right)  \odot\left(
m\right)  \text{)}}}f\nonumber\\
&  \ \ \ \ \ \ \ \ \ \ \left(
\begin{array}
[c]{c}%
\text{here, we have substituted }\left(  \alpha_{1},\alpha_{2},\ldots
,\alpha_{\ell}\right) \\
\text{for }\left(  \alpha_{1},\alpha_{2},\ldots,\alpha_{\ell-1},\alpha_{\ell
}-m\right)  \text{ in the sum}%
\end{array}
\right) \nonumber\\
&  =\sum_{\substack{\left(  \alpha_{1},\alpha_{2},\ldots,\alpha_{\ell}\right)
\in\operatorname*{Comp};\\\ell>0}}\left(  -1\right)  ^{\left\vert \left(
\alpha_{1},\alpha_{2},\ldots,\alpha_{\ell}\right)  \right\vert +m}F_{\left(
\alpha_{1},\alpha_{2},\ldots,\alpha_{\ell}\right)  }R_{\omega\left(  \left(
\alpha_{1},\alpha_{2},\ldots,\alpha_{\ell}\right)  \odot\left(  m\right)
\right)  }^{\perp}f\nonumber\\
&  =\sum_{\substack{\alpha\in\operatorname*{Comp};\\\alpha\text{ is nonempty}%
}}\left(  -1\right)  ^{\left\vert \alpha\right\vert +m}F_{\alpha}%
R_{\omega\left(  \alpha\odot\left(  m\right)  \right)  }^{\perp}f
\label{pf.prop.hmDless.analogue-.16}%
\end{align}
(here, we have substituted $\alpha$ for $\left(  \alpha_{1},\alpha_{2}%
,\ldots,\alpha_{\ell}\right)  $ in the sum).

But%
\begin{align*}
&  \sum_{\alpha\in\operatorname*{Comp}}\left(  -1\right)  ^{\left\vert
\alpha\right\vert }F_{\alpha}^{\setminus m}R_{\omega\left(  \alpha\right)
}^{\perp}f\\
&  =\sum_{\left(  \alpha_{1},\alpha_{2},\ldots,\alpha_{\ell}\right)
\in\operatorname*{Comp}}\left(  -1\right)  ^{\left\vert \left(  \alpha
_{1},\alpha_{2},\ldots,\alpha_{\ell}\right)  \right\vert }F_{\left(
\alpha_{1},\alpha_{2},\ldots,\alpha_{\ell}\right)  }^{\setminus m}%
R_{\omega\left(  \left(  \alpha_{1},\alpha_{2},\ldots,\alpha_{\ell}\right)
\right)  }^{\perp}f\\
&  \ \ \ \ \ \ \ \ \ \ \left(  \text{here, we have renamed the summation index
}\alpha\text{ as }\left(  \alpha_{1},\alpha_{2},\ldots,\alpha_{\ell}\right)
\right) \\
&  =\sum_{\substack{\left(  \alpha_{1},\alpha_{2},\ldots,\alpha_{\ell}\right)
\in\operatorname*{Comp};\\\ell=0\text{ or }\alpha_{\ell}<m}}\left(  -1\right)
^{\left\vert \left(  \alpha_{1},\alpha_{2},\ldots,\alpha_{\ell}\right)
\right\vert }\underbrace{F_{\left(  \alpha_{1},\alpha_{2},\ldots,\alpha_{\ell
}\right)  }^{\setminus m}}_{\substack{=0\\\text{(since }\ell=0\text{ or
}\alpha_{\ell}<m\text{)}}}R_{\omega\left(  \left(  \alpha_{1},\alpha
_{2},\ldots,\alpha_{\ell}\right)  \right)  }^{\perp}f\\
&  \ \ \ \ \ \ \ \ \ \ +\underbrace{\sum_{\substack{\left(  \alpha_{1}%
,\alpha_{2},\ldots,\alpha_{\ell}\right)  \in\operatorname*{Comp}%
;\\\alpha_{\ell}=m}}\left(  -1\right)  ^{\left\vert \left(  \alpha_{1}%
,\alpha_{2},\ldots,\alpha_{\ell}\right)  \right\vert }F_{\left(  \alpha
_{1},\alpha_{2},\ldots,\alpha_{\ell}\right)  }^{\setminus m}R_{\omega\left(
\left(  \alpha_{1},\alpha_{2},\ldots,\alpha_{\ell}\right)  \right)  }^{\perp
}f}_{\substack{=\left(  -1\right)  ^{\left\vert \varnothing\right\vert
+m}F_{\varnothing}R_{\omega\left(  \left[  \varnothing,\left(  m\right)
\right]  \right)  }^{\perp}f+\sum_{\substack{\alpha\in\operatorname*{Comp}%
;\\\alpha\text{ is nonempty}}}\left(  -1\right)  ^{\left\vert \alpha
\right\vert +m}F_{\alpha}R_{\omega\left(  \left[  \alpha,\left(  m\right)
\right]  \right)  }^{\perp}f\\\text{(by (\ref{pf.prop.hmDless.analogue-.11}%
))}}}\\
&  \ \ \ \ \ \ \ \ \ \ +\underbrace{\sum_{\substack{\left(  \alpha_{1}%
,\alpha_{2},\ldots,\alpha_{\ell}\right)  \in\operatorname*{Comp}%
;\\\alpha_{\ell}>m}}\left(  -1\right)  ^{\left\vert \left(  \alpha_{1}%
,\alpha_{2},\ldots,\alpha_{\ell}\right)  \right\vert }F_{\left(  \alpha
_{1},\alpha_{2},\ldots,\alpha_{\ell}\right)  }^{\setminus m}R_{\omega\left(
\left(  \alpha_{1},\alpha_{2},\ldots,\alpha_{\ell}\right)  \right)  }^{\perp
}f}_{\substack{=\sum_{\substack{\alpha\in\operatorname*{Comp};\\\alpha\text{
is nonempty}}}\left(  -1\right)  ^{\left\vert \alpha\right\vert +m}F_{\alpha
}R_{\omega\left(  \alpha\odot\left(  m\right)  \right)  }^{\perp}f\\\text{(by
(\ref{pf.prop.hmDless.analogue-.16}))}}}\\
&  =\underbrace{\sum_{\substack{\left(  \alpha_{1},\alpha_{2},\ldots
,\alpha_{\ell}\right)  \in\operatorname*{Comp};\\\ell=0\text{ or }\alpha
_{\ell}<m}}\left(  -1\right)  ^{\left\vert \left(  \alpha_{1},\alpha
_{2},\ldots,\alpha_{\ell}\right)  \right\vert }0R_{\omega\left(  \left(
\alpha_{1},\alpha_{2},\ldots,\alpha_{\ell}\right)  \right)  }^{\perp}f}_{=0}\\
&  \ \ \ \ \ \ \ \ \ \ +\left(  -1\right)  ^{\left\vert \varnothing\right\vert
+m}F_{\varnothing}R_{\omega\left(  \left[  \varnothing,\left(  m\right)
\right]  \right)  }^{\perp}f+\sum_{\substack{\alpha\in\operatorname*{Comp}%
;\\\alpha\text{ is nonempty}}}\left(  -1\right)  ^{\left\vert \alpha
\right\vert +m}F_{\alpha}R_{\omega\left(  \left[  \alpha,\left(  m\right)
\right]  \right)  }^{\perp}f\\
&  \ \ \ \ \ \ \ \ \ \ +\sum_{\substack{\alpha\in\operatorname*{Comp}%
;\\\alpha\text{ is nonempty}}}\left(  -1\right)  ^{\left\vert \alpha
\right\vert +m}F_{\alpha}R_{\omega\left(  \alpha\odot\left(  m\right)
\right)  }^{\perp}f
\end{align*}%
\begin{align*}
&  =\left(  -1\right)  ^{\left\vert \varnothing\right\vert +m}F_{\varnothing
}R_{\omega\left(  \left[  \varnothing,\left(  m\right)  \right]  \right)
}^{\perp}f+\sum_{\substack{\alpha\in\operatorname*{Comp};\\\alpha\text{ is
nonempty}}}\left(  -1\right)  ^{\left\vert \alpha\right\vert +m}F_{\alpha
}R_{\omega\left(  \left[  \alpha,\left(  m\right)  \right]  \right)  }^{\perp
}f\\
&  \ \ \ \ \ \ \ \ \ \ +\sum_{\substack{\alpha\in\operatorname*{Comp}%
;\\\alpha\text{ is nonempty}}}\left(  -1\right)  ^{\left\vert \alpha
\right\vert +m}F_{\alpha}R_{\omega\left(  \alpha\odot\left(  m\right)
\right)  }^{\perp}f\\
&  =\left(  -1\right)  ^{\left\vert \varnothing\right\vert +m}F_{\varnothing
}\underbrace{R_{\omega\left(  \left[  \varnothing,\left(  m\right)  \right]
\right)  }^{\perp}}_{\substack{=R_{\left(  1^{m}\right)  }^{\perp
}\\\text{(since }\omega\left(  \left[  \varnothing,\left(  m\right)  \right]
\right)  =\omega\left(  \left(  m\right)  \right)  =\left(  1^{m}\right)
\text{)}}}f\\
&  \ \ \ \ \ \ \ \ \ \ +\sum_{\substack{\alpha\in\operatorname*{Comp}%
;\\\alpha\text{ is nonempty}}}\left(  -1\right)  ^{\left\vert \alpha
\right\vert +m}F_{\alpha}\underbrace{\left(  R_{\omega\left(  \left[
\alpha,\left(  m\right)  \right]  \right)  }+R_{\omega\left(  \alpha
\odot\left(  m\right)  \right)  }\right)  ^{\perp}}_{\substack{=R_{\omega
\left(  \alpha\right)  }^{\perp}\circ R_{\left(  1^{m}\right)  }^{\perp
}\\\text{(by (\ref{pf.prop.hmDless.analogue-.4}))}}}f\\
&  =\left(  -1\right)  ^{\left\vert \varnothing\right\vert +m}F_{\varnothing
}\underbrace{R_{\left(  1^{m}\right)  }^{\perp}f}_{\substack{=R_{\omega\left(
\varnothing\right)  }^{\perp}\left(  R_{\left(  1^{m}\right)  }^{\perp
}f\right)  \\\text{(since }R_{\omega\left(  \varnothing\right)  }^{\perp
}=\operatorname*{id}\text{ and thus}\\R_{\omega\left(  \varnothing\right)
}^{\perp}\left(  R_{\left(  1^{m}\right)  }^{\perp}f\right)  =R_{\left(
1^{m}\right)  }^{\perp}f\text{)}}}+\sum_{\substack{\alpha\in
\operatorname*{Comp};\\\alpha\text{ is nonempty}}}\left(  -1\right)
^{\left\vert \alpha\right\vert +m}F_{\alpha}\underbrace{\left(  R_{\omega
\left(  \alpha\right)  }^{\perp}\circ R_{\left(  1^{m}\right)  }^{\perp
}\right)  f}_{=R_{\omega\left(  \alpha\right)  }^{\perp}\left(  R_{\left(
1^{m}\right)  }^{\perp}f\right)  }\\
&  =\left(  -1\right)  ^{\left\vert \varnothing\right\vert +m}F_{\varnothing
}R_{\omega\left(  \varnothing\right)  }^{\perp}\left(  R_{\left(
1^{m}\right)  }^{\perp}f\right)  +\sum_{\substack{\alpha\in
\operatorname*{Comp};\\\alpha\text{ is nonempty}}}\left(  -1\right)
^{\left\vert \alpha\right\vert +m}F_{\alpha}R_{\omega\left(  \alpha\right)
}^{\perp}\left(  R_{\left(  1^{m}\right)  }^{\perp}f\right) \\
&  =\sum_{\alpha\in\operatorname*{Comp}}\underbrace{\left(  -1\right)
^{\left\vert \alpha\right\vert +m}}_{=\left(  -1\right)  ^{m}\left(
-1\right)  ^{\left\vert \alpha\right\vert }}F_{\alpha}R_{\omega\left(
\alpha\right)  }^{\perp}\left(  R_{\left(  1^{m}\right)  }^{\perp}f\right)
\ \ \ \ \ \ \ \ \ \ \left(
\begin{array}
[c]{c}%
\text{here, we have incorporated the}\\
\alpha=\varnothing\text{ addend into the sum}%
\end{array}
\right) \\
&  =\left(  -1\right)  ^{m}\sum_{\alpha\in\operatorname*{Comp}}\left(
-1\right)  ^{\left\vert \alpha\right\vert }F_{\alpha}R_{\omega\left(
\alpha\right)  }^{\perp}\left(  R_{\left(  1^{m}\right)  }^{\perp}f\right)  .
\end{align*}
Multiplying both sides of this equality with $\left(  -1\right)  ^{m}$, we
obtain%
\[
\left(  -1\right)  ^{m}\sum_{\alpha\in\operatorname*{Comp}}\left(  -1\right)
^{\left\vert \alpha\right\vert }F_{\alpha}^{\setminus m}R_{\omega\left(
\alpha\right)  }^{\perp}f=\sum_{\alpha\in\operatorname*{Comp}}\left(
-1\right)  ^{\left\vert \alpha\right\vert }F_{\alpha}R_{\omega\left(
\alpha\right)  }^{\perp}\left(  R_{\left(  1^{m}\right)  }^{\perp}f\right)  .
\]
Comparing this with%
\[
\varepsilon\left(  R_{\left(  1^{m}\right)  }^{\perp}f\right)  =\sum
_{\alpha\in\operatorname*{Comp}}\left(  -1\right)  ^{\left\vert \alpha
\right\vert }F_{\alpha}R_{\omega\left(  \alpha\right)  }^{\perp}\left(
R_{\left(  1^{m}\right)  }^{\perp}f\right)
\]
(by Proposition \ref{prop.hmDless.analogue0}, applied to $R_{\left(
1^{m}\right)  }^{\perp}f$ instead of $f$), we obtain%
\[
\left(  -1\right)  ^{m}\sum_{\alpha\in\operatorname*{Comp}}\left(  -1\right)
^{\left\vert \alpha\right\vert }F_{\alpha}^{\setminus m}R_{\omega\left(
\alpha\right)  }^{\perp}f=\varepsilon\left(  R_{\left(  1^{m}\right)  }%
^{\perp}f\right)  .
\]
This proves Proposition \ref{prop.hmDless.analogue-}.
\end{proof}
\end{verlong}

\section{\label{sect.WQSym}Lifts to $\operatorname*{WQSym}$ and
$\operatorname*{FQSym}$}

We have so far been studying the Hopf algebras $\operatorname*{Sym}$,
$\operatorname*{QSym}$ and $\operatorname*{NSym}$. These are merely the tip of
an iceberg; dozens of combinatorial Hopf algebras are currently known, many of
which are extensions of these. In this final section, we shall discuss how
(and whether) our operations $\left.  \prec\right.  $ and $\bel $ as well as
some similar operations can be lifted to the bigger Hopf algebras
$\operatorname*{WQSym}$ and $\operatorname*{FQSym}$. We shall give no proofs,
as these are not difficult and the whole discussion is tangential to this note.

Let us first define these two Hopf algebras (which are discussed, for example,
in \cite{FM}).

We start with $\operatorname*{WQSym}$. (Our definition of
$\operatorname*{WQSym}$ follows the papers of the Marne-la-Vall\'{e}e school,
such as \cite[Section 5.1]{AFNT}\footnote{where $\operatorname*{WQSym}$ is
denoted by $\mathbf{WQSym}$}; it will differ from that in \cite{FM}, but we
will explain why it is equivalent.)

Let $X_{1},X_{2},X_{3},\ldots$ be countably many distinct symbols. These
symbols will be called \textit{letters}. We define a \textit{word} to be an
$\ell$-tuple of elements of $\left\{  X_{1},X_{2},X_{3},\ldots\right\}  $ for
some $\ell\in\mathbb{N}$. Thus, for example, $\left(  X_{3},X_{5}%
,X_{2}\right)  $ and $\left(  X_{6}\right)  $ are words. We denote the empty
word $\left(  {}\right)  $ by $1$, and we often identify the one-letter word
$\left(  X_{i}\right)  $ with the symbol $X_{i}$ for every $i>0$. For any two
words $u=\left(  X_{i_{1}},X_{i_{2}},\ldots,X_{i_{n}}\right)  $ and $v=\left(
X_{j_{1}},X_{j_{2}},\ldots,X_{j_{m}}\right)  $, we define the concatenation
$uv$ as the word \newline$\left(  X_{i_{1}},X_{i_{2}},\ldots,X_{i_{n}%
},X_{j_{1}},X_{j_{2}},\ldots,X_{j_{m}}\right)  $. Concatenation is an
associative operation and the empty word $1$ is a neutral element for it;
thus, the words form a monoid. We let $\operatorname*{Wrd}$ denote this
monoid. This monoid is the free monoid on the set $\left\{  X_{1},X_{2}%
,X_{3},\ldots\right\}  $. Concatenation allows us to rewrite any word $\left(
X_{i_{1}},X_{i_{2}},\ldots,X_{i_{n}}\right)  $ in the shorter form $X_{i_{1}%
}X_{i_{2}}\cdots X_{i_{n}}$.

Notice that $\operatorname*{Mon}$ (the set of all monomials) is also a monoid
under multiplication. We can thus define a monoid homomorphism $\pi
:\operatorname*{Wrd}\rightarrow\operatorname*{Mon}$ by $\pi\left(
X_{i}\right)  =x_{i}$ for all $i\in\left\{  1,2,3,\ldots\right\}  $. This
homomorphism $\pi$ is surjective.

We define $\mathbf{k}\left\langle \left\langle \mathbf{X}\right\rangle
\right\rangle $ to be the $\mathbf{k}$-module $\mathbf{k}^{\operatorname*{Wrd}%
}$; its elements are all families $\left(  \lambda_{w}\right)  _{w\in
\operatorname*{Wrd}}\in\mathbf{k}^{\operatorname*{Wrd}}$. We define a
multiplication on\textbf{ }$\mathbf{k}\left\langle \left\langle \mathbf{X}%
\right\rangle \right\rangle $ by
\begin{equation}
\left(  \lambda_{w}\right)  _{w\in\operatorname*{Wrd}}\cdot\left(  \mu
_{w}\right)  _{w\in\operatorname*{Wrd}}=\left(  \sum_{\left(  u,v\right)
\in\operatorname*{Wrd}\nolimits^{2};\ uv=w}\lambda_{u}\mu_{v}\right)
_{w\in\operatorname*{Wrd}}. \label{eq.WQSym.powerseries-mul}%
\end{equation}
This makes $\mathbf{k}\left\langle \left\langle \mathbf{X}\right\rangle
\right\rangle $ into a $\mathbf{k}$-algebra, with unity $\left(  \delta
_{w,1}\right)  _{w\in\operatorname*{Wrd}}$. This $\mathbf{k}$-algebra is
called the $\mathbf{k}$\textit{-algebra of noncommutative power series in
}$X_{1},X_{2},X_{3},\ldots$. For every $u\in\operatorname*{Wrd}$, we identify
the word $u$ with the element $\left(  \delta_{w,u}\right)  _{w\in
\operatorname*{Wrd}}$ of $\mathbf{k}\left\langle \left\langle \mathbf{X}%
\right\rangle \right\rangle $\ \ \ \ \footnote{This identification is
harmless, since the map $\operatorname*{Wrd}\rightarrow\mathbf{k}\left\langle
\left\langle \mathbf{X}\right\rangle \right\rangle ,\ u\mapsto\left(
\delta_{w,u}\right)  _{w\in\operatorname*{Wrd}}$ is a monoid homomorphism from
$\operatorname*{Wrd}$ to $\left(  \mathbf{k}\left\langle \left\langle
\mathbf{X}\right\rangle \right\rangle ,\cdot\right)  $. (However, it fails to
be injective if $\mathbf{k}=0$.)}. The $\mathbf{k}$-algebra $\mathbf{k}%
\left\langle \left\langle \mathbf{X}\right\rangle \right\rangle $ becomes a
topological $\mathbf{k}$-algebra via the product topology (recalling that
$\mathbf{k}\left\langle \left\langle \mathbf{X}\right\rangle \right\rangle
=\mathbf{k}^{\operatorname*{Wrd}}$ as sets). Thus, every element $\left(
\lambda_{w}\right)  _{w\in\operatorname*{Wrd}}$ of $\mathbf{k}\left\langle
\left\langle \mathbf{X}\right\rangle \right\rangle $ can be rewritten in the
form $\sum_{w\in\operatorname*{Wrd}}\lambda_{w}w$. This turns the equality
(\ref{eq.WQSym.powerseries-mul}) into a distributive law (for infinite sums),
and explains why we refer to elements of $\mathbf{k}\left\langle \left\langle
\mathbf{X}\right\rangle \right\rangle $ as \textquotedblleft noncommutative
power series\textquotedblright. We think of words as noncommutative analogues
of monomials.

The \textit{degree} of a word $w$ will mean its length (i.e., the integer $n$
for which $w$ is an $n$-tuple). Let $\mathbf{k}\left\langle \left\langle
\mathbf{X}\right\rangle \right\rangle _{\operatorname*{bdd}}$ denote the
$\mathbf{k}$-subalgebra of $\mathbf{k}\left\langle \left\langle \mathbf{X}%
\right\rangle \right\rangle $ formed by the \textit{bounded-degree
noncommutative power series}\footnote{A noncommutative power series $\left(
\lambda_{w}\right)  _{w\in\operatorname*{Wrd}}\in\mathbf{k}\left\langle
\left\langle \mathbf{X}\right\rangle \right\rangle $ is said to be
\textit{bounded-degree} if there is an $N\in\mathbb{N}$ such that every word
$w$ of length $>N$ satisfies $\lambda_{w}=0$.} in $\mathbf{k}\left\langle
\left\langle \mathbf{X}\right\rangle \right\rangle $. The surjective monoid
homomorphism $\pi:\operatorname*{Wrd}\rightarrow\operatorname*{Mon}$
canonically gives rise to surjective $\mathbf{k}$-algebra homomorphisms
$\mathbf{k}\left\langle \left\langle \mathbf{X}\right\rangle \right\rangle
\rightarrow\mathbf{k}\left[  \left[  x_{1},x_{2},x_{3},\ldots\right]  \right]
$ and $\mathbf{k}\left\langle \left\langle \mathbf{X}\right\rangle
\right\rangle _{\operatorname*{bdd}}\rightarrow\mathbf{k}\left[  \left[
x_{1},x_{2},x_{3},\ldots\right]  \right]  _{\operatorname*{bdd}}$, which we
also denote by $\pi$. Notice that the $\mathbf{k}$-algebra $\mathbf{k}%
\left\langle \left\langle \mathbf{X}\right\rangle \right\rangle
_{\operatorname*{bdd}}$ is denoted $R\left\langle \mathbf{X}\right\rangle $ in
\cite[Section 8.1]{Reiner}.

If $w$ is a word, then we denote by $\operatorname*{Supp}w$ the subset%
\[
\left\{  i\in\left\{  1,2,3,\ldots\right\}  \ \mid\ \text{the symbol }%
X_{i}\text{ is an entry of }w\right\}
\]
of $\left\{  1,2,3,\ldots\right\}  $. Notice that $\operatorname*{Supp}%
w=\operatorname*{Supp}\left(  \pi\left(  w\right)  \right)  $ is a finite set.

A word $w$ is said to be \textit{packed} if there exists an $\ell\in
\mathbb{N}$ such that $\operatorname*{Supp}w=\left\{  1,2,\ldots,\ell\right\}
$.

For each word $w$, we define a packed word $\operatorname*{pack}w$ as follows:
Replace the smallest letter\footnote{We use the total ordering on the set
$\left\{  X_{1},X_{2},X_{3},\ldots\right\}  $ given by $X_{1}<X_{2}%
<X_{3}<\cdots$.} that appears in $w$ by $X_{1}$, the second-smallest letter by
$X_{2}$, etc..\footnote{Here is a more pedantic way to restate this
definition: Write $w$ as $\left(  X_{i_{1}},X_{i_{2}},\ldots,X_{i_{\ell}%
}\right)  $, and let $I=\operatorname*{Supp}w$ (so that $I=\left\{
i_{1},i_{2},\ldots,i_{\ell}\right\}  $). Let $r_{I}$ be the unique increasing
bijection $\left\{  1,2,\ldots,\left\vert I\right\vert \right\}  \rightarrow
I$. Then, $\operatorname*{pack}w$ denotes the word $\left(  X_{r_{I}%
^{-1}\left(  i_{1}\right)  },X_{r_{I}^{-1}\left(  i_{2}\right)  }%
,\ldots,X_{r_{I}^{-1}\left(  i_{\ell}\right)  }\right)  $.} This word
$\operatorname*{pack}w$ is called the \textit{packing} of $w$. For example,
$\operatorname*{pack}\left(  X_{3}X_{1}X_{6}X_{1}\right)  =X_{2}X_{1}%
X_{3}X_{1}$.

\begin{noncompile}
Two words $u$ and $v$ are said to be \textit{pack-equivalent} if and only if
$\operatorname*{pack}u=\operatorname*{pack}v$.
\end{noncompile}

For every packed word $u$, we define an element $\mathbf{M}_{u}$ of
$\mathbf{k}\left\langle \left\langle \mathbf{X}\right\rangle \right\rangle
_{\operatorname*{bdd}}$ by%
\[
\mathbf{M}_{u}=\sum_{\substack{w\in\operatorname*{Wrd};\\\operatorname*{pack}%
w=u}}w.
\]
(This element $\mathbf{M}_{u}$ is denoted $P_{u}$ in \cite[Section 5.1]%
{AFNT}.) We denote by $\operatorname*{WQSym}$ the $\mathbf{k}$-submodule of
$\mathbf{k}\left\langle \left\langle \mathbf{X}\right\rangle \right\rangle
_{\operatorname*{bdd}}$ spanned by the $\mathbf{M}_{u}$ for all packed words
$u$. It is known that $\operatorname*{WQSym}$ is a $\mathbf{k}$-subalgebra of
$\mathbf{k}\left\langle \left\langle \mathbf{X}\right\rangle \right\rangle
_{\operatorname*{bdd}}$ which can furthermore be endowed with a Hopf algebra
structure (the so-called \textit{Hopf algebra of word quasisymmetric
functions}) such that $\pi$ restricts to a Hopf algebra surjection
$\operatorname*{WQSym}\rightarrow\operatorname*{QSym}$. Notice that
$\pi\left(  \mathbf{M}_{u}\right)  =M_{\operatorname*{Parikh}\left(
\pi\left(  u\right)  \right)  }$ for every packed word $u$, where the Parikh
composition $\operatorname*{Parikh}\mathfrak{m}$ of any monomial
$\mathfrak{m}$ is defined as in the proof of Proposition
\ref{prop.QSym.closed}.

The elements $\mathbf{M}_{u}$ with $u$ ranging over all packed words form a
basis of the $\mathbf{k}$-module $\operatorname*{WQSym}$, which is usually
called the \textit{monomial basis}\footnote{Sometimes it is parametrized not
by packed words but instead by set compositions (i.e., ordered set partitions)
of sets of the form $\left\{  1,2,\ldots,n\right\}  $ with $n\in\mathbb{N}$.
But the packed words of length $n$ are in a 1-to-1 correspondence with set
compositions of $\left\{  1,2,\ldots,n\right\}  $, so this is merely a matter
of relabelling.}. Furthermore, the product of two such elements can be
computed by the well-known formula\footnote{This formula appears in
\cite[Proposition 4.1]{MNT}.}%
\begin{equation}
\mathbf{M}_{u}\mathbf{M}_{v}=\sum_{\substack{w\text{ is a packed
word;}\\\operatorname*{pack}\left(  w\left[  :\ell\right]  \right)
=u;\ \operatorname*{pack}\left(  w\left[  \ell:\right]  \right)
=v}}\mathbf{M}_{w}, \label{eq.WQSym.prod}%
\end{equation}
where $\ell$ is the length of $u$, and where we use the notation $w\left[
:\ell\right]  $ for the word formed by the first $\ell$ letters of $w$ and we
use the notation $w\left[  \ell:\right]  $ for the word formed by the
remaining letters of $w$. This equality (which should be considered a
noncommutative analogue of (\ref{eq.rmk.smaps.2}), and can be proven
similarly) makes it possible to give an alternative definition of
$\operatorname*{WQSym}$, by defining $\operatorname*{WQSym}$ as the free
$\mathbf{k}$-module with basis $\left(  \mathbf{M}_{u}\right)  _{u\text{ is a
packed word}}$ and defining multiplication using (\ref{eq.WQSym.prod}). This
is precisely the approach taken in \cite[Section 1.1]{FM}.

The Hopf algebra $\operatorname*{WQSym}$ has also appeared under the name
$\operatorname*{NCQSym}$ (\textquotedblleft quasisymmetric functions in
noncommuting variables\textquotedblright) in \cite[Section 5.2]{BerZab} and
other sources.

We now define five binary operations $\left.  \prec\right.  $, $\circ$,
$\left.  \succ\right.  $, $\bel $, and $\tvi $ on $\mathbf{k}\left\langle
\left\langle \mathbf{X}\right\rangle \right\rangle $.

\begin{definition}
\label{def.five-ops}\textbf{(a)} We define a binary operation $\left.
\prec\right.  :\mathbf{k}\left\langle \left\langle \mathbf{X}\right\rangle
\right\rangle \times\mathbf{k}\left\langle \left\langle \mathbf{X}%
\right\rangle \right\rangle \rightarrow\mathbf{k}\left\langle \left\langle
\mathbf{X}\right\rangle \right\rangle $ (written in infix notation) by the
requirements that it be $\mathbf{k}$-bilinear and continuous with respect to
the topology on $\mathbf{k}\left\langle \left\langle \mathbf{X}\right\rangle
\right\rangle $ and that it satisfy%
\[
u\left.  \prec\right.  v=
\begin{cases}
uv, & \text{if }\min\left(  \operatorname*{Supp}u\right)  <\min\left(
\operatorname*{Supp}v\right)  ;\\
0, & \text{if }\min\left(  \operatorname*{Supp}u\right)  \geq\min\left(
\operatorname*{Supp}v\right)
\end{cases}
\]
for any two words $u$ and $v$.

\textbf{(b)} We define a binary operation $\circ:\mathbf{k}\left\langle
\left\langle \mathbf{X}\right\rangle \right\rangle \times\mathbf{k}%
\left\langle \left\langle \mathbf{X}\right\rangle \right\rangle \rightarrow
\mathbf{k}\left\langle \left\langle \mathbf{X}\right\rangle \right\rangle $
(written in infix notation) by the requirements that it be $\mathbf{k}%
$-bilinear and continuous with respect to the topology on $\mathbf{k}%
\left\langle \left\langle \mathbf{X}\right\rangle \right\rangle $ and that it
satisfy%
\[
u\circ v=
\begin{cases}
uv, & \text{if }\min\left(  \operatorname*{Supp}u\right)  =\min\left(
\operatorname*{Supp}v\right)  ;\\
0, & \text{if }\min\left(  \operatorname*{Supp}u\right)  \neq\min\left(
\operatorname*{Supp}v\right)
\end{cases}
\]
for any two words $u$ and $v$.

\textbf{(c)} We define a binary operation $\left.  \succ\right.
:\mathbf{k}\left\langle \left\langle \mathbf{X}\right\rangle \right\rangle
\times\mathbf{k}\left\langle \left\langle \mathbf{X}\right\rangle
\right\rangle \rightarrow\mathbf{k}\left\langle \left\langle \mathbf{X}%
\right\rangle \right\rangle $ (written in infix notation) by the requirements
that it be $\mathbf{k}$-bilinear and continuous with respect to the topology
on $\mathbf{k}\left\langle \left\langle \mathbf{X}\right\rangle \right\rangle
$ and that it satisfy%
\[
u\left.  \succ\right.  v=
\begin{cases}
uv, & \text{if }\min\left(  \operatorname*{Supp}u\right)  >\min\left(
\operatorname*{Supp}v\right)  ;\\
0, & \text{if }\min\left(  \operatorname*{Supp}u\right)  \leq\min\left(
\operatorname*{Supp}v\right)
\end{cases}
\]
for any two words $u$ and $v$.

\textbf{(d)} We define a binary operation $\bel :\mathbf{k}\left\langle
\left\langle \mathbf{X}\right\rangle \right\rangle \times\mathbf{k}%
\left\langle \left\langle \mathbf{X}\right\rangle \right\rangle \rightarrow
\mathbf{k}\left\langle \left\langle \mathbf{X}\right\rangle \right\rangle $
(written in infix notation) by the requirements that it be $\mathbf{k}%
$-bilinear and continuous with respect to the topology on $\mathbf{k}%
\left\langle \left\langle \mathbf{X}\right\rangle \right\rangle $ and that it
satisfy%
\[
u \bel v=
\begin{cases}
uv, & \text{if }\max\left(  \operatorname*{Supp}u\right)  \leq\min\left(
\operatorname*{Supp}v\right)  ;\\
0, & \text{if }\max\left(  \operatorname*{Supp}u\right)  >\min\left(
\operatorname*{Supp}v\right)
\end{cases}
\]
for any two words $u$ and $v$.

\textbf{(e)} We define a binary operation $\tvi :\mathbf{k}\left\langle
\left\langle \mathbf{X}\right\rangle \right\rangle \times\mathbf{k}%
\left\langle \left\langle \mathbf{X}\right\rangle \right\rangle \rightarrow
\mathbf{k}\left\langle \left\langle \mathbf{X}\right\rangle \right\rangle $
(written in infix notation) by the requirements that it be $\mathbf{k}%
$-bilinear and continuous with respect to the topology on $\mathbf{k}%
\left\langle \left\langle \mathbf{X}\right\rangle \right\rangle $ and that it
satisfy%
\[
u \tvi v=
\begin{cases}
uv, & \text{if }\max\left(  \operatorname*{Supp}u\right)  <\min\left(
\operatorname*{Supp}v\right)  ;\\
0, & \text{if }\max\left(  \operatorname*{Supp}u\right)  \geq\min\left(
\operatorname*{Supp}v\right)
\end{cases}
\]
for any two words $u$ and $v$.
\end{definition}

The first three of these five operations are closely related to those defined
by Novelli and Thibon in \cite{NoThi05}; the main difference is the use of
minima instead of maxima in our definitions.

The operations $\left.  \prec\right.  $, $\bel $ and $\tvi $ on
$\operatorname*{WQSym}$ lift the operations $\left.  \prec\right.  $, $\bel $
and $\tvi $ on $\operatorname*{QSym}$. More precisely, any $a\in
\mathbf{k}\left\langle \left\langle \mathbf{X}\right\rangle \right\rangle $
and $b\in\mathbf{k}\left\langle \left\langle \mathbf{X}\right\rangle
\right\rangle $ satisfy%
\begin{align*}
\pi\left(  a\right)  \left.  \prec\right.  \pi\left(  b\right)   &
=\pi\left(  a\left.  \prec\right.  b\right)  =\pi\left(  b\left.
\succ\right.  a\right)  ;\\
\pi\left(  a\right)  \bel \pi\left(  b\right)   &  =\pi\left(  a
\bel b\right)  ;\\
\pi\left(  a\right)  \tvi \pi\left(  b\right)   &  =\pi\left(  a
\tvi b\right)
\end{align*}
(and similar formulas would hold for $\circ$ and $\left.  \succ\right.  $ had
we bothered to define such operations on $\operatorname*{QSym}$). Also, using
the operation $\left.  \succeq\right.  $ defined in Remark \ref{rmk.dendri},
we have%
\[
\pi\left(  a\right)  \left.  \succeq\right.  \pi\left(  b\right)  =\pi\left(
a\left.  \succ\right.  b+a\circ b\right)  \ \ \ \ \ \ \ \ \ \ \text{for any
}a\in\mathbf{k}\left\langle \left\langle \mathbf{X}\right\rangle \right\rangle
\text{ and }b\in\mathbf{k}\left\langle \left\langle \mathbf{X}\right\rangle
\right\rangle .
\]

We now have the following analogue of Proposition \ref{prop.QSym.closed}:

\begin{proposition}
\label{prop.WQSym.closed}Every $a\in\operatorname*{WQSym}$ and $b\in
\operatorname*{WQSym}$ satisfy $a\left.  \prec\right.  b\in
\operatorname*{WQSym}$, $a\circ b\in\operatorname*{WQSym}$, $a\left.
\succ\right.  b\in\operatorname*{WQSym}$, $a \bel b\in\operatorname*{WQSym}$
and $a \tvi b\in\operatorname*{WQSym}$.
\end{proposition}

The proof of Proposition \ref{prop.WQSym.closed} is easier than that of
Proposition \ref{prop.QSym.closed}; we omit it here. In analogy to Remark
\ref{rmk.QSym.closed} and to (\ref{eq.WQSym.prod}), let us give explicit
formulas for these five operations on the basis $\left(  \mathbf{M}%
_{u}\right)  _{u\text{ is a packed word}}$ of $\operatorname*{WQSym}$:

\begin{remark}
\label{rmk.WQSym.closed}Let $u$ and $v$ be two packed words. Let $\ell$ be the
length of $u$. Then:

\textbf{(a)} We have%
\[
\mathbf{M}_{u}\left.  \prec\right.  \mathbf{M}_{v}=\sum_{\substack{w\text{ is
a packed word;}\\\operatorname*{pack}\left(  w\left[  :\ell\right]  \right)
=u;\ \operatorname*{pack}\left(  w\left[  \ell:\right]  \right)
=v;\\\min\left(  \operatorname*{Supp}\left(  w\left[  :\ell\right]  \right)
\right)  <\min\left(  \operatorname*{Supp}\left(  w\left[  \ell:\right]
\right)  \right)  }}\mathbf{M}_{w}.
\]

\textbf{(b)} We have%
\[
\mathbf{M}_{u}\circ\mathbf{M}_{v}=\sum_{\substack{w\text{ is a packed
word;}\\\operatorname*{pack}\left(  w\left[  :\ell\right]  \right)
=u;\ \operatorname*{pack}\left(  w\left[  \ell:\right]  \right)
=v;\\\min\left(  \operatorname*{Supp}\left(  w\left[  :\ell\right]  \right)
\right)  =\min\left(  \operatorname*{Supp}\left(  w\left[  \ell:\right]
\right)  \right)  }}\mathbf{M}_{w}.
\]

\textbf{(c)} We have%
\[
\mathbf{M}_{u}\left.  \succ\right.  \mathbf{M}_{v}=\sum_{\substack{w\text{ is
a packed word;}\\\operatorname*{pack}\left(  w\left[  :\ell\right]  \right)
=u;\ \operatorname*{pack}\left(  w\left[  \ell:\right]  \right)
=v;\\\min\left(  \operatorname*{Supp}\left(  w\left[  :\ell\right]  \right)
\right)  >\min\left(  \operatorname*{Supp}\left(  w\left[  \ell:\right]
\right)  \right)  }}\mathbf{M}_{w}.
\]

\textbf{(d)} We have%
\[
\mathbf{M}_{u} \bel \mathbf{M}_{v}=\sum_{\substack{w\text{ is a packed
word;}\\\operatorname*{pack}\left(  w\left[  :\ell\right]  \right)
=u;\ \operatorname*{pack}\left(  w\left[  \ell:\right]  \right)
=v;\\\max\left(  \operatorname*{Supp}\left(  w\left[  :\ell\right]  \right)
\right)  \leq\min\left(  \operatorname*{Supp}\left(  w\left[  \ell:\right]
\right)  \right)  }}\mathbf{M}_{w}.
\]
The sum on the right hand side consists of two addends (unless $u$ or $v$ is
empty), namely $\mathbf{M}_{uv^{+h-1}}$ and $\mathbf{M}_{uv^{+h}}$, where
$h=\max\left(  \operatorname*{Supp}u\right)  $, and where $v^{+j}$ denotes the
word obtained by replacing every letter $X_{k}$ in $v$ by $X_{k+j}$.

\textbf{(e)} We have%
\[
\mathbf{M}_{u} \tvi \mathbf{M}_{v}=\sum_{\substack{w\text{ is a packed
word;}\\\operatorname*{pack}\left(  w\left[  :\ell\right]  \right)
=u;\ \operatorname*{pack}\left(  w\left[  \ell:\right]  \right)
=v;\\\max\left(  \operatorname*{Supp}\left(  w\left[  :\ell\right]  \right)
\right)  <\min\left(  \operatorname*{Supp}\left(  w\left[  \ell:\right]
\right)  \right)  }}\mathbf{M}_{w}.
\]
The sum on the right hand side consists of one addend only, namely
$\mathbf{M}_{uv^{+h}}$.
\end{remark}

Let us now move on to the combinatorial Hopf algebra $\operatorname*{FQSym}$,
which is known as the \textit{Malvenuto-Reutenauer Hopf algebra} or the
\textit{Hopf algebra of free quasi-symmetric functions}. We shall define it as
a Hopf subalgebra of $\operatorname*{WQSym}$. This is not identical to the
definition in \cite[Section 8.1]{Reiner}, but equivalent to it.

For every $n\in\mathbb{N}$, we let $\mathfrak{S}_{n}$ be the symmetric group
on the set $\left\{  1,2,\ldots,n\right\}  $. (This notation is identical with
that in \cite{Reiner}. It has nothing to do with the $\mathfrak{S}_{\alpha}$
from \cite{BBSSZ}.) We let $\mathfrak{S}$ denote the disjoint union
$\bigsqcup_{n\in\mathbb{N}}\mathfrak{S}_{n}$. We identify permutations in
$\mathfrak{S}$ with certain words -- namely, every permutation $\pi
\in\mathfrak{S}$ is identified with the word $\left(  X_{\pi\left(  1\right)
},X_{\pi\left(  2\right)  },\ldots,X_{\pi\left(  n\right)  }\right)  $, where
$n$ is such that $\pi\in\mathfrak{S}_{n}$. The words thus identified with
permutations in $\mathfrak{S}$ are precisely the packed words which do not
have repeated elements.

For every word $w$, we define a word $\operatorname*{std}w\in\mathfrak{S}$ as
follows: Write $w$ in the form $\left(  X_{i_{1}},X_{i_{2}},\ldots,X_{i_{n}%
}\right)  $. Then, $\operatorname*{std}w$ shall be the unique permutation
$\pi\in\mathfrak{S}_{n}$ such that, whenever $u$ and $v$ are two elements of
$\left\{  1,2,\ldots,n\right\}  $ satisfying $u<v$, we have $\left(
\pi\left(  u\right)  <\pi\left(  v\right)  \text{ if and only if }i_{u}\leq
i_{v}\right)  $. Equivalently (and less formally), $\operatorname*{std}w$ is
the word which is obtained by

\begin{itemize}
\item replacing the leftmost smallest letter of $w$ by $X_{1}$, and marking it
as \textquotedblleft processed\textquotedblright;

\item then replacing the leftmost smallest letter of $w$ that is not yet
processed by $X_{2}$, and marking it as \textquotedblleft
processed\textquotedblright;

\item then replacing the leftmost smallest letter of $w$ that is not yet
processed by $X_{3}$, and marking it as \textquotedblleft
processed\textquotedblright;

\item etc., until all letters of $w$ are processed.
\end{itemize}

For instance, $\operatorname*{std}\left(  X_{3}X_{5}X_{2}X_{3}X_{2}%
X_{3}\right)  =X_{3}X_{6}X_{1}X_{4}X_{2}X_{5}$ (which, regarded as
permutation, is the permutation written in one-line notation as $\left(
3,6,1,4,2,5\right)  $).

We call $\operatorname*{std}w$ the \textit{standardization} of $w$.

Now, for every $\sigma\in\mathfrak{S}$, we define an element $\mathbf{G}%
_{\sigma}\in\operatorname*{WQSym}$ by%
\[
\mathbf{G}_{\sigma}=\sum_{\substack{w\text{ is a packed word;}%
\\\operatorname*{std}w=\sigma}}\mathbf{M}_{w}=\sum_{\substack{w\in
\operatorname*{Wrd};\\\operatorname*{std}w=\sigma}}w.
\]
(The second equality sign can easily be checked.) Then, the $\mathbf{k}%
$-submodule of $\operatorname*{WQSym}$ spanned by $\left(  \mathbf{G}_{\sigma
}\right)  _{\sigma\in\mathfrak{S}}$ turns out to be a Hopf subalgebra, with
basis $\left(  \mathbf{G}_{\sigma}\right)  _{\sigma\in\mathfrak{S}}$. This
Hopf subalgebra is denoted by $\operatorname*{FQSym}$. This definition is not
identical with the one given in \cite[Section 8.1]{Reiner}; however, it gives
an isomorphic Hopf algebra, as our $\mathbf{G}_{\sigma}$ correspond to the
images of the $G_{\sigma}$ introduced in \cite[Section 8.1]{Reiner} under the
embedding $\operatorname*{FQSym}\rightarrow R\left\langle \left\{
X_{i}\right\}  _{i\in I}\right\rangle $ also defined therein.

Only two of the five operations $\left.  \prec\right.  $, $\circ$, $\left.
\succ\right.  $, $\bel $, and $\tvi $ defined in Definition \ref{def.five-ops}
can be restricted to binary operations on $\operatorname*{FQSym}$:

\begin{proposition}
\label{prop.FQSym.closed}Every $a\in\operatorname*{FQSym}$ and $b\in
\operatorname*{FQSym}$ satisfy $a\left.  \succ\right.  b\in
\operatorname*{FQSym}$ and $a \bel b\in\operatorname*{FQSym}$.
\end{proposition}

Moreover, we have the following explicit formulas on the basis $\left(
\mathbf{G}_{\sigma}\right)  _{\sigma\in\mathfrak{S}}$:

\begin{remark}
\label{rmk.FQSym.closed}Let $\sigma\in\mathfrak{S}$ and $\tau\in\mathfrak{S}$.
Let $\ell$ be the length of $\sigma$ (so that $\sigma\in\mathfrak{S}_{\ell}$).

\textbf{(a)} We have%
\[
\mathbf{G}_{\sigma}\left.  \succ\right.  \mathbf{G}_{\tau}=\sum_{\substack{\pi
\in\mathfrak{S};\\\operatorname*{std}\left(  \pi\left[  :\ell\right]  \right)
=\sigma;\ \operatorname*{std}\left(  \pi\left[  \ell:\right]  \right)
=\tau;\\\min\left(  \operatorname*{Supp}\left(  \pi\left[  :\ell\right]
\right)  \right)  >\min\left(  \operatorname*{Supp}\left(  \pi\left[
\ell:\right]  \right)  \right)  }}\mathbf{G}_{\pi}.
\]

\textbf{(b)} We have%
\[
\mathbf{G}_{\sigma} \bel \mathbf{G}_{\tau}=\sum_{\substack{\pi\in
\mathfrak{S};\\\operatorname*{std}\left(  \pi\left[  :\ell\right]  \right)
=\sigma;\ \operatorname*{std}\left(  \pi\left[  \ell:\right]  \right)
=\tau;\\\max\left(  \operatorname*{Supp}\left(  \pi\left[  :\ell\right]
\right)  \right)  \leq\min\left(  \operatorname*{Supp}\left(  \pi\left[
\ell:\right]  \right)  \right)  }}\mathbf{G}_{\pi}.
\]
The sum on the right hand side consists of one addend only, namely
$\mathbf{G}_{\sigma\tau^{+\ell}}$.
\end{remark}

The statements of Remark \ref{rmk.FQSym.closed} can be easily derived from
Remark \ref{rmk.WQSym.closed}. The proof for \textbf{(a)} rests on the
following simple observations:

\begin{itemize}
\item Every word $w$ satisfies $\operatorname*{std}\left(
\operatorname*{pack}w\right)  =\operatorname*{std}w$.

\item Every $n\in\mathbb{N}$, every word $w$ of length $n$ and every $\ell
\in\left\{  0,1,\ldots,n\right\}  $ satisfy
\[
\operatorname*{std}\left(  \left(  \operatorname*{std}w\right)  \left[
:\ell\right]  \right)  =\operatorname*{std}\left(  w\left[  :\ell\right]
\right)  \ \ \ \ \ \ \ \ \ \ \text{and}\ \ \ \ \ \ \ \ \ \ \operatorname*{std}%
\left(  \left(  \operatorname*{std}w\right)  \left[  \ell:\right]  \right)
=\operatorname*{std}\left(  w\left[  \ell:\right]  \right)  .
\]

\item Every $n\in\mathbb{N}$, every word $w$ of length $n$ and every $\ell
\in\left\{  0,1,\ldots,n\right\}  $ satisfy the equivalence%
\begin{align*}
&  \ \left(  \min\left(  \operatorname*{Supp}\left(  w\left[  :\ell\right]
\right)  \right)  >\min\left(  \operatorname*{Supp}\left(  w\left[
\ell:\right]  \right)  \right)  \right) \\
&  \Longleftrightarrow\ \left(  \min\left(  \operatorname*{Supp}\left(
\left(  \operatorname*{std}w\right)  \left[  :\ell\right]  \right)  \right)
>\min\left(  \operatorname*{Supp}\left(  \left(  \operatorname*{std}w\right)
\left[  \ell:\right]  \right)  \right)  \right)  .
\end{align*}

\end{itemize}

The third of these three observations would fail if the greater sign were to
be replaced by a smaller sign; this is essentially why $\operatorname*{FQSym}%
\subseteq\operatorname*{WQSym}$ is not closed under $\left.  \prec\right.  $.

The operation $\left.  \succ\right.  $ on $\operatorname*{FQSym}$ defined
above is closely related to the operation $\left.  \succ\right.  $ on
$\operatorname*{FQSym}$ introduced by Foissy in \cite[Section 4.2]{Foissy07}.
Indeed, the latter differs from the former in the use of $\max$ instead of
$\min$.

\section{\label{sect.epilogue}Epilogue}

We have introduced five binary operations $\left.  \prec\right.  $, $\circ$,
$\left.  \succ\right.  $, $\bel $, and $\tvi $ on $\mathbf{k}\left[  \left[
x_{1},x_{2},x_{3},\ldots\right]  \right]  $ and their restrictions to
$\operatorname*{QSym}$; we have further introduced five analogous operations
on $\mathbf{k}\left\langle \left\langle \mathbf{X}\right\rangle \right\rangle
$ and their restrictions to $\operatorname*{WQSym}$ (as well as the
restrictions of two of them to $\operatorname*{FQSym}$). We have used these
operations (specifically, $\left.  \prec\right.  $ and $\bel $) to prove a
formula (Corollary \ref{cor.zabrocki}) for the dual immaculate functions
$\mathfrak{S}_{\alpha}^{\ast}$. Along the way, we have found that the
$\mathfrak{S}_{\alpha}^{\ast}$ can be obtained by repeated application of the
operation $\left.  \prec\right.  $ (Corollary \ref{cor.dualImm.dend.explicit}%
). A similar (but much more obvious) result can be obtained for the
fundamental quasisymmetric functions: For every $\alpha=\left(  \alpha
_{1},\alpha_{2},\ldots,\alpha_{\ell}\right)  \in\operatorname*{Comp}$, we have%
\[
F_{\alpha}=h_{\alpha_{1}} \tvi h_{\alpha_{2}} \tvi \cdots\tvi h_{\alpha_{\ell
}} \tvi 1
\]
(we do not use parentheses here, since $\tvi $ is associative). This shows
that the $\mathbf{k}$-algebra $\left(  \operatorname*{QSym}, \tvi \right)  $
is free. Moreover,%
\[
F_{\omega\left(  \alpha\right)  }=e_{\alpha_{\ell}} \bel e_{\alpha_{\ell-1}}
\bel \cdots\bel e_{\alpha_{1}} \bel 1,
\]
where $e_{m}$ stands for the $m$-th elementary symmetric function; thus, the
$\mathbf{k}$-algebra $\left(  \operatorname*{QSym}, \bel \right)  $ is also
free.\footnote{We owe these two observations to the referee.} (Incidentally,
this shows that $S\left(  a \tvi b\right)  =S\left(  b\right)  \bel S\left(
a\right)  $ for any $a,b\in\operatorname*{QSym}$. But this does not hold for
$a,b\in\operatorname*{WQSym}$.)

One might wonder what \textquotedblleft functions\textquotedblright\ can be
similarly constructed using the operations $\left.  \prec\right.  $, $\circ$,
$\left.  \succ\right.  $, $\bel $, and $\tvi $ in $\operatorname*{WQSym}$,
using the noncommutative analogues $H_{m}=\sum_{i_{1}\leq i_{2}\leq\cdots\leq
i_{m}}X_{i_{1}}X_{i_{2}}\cdots X_{i_{m}}=\mathbf{G}_{\left(  1,2,\ldots
,m\right)  }$ and $E_{m}=\sum_{i_{1}>i_{2}>\cdots>i_{m}}X_{i_{1}}X_{i_{2}%
}\cdots X_{i_{m}}=\mathbf{G}_{\left(  m,m-1,\ldots,1\right)  }$ of $h_{m}$ and
$e_{m}$. (These analogues actually live in $\operatorname*{NSym}$, where
$\operatorname*{NSym}$ is embedded into $\operatorname*{FQSym}$ as in
\cite[Corollary 8.1.14(b)]{Reiner}; but the operations do not preserve
$\operatorname*{NSym}$, and only two of them preserve $\operatorname*{FQSym}%
$.) However, it seems somewhat tricky to ask the right questions here; for
instance, the $\mathbf{k}$-linear span of the $\left.  \succ\right.  $-closure
of $\left\{  H_{m}\ \mid\ m\geq0\right\}  $ is not a $\mathbf{k}$-subalgebra
of $\operatorname*{FQSym}$ (since $H_{2}H_{1}$ is not a $\mathbf{k}$-linear
combination of $H_{3}$, $H_{1}\left.  \succ\right.  \left(  H_{1}\left.
\succ\right.  H_{1}\right)  $, $\left(  H_{1}\left.  \succ\right.
H_{1}\right)  \left.  \succ\right.  H_{1}$, $H_{1}\left.  \succ\right.  H_{2}$
and $H_{2}\left.  \succ\right.  H_{1}$).

On the other hand, one might also try to write down the set of identities
satisfied by the operations $\cdot$, $\left.  \prec\right.  $, $\circ$,
$\left.  \succeq\right.  $, $\bel $ and $\tvi $ on the various spaces
($\mathbf{k}\left[  \left[  x_{1},x_{2},x_{3},\ldots\right]  \right]  $,
$\operatorname*{QSym}$, $\mathbf{k}\left\langle \left\langle \mathbf{X}%
\right\rangle \right\rangle $, $\operatorname*{WQSym}$ and
$\operatorname*{FQSym}$), or by subsets of these operations; these identities
could then be used to define new operads, i.e., algebraic structures
comprising a $\mathbf{k}$-module and some operations on it that imitate (some
of) the operations $\cdot$, $\left.  \prec\right.  $, $\circ$, $\left.
\succeq\right.  $, $\bel $ and $\tvi $. For instance, apart from being
associative, the operations $\bel $ and $\tvi $ on $\mathbf{k}\left\langle
\left\langle \mathbf{X}\right\rangle \right\rangle $ satisfy the identity%
\begin{equation}
\left(  a \bel b\right)  \tvi c+\left(  a \tvi b\right)  \bel c=a \bel \left(
b \tvi c\right)  +a \tvi \left(  b \bel c\right)  \label{eq.epilogue.bel-tvi}%
\end{equation}
for all $a,b,c\in\mathbf{k}\left\langle \left\langle \mathbf{X}\right\rangle
\right\rangle $. This follows from the (easily verified) identities%
\begin{align}
\left(  a \bel b\right)  \tvi c-a \bel \left(  b \tvi c\right)   &
=\varepsilon\left(  b\right)  \left(  a \tvi c-a \bel c\right)
;\label{eq.epilogue.bel-tvi-1a}\\
\left(  a \tvi b\right)  \bel c-a \tvi \left(  b \bel c\right)   &
=\varepsilon\left(  b\right)  \left(  a \bel c-a \tvi c\right)  ,
\label{eq.epilogue.bel-tvi-1b}%
\end{align}
where $\varepsilon:\mathbf{k}\left\langle \left\langle \mathbf{X}\right\rangle
\right\rangle \rightarrow\mathbf{k}$ is the map which sends every
noncommutative power series to its constant term. The equality
(\ref{eq.epilogue.bel-tvi}) (along with the associativity of $\bel $ and
$\tvi $) makes $\left(  \mathbf{k}\left\langle \left\langle \mathbf{X}%
\right\rangle \right\rangle , \bel , \tvi \right)  $ into what is called an
$As^{\left\langle 2\right\rangle }$\textit{-algebra} (see \cite[p.
39]{Zinbie10}).
Is $\operatorname*{QSym}$ or $\operatorname*{WQSym}$ a free $As^{\left\langle
2\right\rangle }$-algebra?\footnote{\textbf{Update (2026):} No. None of the
$As^{\left\langle 2\right\rangle }$-algebras $\left(  \mathbf{k}\left\langle
\left\langle \mathbf{X}\right\rangle \right\rangle , \bel , \tvi \right)  $,
$\operatorname*{QSym}$ and $\operatorname*{WQSym}$ is free.
\par
The reason for this is that all three of these $As^{\left\langle
2\right\rangle }$-algebras satisfy the additional ``quasi-identity'' saying
that for any $a, b, c$ in the algebra, we have
\begin{equation}
\left(  a \bel b\right)  \tvi c-a \bel \left(  b \tvi c\right)  \in\mathbf{k}
\left(  a \tvi c-a \bel c\right)  \label{eq.epilogue.bel-tvi-2}%
\end{equation}
(which follows from \eqref{eq.epilogue.bel-tvi-1a}), but a free
$As^{\left\langle 2\right\rangle }$-algebra (with at least one generator) does
not. The latter can be shown as follows:
\par
First we introduce a way to construct many $As^{\left\langle 2\right\rangle }%
$-algebras: Let $S$ be a $\mathbf{k}$-algebra, and let $I$ be an $\left(
S,S\right)  $-bimodule with a $\mathbf{k}$-bilinear associative multiplication
that is furthermore associative with respect to the left and right actions of
$S$ (that is, satisfies $s\left(  ij\right)  = \left(  si\right)  j$ and
$\left(  ij\right)  s = i\left(  js\right)  $ and $\left(  is\right)  j =
i\left(  sj\right)  $ for all $s \in S$ and $i, j \in I$). (For example, $I$
can be an isomorphic copy of $S$, with the multiplication of $S$ being reused
as both $S$-module structures and as multiplication in $I$.) Then, on the
$\mathbf{k}$-module $R := S \oplus I$, we define two binary operations $\bel$
and $\tvi$ by
\begin{align*}
\left(  s, i\right)  \bel \left(  t, j\right)   &  = \left(  st, sj +
it\right)  \qquad\text{for all } s, t \in S \text{ and } i, j \in I;\\
\left(  s, i\right)  \tvi \left(  t, j\right)   &  = \left(  st, sj + it +
ij\right)  \qquad\text{for all } s, t \in S \text{ and } i, j \in I.
\end{align*}
Then, $R$ becomes an $As^{\left\langle 2\right\rangle }$-algebra with these
two operations. It is easy to see that \eqref{eq.epilogue.bel-tvi-2} does not
hold in general for such an algebra, even if it is just generated by $1$
element. (For instance, if $S$ is the polynomial ring $\mathbf{k}\left[
x\right]  $ and $I$ is a copy of $S$, then the $As^{\left\langle
2\right\rangle }$-subalgebra of $R = S \oplus I$ generated by the single
element $\left(  x,1\right)  $ can easily be shown to contain any $\left(
x^{i}, x^{j}\right)  $ with $i>0$, and then we can see that
\eqref{eq.epilogue.bel-tvi-2} is violated already for $a = b = c = \left(
x,1\right)  $.) Hence, \eqref{eq.epilogue.bel-tvi-2} cannot hold for a free
$As^{\left\langle 2\right\rangle }$-algebra unless it is free on $0$
generators, i.e., trivial. Hence, if any of the $As^{\left\langle
2\right\rangle }$-algebras $\left(  \mathbf{k}\left\langle \left\langle
\mathbf{X}\right\rangle \right\rangle , \bel , \tvi \right)  $,
$\operatorname*{QSym}$ and $\operatorname*{WQSym}$ were free, then it would be
trivial, which is absurd.} What if we add the existence of a common neutral
element for the operations $\bel$ and $\tvi$ to the axioms of this
operad?\footnote{\textbf{Update (2026):} The answer is still ``no'', for the
same reason as in the previous footnote. (Note that the $As^{\left\langle
2\right\rangle }$-algebra $R = S \oplus I$ is unital with $\left(  1,0\right)
$ acting as common neutral element for both operations $\bel$ and $\tvi$.)}

The equalities \eqref{eq.epilogue.bel-tvi-1a} and
\eqref{eq.epilogue.bel-tvi-1b} also show that the positive part of $\left(
\mathbf{k}\left\langle \left\langle \mathbf{X}\right\rangle \right\rangle ,
\bel , \tvi \right)  $ (that is, the $\mathbf{k}$-submodule consisting of the
series with constant term $0$) is an \emph{$As^{\left(  2\right)  }$-algebra}
as defined in \cite[p. 38]{Zinbie10}. Here, again, one can ask about the freeness:

\begin{question}
Is the positive part of $\left(  \mathbf{k}\left\langle \left\langle
\mathbf{X}\right\rangle \right\rangle , \bel , \tvi \right)  $ a free
$As^{\left(  2\right)  }$-algebra? What about the positive parts of
$\operatorname*{QSym}$ and $\operatorname*{WQSym}$?
\end{question}

It is not hard to see that the answer is positive for the positive part of
$\operatorname*{QSym}$: it is a free $As^{\left(  2\right)  }$-algebra on one
generator, which is the quasisymmetric function $F_{\left(  1\right)  }$.
(This follows easily from the formulas $F_{\alpha} \bel F_{\beta}%
=F_{\alpha\odot\beta}$ and $F_{\alpha} \tvi F_{\beta}=F_{\left[  \alpha
,\beta\right]  }$, which hold for any two nonempty compositions $\alpha$ and
$\beta$.)

\end{document}